\documentclass[reqno,draft,intlimits]{amsart}
\usepackage{amsmath,pst-all,amsfonts,graphicx,pstricks,graphics,color}
\usepackage{amssymb}
\usepackage[active]{srcltx}

\usepackage[reqno]{amsmath}
\usepackage{amssymb}

\usepackage{wrapfig}
\usepackage[all]{xy}

\makeatletter
\long\def\@makecaption#1#2{%
  \vskip3pt
  \sbox\@tempboxa{\small#1. #2}%
    \ifdim \wd\@tempboxa >\hsize
    \small#1. #2\par
  \else
    \global \@minipagefalse
    \hb@xt@\hsize{\hfil\box\@tempboxa\hfil}%
  \fi
  \vskip0pt}
\makeatother

\def\figone{%
\begin{picture}(50,110) \multiput(10,123)(7,0){10}{-} \put(78,123){0}
\multiput(10,80)(7,0){10}{-} \put(78,80){1} \multiput(10,40)(7,0){10}{-}
\put(78,40){2}  \multiput(10,0)(7,0){10}{-} \put(78,0){3}
\put(35,123){\circle*{3}} \put(35,123){\line(-2,-3){27}}
\put(35,123){\line(2,-5){33}} \put(35,123){\line(0,-1){40}}
\put(9,83){\circle*{3}} \put(9,83){\line(0,-1){40}} \put(9,83){\line(1,-2){40}}
\put(35,83){\circle*{3}}
\put(51,83){\circle*{3}} \put(51,83){\line(0,-1){40}} \put(51,83){\line(-1,-6){7}}
\put(9,43){\circle*{3}}
\put(29,43){\circle*{3}} \put(29,43){\line(0,-1){40}}
\put(29,43){\line(-1,-2){20}} \put(44,43){\circle*{3}}
\put(51,43){\circle*{3}}
\put(67,43){\circle*{3}} \put(67,43){\line(0,-1){40}} \put(67,43){\line(-1,-6){7}}
\put(9,3){\circle*{3}} \put(29,3){\circle*{3}} \put(49,3){\circle*{3}}
\put(60,3){\circle*{3}} \put(67,3){\circle*{3}}
\end{picture}}

\def\figvert{%
\begin{picture}(1,3)
\put(1,0){\line(0,1){3}}
\end{picture}}

\def\vertpunkt{%
\begin{picture}(10,10)
\multiput(0,0)(0,7){12}{\figvert}
\end{picture}}

\def\treea{%
\begin{picture}(10,10)
\put(0,0){\circle*{3}}
\put(0,0){\line(5,3){60}}
\put(60,35){\circle*{3}}
\put(0,0){\line(6,-1){60}}
\put(60,-10){\circle*{3}}
\put(0,0){\line(1,0){40}}
\put(0,0){\line(6,1){40}}
\put(0,0){\line(5,2){40}}
\put(45,3){\vdots}
\end{picture}}
\def\figtwo{%
\begin{picture}(365,90)
\multiput(0,0)(60,0){7}{\vertpunkt}
\put(0,0){\circle*{3}}
\put(0,0){\line(1,0){180}}
\put(0,0){\line(3,2){60}}
\put(60,40){\treea}
\put(60,0){\circle*{3}}
\put(60,0){\line(5,2){60}}
\put(120,24){\treea}
\put(180,0){\circle*{3}}
\put(180,0){\line(5,2){40}}
\multiput(180,0)(4,0){15}{.}
\put(240,0){\circle*{3}}
\put(240,0){\line(5,3){60}}
\put(300,36){\treea}
\end{picture}}

\def\figthree{%
\begin{picture}(130,40)
\put(43,25){\circle*{3}} \put(17,13){\circle*{3}} \put(69,13){\circle*{3}} \put(43,25){\line(-2,-1){27}}
\put(43,25){\line(2,-1){27}} \put(13,19){{\text{\small A}}} \put(66,19){{\text{\small B}}}
\put(43,29){{\text{\small C}}} \put(19,0){$\lfloor A,B|C\rceil$} \put(19,0){$\lfloor A,B|C\rceil$}
\put(89,13){\circle*{3}} \put(115,25){\circle*{3}} \put(89,13){\vector(2,1){25}}
\put(85,19){{\text{\small D}}} \put(115,29){{\text{\small E}}} \put(91,0){$\lfloor D|E\rceil$}
\end{picture}}

\newtheorem{thm}{Theorem}[section] 
 
\newtheorem{proc}{Proposition}[section] 
        
\newtheorem{cor}{Corollary}[section]  
          
\newtheorem{lem}{Lemma}[section]        
               
\newtheorem{defi}{Definition}[section]
\newtheorem{rmk}{Remark}[section]
\newtheorem{ex} {Example}[section]

\newtheorem*{conj}{Conjecture}

\renewcommand{\a}{\alpha}

\def\b{\beta}

\def\e{\varepsilon} 
\def\i{\iota}

\def\ls{{\rm [\![}}
\def\rs{{\rm ]\!]}}
  
 \DeclareMathOperator{\id}{id} 
\DeclareMathOperator{\Div}{div} \DeclareMathOperator{\Ker}{Ker} 
  
 \DeclareMathOperator{\im}{Im} 
 \DeclareMathOperator{\const}{const}

\DeclareMathOperator{\Hom}{Hom} \DeclareMathOperator{\defin}{def} 
  
\DeclareMathOperator{\End}{End}

\newcommand{\tr}{\mathrm{tr}}
\newcommand{\ad}{\mathrm{ad}}
\newcommand{\df}{\stackrel{\mathrm{def}}{=}}

\newcommand{\pit}{\pitchfork}

\newcommand{\A}{\mathcal A}
\newcommand{\V}{\mathcal V}

\newcommand{\R}{\mathbb R}
\newcommand{\Z}{\mathbb Z}
\newcommand{\C}{\mathbb C}
\newcommand{\Q}{\mathbb Q}
\newcommand{\X}{\mathbb X}
\newcommand{\dF}{\mathbb F}

\newcommand{\gk}{\boldsymbol{k}}

\newcommand{\bg}{\boldsymbol{\gamma}}

\newcommand{\gG}{\boldsymbol{\mathfrak g}}
\newcommand{\gH}{\boldsymbol{\mathfrak h}}
\newcommand{\gS}{\boldsymbol{\mathfrak S}}
\newcommand{\gs}{\boldsymbol{\mathfrak s}}
\newcommand{\gA}{\boldsymbol{\mathfrak a}}

\newcommand{\gR}{\boldsymbol{\mathfrak r}}
\newcommand{\gl}{\boldsymbol{\mathfrak l}}
\newcommand{\go}{\boldsymbol{\mathfrak o}}
\newcommand{\gu}{\boldsymbol{\mathfrak u}}
\newcommand{\gp}{\boldsymbol{\mathfrak p}}
\newcommand{\DEF}{\overset\defin =}

\DeclareFontFamily{OT1}{wncyi}{} \DeclareFontShape{OT1}{wncyi}{m}{it}{
   <5> <6> <7> <8> <9> gen * wncyi
   <10> <10.95> <12> <14.4> <17.28> <20.74> <24.88> wncyi10
  }{}
\DeclareSymbolFont{cyrletters}{OT1}{wncyi}{m}{it} 
\DeclareSymbolFontAlphabet{\cyrmath}{cyrletters} 
\DeclareMathSymbol{\rE}{\cyrmath}{cyrletters}{003} 
\DeclareMathSymbol{\rD}{\cyrmath}{cyrletters}{068} 
\DeclareMathSymbol{\rG}{\cyrmath}{cyrletters}{017} 
\DeclareMathSymbol{\rI}{\cyrmath}{cyrletters}{073} 
\DeclareMathSymbol{\rL}{\cyrmath}{cyrletters}{076} 
\DeclareMathSymbol{\rZ}{\cyrmath}{cyrletters}{090}

\newdimen\theight
\def \refright#1{%
             \vadjust{\setbox0=\hbox{\quad\vtop{\hsize5cm\bf\noindent #1}}%
             \theight=\ht0
             \advance\theight by \dp0    \advance\theight by \lineskip
             \kern -\theight \vbox to \theight{\rightline{\rlap{\box0}}%
             \vss}%
             }}%

\begin{document}

\newcommand{\PsP}{Poisson structure $P$}
\def\ldb{{\rm [\![}}
\def\rdb{{\rm ]\!]}}
\def\a{\alpha}
\def\o{\omega}
\def\r{\rho}
\def\th{\theta}
\def\s{\sigma}
\def\t{\tau}
\def\be{\begin{equation}}
\def\ee{\end{equation}}
\def\bea{\begin{eqnarray}}
\def\eea{\end{eqnarray}}
\def\nn{\nonumber}

\begin{center}
{\LARGE \bf  ASSEMBLING LIE ALGEBRAS FROM LIEONS.} \footnote{.}
\end{center}
\vspace{.5cm}
\begin{center}

{\Large A.M.Vinogradov}
\end{center}

\bigskip

\begin{center}
$^{1}$
Levi-Civita Institute, 83050 Santo Stefano del Sole (AV), Italia.
\end{center}
\smallskip

\vspace{1.0cm}

\noindent {\bf Abstract.}
If a Lie algebra structures $\gG$ on a vector space is the sum of a family of mutually compatible
Lie algebra structures $\gG_i$, we say that $\gG$ is \emph{simply assembled} from $\gG_s$'s. By 
repeating this procedure several times one gets a family of Lie algebras  \emph{assembled} from 
$\gG_s$'s. The central result of this paper is that any finite dimensional Lie algebra over $\R$ or
$\C$ can be assembled from two constituents, called $\between$- and $\pitchfork$-\emph{lieons}.
A lieon is the direct sum of an abelian Lie algebra with a 2-dmensional nonabelian Lie algebra or 
with the 3-dimensional Heisenberg algebra.

Some techniques of disassembling Lie algebras are introduced and various results concerning assembling-disassembling procedures are obtained. In particular, it is shown how classical Lie 
algebras are assembled from lieons and is obtained the complete list of Lie algebras, which 
can be simply assembled from lieons.

\vspace{5.0cm}

\pagebreak

\tableofcontents

\section{Introduction}
According to the modern view, matter is of a compound nature. The constituents, elementary 
particles,  are characterized by their symmetry properties. Since these properties are formalized 
in terms of Lie algebras, one may hypothesize that the compound nature of matter is somehow 
mirrored in the structure of symmetry algebras. This and some other similar considerations lead
to a suspicion that Lie algebras possess, in a sense, compound structure. The study, some results 
of which are presented in this paper, was motivated by this question. The main result we have 
found is that finite dimensional Lie algebras over $\R$ and $\C$ are made from two "elementary 
particles", which we call \emph{lieons}.

Obviously, prior to approaching ``elementary particle theory" of ``Lie matter" the exact meaning 
of ``made from" in the context of Lie algebras should be established. A suggestion of how it could 
be done comes from Poisson geometry. Namely, from the one hand, a Lie algebra is naturally 
interpreted as a linear Poisson structure on its dual. On the other hand, it is natural to think that 
a Poisson structure/bivector $P$ is ``made from" Poisson structures $P_1$ and $P_2$, if 
$P=P_1+P_2$. In such a case  $P_1$ and $P_2$ are called \emph{compatible}. So, by translating this idea into the language of Lie algebras  we get the following.
 
Lie algebra structures $[\cdot,\cdot]_1$ and $[\cdot,\cdot]_2$ on a vector space $V$ are \emph{compatible}, if $[\cdot,\cdot]_1+[\cdot,\cdot]_2$ is  a Lie algebra structure as well. 
If a Lie algebra structure $[\cdot,\cdot]$ is presented in the form
$$
[\cdot,\cdot]=[\cdot,\cdot]_1+\dots+[\cdot,\cdot]_m
$$
with mutually compatible structures $[\cdot,\cdot]_i$'s, we speak of a \emph{disassembling} 
of $[\cdot,\cdot]$, i.e.,  that $[\cdot,\cdot]$ is ``made from" $[\cdot,\cdot]_i$'s. Note that 
compatibility is not a transitive notion. Hence it has sense to go ahead by disassembling all \emph{compounds} $[\cdot,\cdot]_i$. This way we get a 2-step disassembling of $[\cdot,\cdot]$, 
and so on.  The central result of this paper states that any $n$-dimensional Lie algebras can be 
in this sense disassembled  up to  \emph{lieons}
$\between_n$ and $\pitchfork_n$, called $\between$- and $\pitchfork$-lieons, respectively. 
Here $\between_n=\between\oplus\gamma_{n-2}$ and 
$\pitchfork_n=\pitchfork\oplus\gamma_{n-3}$ with $\gamma_m$ standing for the $m$-dimensional
abelian Lie algebra, $\between$ for a non-abelian $2$-dimensional algebra (all such algebras
are isomorphic) and $\pitchfork$ for the $3$-dimensional Heisenberg algebra. For instance, the 
algebra $\gu(2)=\gs\go(3)$ can be disassembled into 3 pieces each of which is equivalent to 
$\pitchfork$. Speculatively, one might interpret this fact by saying that $\gu(2)$, the symmetry 
algebra  of nucleons, is composed of 3 $\pitchfork$-lieons, each of them is the symmetry algebra 
of a hypothetical particle called, say, ``quark", etc.  We, however, do not discuss eventual physical applications in this paper by concentrating only on purely mathematical questions.

The concept of compatible Poisson structures originates in  F.\,Magri's work \cite{Mag} on 
bihamitonian systems, and was subsequently developed and exploited  by many authors in the 
context of integrable systems and Poisson geometry.  However, as far as we know, it was not systematically studied in theory of Lie algebras. Also, it worth mentioning that translation 
of techniques and constructions of Poisson geometry into context of Lie algebras is very useful,
and we exploit this possibility at full scale. This is why some parts of this paper are dedicated to 
necessary elements of Poisson geometry.

The contents, results and organization of the paper are as follows. Generalities concerning 
differential forms, multivector fields and Schouten bracket formalism we need are collected
in section\,\ref{preliminaries}. Here we recall notions of compatibility of Poisson and Lie 
algebra structures and discuss their simplest properties.

Modularity properties of Poisson and Lie algebra structures are considered in section\,\ref{mdlr}.
Here we specify the general compatibility conditions for unimodular Poisson structures and as a 
result prove that the Lie rank of an unimodular Lie algebra is strictly lesser than its dimension. 
Then we show that a Poisson structure $P$ can be disassembled into unimodular and 
completely non-unimodular parts. The Poisson bivector corresponding to the non-unimodular 
part is $P_{\nu}\wedge\Xi=\ls P,\nu\Xi\rs$ where $\Xi$ is the modular vector
field of $P$ and $\Xi(\nu)=1, \,\nu\in C^{\infty}(M)$. This bivector is of rank two (if nontrivial) and 
characterized by the property that its unimodular part is trivial. 

The second part of  section\,\ref{mdlr} is dedicated to the \emph{matching problem}: what are
different (up to equivalence) ways to assemble Poisson structures from given ones. This problem
in full generality seems to be very difficult. By this reason, we restrict ourself to a particular case
of two completely non-unimodular structures.  It turns out that even in this case equivalent classes 
of matchings are labeled by functional parameters (proposition\,\ref{InvariantsMatchings}).

The result and constructions concerning modular disassembling of Poisson structures are then
adopted to Lie algebras in section\,\ref{Lie-modlarity}. In particular, we call \emph{modular} 
Lie algebras whose Poisson bivector is completely non-unimodular and show that by subtracting
from a Lie algebra a suitable modular algebra one gets an unimodular algebra. The structure
of modular Lie algebras is very simple. So, in this sense this result reduces the study of general
Lie algebras to unimodular ones. Here we also discuss compatibility conditions for modular and 
unimodular Lie algebras, and, in particular, show that semisimple and modular Lie algebras are 
incompatible. In the second part of this section the matching problem for modular Lie algebras
is solved. In contrast with general Poisson structures, this problem admits a complete solution.
Essentially, matchings of modular Lie algebras are labeled by representations of 2-dimensional
algebras (see theorem\,\ref{LieMathing}).

Section\,\ref{dis-probl} is central in the paper. Here we prove that any finite dimensional Lie
algebra over an algebraically closed ground field of zero characteristic, or over $\R$ can be 
assembled from lieons (theorems\,\ref{C-dis} and \ref{R-dis}). The proof of this result  naturally 
splits into ``solvable" and ``semisimple" parts, and we show that any solvable algebra over 
arbitrary ground field of characteristic zero can be assembled from lieons 
(proposition\,\ref{dis-solv}). This part of the proof is rather simple. On the contrary, the
``semisimple" one is more delicate. As a preliminary step, we reduce this part to the problem 
of disassembling abelian extensions of simple algebras.

The last problem is, essentially, a question on representations of simple Lie algebras, and as
such could be analyzed on the basis of the well-known description of them. However, such an
approach would be rather cumbersome and hardly instructive, if not to say ``amoral".  Moreover,
the fact of compound structure of Lie algebras seems to be more fundamental than classification 
of simple Lie algebras and their representations and, by this reason, must logically precede it.
By all these reasons we have chosen another approach. It is based on the 
\emph{stripping procedure} (see subsection\,\ref{striptease}), which reduces the problem to representations of \emph{simplest} algebras, i.e., simple algebras without proper nonabelian 
subalgebras. Simplest algebras do not exists over an algebraically closed ground field of zero characteristic. The only simplest algebra over $\R$ is $\gs\go(3)$.  This is, at the end, 
why the assembling-from-lieons theorem was proven in these two cases. In this connection
it is worth mentioning that the reduction to representations of simplest algebras is based on
representations of $\gs\gl(2)$.

Besides the proof of these two main theorems, some useful disassembling techniques are also 
developed  in section\,\ref{dis-probl}. Being mostly interested to some applications to 
differential geometry and theoretical physics we have been initially restricted in this paper to 
ground fields $\R$ and $\C$. But it turned out that many of developed here constructions and 
techniques works well for arbitrary ground fields too. In particular, they indicate possible 
approaches to the  assembling-from-lieons problem for arbitrary fields. They are briefly discussed
at the end of this section. 

In section\,\ref{1st-level} we study first level Lie algebras, i.e., the algebras that can be assembled 
from lieons in one step. With this purpose, we analyze compatibility conditions of two lieons and 
show that these can be expressed in a purely geometrical manner,  namely, in terms of the 
relative position of subspaces carrying centers and derived algebras of lieons in question. One 
of consequences of this fact is that the structure of first level algebras does not depend on
the ground field in the sense that it is described exclusively in terms of the above-mentioned 
subspaces. The results of this sections are used in section\,\ref{coaxial} where we study 
the ``chemistry" of a special class of Lie algebras, called \emph{coaxial}.

How to assemble classical Lie algebras from $\pitchfork$-lieons over arbitrary ground fields 
of characteristic zero is shown in section\,\ref{classical}. A remarkable fact is that this can done
in no more than 4 steps. More exactly, all simple 3-dimensional Lie algebras can be directly
assembled from 3 $\pitchfork$-lieons, i.e.,  in one step. Simple algebras of higher dimensions
require at least 2 steps. For instance, orthogonal algebras can be assembled from 
$\pitchfork$-lieons in 2 steps.

Sections\,\ref{coaxial} and \ref{clusters} are dedicated to a natural question: what are all
possible combinations of mutually compatible lieons. Informally speaking, we ask what are
simple ``molecules", which can be synthesized from lieons. Essentially, this question is
equivalent to the classification problem: what are maximal families of mutually compatible 
lieons. We solve a simplified version of this problem, when only \emph{coaxial}, i.e., naturally 
related with a chosen base lieons are considered. This version is not only interested by itself
but also gives useful hints toward the general ``chemistry" of Lie algebras. The result we have
obtained looks encouraging. Namely, it turned out that maximal families of mutually compatible 
coaxial lieons, called \emph{clusters}, are composed of \emph{structural groups} surrounded by \emph{casings} and connected by \emph{connectives} like in the usual chemistry. 
 
The distribution of the material in these two sections is such that in the first of them we introduce
basic techniques and necessary terminology, and solve the problem for clusters    
composed only of $\pitchfork$-lions, or only of $\between$-lieons. In the second one we
describe general clusters and on this basis describe the structure of coaxial Lie algebras.
In particular, it turns out that the semisimple part of a coaxial algebra consists of 3-dimensional
simple algebras, and the derived series of its solvable part is of length $\leq 3$. Here we also
give some examples of infinite-dimensional Lie algebras assembled from lions. 

In conclusive section\,\ref{problems} we briefly discuss some problems and perspectives of
the theory we have started in this paper.

\section{Preliminaries}\label{preliminaries}
In this section we collect necessary for the sequel facts concerning the calculus of multivectors and differential forms, Poisson geometry, compatibility of Poisson and Lie algebra structures, etc, and fix the notation. More details concerning  material reported in this section the reader will find in \cite{CabVin, V90}. Everything in this article is assumed to be smooth.

\subsection{Multivectors and differential forms} We use $M$ for an $n$--dimensional manifold and

\begin{enumerate}
\item $D_{*}(M)=\bigoplus_{k\geq 0}D_k(M)$ for the exterior algebra of multivectors on $M$, 
$D(M)=D_1(M)$ for the $C^{\infty}(M)$--module of vector fields on $M$, and $``\wedge"$ for the 
wedge product in $D_{*}(M)$;
\item $\ldb\cdot,\cdot\rdb$ for the Schouten bracket in $D_{*}(M)$;
\item $\Lambda^{*}(M)=\bigoplus_{k\geq 0}\Lambda^{*}(M)$ for the exterior algebra of 
differential forms on $M$ and $``\wedge"$ for the wedge product in it;
\end{enumerate}

If $S$ is a $\Z$-graded object, say, a multivector, then we use $(-1)^{\dots S\dots}$ 
(resp., $(-1)^{\dots \bar{S}\dots}$) for $(-1)^{\dots \mathrm{deg}S\dots}$ (resp., $(-1)^{\dots (\mathrm{deg}S-1)\dots}$).
For instance, if $P\in D_k(M)$ and $Q\in D_l(M)$, then $(-1)^{P\bar{Q}}=(-1)^{k(l-1)}$ and 
$(-1)^{P+\bar{Q}}=(-1)^{k+l-1}$. This notation makes the formulas that involve signs of 
$\Z$-graded objects, more readable. In particular, graded anticommutativity and Jacobi 
identity for the Schouten bracket reads
\begin{equation}\label{Schouten-skew}
\ldb P,Q\rdb=-(-1)^{\bar{P}\bar{Q}}\ldb Q,P \rdb
\end{equation}
\begin{equation}\label{Schouten-Jacobi}
(-1)^{\bar{P}\bar{R}}\ldb P, \ldb Q,R \rdb \rdb+(-1)^{\bar{R}\bar{Q}}\ldb R, \ldb P,Q \rdb \rdb+
(-1)^{\bar{Q}\bar{R}}\ldb Q, \ldb R,P \rdb \rdb=0
\end{equation}

Denote by $\mathrm{Hgr}\;\Lambda^*(M)$ the totality of graded $\R$--linear operators acting on the graded 
space $\Lambda^*(M)$ and by $[\cdot,\cdot]^{gr}$ the graded commutator of such operators. An operator
$\Delta\in \mathrm{Hgr}\;\Lambda^*(M)$ is a \emph{(graded) differential operator} over $\Lambda^*(M)$ if
$$
[\omega_0,[\omega_1,\dots,[\omega_k,\Delta]^{gr},\dots,]^{gr}]^{gr}=0, \quad \forall\omega_o,\omega_1,\dots,\omega_k\in \Lambda^*(M),
$$
where $\omega_i'$s are understood to be left multiplication operators.

Insertion of a multivector $Q\in D_k(M)$ into a differential form $\omega\in\Lambda^l$ we denote by 
$Q\rfloor\omega\in\Lambda^{l-k}$, and by $i_Q : \Lambda^{*}(M)\rightarrow\Lambda^{*}(M)$ the operator
$\omega\mapsto Q\rfloor\omega$, i.e., $i_Q(\omega)=Q\rfloor\omega$. Obviously, 
$$
i_P\circ i_Q=i_{P\wedge Q} \quad \mbox{and} \quad [i_P,i_Q]^{gr}=0,  \;P,Q\in D_{*}(M).
$$

The correspondence  $Q\Leftrightarrow i_Q$ identifies the algebra $D_{*}(M)$ with the algebra of
$C^{\infty}(M)$--linear differential operators  over $\Lambda^{*}(M)$. More exactly,
these operators of order $k$ correspond to $k$--vectors. In terms of this identification the Schouten bracket is described by
the formula
\begin{equation}\label{Schouten}
i_{\ldb P,Q\rdb}=[[i_P,d]^{gr},i_Q]^{gr}=-(-1)^{\mathrm{deg}Q}[i_P,[i_Q,d]^{gr}]^{gr}, \quad P,Q\in D_{*}(M).
\end{equation}

The \emph{Lie derivative operator} along a multivector $Q$ is defined as
\begin{equation}\label{Lie derivative}
L_Q=[i_Q,d]^{gr}:\Lambda^*(M)\rightarrow\Lambda^*(M)
\end{equation}
and (\ref{Schouten}) reads
\begin{equation}\label{Schouten-Lie}
i_{\ldb P,Q\rdb}=[L_P,i_Q]^{gr}=-(-1)^Q[i_P,L_Q]^{gr}.
\end{equation}

Here  the following useful formula  should be mentioned:
\begin{equation}\label{[L_X,i_P]}
[i_Q, L_X]^{gr}=i_{L_X(Q)},\quad X\in D(M), \;Q\in D_{*}(M),
\end{equation}
where $L_X(Q)=\ldb Q,X\rdb$.

The \emph{liezation} operation $L:Q\mapsto L_Q$ is a (graded right) derivation of the algebra $D_{*}(M)$:
\begin{equation}\label{L as derivation}
L_{P\wedge Q}=i_P\circ L_Q+(-1)^Q L_P\circ i_Q.
\end{equation}
Another useful interpretation of the Schouten bracket is easily derived from (\ref{Schouten-Lie}) and (\ref{L as derivation}):
\begin{equation}\label{Schouten via L}
i_{\ldb P,Q\rdb}=(-1)^Q L_{P\wedge Q}-(-1)^Q i_P\circ L_Q-
(-1)^{\bar{P}Q} i_Q\circ L_P.
\end{equation}

A convenient coordinate-wise descriptions of the above operations is as follows. Let $x=(x_1,\dots,x_n)$ be a local chart
on $M$. Instead of the standard local expression 
$$
Q=\sum_{1\leq i_1<\dots<i_k\leq n}a_{ i_1,\dots,i_k}(x)\frac{\partial}{\partial x_{i_1}}\wedge\dots\wedge\frac{\partial}{\partial x_{i_n}}, \quad Q\in D_k(M),
$$
we shall use
\begin{equation}\label{antipolinomials}
Q=\sum_{1\leq i_1<\dots<i_k\leq n}a_{ i_1,\dots,i_k}(x)\xi_{i_1}\dots\xi_{i_n}
\end{equation}
assuming that the variables $\xi_i$'s anticommute, i.e., $\xi_i\xi_j=-\xi_j\xi_i$. This allows 
one to introduce ``partial derivatives" $\frac{\partial}{\partial x_i}$ and $\frac{\partial}{\partial \xi_i}$ 
acting on  skew-commutative polynomials (\ref{antipolinomials}) in $\xi_i$'s. Namely,
the first of them just acts on coefficients $a_{ i_1,\dots,i_k}(x)$, while the second is 
$C^{\infty}(M)$-linear and commutes with the multiplication by $\xi_j$ operator by the  rule
$\frac{\partial}{\partial \xi_i}\circ \xi_j+\xi_j\circ\frac{\partial}{\partial \xi_i}=\delta_{ij}$. In these terms
the Schouten bracket reads
\begin{equation}\label{Schouten in coordinates}
\ldb P,Q \rdb=-\sum_{i}\left(\frac{\partial P}{\partial x_i}\frac{\partial Q}{\partial \xi_i}+
(-1)^P \frac{\partial P}{\partial \xi_i}\frac{\partial Q}{\partial x_i}\right).
\end{equation}
In particular, by introducing the operator $X_P:D_{*}(M)\rightarrow D_{*}(M), \;X_P(Q)=\ldb P,Q \rdb$,
we have
\begin{equation}\label{Schouten-Hamilton field}
X_P=-\sum_{i}\left(\frac{\partial P}{\partial x_i}\frac{\partial}{\partial \xi_i}+
(-1)^P\frac{\partial P}{\partial \xi_i}\frac{\partial }{\partial x_i}\right).
\end{equation}

\subsection{Poisson manifolds}

 Recall that \emph{Poisson structure} on a manifold $M$ is a Lie algebra structure on the $\R$-- vector space $C^{\infty}(M)$
 $$
 (f,g)\mapsto \{f,g\}\in C^{\infty}(M), \quad  f,g\in C^{\infty}(M),
 $$
which at the same time is  a \emph{biderivation}, i.e.,
$$
\{fg,h\}=f\{g,h\}+g\{f,h\} \quad \mbox{and}  \quad\{f,gh\}=g\{f,h\}+h\{f,g\}.
$$
$P\in D_2(M)$ is a \emph{Poisson} bivector if $\ldb P,P\rdb=0$.  The formula
$$
\{f,g\}=P(df,dg), \quad f,g\in C^{\infty}(M).
$$
establishes one-to-one correspondence
between Poisson  bivectors and Poisson structures on $M$. The Poisson bracket associated 
this way with the Poisson bivector $P$ will be denoted by $\{\cdot,\cdot\}_P$.

A Poisson manifold is \emph{nondegenerate}, if the corresponding Poisson bivector is 
nondegenerate, i.e.,  the correspondence
$$
\gamma_P:\Lambda^1(M)\rightarrow D(M), \quad\omega\mapsto P(\omega,\cdot),
$$
is an isomorphism of $C^{\infty}(M)$--modules.  $\gamma_P$ naturally extends to an homomorphism 
of exterior algebras still denoted 
$$
\gamma_P:\Lambda^{*}(M)\rightarrow D_{*}(M).
$$
It is an isomorphism if $P$ is nondegenerate. In this case 
$\gamma_P(P)\in \Lambda^2(M)$ is a symplectic form on $M$. This way the class of nondegenerate 
Poisson manifolds is identified with the class of symplectic manifolds.

The Poisson differential 
$$
\partial_P:D_{*}(M)\rightarrow D_{*+1}(M), \quad \partial_P(Q)=\ldb P,Q\rdb,
$$ 
associated with a Poisson bivector $P$ supplies $D_{*}(M)$ with a cochain complex structure. The vector field
\begin{equation}\label{P-Ham field}
P_f\DEF\partial_P(f)=\ldb P,f\rdb=-\gamma_P(df)=-df\rfloor P
\end{equation}
is called \emph{$P$--Hamiltonian} corresponding to the \emph{Hamiltonian function} f.

The following definition is central for this paper.
\begin{defi}\label{Poisson compatibility}
Poisson structures $P_1$ and $P_2$ on a manifold $M$ are called \emph{compatible} if $P_1+P_2$ is a Poisson structure as well.
\end{defi}
\begin{proc}
Poisson structures $P_1$ and $P_2$ are compatible if one of the following equivalent conditions holds:
\begin{enumerate}
\item $\ldb P_1,P_2\rdb=0$;
\item $sP_1+tP_2, \,s,t\in\R$, is a Poisson structure $\forall s,t$;
\item the bracket $\{\cdot,\cdot\}=s\{\cdot,\cdot\}_{P_1}+t\{\cdot,\cdot\}_{P_2},  \,s,t\in\R$, is a Lie algebras structure
on $C^{\infty}(M)$;
\item $\partial_{P_1}+\partial_{P_2}$ is a differential in $D_{*}(M)$, or, equivalently, $\partial_{P_1}\partial_{P_2}+\partial_{P_2}\partial_{P_1}=0$.
\end{enumerate} 
\end{proc}
\begin{proof}
The first assertion directly follows from $\ldb P_1+P_2,P_1+P_2\rdb=2\ldb P_1,P_2\rdb$, while (2) - (4) are obvious consequences of it.
\end{proof}

\subsection{Lie algebras.} 

In the literature the term "Lie algebra" is commonly used in two different 
meanings, namely, either as a concrete \emph{Lie algebra structure} on a vector space, or as an isomorphism class
of such structures. In various situations in this paper this distinction is essential and we will use "Lie algebra structure"
instead of "Lie algebra" in an ambiguous in this sense context.

Lie algebra structures will be denoted by bold Fraktur characters, say, $\gG, \gH$, etc. The symbol 
$|\gG|$ refers to the supporting $\gG$ vector space. We use square brackets, if necessary with 
various indexes, for Lie product operations.

Let $\gG$ be a Lie algebra over a ground field $\gk$ and  $V=|\gG|$. A Lie algebra structure is 
naturally defined in the algebra $\gk(V^*)$ of polynomials on $V^*=\Hom_{\gk}(V,\gk)$. Namely, 
denoting by $f_v$ the linear function on $V^*$ corresponding to $v\in V$, we define the 
``Poisson bracket" $\{\cdot,\cdot\}$ on linear functions by putting
$$
\{f_v,f_w\}\DEF f_{[v,w]}, \quad v,w,\in V,
$$ 
and extend it onto the whole algebra  as a biderivation. This construction remains valid for any 
larger algebra $A\supset \gk(V^*)$ with the property that any derivation of  $\gk(V^*)$ uniquely 
extends to $A$. For instance, $C^{\infty}(V^*)$ is such an algebra if $\gk=\R$. We shall refer 
to the so-defined Lie algebra as the \emph{Poisson structure on the dual to the Lie algebra} 
$\gG$. The corresponding Poisson bivector on $V^*$ will be denoted $P_{\gG}$.

Let $\{e_i\}$ be a basis in $V$. Put $x_i=f_{e_i}$. Then
\begin{equation}\label{ coord}
\{f,g\}=\sum_{i,j,k}c_{ij}^kx_k\frac{\partial f}{\partial x_i}\frac{\partial g}{\partial x_j}
\end{equation}
where $c_{ij}^k$ are structure constant of $\gG$ in the considered basis, and
\begin{equation}\label{Poisson dual}
P_{\gG}=\sum_{i,j,k}c_{ij}^kx_k\xi_i\xi_j.
\end{equation}
Poisson structures of the form $P_{\gG}$  have linear coefficients in any cartesian chart on $V^*$ and vice versa. By this reason they are also called \emph{linear}. If $Q\in D_{*}(W)$ is a \emph{linear}, i.e., with linear in a cartesian chart coefficients, multivector on a vector space $W$, then
$$
Q_{\theta}=\ldb X_{\theta},Q\rdb, \quad \theta\in W,
$$
where $Q_{\theta}$ is the value of $Q$ at $\theta$ and $X_{\theta}$ is the  corresponding to
$\theta$ constant vector field on $W$. This observation is useful when dealing with linear Poisson structures.

Let $\gG_1$ and $\gG_2$ be Lie algebra structures on a $\gk$-vector space $V$ and $[\cdot,\cdot]_1, [\cdot,\cdot]_2$ the corresponding Lie products. The following is the analogue of definition\,\ref{Poisson compatibility} for Lie algebras.

\begin{defi}
 Lie algebra structures $\gG_1$ and $\gG_2$ are called \emph{compatible} if $[\cdot,\cdot]\DEF[\cdot,\cdot]_1,+[\cdot,\cdot]_2$
 is a Lie product in $V$.
\end{defi}
The  Lie algebra structure on $V$ defined by the Lie product $[\cdot,\cdot]_1,+[\cdot,\cdot]_2$ will be denoted by 
 $\gG_1+\gG_2$ . Obviously, we have
 \begin{proc}\label{Poisson-lie dual}
Lie algebra structures $\gG_1$ and $\gG_2$ on $V$ are compatible if and only if the corresponding Poisson structures on $V^*$ are 
compatible. $\qquad\qquad\qquad\Box$
\end{proc}

\subsection{Lie rank of Poisson manifolds and Lie Algebras.}

Recall that a bivector field $Q\in D_2(M)$ generates  a distribution  (with singularities) on 
$M$. This distribution is defined as a $C^{\infty}(M)$--submodule $D_Q(M)$ of 
$D(M)$ generated by vector fields $Q_f=df\rfloor Q, \;f\in C^{\infty}(M)$.

Geometrically,  $D_Q(M)$ may be viewed as a family of vector spaces 
$M\ni x\mapsto \triangle_Q(x)\subset T_{x}M$ on $M$ where the subspace 
$\triangle_Q(x)\subset T_{x}M$ is generated by vectors of the form 
$Q_{f,x}\in T_{x}M,\quad \forall f \in C^\infty(M).$ The function
$$
M\ni x\mapsto {\rm rank}_Q(x) \DEF\dim\triangle_Q(x)
$$
is, obviousely, lower semicontinous with values in even integers. In 
particular, ${\rm rank}_Q(x)$ is is locally constant except a thin closed
subset in $M$ and reaches its maximum valiue, say 2k, in an open domain of $M$.  
The number $k$ is uniquely characterized as the number such that $Q^k\neq 0$ and 
$Q^{k+1}=0$ (here $Q^{i}$ stands for $i$-th wedge power of $Q$). Alternatively, 
2k is equal to a maximal number of independent vector fields of the form $Q_f.$ 
Below we shall also deal  with polinomial (Poisson) bivectors on a $\gk$--linear space 
${\bf V}$, and the above-said remains valid in this context as well. In particular, in this 
case bivectors have the same rank almost everywhere except an algebraic subveraiety 
of $\bf V.$
\begin{defi}\label{Oppa}
1) A bivector field $Q$ is said to be of \emph{rank 2k} if 
$$
Q^k\neq 0,\quad Q^{k+1}=0
$$
2) A Lie algebra is said to be of Lie rank 2k if the associated linear Poisson 
bivector on its dual is of rank 2k.
\end{defi}

\begin{ex}
If  $n=2k+\epsilon, \quad\epsilon=0, 1,$  then the direct sum of 
$k$ copies of the non-commutative 2-dimensional Lie algebra and the 1-dimensional algebra, 
if $\epsilon=1$, is an n-dimensional Lie algebra of maximal  Lie rank $2k.$ 
\end{ex}

If $Q$ is a Poisson bivector, then $[Q_f,Q_g]=Q_{\{f,g\}}$, i.e.,  the distribution $D_Q(M)$ 
is a Frobenius one. Locally maximal integral submanifolds of $D_Q(M)$ constitute the 
canonical \emph{symplectic foliation} of $M$ associated with $Q$ (see \cite{VKra, Wein}). 
If $Q=P_{\gG}$, then the leaves of this foliation are \emph{orbits of the coadjoint 
representation of $\gG$}.

\section{Modularity of Poisson and Lie algebra structures}\label{mdlr}

In this section we introduce and study a splitting of a Poisson or a Lie algebra structure 
into {\it unimodular} and {\it non-unimodular} parts. This splitting is canonical up to a 
gauge transformation and reduces, in a sense, the study of general Poisson or Lie algebras 
structures to unimodular ones. The central in this construction are notion of modular 
vector field and modular class introduced by J.-L. Koszul \cite{Kz}. In survey 
\cite{Ksmn} the reader will find an extensive bibliography about. 
 
\subsection{Unimodular Poisson structures.}

If $\omega\in\Lambda^n(M)$ is a volume from, then the map
$$
Q\mapsto Q\mathop{\rfloor} \omega,\qquad Q\in  D_*(M),
$$
which we shall call \emph{$\omega$-duality},
is an isomorphism between $C^{\infty}(M)$--modules $D_*(M)$ and 
$\Lambda^*(M).$ In particular, the $(n-2)$--form 
$\alpha=\alpha_{P,\omega}=P\rfloor\omega$, $\omega$-dual to the Poisson bivector 
$P$, completely characterizes this bivector. 
\begin{proc}\label{Dual}
$P \in D_2(M)$ is a Poisson bivector on $M$ if and only if
\begin{equation}\label {DUAL}
d(P\rfloor\alpha)=2P\rfloor d\alpha \quad\mathrm{with}\quad  \alpha=P\rfloor\omega.
\end{equation}
\end{proc}

{\it Proof.} Formula (\ref{Schouten-Lie}) for $P=P_1=P_2$ can be rewritten as
$$
L_{P\wedge P}-2i_P\circ L_{P}= i_{\ldb P,P\rdb}
$$
and hence
\begin{equation}\label{PPP}
L_{P\wedge P}(\omega) -2P\mathop{\rfloor}L_{P}(\omega) = \ldb 
P,P\rdb\mathop{\rfloor}(\omega)
\end{equation}

On the other hand, by (\ref{Lie derivative}), we have
\begin{equation}\label{Pdd}
L_{P}(\omega)=-d(P\mathop{\rfloor}\omega)=-d\alpha
\end{equation}
and
$$
 L_{P\wedge P}(\omega)=-d((i_{P\wedge P)}(\omega))=
 -d(P\mathop{\rfloor}(P\mathop{\rfloor}\omega))=-d(P\mathop{\rfloor}\alpha)
$$
 With these substitutions (\ref{PPP}) takes the form
$$
 d(P\mathop{\rfloor}\alpha)-2P\mathop{\rfloor} d\alpha=
-\ldb P,P\rdb\mathop{\rfloor}\omega
$$
Finally, observe that $Q=0 \Leftrightarrow Q\mathop{\rfloor}\omega=0$ for any $Q\in D_{*}(M)$. 
$\quad\Box$

\begin{defi}\label{UUU}
A Poisson structure on $M$ is called \emph{unimodular} with respect to $\omega$ (shortly, 
\emph{$\omega$--unimodular}) if $ L_{P}(\omega)=0.$
\end{defi}

\begin{proc}\label{Uni}
A Poisson structure $P$ is unimodular if and only if one of the following relations holds
\begin{enumerate}
\item $
d\alpha=0\quad\mbox{with}\quad\alpha=P\mathop{\rfloor}\omega$;
\item $L_{P_{f}}(\omega)=0,\quad\forall\, f\in{\rm C}^\infty(M),$
i.e. Lie derivatives of $\omega$ along all $P$--hamiltonian fields vanish.
\end{enumerate}
\end{proc}

{\it Proof.} In view of (\ref{Pdd})  the first assertion is obvious. Then, by 
applying (\ref{Lie derivative}), we have
$$
\begin{array}{l}
L_{P_{f}}(\omega)=d(P_{f}\mathop{\rfloor}\omega)=
-d((df\mathop{\rfloor}P)\mathop{\rfloor}\omega)=
-d(df\wedge(P\mathop{\rfloor}\omega))=\\df\wedge d(P\mathop{\rfloor}\omega) 
=-df\wedge L_{P}(\omega)
\end{array}
$$
i.e.
\begin{eqnarray}\label{LPf}
L_{P_{f}}(\omega)=-df\wedge L_{ P}(\omega)
\end{eqnarray}
Since $df\wedge\rho=0,\quad\forall\, f\in C^\infty( M),$ implies $\rho=0$ for a 
from $\rho\in \Lambda^k( M),\, k<n,$ we  see that
$$
L_{P_{f}}(\omega)=0,\quad\forall\,f\Longleftrightarrow df\wedge L_{P}(\omega)= 
0,\quad\forall\, f\Longleftrightarrow L_{P}(\omega)=0. \qquad\Box
$$
\begin{rmk}\label{COO}
In fact, formula (\ref{LPf}) shows that unimodularity of $P$ is guaranteed 
by $L_{P_{x_i}}(\omega)=0,\,i=1,\dots,n,$ for a local chart 
$(x_1,\dots,x_n)$ on $M.$
\end{rmk}

Compatibility conditions  for unimodular Poisson structures are simplified as follows.
\begin{proc}\label{PPO}
Let $P_1,P_2,$ be $\omega$-unimodular Poisson structures on $M$. They are 
compatible if and only if $L_{P_1\wedge P_2}(\omega)=0$.
\end{proc}

{\it Proof}. Since $L_{P_1\wedge P_2}(\omega)=0$  formula (\ref{Schouten via L}) applied 
to $\omega$ gives:
$$
P_1\rfloor L_{P_2}(\omega)+P_2\rfloor L_{P_1}(\omega) + \ldb 
P_1,P_2\rdb\mathop{\rfloor}\omega=0
$$
Due to unimodularity of $P_1$ and $P_2$ this for\-mu\-la reduces to $\ldb 
P_1,P_2\rdb\mathop{\rfloor}\omega=0$. which is equivalent to $\ldb P_1,P_2\rdb 
=0.\quad\Box$

\begin{cor}\label{PROD}
Two $\omega$--unimodular Poisson structures $P_1$ and $P_2$ are compatible if 
$P_1\wedge P_2=0.\qquad\Box$
\end{cor}

\begin{cor}\label{prod}
Any two $\omega$--unimodular Poisson structures on a 3-dimensional $M$ are 
compatible.
\end{cor}
\begin{rmk}\label{LOB}
Condition $L_{P_1\wedge P_2} =0$ 
implying, obviously, compatibility of $P_1$ and $P_2$ is not, in fact, weaker 
than the condition $P_1\wedge P_2=0$, since $L_Q=0$ is equivalent to $Q=0$ 
for any $Q\in D_*(M)$.
\end{rmk}

\subsection{Lie rank of unimodular Lie algebras}

The following example shows existence of unimodular odd-dimensional Lie algebras of
maximal possible Lie rank, i.e., of rank $2k$ if $n=2k+1$.
 \begin{ex}
Let $(x_1,x_2,\dots,x_{2k+1})$ be a cartesian chart on $V^*, \;\dim\;V=2k+1$. The bivector
$$
P=x_{2k+1}(\xi_1\wedge\xi_2+\dots +\xi_{2k-1}\wedge\xi_{2k})
$$
is a Poisson one as well as bivectors $P_s=x_{2k+1}(\xi_{2s-1}\wedge\xi_{2s}, \;s=1,\dots,k$,
of  rank one.
Obviously, the rank of $P$ is $2k$, $P_s$'s are mutually compatible and $P=P_1+\dots+P_k$.
So, Lie algebra structures $\gG, \gG_1,\dots, \gG_k$ on $V$ corresponding $P, P_1,\dots, P_k$, respectively,
are mutually compatible, $\gG=\gG_1+\dots+\gG_k$ and the Lie rank of $\gG$ is $2k$. 
\end {ex}

Now we will show that the Lie rank of an unimodular Lie aldgebra can not 
coincide with its dimension.

\begin{proc}\label{XUP}
If $Q$ is an $\omega$-modular Poisson structure on $M$ of rank $Q$ and $n=2k$, 
then $Q^k\rfloor\omega=const\neq 0.$
\end{proc}
{\it Proof.} Since $Q$  is a Poisson structure, then, according to (\ref{Schouten via L}), 
$L_{Q^2}=2i_Q\circ L_Q$. Inductively, one easily finds that
$$
L_{Q^s}=si_{Q^{s-1}}\circ L_Q,\quad\forall s\ge 2
$$
If $Q^k\ne 0$, then the  function $f=Q^k\rfloor\omega=(Q^k,\omega)$ is 
different from zero.

On the other hand,  due to $\omega$-unimodularity of $Q$, $L_Q(\omega)=0$ and 
hence
$$
df=d(Q^k\rfloor\omega)=L_{Q^k}(\omega)=kQ^{k-1}\rfloor L_Q(\omega)=0. 
\quad\quad\Box
$$
\begin{cor}\label{6.1f}
The Lie rank of an unimodular Lie algebra is strictly lesser than its dimension.
\end{cor}
{\it Proof.} Let $P$ be the associated linear Poisson bivector on $\bf V$ and
${\rm dim V}=2k$. If $k$ is the rank of $P$, then
$f=P^k\rfloor \omega=const\ne 0$ where $\omega$ is a cartesian volume form on $\bf 
V$. This contradicts the fact that $f$ is a homogeneous polynomial of order 
$k.\qquad\Box$

\subsection{Modular vector fields}

Let $\omega$ be a volume form and $P$ a Poisson vector on $M$.
\begin{defi}
The vector field $\Xi=\Xi_{P,\omega}$ is uniquely defined by the relation
\begin{equation}\label{POL}
\Xi\mathop{\rfloor}\omega=d(P\rfloor\omega)
\end{equation}
and is called the \emph{modular}  vector field of $P$  with respect to $\omega$ 
(shortly, \emph{$\omega$--modular}).
\end{defi}

It follows from (\ref{Pdd}) and (\ref{LPf})  that
$$
L_{P_f}(\omega)=df\wedge (\Xi\mathop{\rfloor}\omega)
$$
On the other hand, $0=\Xi\mathop{\rfloor}(df\wedge\omega)= 
\Xi(f)\omega-df\wedge(\Xi\mathop{\rfloor}\omega).$ So,
\begin{equation}\label{DIV}
L_{P_f}(\omega)=\Xi(f)\omega \ee i.e. $\Xi(f)={\rm div}_\omega (P_f).$ In other words, 
$\Xi_{P,\omega}$ mesures divergence of $P$-ha\-mil\-to\-ni\-an vec\-tor 
fields with respect to the volume from $\omega.$ This property was the 
original definition of modular fields.

\begin{proc}\label {LIST} The following relations hold for the modular field
$\Xi=\Xi_{P,\omega}$ of a Poisson structure $P$ and its dual form 
$\alpha=\alpha_{P,\omega}=P\rfloor\omega$ :
$$
\begin{array}{lcc}
({\rm i})\quad L_{\Xi}(\alpha)=0\\
({\rm ii})\quad L_P(d\alpha)=0\\
({\rm iii})\quad  L_P(\alpha)=-\Xi\mathop{\rfloor}\alpha\\
({\rm iv})\quad L_\Xi(\omega)=0\\
({\rm v})\quad L_\Xi(P)=\ldb P,\Xi\rdb=\partial_P(\Xi)=0\\
({\rm vi})\quad P_f\mathop{\rfloor}\omega=-df\wedge\alpha;\\
({\rm vii})\quad [\Xi, P_{f}]=P_{\Xi(f)}.
\end{array}
$$
\end{proc}

{\it Proof.} First, we have
$
 L_{\Xi}(\alpha)=d(\Xi\mathop{\rfloor}\alpha)+\Xi\mathop{\rfloor}d\alpha.
$
Since
$$
\Xi\mathop{\rfloor}\alpha= \Xi\mathop{\rfloor}(P\mathop{\rfloor}\omega)= 
P\mathop{\rfloor}(\Xi\mathop{\rfloor}\omega)= P\mathop{\rfloor}d\alpha
$$
then, by (\ref{DUAL}),
$$
d(\Xi\mathop{\rfloor}\alpha)=d(P\mathop{\rfloor}d\alpha)= 
\frac{1}{2}d^2(P\mathop{\rfloor}\alpha)=0
$$
On the other hand,
 $\Xi\mathop{\rfloor}d\alpha=
\Xi\mathop{\rfloor}(\Xi\mathop{\rfloor}\omega)=0$ This proves (i). 

Similarly,
$$
L_P(d\alpha)=[i_P,d]^{gr}(d\alpha)=-d(P\mathop{\rfloor}d\alpha)=0
$$
This proves (ii).

In its turn (iii) directly results  from (\ref{DUAL}):
$$
L_P(\alpha)=[i_P,d](\alpha)=P\mathop{\rfloor}d\alpha- 
d(P\mathop{\rfloor}\alpha)=-P\mathop{\rfloor}d\alpha= 
-\Xi\mathop{\rfloor}\alpha
$$
Concerning (iv) we have:
$$
L_\Xi(\omega)=d(\Xi\mathop{\rfloor}\alpha)=d^2\alpha=0
$$
To prove (v) it suffices to show that
  $L_\Xi(P)\mathop{\rfloor}\omega=0.$ But according to (i) and (\ref{[L_X,i_P]}) we have
 $$0=L_\Xi(\a)=L_\Xi(P\mathop{\rfloor}\o)=
-L_\Xi(P)\mathop{\rfloor}\omega+P\mathop{\rfloor}L_\Xi(\omega)= 
-L_\Xi(P)\mathop{\rfloor}\omega.
$$

Now, a particular case of (\ref{Schouten}) is $i_{\ldb f, P\rdb}=[[f,d]^{gr},i_P]^{gr}=-[df,i_P]$, and one gets (vi) 
by applying this relation to $\omega$. 

Finally, recall the general formula
$$
L_X(\rho\mathop{\rfloor}Q)=-L_X(\rho)\mathop{\rfloor}Q+\rho 
\mathop{\rfloor}L_X(Q)
$$
with $X\in D(M),\quad Q\in D_i(M),\quad \rho\in\Lambda^j(M).$ By specifying 
it to $X=\Xi,\quad Q=P$ and $\o=d\nu$ one finds that
$$
[ \Xi, P_f]=-L_\Xi(P_f)=L_\Xi(df\rfloor P)=-L_{\Xi}(df)\rfloor P+df\rfloor L_{\Xi}(P).
$$
Now the desired result follows from assertion (v) by observing that $L_\Xi(df)=d\,\Xi(f)$. $\qquad\Box$

Additivity is an important property of $\omega$--modular fields.
\begin{proc}\label{PFA}
If Poisson structure $P_1$ and $P_2$ are compatible and $\Xi_1,\Xi_2$
 are their $\omega$--modular fields, respectively, then
$\Xi_1+\Xi_2$ is the $\omega$--modular field of $P_1+P_2.$
\end{proc}

{\it Proof.} Straightforwardly from 
(\ref{DIV}).$\qquad\Box$
\begin{cor}\label{BUU} If $P_i,  \,\Xi_i, \,i=1,2$, are as above, then
\begin{equation}\label{buu}
L_{\Xi_1}(P_2)+L_{\Xi_2}(P_1)=0
\end{equation}
\end{cor}
{\it Proof.} Directly from $L_{\Xi_1}(P_1)=L_{\Xi_2}(P_2)=
L_{\Xi+\Xi_2}(P_1+P_2)=0$ (proposition\;\ref{LIST},(v)). $\qquad\Box$

 \subsection{$\omega$--modular class}
 
The following well-known fact describes dependence of the $\omega$--modular field on $\omega$. 
For completeness we report a short proof of it.
\begin{proc}\label{GAU}
If $\omega^\prime=f\o$ is another volume form on $M$, then
\begin{equation}\label{gau}
\Xi_{P,f\omega}=\Xi_{P,\omega}-P_{\ln|f|}
\end{equation}
\end{proc}
{\it Proof.} The vector field $\Xi_{P,f\omega}$ is determined as solution of
\begin{equation}\label{gag}
\Xi_{P,f\omega}\mathop{\rfloor}(f\omega)=d\alpha_{P,f\omega}
\end{equation}
By putting $\Xi_{P,f\omega}=\Xi_{P,\omega}+Y$ and noticing that 
$\alpha_{P,f\omega}=P\mathop{\rfloor}(f\omega)=f\alpha$ where 
$\alpha=\alpha_{P,\omega}$ one may rewrite  (\ref{gag}) in the form
$$
fY\mathop{\rfloor}\omega=df\wedge\alpha\Longleftrightarrow 
Y\mathop{\rfloor}\omega=d(\ln|f|)\wedge\alpha
$$
(f is nowhere zero, since $f\o$ is a volume form). Now proposition \ref{LIST}, 
(vi), shows that $Y=-P_{\ln|f|}.\qquad\Box$

This result has the following cohomological interpretation. First, observe that 
(proposition (\ref{LIST},(v))
$$
\partial_P(\Xi)=\ldb P,\Xi\rdb=L_\Xi(P)=0,
$$
i.e., $\Xi$ is a 1-cocycle of the complex $\{ D_*(M), \partial_P\}$. Moreover, $P$--hamiltonian
 fields are coboundaries of this complex, namely,
$P_g=\partial_P(g)$. Hence proposition \ref{GAU} yields:

\begin{cor}\label{COH}
The cohomology class of the $\omega$--modular field $\Xi_{P,\omega}$ in 
$\{D_*(M),\partial_P\}$ does not depend on $\omega$ and, therefore, is 
well-defined by $P$. $\qquad\Box$
\end{cor}
\begin{defi}
The $\partial_P$-cohomology class of the $\omega$-modular field is called 
the \emph{modular} class of $P$.
\end{defi}
\begin{cor}\label{HMM}
A Poisson structure $P$  is $\omega$--unimodular with respect to a volume form 
$\omega$ if and only if its modular class vanishes.$\qquad\Box$
\end{cor}
If $P$ is nondegenerate, then $(M, \gamma_P(P))$ is a symplectic
manifold. In this case the isomorphism $\gamma_P:D_*(M)\rightarrow\Lambda^*(M)$  is also an isomorphism of
complexes  $\{\Lambda^*(M),d\}$ and $\{D_*(M),d_P\}$.
Therefore, if $H^1(M)=0,$ then any nondegenerate Poisson structure on $M$ is $\o$--unimodular
with respect to a suitable volume form $\o.$

\subsection{Modular disassembling of a Poisson structure.}

Now we shall show that the $\omega$--modular vector field of a Poisson structure $P$ allows 
one to disassemble this structure, at 
least, locally,  into two parts, one of which is 
$\omega$--unimodular, while  all ``$\omega$--non-- unimodularity" 
of $P$ is concentrated the second part.

\begin {proc}\label{SPLIT}
Let $\Xi$ be the $\omega$--modular vector field of a Poisson structure $P$ and $\nu\in C^\infty(M)$ be such 
that $\Xi(\nu)=1$. Then
\begin{enumerate}
\item $\Xi\wedge P_\nu$ is a Poisson structure compatible with $P$;
\item $L_{\Xi\wedge P_\nu}(\omega)=-L_P(\omega)$;
\item $P+{\Xi\wedge P_\nu}$ is an  $\omega$-unimodular 
compatible with $P$ Poisson structure  and $\nu$  is a Casimir function of it;
\item $P_\nu\wedge \Xi=\ldb P, \nu\Xi\rdb=\partial_P(\nu\Xi)$.
\end{enumerate}
\end{proc}
{\it Proof.} First, from proposition \ref{LIST},(vii), we see that 
$$
\ldb \Xi, P_{\nu}\rdb=[\Xi, P_{\nu}]=-L_{\Xi}(d\nu\rfloor P)=P_{\Xi(\nu)}=P_1=0.
$$
Since the Schouten bracket is a graded biderivation of the exterior algebra $D_{*}(M)$, this implies
$$
\ldb\Xi\wedge P_\nu, \Xi\wedge P_\nu\rdb=0
$$
i.e. $\Xi\wedge P_\nu$ is a Poisson structure. By the same reason we have 
$$
\ldb P,\Xi\wedge P_\nu\rdb=\ldb P,\Xi\rdb\wedge P_\nu- \Xi\wedge\ldb 
P,P_\nu\rdb
$$
The right hand side of this equality vanishes by proposition 
\ref{LIST},(v), and the fact that $P_\nu$ ia a $P$--hamiltonian field. So, 
$\Xi\wedge P_\nu$ is compatible with $P.$

To prove the second assertion we specify (\ref{L as derivation}) to $P=\Xi, \;Q=P_\nu$ 
and then apply the result to $\omega$:
$$
L_{\Xi\wedge P_\nu}(\o)=\Xi\mathop{\rfloor}L_{P_\nu}(\omega)- 
L_\Xi(P_{\nu}\mathop{\rfloor}\omega)
$$
Similarly, formula (\ref{[L_X,i_P]}), specified to $X=\Xi, \;Q=P_\nu$ and then applied to $\omega$, gives
$$
L_\Xi(P_\nu\mathop{\rfloor}\omega)=-L_\Xi(P_\nu)\mathop{\rfloor}\omega+ 
P_\nu\mathop{\rfloor}L_{\Xi}(\omega)
$$
By proposition \ref{LIST},(v), (vii), $L_{\Xi}(\omega)=0$  and $L_\Xi(P_\nu)=0$. Hence
$$
L_{\Xi\wedge P_\nu}(\o)=\Xi\mathop{\rfloor}L_{P_\nu}(\omega)
$$
On the other hand, according to (\ref{DIV}), $L_{P_\nu}(\omega)=\Xi(\nu)\omega=\omega$ and hence
$$
L_{\Xi\wedge P_\nu}(\o)=\Xi\rfloor\omega=d\alpha=-L_P(\omega).
$$

Next, $\omega$--unimodularity of $P+\Xi\wedge P_\nu$
directly follows from assertion (2), while
$$
d\nu\mathop{\rfloor}(P+\Xi\wedge  P_\nu)=d\nu\mathop{\rfloor}P+ 
\Xi(\nu){P_\nu}=0
$$
proves that $\nu$ is a Casimir function of this structure.

The fact that $\partial_P$ is a graded derivation of the exterior algebra $D_{*}(M)$
together with $P_\nu=\ldb P, \nu\rdb, \;\ldb P, \Xi\rdb=0$ (proposition\;\ref{LIST}, (v)) proves the last assertion.
  $\qquad\Box$

Thus  $P$  is presented as the sum
\begin{equation}\label{BRA}
P=(P-P_\nu\wedge\Xi)+P_\nu\wedge\Xi
\end{equation}
of two compatible Poisson structures,
one of which is $\omega$-unimodular and another $\omega$-non-unimodular (if different 
from zero). $P_\nu\wedge\Xi$ is an \emph{$\omega$-modular bivector} associated 
with $P$. Note that $P_\nu\wedge\Xi$ is of rank 2 (if different from zero), i.e., is a 
smallest possible $\omega$--non--unimodular part of $P$.
This may be interpreted as a canonical disassembling of $P$ into $\omega$-unimodular and 
$\omega$-non-unimodular parts.
 
 Obviously, all $\omega$-modular bivectors associated with $P$ form an affine subspace  
 in $D_2(M)$ modelled on the subspace $\{P_f\wedge\Xi \;|\;\Xi(f)=0\}$. 

All $\omega$-modular bivectors associated with $P$ are compatible each other. Indeed, by 
proposition\;\ref{LIST}, (viii), $[P_{f}, \Xi]=0$ 
if $\Xi(f)=\const$. Since the Schouten bracket is a graded biderivation of $D_{*}(M)$, this implies
$$
\ldb P_\nu\wedge\Xi, P_{\mu}\wedge\Xi\rdb=0 \quad \mathrm{if} \quad \Xi(\nu)=\Xi(\mu)=1.
$$

It is remarkable that an $\omega$-modular bivector coincides with its 
$\omega$-non-unimodular part, i.e., is \emph{completely $\omega$-non-modular}.
\begin{proc}\label{IDE} If $P,\omega,\Xi$ and $\nu$ are as above, then
\begin{enumerate}
\item the $\omega$--modular field of the Poisson structure 
 $P_\nu\wedge \Xi$  coincides with  $\Xi$;
\item if $\Xi(\mu)=1$, then  $(P_\nu\wedge\Xi)_\mu\wedge\Xi=P_\nu\wedge\Xi$.
\end{enumerate}
\end{proc}
{\it Proof.} By definition (\ref{POL}), the first assertion is just an interpretation of 
proposition \ref{SPLIT}, (2). The second one follows from:
$$
(P_{\nu}\wedge\Xi)_\mu=-d\mu\mathop{\rfloor}(P_\nu\wedge\Xi)= 
\Xi(\mu)P_\nu-\{\mu,\nu\}\Xi\qquad\Box
$$

So, in contrast to general Poison structures, the non-unimodular part of 
an $\omega$--modular bivector is unique, i.e. does not depend on 
the choice of a normalizing function $\nu$.

Denote by $\{\cdot,\cdot\}_{non}$ (resp., $\{\cdot,\cdot\}_{uni}$) the Poisson 
bracket corresponding to the $\omega$-non-uni-modular
 (resp., $\omega$-unimodular) part of $P$ according to (\ref{BRA}).Then
$$
\{f,g\}_{non}=\{f,\nu\}\Xi(g)-\{g,\nu\}\Xi(f).
$$
If $P$--hamiltonian fields $P_f$ and $P_g$ are $\omega$--divergenceless i.e., 
$\Xi(f)=\Xi(g)=0$ (see (\ref{DIV}), then
$$
\{f,g\}_{non}=0 \Longleftrightarrow \{f,g\}_{uni}=\{f,g\}.
$$
This shows that restriction of the bracket $\{\cdot,\cdot\}_{uni}$ to the
$\omega$--divergenceless part of the original Poisson structure does not depend 
on the choice of the normalizing function $\nu$ (see (\ref{BRA})).

An abstract description of 
$\omega$--modular bivectors is as follows.
\begin{proc}\label{UAB}
Let $X,\Xi\in D(M)$ and the volume form $\omega\in\Lambda^n(M)$ be such that
\begin{equation}\label{uab}
[X,\Xi]=0, \quad L_\Xi(\omega)=0,\quad L_X(\omega)=\omega
\end{equation}
Then $P=X\wedge\Xi$ is a Poisson structure, which coincides with its 
$\omega$--non--unimodular part, and $\Xi$ is the  $\omega$--modular field of it.
\end{proc}

{\it Proof.} First, observe that conditions (\ref{uab}) implies independence of vector fields
$X$ and $\Xi$. Since $\ldb X\wedge\Xi,X\wedge\Xi\rdb 
=2[X,\Xi]\wedge X\wedge\Xi$, the condition $[X,\Xi]=0$ implies that $P$ is a 
Poisson structure. On the other hand, in view of (\ref{[L_X,i_P]}), (\ref{L as derivation}) and 
(\ref{uab}) we have
$$
\begin{array}{l}
L_{P}(\omega)=X\mathop{\rfloor}L_\Xi(\omega)-L_X(\Xi\mathop{\rfloor}\omega)=
L_X(\Xi)\mathop{\rfloor}\omega-\Xi\mathop{\rfloor}L_X(\omega)= 
-\Xi\mathop{\rfloor}\omega
\end{array}
$$
This shows (see (\ref{Pdd}) and (\ref{POL})) that $\Xi$ is the 
$\omega$--modular vector field of $P$.

Since $X$ and $\Xi$  commute, there exists, at least, locally,  a function $\nu$ such that 
$\Xi(\nu)=1,\quad X(\nu)=0$. For such a function $P_\nu=-d\nu\rfloor(X\wedge\Xi)=X,$ i.e., 
locally, $P=P_\nu\wedge\Xi$ and, therefore, $P$ coincides with its $\omega$--non--unimodular
part.  $\qquad\Box$

\begin{defi}\label{PNU} A Poisson bivector described in proposition (\ref{UAB})
 is called an \emph{$\omega$--modular bivector}.
\end{defi}

In what follows we shall assume satsified conditions of proposition\,\ref{UAB} when 
referring to an $\omega$--modular bivector presented in the form $X\wedge\Xi$.

\subsection{Compatibility of $\omega$--modular bivectors.}

Now we shall discuss compatibility conditions involving  $\omega$--modular bivectors.  First, we consider
the inverse to the $\omega$--modular splitting procedure.
\begin {proc}\label{VUB}
An  $\omega$--modular bivector $X\wedge\Xi$
and an unimodular  structure $Q$ are compatible if and only if
$$
\L_\Xi(Q)=0,\quad\Xi\wedge L_X(Q)=0.
$$
\end{proc}
{\it Proof.} First, observe that $\Xi$ is the  $\omega$--modular field of
the Poisson structure $X\wedge\Xi+Q$ (proposition\,\ref{PFA}). 
Therefore, in view of proposition\,\ref{LIST}, (v), 
$0=L_\Xi(X\wedge\Xi+Q)=L_\Xi(Q).$ On the other hand, the compatibility 
condition of $X\wedge\Xi$ and  $Q$ is
$$
\begin{array}{l}
0=\ldb X\wedge\Xi,Q\rdb=X\wedge\ldb\Xi,Q\rdb-\ldb X,Q\rdb
\wedge\Xi=
\L_\Xi(Q)\wedge X+ L_X(Q)\wedge \Xi
\end{array}
$$
Since $\L_\Xi(Q)=0$, this gives  the desired result.$\qquad\Box $
\begin{rmk}
Generally, $X\wedge\Xi$ is not an $\omega$--modular bivector associated with $P=X\wedge\Xi+Q$.
For example, if $M=\R^2,  \,\omega=dx_1\wedge dx_2, X=x_1\xi_1, \,\Xi=\xi_2$ and $Q=\xi_1\xi_2$, then
$X\wedge\Xi$ is an  $\omega$--modular bivector compatible with the unimodular Poisson bivector $Q$. But
in this case $P$ is another $\omega$--modular bivector, i.e., the unimodular part of $P$ is trivial. 
\end{rmk}

Now we shall discuss compatibility of two $\omega$-modular bivectors.
Assume $X_i,\Xi_i\in D(M), i=1,2,$ to be as in proposition (\ref{UAB}). 
By developing the compatibility condition $\ldb P_1,P_2\rdb=0$ of $\omega$-modular 
bivectors $P_1=X_1\wedge \Xi_1$ and $P_2=X_2\wedge 
\bar{\Xi}_2$ we obtain
\begin{eqnarray}\label{kom}
 \ldb X_1,X_2\rdb\wedge\Xi_1\wedge\Xi_2 +
\ldb\Xi_1,\Xi_2 \rdb\wedge X_1\wedge X_2=\nonumber\\
 \ldb X_1,\Xi_2 \rdb\wedge \Xi_1\wedge X_2+
\ldb \Xi_1,X_2 \rdb\wedge X_1\wedge\Xi_2.
\end{eqnarray}

However, formula (\ref{kom}) does not reflect modularity properties of $P_1$ and $P_2$, and
these are to be added. Since (\ref{kom}) guarantees that $P=P_1+P_2$ is a Poisson bivector
with the modular field $\Xi=\Xi_1+\Xi_2$, proposition\,\ref{LIST}  and its consequences are 
valid for $P, \Xi$. 

For instance, formula (\ref{buu}) in the considered case  can be rewritten as
\begin{equation}\label{lko}
\ldb X_2,\Xi_1\rdb\wedge\Xi_2+\ldb X_1,\Xi_2 \rdb\wedge \Xi_1
+(X_1-X_2)\wedge\ldb\Xi_1,\Xi_2 \rdb=0
\end{equation}
Note that (\ref{lko}) is a formal conseguence of (\ref{kom}) and modularity property of 
$P_i$'s. Moreover, by 
multiplying (\ref{lko}) by $X_2$, we can bring formula (\ref{kom})  to the form
\begin{equation}\label{rdc}
\ldb X_1,X_2\rdb\wedge\Xi_1\wedge\Xi_2= \ldb \Xi_1,X_2 \rdb \wedge (X_1-X_2) 
\wedge\Xi_2
\end{equation}
or, similarly, to 
\begin{equation}\label{rds}
\ldb X_1,X_2\rdb\wedge\Xi_1\wedge\Xi_2= \ldb X_1,\Xi_2 \rdb \wedge (X_1-X_2) 
\wedge\Xi_1
\end{equation}
Hence we have
\begin{proc}\label{OMP}
$\omega$-modular bivectors $X_1\wedge\Xi_1$ and  $X_2\wedge\Xi_2$ are 
compatible if and only if (\ref{lko}) and one of formulae (\ref{rdc}) or (\ref{rds}) 
holds.
 $\qquad\Box$
\end {proc}
\begin{rmk}\label{TCN}
Condition (\ref{buu}) is manifestly satisfied if $\Xi_1=\lambda\Xi_2, \quad 0\ne \lambda\in C^{\infty}(M)$,
i.e., two $\omega$-modular bivectors are compatible if their $\omega$-modular 
fields are proprtional.
\end{rmk}

\subsection{On complexity of the matching problem}

Let $\mathcal{P}_i$  (resp.,  $\mathcal{G}_i$), $i=1,\dots, m$, be diffeomorphism 
(resp., isomorphism) classes of Poisson (resp., Lie algebras) structures of the same 
dimension. The \emph{matching problem} is to classify various realizations of $P_i$'s 
(resp., $\gG_i$'s) of these structures on the same manifold (resp., vector space) for
which $P_i$ and $P_j$ (resp., $\gG_i$ and $\gG_j$) are compatible for all $i$ and $j$. 
Such a realization will be called a \emph{matching}. An equivalence of two matchings 
is  defined in an obvious manner, and the matching problem is : what are different, i.e.,
nonequivalent, matchings of given Poisson (resp., Lie algebra) structures?

This problem seems to be rather difficult. Below we shall discuss it for two 
$\omega$--modular bivectors in order to show its complexity.  Namely, we shall solve 
compatibility conditions for $\omega$--modular bivectors $P_1=X_1\wedge \Xi_1$ and 
$P_2=X_2\wedge \Xi_2$ assuming that $\Xi_1$ and $\Xi_2$are independent and  
$[\Xi_1,\Xi_2]=0$. The last assumption is automatically satisfied for Lie algebras. 

With these assummptions (\ref{lko}) becomes
$$
\ldb X_2,\Xi_1\rdb\wedge\Xi_2+\ldb X_1,\Xi_2\rdb\wedge\Xi_1=0,
$$
or, equivalently,
\begin{equation}\label{yhi}
\ldb X_1,\Xi_2\rdb=f_1\Xi_1+\lambda\Xi_2, \qquad
\ldb X_2,\Xi_1\rdb=\lambda\Xi_1+f_2\Xi_2
\end{equation}
for some $f_1, f_2, \lambda\in C^\infty(M)$. Now each of formulae
(\ref{rdc}) and (\ref{rds}) can be brought to the form
$$
(\ldb X_1,X_2\rdb- \lambda( X_1-X_2))\wedge\Xi_1\wedge\Xi_2=0
$$
The last relation is equivalent to
\begin{equation}\label{ihy}
\ldb X_1,X_2\rdb=\lambda( X_1-X_2)+\mu_1\Xi_1-\mu_2\Xi_2
\end{equation}
for some $\mu_1,\mu_2\in C^\infty(M)$.

\begin{lem}\label{UNI}
 If $X_1,X_2,\Xi_1, \Xi_2$
are as above, then functions $f_1,f_2,\lambda,\mu_1,\mu_2\,$ occuring in 
(\ref{yhi}) and (\ref{ihy}) satisfy relations
\begin{eqnarray}\label{erl}
\Xi_1(\lambda)=\Xi_2(\lambda)=
\Xi_1(f_1)=\Xi_2(f_2)=0\nonumber\\
\Xi_1(\mu_2)=2f_2\lambda+X_1(f_2)\nonumber\\
\Xi_2(\mu_1)=2f_1\lambda+X_2(f_1)\nonumber\\
\lambda^2+f_1f_2=-\Xi_1(\mu_1)-X_1(\lambda)= -\Xi_2(\mu_2)-X_2(\lambda).
\end{eqnarray}
\end{lem}
{\it Proof.} These are consequences of  relations (\ref{yhi}), (\ref{ihy}) and Jacobi identities 
involving vector fields $X_1,X_2, \Xi_1,\Xi_2$. For instance,  Jacobi identity for  
$X_1, \Xi_1,\Xi_2$ gives $\Xi_1(\a)=\Xi_1(\lambda)=0$. 
$\qquad\Box$

The gauge $X_i\mapsto X_i+g_i\Xi_i$ with $\Xi_i(g_i)=0$ annihilates $f_i, i=1,2,$ if 
$\Xi_1(g_2)=-f_2,  \,\Xi_2(g_1)=-f_1$. Due to relations $\Xi_i(f_i)=0$ from the above lemma 
such functions $g_i$ exist, at least, locally and relations (\ref{erl}) are simplified to
\begin{eqnarray}\label{erlm}
\Xi_1(\lambda)=\Xi_2(\lambda)=\Xi_1(\mu_2)=\Xi_2(\mu_1)=0,\nonumber\\
\Xi_1(\mu_1)+X_1(\lambda)= \Xi_2(\mu_2)+X_2(\lambda)=-\lambda^2.
\end{eqnarray}
Note that the so-normalized vectors $X_i$'s are uniquely defined up to a transformation 
$X_i\mapsto X_i+\phi_i\Xi_i$ with $\Xi_j(\phi_i)=0, \;i,j,=1,2$. Such a transformation induces 
a transformation $\mu_i\mapsto \mu_i+\psi_i, \,i=1,2,$ with 
$\Xi_j(\psi_i)=0, \;i,j,=1,2$. Here $\psi_1, \psi_2$ arbitrary satisfying the last conditions
functions.
 
Now It is convenient to pass to vectors $Z=X_1-X_2$ and $W=X_1+X_2$ instead of 
$X_1$ and $X_2$ which will be assumed normalized as above. In these terms relations 
(\ref{yhi}) and (\ref{ihy}) become
\begin{eqnarray}\label{newihy}
\ldb Z,\Xi_1\rdb=-\lambda\Xi_1, \quad \ldb Z,\Xi_2\rdb=-\lambda\Xi_2, 
\quad\ldb W,\Xi_i\rdb=\lambda\Xi_i, \nonumber \\
\ldb Z,W\rdb=2\lambda\,Z+2\,\mu_1\Xi_1-2\,\mu_2\Xi_2
\end{eqnarray}
with (see (\ref{erl}))
\begin{equation}\label{newerlm}
Z(\lambda)=\Xi_2(\mu_2)-\Xi_1(\mu_1), \quad W(\lambda)=-(2\lambda^2+\Xi_1(\mu_1)+\Xi_2(\mu_2)).
\end{equation}

Now suppose that the bidimensional foliation generated by $\Xi_1$ and $\Xi_2$ is a 
fibration $\pi:M\rightarrow N$. This takes place, at least, locally. Then formula 
(\ref{newihy}) tells that vector fields $Z$ and $W$ are $\pi$-projectable. Moreover,
it follows from (\ref{newihy}) and (\ref{newerlm}) that $\Xi_i(\mu_i)=0, \;i=1,2,$ and hence
$$
\lambda=\pi^*(z), \quad \Xi_i(\mu_i)=\pi^*(\nu_i), \;i=1,2,  \quad \mbox{for some} \quad z, \,\nu_1, \,\nu_2\in\,C^{\infty}(N). 
$$

Let $\bar{Z}=\pi(Z), \; \bar{W}=\pi(W), \;u=\nu_2-\nu_1$ and $v=\nu_1+\nu_2$. Then
\begin{equation}\label{projrel}
\bar{Z}(z)=u, \quad \bar{W}(z)=-(2z^2+u), \quad [\bar{Z},\bar{W}]=2z\bar{Z}.
\end{equation}
We shall explicitly solve these relations assuming  that functions $z, u$ and $v$  are
functionally independent (the \emph{generic} case). To this end note that there is a local 
chart $(z,u,v, y_1,\dots,y_{n-5})$ on $N$ such that $\bar{Z}(y_i)=\bar{W}(y_i)=0, \forall\,i$, 
and in this chart
\begin{equation}\label{exactfields}
\bar{Z}=u\partial_z+\alpha\partial_u+\beta\partial_v, \quad \bar{W}=
-(2z^2+v)\partial_z+a\partial_u+b\partial_v
\end{equation}
with $\alpha, \beta, a$ and $b$ being some functions of $z, u,$ and $v$. By using first 
two relations in (\ref{projrel})
and (\ref{exactfields}) one finds that $[\bar{Z}, \bar{W}](z)=-(4zu+a+\beta)$. On the other 
hand, the third relation in (\ref{projrel}) gives  $[\bar{Z}, \bar{W}](z)=2zu$. Hence
\begin{equation}\label{a}
a=-(\beta+6zu)
\end{equation}
Next, by substituting fields (\ref{exactfields}) to $[\bar{Z},\bar{W}]=2z\bar{Z}$ and taking 
into account (\ref{a}) one obtains
\begin{equation}\label{pre-b}
\bar{Z}(a)-\bar{W}(\alpha)=2z\alpha, \quad \bar{Z}(b)-\bar{W}(\beta)=2z\beta
\end{equation}
The first of these equations can be resolved with respect to $b$:
\begin{equation}\label{b}
b=\frac{1}{\alpha_v}\psi(z,u,v,\alpha,\beta, \alpha_z,\dots,\beta_v)
\end{equation}
with $\psi$ being a polynomial of variables $z,u, v$, functions $\alpha$ and $\beta$ and 
their first order derivatives. So, coefficients $a$ and $b$ of $\bar{W}$ are completely 
determined by $\alpha$ and $\beta$.

Now, in view of (\ref{a}) and (\ref{b}), the second of equations (\ref{pre-b}) becomes a 
relation of the form
\begin{equation}\label{Phi}
\Phi(z,u,v,\alpha,\beta, \alpha_z,\dots,\beta_v)=0
\end{equation}
with $\Phi$ being a rational function of all involved arguments. The interested reader will 
easily find explicit expressions for $\phi$ and $\Phi$ which are not so instructive to be 
reported here.
\begin{proc}\label{InvariantsMatchings}
Functions $$\lambda=\pi^*(z), \quad\Xi_1(\mu_1)=-\frac{1}{2}\pi^*(u+v), \quad\Xi_2(\mu_2)=\frac{1}{2}\pi^*(u+v)$$ are differential invariants of matchings of generic $\omega$--modular bivectors with respect to diffeomorphisms preserving the
volume form $\omega$ as well as functions $\pi^*\alpha$ and $\pi^*\beta$ of variables 
$\pi^*(z), \pi^*(u), \pi^*(v)$ which are subject to differential relation (\ref{Phi}).
\end{proc}
\begin{proof}
The first assertion of this proposition is obvious from the above discussion. Also, vector 
fields $\bar{Z}$ and $\bar{W}$ are differential invariants of the problem. Hence their 
components $\alpha$ and $\beta$ in the invariant chart $(z,u,v)$
are differential invariants as well assuming that all vector fields $\bar{Z}, \,\bar{W}$ 
resolving relations (\ref{projrel}) can be lifted to vector fields $Z, W$ on $M$ that resolve  
relations (\ref{newihy}) and (\ref{newerlm}).

To prove the last assertion consider a local chart $(x_1, x_2, z,u,v,...)$ in which 
$\Xi_i=\partial_{x_i} \,i=1,2$. Then vector fields $\bar{Z}, \,\bar{W}$  can be lifted to 
vector fields $\tilde{Z}, \,\tilde{W}$ on $M$ such that  $\tilde{Z}(x_i)=\tilde{W}(x_i)=0, \,i=1,2$. 
Vector fields $Z, W$ on $M$ are of the form
$$
Z=\tilde{Z}+\lambda(x_1\partial_{x_1}-x_2\partial_{x_2})+\phi_1\partial_{x_1}+\phi_2\partial_{x_2}, \;
W=\tilde{W}-\lambda(x_1\partial_{x_1}+x_2\partial_{x_2})+\psi_1\partial_{x_1}+\psi_2\partial_{x_2}
$$
with $\Xi_i(\phi_j)=\Xi_i(\psi_j)=0, \;i,j=1,2$. By using the gauge 
$X_1\mapsto X_1+\phi_1\Xi_1, \,X_2\mapsto X_2-\phi_2\Xi_2$ 
we eliminate functions $\phi_1$ and $ \phi_2$. Now all relations in (\ref{newihy}) are 
satisfied except the last one, which is satisfied if
$$
\tilde{Z}(\psi_1)-\lambda\psi_1=\mu_1-\Xi_1(\mu_1)x_1 \quad \mathrm{and} \quad  
\tilde{Z}(\psi_2)+\lambda\psi_2=\Xi_2(\mu_2)x_2-\mu_2.
$$
Obviuosly, these equations admit (local) solutions in a neighborhood of a regular point of $\tilde{Z}$.
\end{proof}

\begin{rmk}Proposition \ref{InvariantsMatchings} tells that matchings of two modular 
Poisson structures depend on functional parameters even when their modular vectors 
commute. If these do not commute the situation becomes much more complicated.
\end{rmk}

\section{Modular structure of Lie algebras.}\label{Lie-modlarity}

Now we shall specify results of the preceding section to linear Poisson structures, i.e.., to 
Lie algebras.  In this case  $M$ is replaced by the dual $V^*$ of an $n$-dimensional vector
space $V$ over a ground field $\gk$. Being algebraically formal the results of the preceding 
section remain valid in the differential calculus over the algebra 
$\gk[V^*]$ of polynomials on $V^*$.

\subsection{Modular disassembling of Lie algebras.}

The \emph{cartesian volume from} 
$\omega=dx_1\wedge\dots\wedge dx_n$ associated with a standard cartesian chart 
$(x_1,\dots,x_n)$ on $V^*$ is well-defined up to a scalar factor. Obviously, the concept 
of $\omega$--modulatity does not change when passing from $\omega$ to $\lambda\omega, \,0\neq\lambda\in\gk$.
So, the {\it cartesian modularity}, i.e.,  $\omega$--modularity with respect
to a cartesian volume form $\omega$, is well--defined on $V^*$ and will be simply referred to as 
\emph{modularity}. In this section we shall  only deal with polynomial tensor fields on $V^*$ and 
use adjectives \emph{constant}, \emph{linear}, etc, by referring to coefficients of these fields.

Below $P$ stands for a linear Poisson structure on $V^*$ which is identified with a 
Lie algebra structure on $V$ (see n.\,2.3). The differential form $\alpha=\alpha_P=
P\mathop{\rfloor}\omega$ is linear, while $d\alpha$ is constant as well as the modular 
vector field $\Xi=\Xi_P$. It is easy to see that $\Xi$ does not depend on the
choice of a cartesian volume form. Being constant the field $\Xi$ is identified with a 
vector $\theta=\theta_P\in V^*$ called the \emph{modular vector} of  $P$ or of the 
corresponding Lie algebra.

Since $\Xi$ is constant, a function $\;\nu\;$ such that $\Xi(\nu)=1$ can be chosen linear 
and, therefore, identified with a vector $v\in V$ such that $\theta(v)=1$. The Poisson 
bivector $P_\nu\wedge\Xi$ is  linear and hence corresponds to a Lie algebra structure 
on $V$. Obviously, it is well--defined by $P$. Therefore, the disassembling 
(\ref{BRA}) defines a disassembling of the Lie algebra associated with $P$ into 
unimodular and non-unimodular parts.  Namely,
\begin{equation}
\gG=\gG_{uni}+\gG_{non} \quad\mathrm{with} \quad P=P_{\gG}, \quad P_{\gG_{uni}}=
P+P_{\nu}\wedge\Xi, \quad
P_{\gG_{non}}=P_{\nu}\wedge\X. 
\end{equation}

A direct description of this disassembling in terms of the Lie algebra $\gG$ is as follows. 
First, recall that with a linear vector field $X$ on $V^*$  a linear operator $A:V\rightarrow V$ 
is naturally associated. Namely, by identifying vectors of $V$ with linear functions on $V^*$, 
this becomes a tautology, namely, $A(u)=X(u)$. In particular, 
for $X=P_{\nu}$ we have 
$$
A(u)\DEF\ldb u,\nu\rdb=P(du,d\nu),\quad u\in{\bf V^*}.
$$
i.e., $A=-\ad_{\gG}\nu$. Now the characteristic property (\ref{DIV}) of $\Xi$ is translated as
\begin{equation}\label{tra}
\theta(u)=-\tr(\ad_{\gG}u)
\end{equation}
This formula may be considered as a direct definition of $\theta$. It also tells that 
unimodular Lie algebras are those for which  operators of the adjoint representation
are traceless.

In these terms, 
the Lie algebra structure
 $\gG_{non}$ corresponding to $P_\nu\wedge\Xi$ reads
\begin{equation}\label{non}
[u,v]_{non}=\th(u)A(v)-\th(v)A(u), \quad u,v\in V, \quad A=\ad_{\gG}\nu,
\end{equation}
or, alternatively, $[u,v]_{non}=\th(u)[\nu,v]-\th(v)[\nu,u]$.
\begin{proc}\label{CAH} The operator $A=\ad_{\nu}$
and  $\theta\in  V^*$ satisfy relations: 
\be\label{usl} 
A^*(\th)=0, \quad \tr\,A=-1, \quad A(\nu)=0, \quad\th(\nu)=1
\ee 
Conversely, if $\theta\in V^*, \,\nu\in V$ and $A: V\rightarrow V$ satisfy the above relations, 
then formula (\ref{non}) defines a Lie algebra, which coincides with its non--unimodular part.
\end{proc}
{\it Proof.} The relations to prove are just translations of relations  (\ref{uab}) for 
$X=P_{\nu}$ and $P_{\nu}(\nu)=0$ (see proposition \ref{UAB}).

If, conversely,
$A^*(\theta)=0$, i.e., $\theta(Au)=0, \forall u\in V$, then   (\ref{non}) defines
 a Lie algebra structure $\gH$ on $V$ . Relations   $A(\nu)=0, \;\theta(\nu)=1$ 
imply that $A=ad_{\gH}\nu$. Finally, definition (\ref{tra}) shows that the 
modular vector of $\gH$ coincides with $\theta$, since 
the trace of the operator $u\longmapsto \theta(u)z, 
\;z\in V$, is equal to $\theta(z).\quad\Box$

Relations (\ref{usl}) except $\tr\, A=-1$ mean that
 $V$ splits  into the direct sum of two subspaces
 $ W_0=\ker\,\th$ and  $ W_1=\{\lambda\nu\, |\,\lambda\in \R\}$
of dimensions $n-1$ and 1, respectively. They are invariant with respect to $A$ 
and $\; W_1\subset\mbox{ker}\,A.$ This shows that a Lie algebra defined by 
(\ref{non})
 with $A$ and $\theta$ satisfying (\ref{usl}) is uniquely up to isomorphism determined
 by the operator $A_0=A|_{W_0}:W_0\rightarrow W_0,\quad\rm{tr}\,A_0=1.$ Indeed, 
 let $W_0$ and $W_1$ be vector spaces, 
$\rm{dim}\,W_0=n-1, \rm{dim}\,W_1=1,$ and $0\neq e\in W_1.$ Then with any linear 
operator $A_0:W_0\rightarrow W_0$ one can associate a Lie 
algebra structure on $W=W_0\oplus W_1$ given by the relations
\begin{equation}\label{DKL}
[u,v]=0\quad\rm{for}\quad u,v\in W_0\quad\rm{and}\quad[u,e]=A_0(u)
\end{equation}
This structure is isomorphic to that given by (\ref{non}) and (\ref{usl}) if $\rm{tr}\,A_0\neq 0$. 

\begin{defi}\label{FRC}
A Lie algebra defined by (\ref{DKL}) with $\rm{tr}\,A_0\neq 0$ is called 
\emph{modular}.
\end{defi}

Observe that the product in a modular Lie algebra is of the form (\ref{non}) with $\theta$ and $A$ satisfying relations (\ref{usl}) for some $\nu\in V$. Denote this algebra by $\gl_{A,\theta,\nu}$.

Thus any  Lie algebra structure is the sum of a modular Lie algebra and a compatible with it unimodular one.

\subsection{Compatibility of modular and unimodular Lie algebras.}

Now we shall discuss compatibility conditions of a modular Lie algebra and an unimodular one. 
A modular Lie algebra is of the form $\gG_{X\wedge\Xi}$ where $X$ and $\Xi$ 
are commuting linear and constant vector fields, respectively, satisfying relations (\ref{uab}). 
The product in this algebra is given by (\ref{non}) where $A:V\rightarrow V$ is corresponding 
to $X$ operator.
\begin{proc}
An unimodular Lie algebra algebra $\gG$ is compatible with the algebra 
$\gl_{A,\theta,\nu}$ if and only if
\begin{equation}\label{comp-mod-uni}
\theta([\gG,\gG])=0 \; \mathrm{and} \; \theta(u)([Av,w]+[v,Aw]-A[v,w])+\mathrm{cycle}=0, \;\forall  u,v,w\in V,
\end{equation}
where $[\cdot,\cdot]$ is the product in $\gG$.
\end{proc}
\begin{proof}
Let $Y\in D(M), \,Q\in D_2(M), \omega,\rho\in \Lambda^1(M)$. Recall the general formula
$$
L_Y(Q)(\omega,\rho)=Q(L_Y(\omega),\rho)+Q(\omega,L_Y(\rho))-Y(Q(\omega,\rho))
$$
By specifying it to $M=V^*, \,Y=\Xi, \,Q=P_{\gG}, \,\omega=du, \,\rho=dv$ 
we find
$$
L_{\Xi}(P_{\gG})(du,dv)=P_{\gG}(d(\theta(u)),dv)+P_{\gG}(du,d(\theta(v)))-\theta([u,v])=-\theta([u,v]),
$$
since $\theta(u)$ and $\theta(v)$ are constant. Hence the condition $L_{\Xi}(P)=0$ of 
proposition\,\ref{LIST} specifies to $\theta([\gG,\gG])=0$. Similarly, for  $Y=X$ we obtain
$$
L_X(P_{\gG})(du,dv)=[Au,v]+[u,Av]-A([u,v])
$$ 
Now it is easy to see that the second relation we have to prove is identical to 
the relation $\Xi\wedge L_X(P)=0$ of proposition\,\ref{VUB}.
\end{proof}
\begin{cor}
If $\gG$ is unimodular and $\gG=[\gG,\gG]$, then no modular Lie algebra is compatible with $\gG$.
In particular, a semi-simple Lie algebra can not be compatible with a modular Lie algebra.
\end{cor}

\subsection{Matching modular Lie algebras.}

 Let $X_1\wedge\Xi_1, X_2\wedge\Xi_2$ be linear Poisson bivectors on 
$V^*$ with $X_i$ and $\Xi_i$ as in sec.\,3.7. 
Being  the modular field of $X_i\wedge\Xi_i$ $\Xi_i$ is a constant 
vector field on $V$ and, therefore, $[\Xi_1,\Xi_2]=0.$ By this reason 
$X_1,X_2,\Xi_1,\Xi_2$ are subject to relations (\ref{yhi}) and (\ref{ihy}) and as 
a consequence, to lemma (\ref{UNI}). Since $X_1,X_2$ are linear vector 
fields on $V^*$ functions $f_1, f_2$ and $\lambda$ in (\ref{yhi}) 
are constant while $\mu_1,\mu_2$ in (\ref{ihy}) are linear. By a suitable gauge 
$X_i\mapsto X_i+g_i\Xi_i, \,i=1,2,$ with linear $g_i$'s functions $f_1$ and $f_2$ 
can be eliminated. In this case relations in lemma\,(\ref{UNI}) reduce to
\begin{equation}\label{ekl}
\Xi_1(\mu_2)= \Xi_2(\mu_1)=0, \quad\Xi_1(\mu_1)= 
\Xi_2(\mu_2)=-\lambda^2.
\end{equation}
Let $\bar V^*$ be the quotient  by the 2-dimensional subspace 
$\mathrm{span}(\Xi_1,\Xi_2)$ space of $V^*$. Then relations (\ref{yhi}) and 
(\ref{ihy}) show that vector fields $X_1$ and 
$X_2$ project to some vector fields $\bar{X_1},\bar{X_2}$ on $\bar{V}^*$, 
respectively, and 
$[\bar{X_1},\bar{X_2}]=\lambda(\bar{X_1}-\bar{X_2}).$ So, $\bar{X_1}$ and $\bar{X_2}$ 
generates a 2-dimensional Lie algebra on $\bar V^*$. 

Let $V_1^*$ be a complement of $V_0^*=\rm{span}(\theta_1,\theta_2)$ in $V^*$ and 
$\pi_i:V^*\rightarrow V_i^*, \,i=0,1$, be the corresponding projections. A projectable on $\bar{V}^*$
vector field $X\in D(V^*)$ can be presented in the form
$$
X=X_0+a_1\Xi_1+a_2\Xi_2, \quad a_1, a_2\in C^{\infty}(V^*)
$$ 
where $X_0$ is parallel to $V_1^*$. If $X$ is linear, then $X_0, a_1$ and $a_2$ are linear too.
So, vector fields $Z=X_1-X_2$ and $W=X_1+X_2$ can be presented as
\begin{equation}\label{fieldsDeco}
Z=Z_0+\alpha_1\Xi_1+\alpha_2\Xi_2, \quad W=W_0+\beta_1\Xi_1+\beta_2\Xi_2.
\end{equation}
In these terms relations (\ref{newihy}) are equivalent to
\begin{eqnarray}\label{1-part}
-\Xi_1(\alpha_1)=\Xi_1(\beta_1)=\Xi_2(\alpha_2)=\Xi_2(\beta_2)=-\lambda, \nonumber \\
\quad \Xi_1(\alpha_2)=\Xi_1(\beta_2)=\Xi_2(\alpha_1)= \Xi_2(\beta_1)=0.
\end{eqnarray} 
\begin{eqnarray}\label{0-part}
Z_0(\beta_1)-W_0(\alpha_1)-\lambda(3\alpha_1+\beta_1)=2\mu_1, \nonumber \\
Z_0(\beta_2)-W_0(\alpha_2)+\lambda(\beta_2-3\alpha_2)=-2\mu_2.
\end{eqnarray} 
Point out that linear functions $\alpha_i$'s and $\beta_i$'s depend on the choice of the 
complement $V_1^*$.

To proceed on we need the following lemma
\begin{lem}\label{div X_0}
Let $\omega$ be a cartesian volume form on $V_1^*$ and 
$\omega=\pi_0^*(\omega_0)\wedge\pi_1^*(\omega_1)$
with $\omega_i$ being a cartesian volume form on $V_i^*$. Then
$$
L_X(\omega)=(\Div_{\omega_0}\bar{X}_0+\Xi_1(\alpha_1)+\Xi_2(\alpha_2))\omega
$$
where $\bar{X}_0$ is the restriction of $X_0$ to $V_1^*$. 
\end{lem}
\begin{proof}
Since $X_0\rfloor\pi_0^*(\omega_0)=\Xi_0\rfloor\omega=\Xi_1\rfloor\omega=0$, we have
$$L_{X_0}(\pi_0^*(\omega_0))=0 \quad\mathrm{and} \quad L_{\alpha_i\Xi_i}(\omega)=d\alpha\wedge(\Xi_i\rfloor\omega)=
\Xi_i(\alpha_i)\omega
$$
It remains to note that
$$L_{X_0}(\omega)=\pi_0^*(\omega_0)\wedge L_{X_0}(\pi_1^*(\omega_1))=
\pi_0^*(\omega_0)\wedge\pi_1^*(L_{\bar{X}_0}\omega_1)=\Div_{\omega_0}\bar{X}_0\cdot\omega.
$$
\end{proof}

A linear function $\varphi$ on $V^*$ can be decomposed into the sum 
$\varphi=\varphi^0+\varphi^1$ where $\varphi^i$ 
is linear and vanishes on  $V_i^*, \,i=1,2$. Accordingly, relations (\ref{1-part}) and 
(\ref{0-part}) split into two parts.
First of them explicitly describes functions $\alpha_i^1, \beta_i^1, \,i=1,2$, while the 
second one, in view of (\ref{ekl}), put no additional restrictions on these functions.      
We can assume that $\alpha_1^0=\alpha_2^0=0$ by making use of 
the gauge $X_1\mapsto X_1+\alpha_1^0\Xi_1, \;X_2\mapsto X_2-\alpha_2^0\Xi_2$, 
which is still at our disposal.  In this normalization relations (\ref{0-part}) become
\begin{eqnarray}\label{reduced 0-part}
Z_0(\beta_1^0)-\lambda\beta_1^0=2\mu_1^0,  \quad\quad
Z_0(\beta_2^0)+\lambda\beta_2^0=-2\mu_2^0.
\end{eqnarray} 

Now we are ready to describe matchings of two modular Lie algebras. If $\lambda\neq 0$ 
they are characterized by  ordered quadruples $(\mathcal{V},A,B,\lambda)$ where 
$\mathcal{V}$ is a vector space, $A,B:\mathcal{V}\rightarrow \mathcal{V}$
are linear operators such that $[A,B]=2\lambda A$ and $\tr\,B=2(1+\lambda)$.
If $\lambda=0$, then the matching is characterized by the
quintuple $(\mathcal{V},A,B,\nu_1,\nu_2)$ with commuting 
$A,B:\mathcal{V}\rightarrow \mathcal{V}$ such that $\tr\,A=0, \tr\,B=2$, and
$\nu_1,\nu_2\in\Ker\,A$. Quadruples $(\mathcal{V},A,B,\lambda)$ and 
$(\mathcal{V}',A',B',\lambda')$ are \emph{equivalent} if $\lambda=\lambda'$ and 
there exists an isomorphism $\Phi:\mathcal{V}\rightarrow \mathcal{V}'$
such that $ A'=\Phi B\Phi^{-1}, B'=\Phi B\Phi^{-1}$. Similarly, quintuples 
$(\mathcal{V},A,B,\nu_1,\nu_2)$ and $ (\mathcal{V}',A',B',\nu_1',\nu_2')$ are 
\emph{equivalent} if, additionally, $\nu_i'=\Phi^{-1}\circ\nu_i$. In particular, the group 
$\gG\gl({\mathcal{V}})$ naturally acts on quadruples and quintuples defined on 
${\mathcal{V}}$ and their equivalence classes are  labeled by orbits of this action.

The quadruple (resp., quintuple) associated with a compatible pair $X_1\wedge \Xi_1, \;
X_2\wedge \Xi_2$ is constructed on 
$\mathcal{V}=\bar{V}=V^*/\mathrm{span}(\theta_1,\theta_2)$ with $A$ and $B$ being
restrictions of  $\bar{Z}=\bar{X_1}-\bar{X_2}$ 
and $\bar{W}=\bar{X_1}+\bar{X_2}$ to linear functions on $\mathcal{V}$  
(resp., with $\nu_i=\beta_i^0$ for a suitable choice of $V_1^*$), respectively. 
\begin{thm}\label{LieMathing}
Matchings of $n$--dimensional modular Lie algebras are classified by equivalence classes of
quadruples $(\mathcal{V},A,B,\lambda)$ (resp., quintuples $(\mathcal{V},A,B,\nu_1,\nu_2)$), 
if $\lambda\neq 0$ (resp., if $\lambda=0$).
\end{thm}
\begin{proof}
Choose the complement $V_1^*$ so that $\mu_1^0=\mu_2^0=0$. It is not difficult to see that such $V_1^*$ 
exists and is unique. Then functions $\mu_i=\mu_i^1$ are uniquely  defined by relations (\ref{ekl}), 
and (\ref{reduced 0-part}) simplifies to
\begin{eqnarray}\label{simplified 0-part}
A(\beta_1^0)-\lambda\beta_1^0=0,  \quad\quad
A(\beta_2^0)+\lambda\beta_2^0=0.
\end{eqnarray} 
The commutation relation $[\bar{Z},\bar{W}]=2\lambda\bar{Z}$ implies $ [A,B]=2\lambda A$. Hence 
$A$ is nilpotent if $\lambda\neq 0 \;\Rightarrow \;A\pm\lambda\id$ is nondegenerate $\;\Rightarrow \;$
the only solution of (\ref{simplified 0-part}) is $\beta_1^0=\beta_2^0=0$.  Moreover, Lemma 
\ref{div X_0} shows that $\tr\,B=\Div_{\omega_0}\bar{X}_0=2(1+\lambda)$. So, in this case 
the matching is completely characterized by the quadruple $(V_1^*\approx\bar{V},A,B,\lambda)$.
If $\lambda=0$, then (\ref{simplified 0-part}) just tells that $\beta_1^0, \,\beta_2^0\in\Ker\A$ and 
we have no other restrictions on these functions. In this case the matching is completely
characterized by the quintuple $(V_1^*\approx\bar{V},A,B, \beta_1^0,\beta_2^0)$. 

Conversely, by starting from an abstract quadruple (resp., quintuple) one can construct a pair of compatible modular Lie algebras, which is characterized by an equivalent to it quadruple 
(resp., quintuple). Indeed, let $\iota:\mathcal{V}\rightarrow V^*$ be an imbedding. Put 
$V_1^*=\im\,\iota$ and choose a complement $V_0^*$ of $V_1^*$ together with two 
independent vectors $\theta_1,\theta_2\in V_0^*$. Then $\Xi_i$
is defined as the corresponding to $\theta_i$ constant vector field.

Next, if $H:\mathcal{V}\rightarrow\mathcal{V}$  is an operator, then the operator 
$\tilde{H}:V_1^*\rightarrow V_1^*$ is defined as the direct sum of the operator 
$\iota\circ H\circ\iota^{-1}$
on $V_1^*$ and the zero operator on $V_0^*$. Denote by $Y_H$  the corresponding to $\tilde{H}$
linear vector field on $V^*$ and put $Z_0=Y_A, \,W_0=Y_B$.  Finally, define functions
$\alpha_i, \,\beta_i$ and $\mu_i$ by putting 
$$\alpha_i^0=\mu_i^0=0, \,\alpha_1^1=\lambda\varphi_1,  
\,\alpha_2^1=-\lambda\varphi_2, \,\beta_i^1=-\lambda\varphi_i, \,\mu_i^1=-\lambda^2\varphi_i
$$ 
with $\varphi_i$ being the linear function on $V^*$ vanishing on $V_1^*$ and such that 
$\varphi_i(\theta_j)=\delta_{ij}$, and $\beta_i^0=0$ (resp., $\beta_i^0=\pi_1^*(\nu_i)$ if 
$\lambda\neq 0$ (resp., if $\lambda=0$).

Vector fields $Z$ and $W$ (and, therefore, $X_1=\frac{1}{2}(Z+W), X_2=\frac{1}{2}(W-Z))$ are defined by
formula (\ref{fieldsDeco}) with  $\alpha_i$'s and $\beta_i$'s as above. Now a direct check shows that the 
so-constructed linear bivectors $X_i\wedge\Xi_i$ are compatible and modular with $X_i$'s and $\Xi_i$'s  
as in sec.\,3.7. 
\end{proof}

\begin{rmk}
Since representations of bidimensional Lie algebras are well-known, theorem \ref{LieMathing} 
gives an exhaustive description of matchings of modular Lie algebras. Also, it is worth stressing 
that for a given representation the 
trace relation $\tr\,B=2(1+\lambda)$ describes the spectrum of the parameter $\lambda$.
\end{rmk}
\begin{rmk}
It is easy  see that matchings with proportional $\Xi_1$ and $\Xi_2$ are completely characterized 
by quadruples $(\mathcal{V}, A_1,A_2, \nu)$ such that 
$\dim\,\mathcal{V}=n-1, \tr\,A_1=\tr\,A_2=1$ and $0\neq\nu\in\gk$. Namely, 
$\mathcal{V}=V^*/\mathrm{span}(\theta_1), \Xi_2=\nu\Xi_1$ and $A_i$ is the restriction of 
the projected on $\mathcal{V}$ vector field $X_i$ to linear functions on $\mathcal{V}$.  
\end{rmk}

\section{The disassembling problem}\label{dis-probl}

This section is central in this paper. Here we discuss how a Lie algebra can be gradually 
disassembled into some other  Lie algebras.   Here we introduce some basic disassembling 
techniques and then prove  (theorems \ref{C-dis} and \ref{R-dis}) that  any 
finite-dimensional Lie algebra over an algebraically closed field of zero characteristic or 
over  $\R$ can be  assembled in few steps from lions (see below).

In this section ``Lie algebra" refers to a finite dimensional Lie algebra 
over a ground field $\gk$. We start with necessary terminology in order to properly 
state the {\it disassembling  problem}  .  

\subsection{Statement of the problem}\label{statemant}

\par\medskip
\textit{Simple disassemblings and lieons.} 

\begin{defi}\label{base-dis}
A \emph {simple disassembling} of a Lie algebra structure $\gG$ on a vector 
space $V$ is a representation of it as the sum 
\begin{equation}\label{dec}
\gG = \gG_1 + \dots +\gG_k
\end{equation}
of mutually compatible Lie algebra structures $\gG_i$'s on $V$. Lie algebras 
$\gG_i$'s figuring in (\ref{dec}) are called \emph{primary constituents} of 
$\gG$. 
\end{defi}

In such a situation we speak,  slightly abusing the language,  on a (simple) 
disassembling of the Lie algebra $\gG$ into algebras $\gG_1,...,\gG_k$ or, 
alternatively, that  $\gG$ is assembled from $\gG_1,...,\gG_k$. Accordingly,  
we write 
\begin{equation}\label{2-dec}
 \gG_1 + \dots +\gG_k = \gH_1  + \dots + \gH_l,
\end{equation}
in order to express one of the following two facts:
\begin{itemize}
\item a Lie algebra structure on a vector space $V$ admits two 
(different) disassemblings into Lie algebras structures $\gG_i$'s and 
$\gH_i$'s, respectively; 
\item Lie algebras assembled from Lie algebras $\gG_i$'s and $\gH_i$'s, 
respectively, are isomorphic. 
\end{itemize} 

Having disassembled a Lie algebra $\gG$ into constituents $\gG_1,\dots , \gG_k$ 
it is natural to look for further disassembling of $\gG_i$'s and so on. This 
way one gets {\it secondary, ternary}, etc, constituents. The procedure, which is
inverse to such a multi-step disassembling one, will be called an \emph{assembling} 
procedure.
A natural questions arising in this connection is: 
\begin{quote}
{\it What are \emph{``finest" (``simplest")} constituents of which any Lie 
algebra over a given ground field can be assembled?} 
\end{quote}
It will be refereed to as the \emph{disassembling problem}.  

It is not difficult to come to the conclusion that the following Lie algebras 
must be in the list of these ``finest" algebras: 
\begin{itemize}
  \item the 1-dimensional Lie algebra $\boldsymbol{\gamma}$,
  \item the unique non-abelian  2-dimensional Lie algebra $\between$,
  \item the 3-dimensional Heisenberg (over $\gk$) algebra $\pitchfork$.
\end{itemize}
In terms of generators algebras $\between$ and $\pitchfork$ are described 
as follows: 
\begin{align*}
 \between &= \{e_1, e_2 \mid [e_1,e_2]=e_2\}  \\
  \pitchfork &= \{\e_1, \e_2, \e_3 \mid [\e_1,\e_2]=\e_3,\ [\e_1,\e_3]=[\e_2,\e_3]=0\}
  \end{align*}
\noindent They are ``simplest" in any 
reasonable sense of this word. In particular, $\between$ is the ``simplest" 
non-unimodular algebra, while $\pitchfork$ is the ``simplest" nontrivial unimodular one. 

Denote by ${\between}_n, n\geqslant 2$, (resp., ${\pitchfork}_n, n\geqslant 3$) 
the direct sum of $\between$ (resp., of $\pitchfork$) and the 
$(n-2)$-dimensional (resp., $(n-3)$-dimensional) abelian Lie algebra. We will 
also use $\boldsymbol{\gamma}_n$ for $n$-dimensional abelian Lie algebra. 
\begin{defi}\label{lieons}
Lie algebras $\pitchfork_n$ and $\between_n$ are called $n$-dimensional 
$\pitchfork$- and $\between$-\emph{lieons}, respectively.
\end{defi}  

Solution of the disassembling problem for 3-dimensional Lie algebras is not 
difficult (see \cite{MVV}) and is as follows.
\begin{ex}
Any unimodular 3-dimensional Lie algebra can be simply assembled from $l$ 
copies of $\pitchfork$, $l\leq 3$. Any non-unimodular 3-dimensional Lie algebra 
can be simply assembled from $l$ copies of $\pitchfork$, $l\leq 2$, and one 
copy of $\between_3$ . In this sense one can say 
that all 3-dimensional Lie algebras are assembled from $\between$'s and  
$\pitchfork$'s with help of $\boldsymbol{\gamma}$ (in order to construct 
$\between_3$ from $\between$). 
\end{ex} 
\begin{wrapfigure}[13]{o}{3cm}
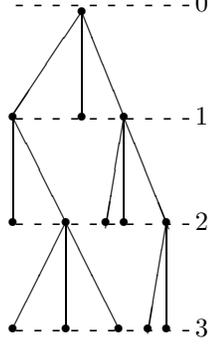
 
\figone
\caption{An a-sche\-me of length~3.}
\end{wrapfigure}

\noindent In this connection see also \cite{Mor} for an explicit description of the algebraic variety 
$\mathrm{Lie}(3)$ of all Lie algebra structures on a 3-dimensiona vector space.

Sometimes it is more expressive to use "chemical" formulas like 
\begin{equation}\label{2h=2b}
 2\pitchfork = 2\between + 2\boldsymbol{\gamma}.
\end{equation}
This formula, which will be proven below, is synonymous to $\pitchfork + 
\pitchfork = \between_3 + \between_3.$ It should be stressed that formulas like 
(\ref{2h=2b}) tell only that a Lie algebra can be in a way assembled both from 
algebras indicated in its left-hand side and from those in its right-hand side. 
\newline
\par\bigskip
\noindent\textit{Assemblage schemas.} 

\noindent Now we pass to a necessary bureaucracy. An {\it assembling scheme} 
(shortly, an {\it a-scheme}) $\gS$ is a finite graph, whose set of vertices 
$vert\,\gS$ is a disjoint union of nonempty subsets $vert_s\gS, \,s=0,\dots ,m$, 
called {\it levels}, such that 
\begin{enumerate}
\item $vert_0\gS$ consists of only one vertex $o_{\gS}$, called
the {\it origin} of $\gS$. 
\item Edges of $\gS$ connect vertices of consecutive levels only. If 
$v_0\in vert_s\gS$ and \\ $v_1\in vert_{s+1}\gS$ are  ends of an edge, 
then they are called \\ the {\it origin} and the {\it end} of this edge, 
respectively. 
\item Any vertex $v\in vert_s\gS , s>0$,  is the end of only one
edge. 
\item None of vertices $v\in vert_s\gS , s<m$, is the origin of only one edge. 
A vertex which is not the origin of an edge is called an {\it end} of $\gS$; 
\end{enumerate} 
The number $m$ is called the {\it length} of $\gS$ and denoted by $|\gS|$. 
Obviously, $\gS$ is a connected graph and there exists at most one edge 
connecting two given vertices of it. All vertices in $vert_m\,\gS$ are ends. 
A-schemes $\gS$ and $\gS'$ are \emph{equivalent} if they are equivalent
as graphs. 

\par\medskip
\noindent\textit{Multi-step disassemblings.} 

\begin{defi}
 Let $\gG$ be a Lie algebra structure on a vector space $V$ and
 $\gS$ be an $a$-scheme. A system $\{\gG_v\}, v\in vert\,\gS$,
 of Lie algebra structures on $V$ is called an $m$-step ($m=|\gS|$)
 \emph{disassembling} of $\gG$ if
 \begin{enumerate}
\item $\gG =\gG_{o_{\gS}}$;
\item If $v_1,\dots,v_p\in vert\,\gS$ are ends of edges having the common
origin v, then structures $\gG_{v_1},\dots ,\gG_{v_p}$ are mutually compatible 
and $\gG_v = \gG_{v_1}+\dots +\gG_{v_p}$. 
 \end{enumerate}
\end{defi}

$\gS$ is the \emph{scheme} of this assembling and we shall speak on a $\gS$-disassembling 
 in order to stress an instance of that.
 The structure $\gG_v, v\in vert_s\gS,$ is called an {\it ($s$-level) term} of the 
$\gS$-disassembling 
It is an {\it end term} of it if $v$ is an end point of $\gS$. We say that $\gG$ 
\emph{is assembled from Lie algebras} $\gG_1,\dots,\gG_r$ if these are in one-to-one 
correspondence with end terms of a disassembling of $\gG$. 

It is worth stressing that if $v$ and $w$ are not ends of two edges of a common 
origin, then  $\gG_v$ and $\gG_w$ are not, generally, compatible.

\begin{defi}\label{finest}
\begin{enumerate}
\item A disassembling of $\gG$ is called \emph{complete} if all its end terms are 
isomorphic either to $\between_n$, or to $\pitchfork_n$. 
\item Disassemblings of isomorphic Lie algebras $\gG$ and $\gH$ are called 
\emph{equivalent} if there exist an equivalence $\sigma:vert\,\gS\rightarrow vert\,\gS'$ 
of the corresponding a-schemes and an isomorphism $\varphi:\gG\rightarrow\gH$ 
which is also an isomorphism of $\gG_v$ onto $\gH_{\sigma(v)}, \,\forall v\in vert\,\gS$.
\end{enumerate}
\end{defi}

Obviously, nonequivalent disassemblings can have equivalent a-schemes. 
The above-mentioned mentioned formula $\pitchfork+\pitchfork=\between_3+\between_3$
is a simple example of that. 

\par\medskip
\noindent\textit{The disassembling problem.}

\noindent  The \emph{disassembling problem} is the question:
\begin{quote}
{\it Whether a given Lie algebra can be completely disassembled?}
\end{quote}
Below we shall develop some disassembling techniques and prove that any Lie 
algebra over an algebraically closed field or over $\R$ 
can be completely disassembled. This result confirms that lieons are 
elementary constituents of which all Lie algebras are made. 
By mimicking physical terminology one may say that $\gamma$ creates the 
necessary ``vacuum", which makes possible interactions between constituents  
$\between$ and $\pitchfork$ of  ``Lie matter".

In this connection it should be mentioned  that the number of elementary 
constituents for Lie algebras can not be reduced to one. Indeed, according 
to proposition\,\ref{PFA}, by 
assembling unimodular Lie algebras one can only get unimodular ones. 
So, only unimodular Lie algebras can be assembled from $\pitchfork$-lieons. 
On the other hand, it is not difficult to show (see \cite{MVV})
that $\pitchfork$ can not be assembled only from $\between$-lieons  (compare 
with (\ref{2h=2b})\,!). Hence the algebra $\pitchfork$ can not be excluded from 
the list of ``finest" Lie algebras. 

Now we pass to some basic techniques and constructions that will be used in our 
analysis of the disassembling problem. 

\par\bigskip
\subsection{Reduction to Solvable and Semisimple Algebras.} \label{ReductionSolvable}

\noindent First of all, we shall show that the problem naturally splits into 
``solvable" and ``semisimple" parts. The semidirect sum of a Lie algebra $\gA$
and of the abelian Lie algebra structure on a vector space $V$, which is defined  
by a representation $\rho:\gA\rightarrow \rm{End}\,V$ will be denoted by 
$\gA\oplus_{\rho}V$.
\begin{proc}\label{s-tr}
Let algebra $\gG$  be the semidirect sum of an its subalgebra $\gG_0$ and an 
ideal $\gH$. Identifying $|\gG|$ and $|\gG_0|\oplus |\gH|$ we have the 
simple disassembling 
\begin{equation}\label{csp}
\gG =(\gG_0 \oplus_{\rho}|\gH |)\,+\,(\bg_m\oplus \gH),
\end{equation}
with $\rho$  being the canonical representation of $\gG_0$ in $|\gH|$\,  and \, 
$\bg_m, \,m=\dim\,\gG_0$,  the abelian structure on $|\gG_0|$. 
\end{proc}
\begin{proof}
By construction. 
\end{proof} 

Now, apply proposition \ref{s-tr} to the Levi-Malcev decomposition 
$\gG=\gH\oplus_{\rho}\gR$   of a Lie algebra $\gG$. Here $\gR$ 
is the radical of $\gG$  and $\gH\subset\gG$ a 
complementing $\gR$ semisimple subalgebra. 
\begin{cor}\label{dis-reduct}
The disassembling problem for Lie algebras over a field $\gk$ of 
characteristic zero reduces to that for solvable algebras (over $\gk$) and for 
abelian extensions of semisimle algebras, i.e., algebras of the form 
$\gH\oplus_{\rho}W$ with $\rho:\gH\rightarrow \rm{End}\,W$  
being a finite-dimensional representation of a semisimple algebra $\gH$ (over 
$\gk$). 
\end{cor}

\par\smallskip
\noindent\textit{Solvable Algebras.} 

\noindent The first of these two problems admits a simple solution. 
\begin{proc}\label{dis-solv}
Any solvable Lie algebra over a field $\gk$ can be completely disassembled. 
\end{proc}
\begin{proof}
Let $\gG$  be a solvable algebra. Then any subspace of $|\gG|$ containing the 
derived algebra $[\gG , \gG]$ is, obviously, an ideal of $\gG$. Consider such 
an ideal $\gs$ of codimension one and a complementing it one-dimensional 
subspace, which is automatically a subalgebra of $\gG$. This makes evident that 
$\gG$ is a semidirect product $\bg\oplus_{\rho}\gs$. By applying to it proposition 
\ref{s-tr} we see that $\gG$ can be disassembled into 
two structures, one of which is $\bg\oplus\gs$ with $\gs$ being a solvable algebra, 
while  the other one is of the form $\bg\oplus_{\rho}V$ with $\rho$ being a 
representation of $\bg$ in the vector space $V=|\gs|$. Now obvious induction 
arguments reduce the problem to disassembling of algebras of the latter type. 
This can be done as follows. 

Fix a base element $\nu\in \bg$ and put $A=\rho(\nu):V\rightarrow V, 
\,\,0\neq\nu\in \bg$. The product in the algebra $\bg\oplus_{\rho}V$ is  given by
$$
[\nu,v]=Av, \,\,[v_1,v_2]=0, \,\,\,\nu\in\bg, \,v_1,v_2\in V.
$$
i.e., is completely determined by the operator $A$. Denote the so-defined algebra 
by $\Gamma_A$. Since the operator $A$ in this construction is defined up to a 
scalar factor, algebras $\Gamma_A$  and $\Gamma_{\lambda A}, \,\,\lambda\in 
\gk$,  \,are isomorphic. So, it remains to show that algebras $\Gamma_A$'s can 
be completely disassembled. 

First, note that $\between=\Gamma_A$  if $A$ is the identity operator on an 
1-dimensional vector space $V$  and $\pitchfork =\Gamma_A$  if $A$ is a 
nontrivial nilpotent operator on a 2-dimensional vector space $V$. 

Second, if $\|a_{ij}\|$  is the matrix of $A$ in a basis $\{ e_i\}$ of $V$, 
then $A=\sum a_{ij}E_{ij}$, \,where the operator $E_{ij}:V\rightarrow V$   is 
defined by $E_{ij}(e_i)=e_j$ and $E_{ij}(e_k)=0, \,\,k\neq i$. Now it is easy 
to see that the structure $\Gamma_{E_{ij}}$ is isomorphic to $\pitchfork_n$,  
if $i\neq j$,  and to $\between_n$, if $i=j$. Finally,  since 
$\Gamma_{a_{ij}E_{ij}}=\Gamma_{E_{ij}}$, if $a_{ij}\neq 0$, then 
\begin{equation}\label{g-dec}
\Gamma_A=\sum_{i,j,\,a_{ij}\neq 0}\Gamma_{E_{ij}},
\end{equation}
which is the desired disassembling. 
\end{proof}

Disassembling (\ref{g-dec}) depends on the choice of a base in $V$. This fact 
can be used to illustrate nonuniqueness of complete disassemblings of a given 
algebra. For instance, let $\dim\,V=2$ and $A:V\rightarrow V$ be an operator, 
with eigenvalues $\pm 1$. Then in the basis of eigenvectors $e_1$ and $e_2$  
disassembling (\ref{g-dec}) is $\Gamma_A=\Gamma_{E_{11}}+\Gamma_{E_{22}}$, 
i.e., symbolically, $\Gamma_A=2\between_3$. On the other hand, in the basis 
$\{e_1 + e_2, \,\,e_1-e_2\}$ we have 
$\Gamma_A=\Gamma_{E_{12}}+\Gamma_{E_{21}}\Leftrightarrow \Gamma_A=2\pitchfork$. 
This proves formula (\ref{2h=2b}). 

The d-scheme of the complete disassembling procedure for a solvable algebra 
described above is as in fig.\,2. 

\bigskip
\begin{figure}[h]
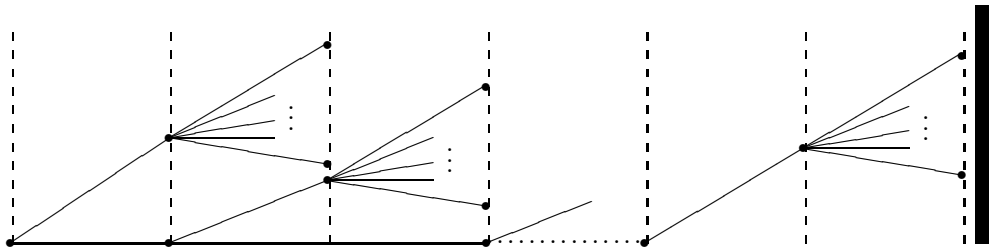

\figtwo
\caption{A-scheme of a complete disassembling of a solvable Lie algebra.}
\end{figure}

\par\medskip
\noindent\textit{From Semisimple to Simple Algebras.} 

\noindent The disassembling problem  for algebras $\gG\oplus_{\rho}V$  with
semisimple $\gG$ is easily reduced to  that for simple $\gG$.
Indeed, observe that the direct sum of Lie algebras 
$\gG=\gG_1\oplus\dots\oplus\gG_k$ is a natural assemblage 
$\gG=\bar{\gG}_1+\dots +\bar{\gG}_k$, where the structure $\bar{\gG}_i$ on 
$|\gG|=|\gG_1|\oplus\dots|\gG_k|$ is the direct sum of abelian structures on 
$|\gG_j|$'s for $ j\neq i$ and the structure $\gG_i$ on $|\gG_i|$. If  
$\rho$ is a representation of $\gG$ in $V$, then the representation $\rho_i$
of $\bar{\gG_i}$ in $V$ is defined to be trivial on $|\gG_j|, \;j\neq i,$ and coinciding 
with $\rho$ on $\gG_i$. Then we have 
\begin{proc}\label{dec-sum}
\begin{equation}\label{dec-repr}
\gG\oplus_{\rho}V=\bar{\gG}_1\oplus_{\rho_1}V+\dots+\bar{\gG}_k\oplus_{\rho_k}V,
\end{equation}
 is a simple disassembling of $\gG\oplus_{\rho}V$. 
\end{proc} 
\begin{proof}
By construction. 
\end{proof}

\begin{cor}\label{dec-ss}
Let $\gG$ be a semisimple Lie algebra and $\gG=\gG_1\oplus\dots\oplus\gG_k$  
its decomposition into a sum of simple algebras. If $\rho$ is a 
representation of $\gG$ in $V$, then 
\begin{equation}
\gG\oplus_{\rho}V=(\gG_1\oplus\gamma_{l_1})\oplus_{\rho_1}V+\dots+
(\gG_k\oplus\gamma_{l_k})\oplus_{\rho_k}V,
\end{equation} 
where $l_i=\dim\;\gG-\dim\;\gG_i, \,i=1,\dots,k$, is a simple disassembling of 
$(\gG_i\oplus_{\rho}V)$. 
\end{cor}
\begin{proof} Just to observe that in the considered case the algebra $\bar{\gG}_i$ in
(\ref{dec-repr}) is isomorphic to $\gG_i\oplus\gamma_{l_i}, \,i=1,\dots,k$. 
\end{proof}

Thus, proposition \ref{dis-solv} and corollary \ref{dec-ss} reduce the 
disassembling problem to abelian extensions of of simple algebras, i.e., Lie algebras 
of the form $\gG\oplus_\rho V$ with simple 
$\gG$. The {\it stripping procedure} we are passing to describe will be  our main tool 
in disassembling  algebras of this kind. 

\subsection{The stripping procedure}\label{striptease}

First, we shall introduce some special algebras which are used in this procedure.

\par\medskip
\noindent\textit{Dressing algebras.} 

\noindent A {\it dressing algebra} is defined on the direct sum $W_0\oplus W$ 
of two vector spaces $W_0$ and $W$ by means of a bilinear skew-symmetric 
${W_0}$-valued form $\beta :W\times W\rightarrow W_0$. The product in this algebra
is defined by formula   
\begin{equation}
[(w_0,w),(w_0',w')]=(\beta (w,w'),0),\quad w_0, w_0^{\prime}\in W_0,\,\,w, w^{\prime}\in W.
\end{equation}
Denote the so-defined algebra by $\gA_{\beta}$. If $\dim\,W=2, \,\,\dim\,W_0=1$ 
and $\beta\neq 0$, then $\gA_{\beta}$ is isomorphic to 
$\pitchfork$.

A Lie algebra $\gA$ is isomorphic to a dressing one iff the derived subalgebra 
$[\gA,\gA]$ belongs to its center. Indeed, in such a case one can take the 
center for $W$ and any complementary to the center subspace for $W_0$. 

\begin{proc}\label{dress}
Let $\beta$  and $\beta '$ be $W_0$-valued skew-symmetric bilinear forms on 
$W$. Then Lie algebras $\gA_{\beta}$  and $\gA_{\beta '}$ are compatible. 
Moreover, any dressing algebra can be simply disassembled into a number of 
$\pitchfork$-lieons. 
\end{proc}
\begin{proof} Obviously,
\begin{equation}
\gA_{\beta + \beta '}=\gA_{\beta}+\gA_{\beta '}.
\end{equation}
Hence $\gA_{\beta}$  and $\gA_{\beta '}$ are compatible. 

Choose a base $e_1,\dots,e_m$ in $W_0$  and a base
$\varepsilon_1\dots,\varepsilon_{n-m}$ in $W$. Then 
$$
\beta =\sum_{i,j,k}\beta_{ij}^k\varepsilon_k
$$
for some $\gk$-valued skew-symmetric bilinear forms $\beta_{ij}^k$ on $W$ 
such that $\beta_{ij}^k(e_p,e_q)=0$ if $(p,q)$ differs from $(i,j)$ and $(j,i)$. If 
$\beta_{ij}^k\neq 0$, then the algebra $\gA_{\beta_{ij|k}}$ with 
$\beta_{ij|k}=\beta_{ij}^k\varepsilon_k$ is isomorphic to $\pitchfork_n$. 
Hence, 
\begin{equation}\label{dec-dress}
\gA_{\beta}={\sum_{i,j,k}}'\gA_{\beta_{ij|k}}, \quad
\end{equation}
where the summation $\sum_{i,j,k}'$ is extended on all triples $i,j,k$ for 
which $\beta_{ij}^k\neq 0$. 
\end {proof}

\medskip
\noindent\textit{D-pairs and involutions}.

\noindent Let now $\gG$ be a Lie algebra, $\gs$ an its subalgebra and $W$  be a 
complement of $|\gs |$ in $|\gG|$. If $[W,W]\subset\gs$ and $[\gs,W]\subset W$ 
the pair $(\gs,W)$  \,is called a ${\it d-pair}$ in $\gG$. A $d$-pair $(\gs,W)$ 
is {\it trivial} if $W$ \,is an abelian subalgebra. 

The dressing algebra $\gA_{\beta}$  defined on $|\gG|$ with $W_0=|\gs|$ and 
$\beta(w_1,w_2)=[w_1,w_2], \\ \,w_1,w_2\in W$ will be called {\it associated} with 
$(\gs,W)$. 

\begin{ex}\label{pair}
Let $V$ be a vector space and $V_1, V_2$ be its subspaces complementary one 
to  another. Consider the  subalgebra $\gs=\gs(V_1,V_2)$ of the Lie algebra 
$\gG\gl(V)$ composed of operators, leaving $V_1$ and $V_2$ invariant. The 
linear subspace $W=W(V_1,V_2)$ formed by operators, sending $V_1$ to $V_2$ and 
conversely, is a complement of $\gs(V_1,V_2)$ in $\gG\gl(V)$. Then ($\gs,W$) is 
a $d$-pair in $\gG\gl(V)$. 
\end{ex}
\begin{rmk}
A $d$-pair $(\gs,W)$ in $\gG$ supplies $\gG$ with a structure of a graded 
$\dF_2$-algebra $(\dF_2=\Z/2\Z)$ and vise versa. Namely, if
$\gG$=$\gG_0\oplus\gG_1$, then $\gs=\gG_0, \,W=\gG_1$. 
\end{rmk}
The involution $I:|\gG| \rightarrow |\gG|, \,I^2=id_{|\gG |}$ with $|\gs |$ and $W$ being 
eigenspaces corresponding to eigenvalues $1$ and $-1$, respectively, is naturally 
associated with a $d$-pair $(\gs,W)$. Obviuosly, $I$ is an automorphism of 
$\gG$. Conversely, $\pm 1$-eigenspaces of an involutive automorphism $I$ of 
$\gG$ form a $d$-pair in $\gG$. So, there is a one-to-one correspondence between
d-pairs and involutions of $\gG$. It depends on the context, which of these points of
view is more convenient.
\begin{ex}\label{d-gl-pair}
The matrix transposition $T:M\mapsto M^t, \,M\in \gG\gl(n,\gk)$ is an 
anti-automorphism of $\gG\gl(n,\gk)$, i.e., $[M,N]^t=[N^t,M^t]$. So,  
$\mathfrak{t}=-T$ is an involution  of $\gG\gl(n,\gk)$. The $1$-eigenspace of 
$\mathfrak{t}$ is formed by skew-symmetric matrices and is identified with the 
special orthogonal subalgebra $\gs\go(n,\gk)\subset \gG\gl(n,\gk)$, while the 
$(-1)$-eigenspace $S(n,\gk)$ consists of symmetric matrices. So, 
$(\gs\go(n,\gk),S(n,\gk))$ is a $d$-pair in $\gG\gl(n,\gk)$. The subalgebra 
$\gs\gl(n,\gk)\subset\gG\gl(n,\gk)$ of traceless matrices is 
$\mathfrak{t}$-invariant. Hence $\mathfrak{t}_0=\mathfrak{t}|_{\gs\gl(n,\gk)}$ 
is an involution in $\gs\gl(n,\gk)$ and $(\gs\go(n,\gk),S_0(n,\gk))$ with 
$S_0(n,\gk))$ being the space of symmetric traceless matrices is the 
corresponding $d$-pair.  
\end{ex} 

\bigskip
\noindent\textit{The stripping procedure}. 

 \noindent The following evident fact is, nevertheless, one of most efficient in
disassembling techniques.
 \begin{lem}[The Stripping Lemma]\label{str}
Let $\gG$ be a Lie algebra, $(\gs,W)$  be a $d$-pair in it and $\gA_{\beta}$ be 
the associated dressing algebra. Then 
\begin{equation}\label{strip}
\gG=(\gs\oplus_{\rho}W)+\gA_{\beta},
\end{equation}
with $\rho$ being the restriction of the adjoint representation of $\gG$ to 
$\gs$  is a simple disassembling of $\gG$. 
 \end{lem}
 \begin{proof} By construction.
 \end{proof}
 \begin{rmk}The dressing algebra $\gA_{\beta}$ may be viewed as a ``mantle" that 
 covers ``shoulders" $\gs$ of $\gs\oplus_{\rho}W$. This motivates the  terminology. 
 According to proposition\,\ref{dress},  $\gA_{\beta}$ can be completely 
 disassembled. This reduces the disassembling problem  for $\gG$ to a simpler algebra, 
 namely, $\gs\oplus_{\rho}W$. 
 \end{rmk}

The {\it stripping procedure} consists of consecutive applications of the 
Stripping Lemma which gradually simplify appearing in its course algebras.
In order to give due rigor  to the term "simplification", we 
define the {\it complexity} $l(\gG)$ of a Lie algebra  $\gG$ as the dimension 
of its "semi-simple part", i.e., of an its Levi subalgebra. Algebras of complexity zero are 
solvable and, according to proposition \ref{dis-solv}, can be completely 
disassembled. Therefore, we see that
\begin{quote}
\emph{all Lie algebras over a given ground field $\gk$ can be 
completely disassembled if any algebra of the form $\gG\oplus_{\rho}V$ with 
simple $\gG$ admits a $d$-pair $(\gs,W)$ such that $l(\gs)<l(\gG)$.} 
\end{quote}
Indeed, the 
dressing algebra $\gA_{\beta}$  in the corresponding disassembling 
$\gG\oplus_{\rho}V=\gs\oplus_{\rho'}W+\gA_{\beta}$ (see (\ref{strip})) can be 
completely disassembled (proposition (\ref{dec-dress})).  On the other hand, 
$l(\gs\oplus_{\rho'}W)=l(\gs)<l(\gG)$. So, by applying proposition \ref{s-tr} 
to the Levi-Malcev decomposition of the algebra $\gs\oplus_{\rho'}W$ we reduce 
the problem to an algebra of the form $\bar{\gG}\oplus_{\bar{\rho}}\bar{V}$ with 
$\bar{\gG}$ being the semisimple  part of $\gs$, since 
$l(\bar{\gG}\oplus_{\bar{\rho}}\bar{V})=l(\bar{\gG})=l(\gs)<\l(\gG)$. Finally, 
according to Proposition \ref{dec-sum}, the algebra 
$\bar{\gG}\oplus_{\bar{\rho}}\bar{V}$ disassembles into algebras of the form 
$\gH\oplus_{\tau}U$ with $l(\gH)\leq l(\bar{\gG})$ and simple $\gH$. 

We shall call a $d$-pair $(\gs,W)$ in a Lie algebra $\gG$ as well as the 
corresponding to it $d$-involution {\it simplifying} if the complexity of $\gs$ 
is lesser than that of $\gG$. In the rest of this section we shall 
concentrate on existence of simplifying $d$-pairs for Lie algebras of the 
form $\gG\oplus_{\rho}V$ \,with simple $\gG$.  

\bigskip
\noindent\textit{Multi-involution disassembling.} 

\noindent 
Keeping in mind that the disassembling problem is reduced to abelian extensions of simple 
Lie algebras we shall adopt the Stripping Lemma to semidirect products 
$\gG=\gH\oplus_{\rho}V$. It is convenient to put the question in a more general context.

Let $P_1,\dots ,P_l$ be commuting involutions of a Lie algebra $\gG$. Denote by 
$\dF_2^l$ the algebra of $\dF_2$-valued $l$-vectors with coordinate-wise 
multiplication. Let $\varsigma=(\varsigma_1,\dots,\varsigma_l)\in\dF_2^l$.
The common eigenspace  of involutions $P_1,\dots,P_l$, which correspond to 
their eigenvalues $\lambda_i=(-1)^{\varsigma_i},  \;i=1,\dots,l$, 
will be denoted by $|\gG |_{\varsigma}$. Then 
\begin{equation}
|\gG |=\bigoplus_{\varsigma\in \dF_2^l}|\gG |_{\varsigma}
\end{equation}
Obviously, $[|\gG |_{\varsigma},|\gG |_{\sigma}]=|\gG |_{\varsigma+\sigma}$.
Associate with any $\varsigma\in \dF_2^l$  a skew-symmetric algebra 
structure denoted  by $\gG_{\varsigma}$ on $|\gG |$ with the product $[\cdot,\cdot]_{\varsigma}$ 
defined  on homogenous elements by the formula 
\begin{equation}
[u,v]_{\varsigma}=[u,v], \;\mbox{if} \;\;\xi \cdot\tau=\varsigma, 
\quad u\in |\gG |_{\xi}, \;v\in |\gG |_{\tau},\quad
\mbox{and zero otherwise}.
\end{equation}

\begin{proc}\label{multi-inv} $\gG_{\varsigma}$ is a Lie algebra structure on $|\gG |$. 
\end{proc}
\begin{proof}First, we have to check the Jacobi identity for the bracket 
$[\cdot,\cdot]_{\varsigma}$.  As it directly follows from the definition,
the double bracket $[u,[v,w]_{\varsigma}]_{\varsigma}$
with $\,u\in\gG_{\mu}, 
v\in\gG_{\nu}, w\in\gG_{\xi},$ can be different from zero only if $\varsigma=0$ 
and $\mu\nu=\nu\xi=\xi\mu=0$. In this case $[\cdot,\cdot]_{\varsigma}=[u,[v,w]]$
and  $[u,[v,w]_{\varsigma}]_{\varsigma}+cycle=[u,[v,w]]+cycle=0$.
On the other hand, if $\varsigma, \mu, \nu, \xi$ do not satisfy the above condition, 
all double commutators of elements $u,v$ and $w$ with respect to 
$[\cdot,\cdot]_{\varsigma}$ vanish. 
\end{proof}

Generally, Lie algebras $\gG_{\varsigma}$'s are not mutually compatible. Nevertheless, 
some their combinations implicitly appear in the \emph{multi-involution disassembling 
procedure} which is described below.

Let $I_0$ be the involution of $\gG=\gH\oplus_{\rho}V$ whose proper subspaces 
corresponding to eigenvalues $1$ and $-1$ are $|\gH|$ and $V$, respectively. An 
involution $I$ (or the corresponding d-pair) of $\gG$, which commutes with $I_0$, 
will be called \emph{adopted} (to the semidirect sum structure of $\gG$).
Obviously, both $|\gH|$ and $V$ are $I$-invariant. Let 
$\gH=\gH_0\oplus\gH_1, \,V=V_0\oplus V_1$ be the splittings into proper subspaces  
corresponding to eigenvalues $1$ and $-1$ of $I$, respectively.  The associated with 
$I$ d-pair is $(\gH_0\oplus_{\rho_0} V_0, |\gH_1\oplus_{\rho_1} V_1|)$ where $\rho_i$ 
stands for the restriction of $\rho$ to $\gH_i$ and $V_i$. By removing, according to the 
Stripping Lemma, the associated dressing algebra we get the Lie algebra 
\begin{equation}\label{newSemidirect}
(\gH_0\oplus_{\rho_0}V_0)\oplus_{\varrho}(|\gH_1|\oplus V_1)
\end{equation}
with the representation $\varrho$ defined by formulas
$$
\varrho(h_0)(h_1)=[h_0,h_1],  \;\varrho(h_0)(v_1)=\rho(h_0)(v_1), \;\varrho(v_0)(h_1)=
-\rho(h_1)(v_0), \;\varrho(v_0)(v_1)=0
$$
where $h_i\in \gH_i, v_i\in V_i, \;i=0,1$.  On the other hand, algebra (\ref{newSemidirect}) 
may be viewed as the semidirect product of $\gH_0$ and the ideal $\mathcal{I}$ whose 
support is $V_0\oplus|\gH_1|\oplus V_1$. The product $[\cdot,\cdot]'$ in this ideal is such 
that $[V_0,|\gH_1|]'\subset V_1, \,[V_0,V_1]'=[|\gH_1|,V_1]'=0$. So, $\mathcal{I}$ is 
nilpotent and as such can be completely disassembled. Since
$$
(\gH_0\oplus_{\rho_0}V_0)\oplus_{\varrho}(|\gH_1|\oplus V_1)=
\gH_0\oplus_{\varrho_0}|\mathcal{I}|+\gamma_m\oplus\mathcal{I}, \quad m=\dim\,\gH_0, 
$$
with $\varrho_0$ being the direct sum of natural actions of $\gH_0$ on $|\gH_1|, \,V_0$ 
and $V_1$ the disassembling problem for the algebra (\ref{newSemidirect}) and, 
therefore, for $\gG$ is reduced to that for $\gH_0\oplus_{\varrho_0}|{\mathcal{I}}|$. The 
passage from $\gH\oplus_{\rho}V$ to $\gH_0\oplus_{\varrho_0}|{\mathcal{I}}|$ will be 
called the \emph{stripping of the semidirect product $\gH\oplus_{\rho}V$ by $I$}.
\begin{proc}\label{stripping-semi}
Let $P_1,\dots ,P_l$ be commuting involutions of a Lie algebra $\gG$. Then $\gG$ can be assembled from lieons
and the algebra $\gG_{(0,\dots,0)}\oplus_{\rho} W$ where $W=\oplus_{0\neq\zeta}|\gG_{\zeta}|$
and $\rho$ is the direct sum of natural actions of $\gG_{(0,\dots,0)}$ on $\gG_{\zeta}$'s.
\end{proc}
\begin{proof}
This is an inductive procedure. First, we use the involution $I_1$ to show (by the 
Stripping Lemma) that $\gG$ can be assembled from lions and the algebra  $\gG_0\oplus_{\varrho_0}W_0, \,W_0=|\gG_1|$, where $|\gG_0|$ and $|\gG_1|$ are 
proper subspaces of $I_1$ corresponding to eigenvalues $1$ and $-1$, respectively, 
and $\varrho_0$ is a natural action of $\gG_0$ on $|\gG_1|$. Since the involution 
$I_2$ commute with $I_1$, it leaves invariant both $|\gG_0|$ and $|\gG_1|$ and, 
therefore, induces an adopted involution $I$ of the algebra 
$\gG_0\oplus_{\varrho_0}W_0$. Now, by stripping the semidirect product 
$\gG_0\oplus_{\varrho_0}W_0$ by $I$, we see that $\gG$ is assembled from lions 
and $\gG_{(0,0)}\oplus_{\varrho_1}W_1$ where $W_1=\oplus_{0\neq\zeta}|\gG_{\zeta}|$ 
with $\zeta\in\dF_2^2$. By continuing this process we get the desired result.
\end{proof}

\bigskip
\noindent\textit{Complete disassembling of classical Lie algebras.}

\noindent One simple application of proposition\,\ref{stripping-semi}
is the following.
\begin{proc}\label{dis-classics}
Lie algebras  $\gs\gl(n,\gk), \,\go(n,\gk), \,\gs\go(n,\gk), \,\gu, \,\gs\gu$ can 
be completely disassembled.
\end{proc}
\begin{proof}
Let $\{e_1,\dots,e_n\}$ be a basis in a $\gk$-vector space $V$. We shall 
identify operators from $\mathrm{End}\,V$ and their matrices in this basis. 
Consider the $d$-pair $(\gs_j,W_j)$ in $\gG=\gG\gl(V)$ associated, 
by the construction of example \ref{pair}, with subspaces 
$V_1=\mathrm{span}\{e_j\}$ and $V_2=\mathrm{span}\{e_1,\dots,\hat 
e_j,\dots,e_n\}$, and denote by $I_j$ the corresponding to it involution of 
$\gG\gl(V)$. Then involutions $I_1,\dots,I_n$ commute each other.  

It is easy to see that the common proper subspace  
$|\gG_{\zeta}|, \zeta\in \dF_2^n$, of involutions $I_1,\dots,I_n$, is 
different from zero iff the unit occurs in $\zeta$ either zero, or $2$ 
times. In the first case the subspace $|\gG_{(0,\dots,0)}|$ is 
composed of operators for which $e_1,\dots,e_n$ are eigenvectors, i.e., 
it consists of diagonal matrices. In the second case, let $\zeta=(ij)\in \dF_2^n, 
i\neq j$, be the $\dF_2$--vector with two nonzero components on the $i$-th 
and $j$-th places. Then the subspace $|\gG_{(ij)}|$ consists of operators 
$\lambda \varepsilon_{ij}+\mu\varepsilon_{ji}, \;\lambda,\mu\in \gk$, with 
$\varepsilon_{ij}$ being the operator sending $e_j$ to $e_i$ and 
annihilating  $e_k$'s for $k\neq j$. So, in the considered situation the algebra 
$\gG_{(0,\dots,0)}$ is abelian. Therefore, the algebra $\gG_{(0,\dots,0)}\oplus_{\rho} W$,
is solvable and as such can be completely disassembled. Now it directly follows from 
proposition\,\ref{stripping-semi} that the algebra $\gG\gl(n,\gk)=\gG\gl(V)$ can be 
completely disassembled too.

Algebras $\gs\gl(n,\gk), \,\go(n,\gk), \,\gs\go(n,\gk)$ are invariant with respect to 
the above constructed involutions $I_j$'s. Obviously, for each of them the subspace
$|\gG_{\zeta}|$ is a subspace of the corresponding subspace for the algebra $\gG\gl(n,k)$.
In particular, this shows that the algebra $\gG_{(0,\dots,0)}\oplus W$ is solvable, and
proposition\,\ref{stripping-semi} gives the desired result.

In order to completely disassemble the symplectic algebra $\gs\gp(n,\gk)$ the preceding
procedure must be slightly modified. Let $\sigma(u,v)$ be a symplectic form on 
$V, \,\dim\,V=2n$. We interpret the algebra  $\gs\gp(n,\gk)$ as the algebra
$$
\gs\gp(\sigma)=\{A\in\mathrm{End}\;V \;|\; \sigma(Au,v)+\sigma(u,Av)=0, \,\forall u,v\in V\}.
$$
Let $V=V_1\oplus\dots\oplus V_n, \,\dim\,V_i=2, \forall i$, be a 
$\sigma$--orthogonal decomposition of $V$ and $P_i:V\rightarrow V$ the associated 
projector on $V_i$. Then $I_i=\id_V-2P_i$ is an involution of $\gG\gl(V)$. It is easy to 
see that involutions $I_i$'s commute and their common proper subspace, on which they 
all are the identity, is $\gs\gp(\sigma_1)\oplus,\dots,\gs\gp(\sigma_n)$ with  \quad 
$\sigma_i=\sigma|_{V_i}$ (=$\gG_{(0.\dots,0)}$ in the notation of 
proposition\,\ref{stripping-semi}). Note that $\gs\gp(\sigma_i)$ is isomorphic to 
$\gs\gp(2,\gk)$. So, in the considered case, in the contrast with the preceding case 
the algebra $\gG_{(0.\dots,0)}\oplus W$ of proposition\,\ref{stripping-semi} is not 
solvable. So, it can not be completely disassembled on the basis of our previous 
results. This small difficulty can be resolved by introducing an additional involution.  

Namely, let $J_i:V_i\rightarrow V_i$ be a complex structure on $V_i$ compatible with 
$\sigma_i$, i.e., $J_i^2=-\id_{V_i}$ and $\sigma_i(J_iu,v)+\sigma_i(u,J_iv)=0$, and 
$J=J_1\oplus\dots\oplus J_n$. Then $\sigma(Ju,v)+\sigma(u,Jv)=0$ and 
$I_0: \End\,V\rightarrow\End\,V, \,A\mapsto -JAJ$ is an involution  which leaves invariant 
the subalgebra $\gs\gp(\sigma)$. Moreover, $I_0$ commutes with involutions 
$I_1,\dots,I_n$, and their common proper subspace on which they act as identity is the abelian subalgebra $\gH$ composed of elements $\lambda_1J_1\oplus\dots\oplus\lambda_nJ_n, 
\,\lambda_1,\dots,\lambda_n\in\gk$.   Now, by applying proposition\,\ref{stripping-semi} to 
involutions $I_0,I_1,\dots,I_n$ and taking into account that in this case the algebra 
$\gG_{(0.\dots,0)}\oplus W=\gH\oplus W$ is solvable, we see that $\gs\gp(\sigma)$ can 
be completely disassembled.

Similar arguments shows that $\gu(n)$ and $\gs\gu(n)$ can be assembled from lions. 
To this end, interpret an $n$--dimensional vector space as a $2n$--dimensional 
$\R$--vector space $V$ supplied with a complex structure $J, \,J^2=-\id_V$ and split 
$V$ into a direct sum of 2-dimensional $J$-invariant subspaces. Then  consider, as 
above, the corresponding involutions $I_1,\dots,I_n$. In this case the algebra 
$\gG_{(0.\dots,0)}$ is abelian. Namely, it consists, exactly as before, of elements 
$\lambda_1J_1\oplus\dots\oplus\lambda_nJ_n, \,\lambda_1,\dots,\lambda_n\in\gk$, 
where $J_i=J|_{V_i}$.  Hence proposition\,\ref{stripping-semi} gives the desired result.
\end{proof}

This proof of proposition\,\ref{stripping-semi} is not very constructive in the sense that 
the corresponding a--scheme is rather complicated to efficiently work with. We reported 
it here with the aim to illustrate the Stripping Lemma at work. A short and constructive 
procedure of complete disassembling of
classical Lie algebras will be described in section\,\ref{classical}.

In fact, any simple Lie algebra $\gG$ over a ground field $\gk$ possesses a nontrivial involution. 
So, the next
natural question is wether it can be extended to the algebra $\gG\oplus_{\rho}V$.

\subsection{Extensions of $d$-pairs/involutions.} 

An extension of a $d$-pair in a Lie algebras $\gG$ to an adopted d-pair
in $\gG\oplus_{\rho}V$ is described as follows. Let $\rho$ be a representation of 
$\gG$ in a vector space $V$ and $(\gs,W)$  a $d$-pair in $\gG$. A decomposition 
$V=V_0\oplus V_1$ will be called an $\rho$-{\it extension} of $(\gs,W)$  if 
\begin{enumerate}
\item $V_i$ \,is invariant with respect to operators $\rho (s), s\in
\gs, i=0,1$; 
\item $\rho(w)(V_0)\subset V_1,\,\, \rho(w)(V_1)\subset V_0$, \,if  $w\in W$. 
\end{enumerate}
\begin{lem}
If $V=V_0\oplus V_1$ is a $\rho$--extension of $(\gs,W)$, then 
$(\gs\oplus_{\rho|_{_{\gs}}}V_0,\,W\oplus V_1)$ is a $d$-pair in the Lie 
algebra $\gG\oplus_{\rho}V$. 
\end{lem}
\begin{proof}
Obviously, $(\gs\oplus_{\rho|_{_{\gs}}}V_0) $  is a subalgebra of 
$\gG\oplus_{\rho}V$. Denoting by $[\cdot,\cdot]_{\rho}$ the Lie product in 
$\gG\oplus_{\rho}V$ we have
\begin{align*}
[(s,v_0),(w,v_1)]_{\rho} &=([s,w],\rho(s)(v_1)-\rho(w)(v_0))\in
W\oplus V_1\\
[(w,v_1),(w',v_1')]_{\rho} &=([w,w'],\rho(w)(v_1')-\rho(w')(v_1)) \in 
\gs\oplus_{\rho|_{_{\gs}}}V_0,
\end{align*}
i.e., that 
$$
[(\gs\oplus_{\rho|_{_{\gs}}}V_0,\,W\oplus V_1]_{\rho}\subset W\oplus V_1,\quad 
[W\oplus V_1,W\oplus V_1]_{\rho}\subset\gs\oplus_{\rho|_{_{\gs}}}V_0.
$$
\end{proof}

Let $(\gs,W)$ and $\rho$ be as above. A linear operator $A:V\rightarrow V$ is 
called {\it splitting} (with respect to $(\gs,W)$ and $\rho$) if 
\begin{equation}\label{spl-oper}
\rho(s)\circ A - A\circ\rho(s)=0,\quad \rho(w)\circ A+A\circ\rho(w)=0,\quad 
s\in\gs,\quad w\in W.
\end{equation}
In particular, $A$ is an endomorphism of the $\gs$-module $(V,\rho|_{_{s}})$. A 
splitting operator $A$ is called a {\it splitting involution} if, in addition, 
$A^2=id_V$. The splitting involution $id_{V_1}\oplus (-id_{V_2})$ is naturally 
associated with a $\rho$-extension $V=V_0\oplus V_1$ and vice versa. Note also 
that splitting operators with respect to $(\gs,W)$ \,and $\rho$ \, form a 
vector space which will be denoted by $\mathcal{S}_1=\mathcal{S}_1(\gs,W,\rho)$, and that 
the product of two splitting operators is an endomorphism of the $\gG$-module $(V,\rho)$. 
So, denoting by $\mathcal{S}_0(\rho)$ the algebra of these endomorphisms we see that
$$
\mathcal{S}(\gs,W,\rho)=\mathcal{S}_0(\rho)\oplus\mathcal{S}_1(\gs,W,\rho)
$$
is an associative $\dF_2$--graded algebra

Denote by $V_{(\lambda)}$ the root space of $A$, corresponding to an eigenvalue 
$\lambda\in \gk$ \, of $A$. 
\begin{lem}
Let $A$ be a splitting operator. Then 
$$
\rho(s)(V_{(\lambda)})\subset V_{(\lambda)}, \quad 
\rho(w)(V_{(\lambda)})\subset V_{(-\lambda)}, \quad \mathrm{if}\quad s\in\gs,\quad w\in W.
$$
\end{lem}\label{pre-spl}
\begin{proof}
Let $I=id_V$. Then, obviously, 
$$
(\lambda I-A)\circ\rho(s)=\rho(s)\circ (\lambda I-A),\quad (\lambda 
I+A)\circ\rho(w)=\rho(w)\circ (\lambda I-A)
$$
for any $s\in\gs,\quad w\in W$. Since $V_{(\lambda)}= \ker (\lambda I-A)^k$ for 
a suitable $k$, the assertion directly follows from the above relations. 
\end{proof}
\begin{cor}\label{ospl}
Assume that eigenvalues of a nondegenerate splitting operator $A$  belong to $\gk$ 
and divide them into two parts $\Lambda_0$ and $\Lambda_1$, in such a way 
that opposite eigenvalues $\lambda$ and $-\lambda$ do not belong to the same 
part. Then the pair 
$$
V_0=\sum_{\lambda\in\Lambda_0}V_{(\lambda)},\quad 
V_1=\sum_{\lambda\in\Lambda_1}V_{(\lambda)}
$$
is an extension of $(\gs,W)$. In particular, if $A$ is a splitting involution, 
then the pair $(V_{(1)},V_{(-1)})$ \,is an extension of $(\gs,W)$. 
\end{cor}
\begin{proof}
Straightforwardly from the above lemma. 
\end{proof}
\begin{rmk}
Let $A, V_0$ and $V_1$ be  as in the 
above corollary. Then one gets a splitting involution just by declaring $V_0$ 
and $V_1$ to be its proper subspaces corresponding to eigenvalues $1$ and $-1$, 
respectively. 
\end{rmk}

The following fact helps in searching for non-trivial 
splitting operators. 
\begin{proc}\label{ext-ext}
Let $\bar{\gk}$ \,be an extension of the ground field $\gk$ and 
$\bar{\gs},\bar{W},\bar{\rho}$ be the corresponding extensions of
$\gs, W$ and $\rho$, respectively. Then 
\begin{enumerate}
\item
$\mathcal{S}_0(\bar{\rho})=\mathcal{S}_0(\rho)\otimes_{\gk}\bar{\gk},  
\quad\mathcal{S}_1(\bar{\gs},\bar{W},\bar{\rho})=\mathcal{S}_1(\gs,W,\rho)
\otimes_{\gk}\bar{\gk}$;  
\item If \;$\mathcal{S}_1(\bar{\gs},\bar{W},\bar{\rho})$ \,contains a
nondegenerate operator, then $\mathcal{S}_1(\gs,W,\rho)$ also contains a 
such one. 
\end{enumerate}
\end{proc}
\begin{proof}
Obviously, $\mathcal{S}_1(\gs,W,\rho)\otimes_{\gk}\bar{\gk}\subset 
\mathcal{S}_1(\bar{\gs},\bar{W},\bar{\rho})$. But $\mathcal{S}_1(\bar{\gs},\bar{W},\bar{\rho})$ 
\,is the solution space of linear system (\ref{spl-oper}) interpreted as a 
system over $\bar{\gk}$. So, its dimension over $\bar{\gk}$ coincides with 
that of  the solution space of (\ref{spl-oper}) over $\gk$, i.e., with $\mathcal{S}_1(\gs,W,\rho)$.
Hence $\mathcal{S}_1(\gs,W,\rho)\otimes_{\gk}\bar{\gk}= 
\mathcal{S}_1(\bar{\gs},\bar{W},\bar{\rho})$. Similar arguments prove that
$\mathcal{S}_0(\bar{\rho})=\mathcal{S}_0(\rho)\otimes_{\gk}\bar{\gk}$.

To prove the second assertion consider the polynomial 
$P(t)=\det\,(t_1A_1+\dots+t_mA_m)$ \,in variables $t_i$'s with $A_1\dots 
A_m$ \,being a base of $\mathcal{S}_1(\gs,W,\rho)$. Zeros $t=(t_1,\dots,t_m)$ of 
$P(t)$ with $t_i\in\gk$ correspond to degenerate operators in 
$\mathcal{S}_1(\gs,W,\rho)$. Since $A_1\dots A_m$\, is also a base of 
$\mathcal{S}_0(\bar{\gs},\bar{W},\bar{\rho})$, zeros of $P(t)$ \,with $t_i$'s in 
$\bar{\gk}$ gives degenerate operators in 
$\mathcal{S}_1(\bar{\gs},\bar{W},\bar{\rho})$. So, since 
$\mathcal{S}_1(\bar{\gs},\bar{W},\bar{\rho})$ contains non-degenerate operators 
the polynomial $P(t)$  is nonzero.  But being of zero characteristic $\gk$ is infinite. 
Hence $P(t)$ is a nonzero function on $\gk^m$. 
\end{proof}

Since the structure of representations of simple algebras over 
algebraically closed fields is well-known, this proposition is of help
when looking for $d$-pairs for abelian extensions of simple algebras
over arbitrary ground fields.

\subsection{Some properties of the algebra $\mathcal{S}(\gs,W,\rho)$}
 In this subsection we keep the notation of the previous one.
\begin{lem}\label{k-rootsInS_1}
Let $0\neq A\in\mathcal{S}_1(\gs,W,\rho)$ and $\rho$ is irreducible. If one of 
eigenvalues $\lambda$ of $A$ belongs to $\gk$, then 
\begin{enumerate}
\item $V=\Ker(A^2-\lambda^2I)$ and $\lambda\neq 0$;
\item $V_0=\Ker(A-\lambda I, \,V_1=\Ker(A+\lambda I)$ is a $\rho$--extension of 
$(\gs, W)$. 
\end{enumerate}
\end{lem}
\begin{proof} First, we have $0\neq \Ker(A-\lambda I)\subset\Ker(A^2-\lambda^2I)$.
On the other hand, 
$$
\Ker(A-\lambda I)\stackrel{\rho(w)}{\longleftrightarrow}\Ker(A+\lambda I), \quad \forall w\in W.
$$
Therefore $\Ker(A^2-\lambda^2I)$ is $\rho$--invariant, and $V=\Ker(A^2-\lambda^2)$,
since $\rho$ is irreducible. If $\lambda=0$, then, obviously,
$\Ker\,A$ is $\rho$--invariant and hence $ \Ker\,A=V$, i.e., $A=0$ in contradiction 
with the assumption.

The second assertion is obvious in view of corollary\,\ref{ospl}.
\end{proof}
An immediate consequence of this lemma is
\begin{cor}\label{cor-k-rootsInS_1}
Let $\gk$ be algebraically closed and $\rho$ irreducible. If $\mathcal{S}_1(\gs, W, \rho)$ 
is nontrivial, then there is a $\rho$--extension of $(\gs, W)$. 
\end{cor}
\begin{proc}\label{S-asTelo}
Let $\gG$ be simple and $\rho$ irriducible. Then
\begin{enumerate}
\item $\mathcal{S}_0(\rho)$ is a division algebra (over $\gk$).
\item If $\mathcal{S}(\gs, W, \rho)$ is not a division algebra, then
the d-pair $(\gs, W)$ admits a $\rho$--extension.
\end{enumerate}
\end{proc}
\begin{proof}
The first assertion is the classical Schur lemma. Next, let $A=A_0+A_1
\in\mathcal{S}(\gs, W, \rho)$ with $\,A_0\in\mathcal{S}_0(\rho), \,  \,A_1\in\mathcal{S}_1(\gs, W, \rho)$
be a degenerate operator. The first assertion of the proposition implies that $A_1\neq 0$ and $A_0^{-1}\in\mathcal{S}_0(\rho)$, if $A_0\neq 0$. In this case the operator $B=AA_0^{-1}=I+A_1A_0^{-1}$
is degenerate too, and hence one of eigenvalues of $B_1= A_1A_0^{-1}\in 
\mathcal{S}_1(\gs, W, \rho)$ is $-1$. Now  lemma \ref{k-rootsInS_1} proves the assertion.
Moreover, this lemma shows that the assumption $A_0\neq 0$ takes place. Indeed, assuming the contrary we see that $A_1$ is degenerate and, therefore, one of its eigenvalues is $0$ in 
contradiction with the lemma.
\end{proof}

Now we shall specify the above results to the case $\gk=\R$.
\begin{proc}\label{AlmExt}
Let the $\gG$ be a simple Lie algebra over $\R$, $\rho:\gG\to \End\,V$ an irreducible 
representation of $\gG$, and $(\gs, W)$ a d-pair in $\gG$. If $\mathcal{S}_1(\gs, W, \rho)$ 
is not trivial, then $(\gs, W)$
admits a $\rho$--extension except, possibly, the case when  $\mathcal{S}(\gs, W, \rho)$
is isomorphic to $\C$.
\end{proc}
\begin{proof}
Proposition \ref{S-asTelo} allows us to restrict to the case when  $\mathcal{S}(\gs, W, \rho)$ is 
a division algebra. Since $\mathcal{S}_1(\gs, W, \rho)$ is nontrivial, the dimension of this 
algebra is greater than 1. So, by  the classical Frobenius theorem, this algebra is
isomorphic either to $\C$, or to $\Q$, and we have to analyze only the second alternative. 

In this case, as it is easy to see, $\mathcal{S}_0(\rho)$ is isomorphic to $\C$, 
and $V$ acquires a structure of a vector 
space over $\C$ by means of the operator $J\in\End\,V$ that corresponds to $\sqrt{-1}$
via this isomorphism. Denote by $V_{\C}$ this complex vector space. The representation
$\rho$ naturally extends to a representation $\rho_{\C}:\gG^{\C}\to \End_{\C}V_{\C}$ of the 
complexification $\gG^{\C}=\gG\otimes_{\R}\C$ of $\gG$ in $V_{\C}$: 
$$
\rho_{\C}(x\otimes \sqrt{-1})\df J(\rho(x)), \,x\in\gG.
$$
By corollary \ref{cor-k-rootsInS_1}, the d-pair $(\gs\otimes_{\R}\C, W\otimes_{\R}\C)$ 
in $\gG^{\C}$ admits a $\rho_{\C}$--extension whose restriction to $\rho$ is, obviously, 
an $\rho$--extension of $(\gs, W)$.
\end{proof}

\subsection{D-pairs associated with 3-dimensional simple subalgebras}

Here we shall construct  simplifying d-pairs for abelian extensions of Lie algebras
possessing a simple $3$--dimensional subalgebra. 

First, we shall collect some necessary facts about simple $3$-dimensional 
algebras (see, for instance, \cite{Jac}). Let $\gH$ be a such one and 
$h\in \gH$ a regular element of it. Then there exists a base 
$(e_1,e_2,e_3=h)$ in $|\gH|$ such that $[e_1,e_2]=h, \;[h,e_1]=\alpha e_2, 
\;[h,e_2]=\beta e_1, \;, \alpha,\beta\in\gk$ (the ground field), 
\;$\alpha\beta\neq 0$. Put $\kappa=\alpha\beta$. So, the characteristic 
polynomial of $\textrm{ad}\,h$ is $t(t^2-\kappa)$. If $\kappa=\lambda^2, 
\;\lambda\in \gk$, then $\gH$ splits and there exists a base 
$(h'=2\lambda^{-1}h,x,y)$ \;of $\gH$,\;called a $\gs\gl_2$-triple, such that 
$[h',x]=2x, \;[h',y]=-2y, \;[x,y]=h'$. If $\kappa$ is not a square in 
$\gk$, i.e., the polynomial $t^2-\kappa$ is irreducible, consider the 
extension $\bar{\gk}$ \;of $\gk$ \; by adding to $\gk$ the roots of 
$t^2-\kappa$. We still denote these roots by 
$\pm\lambda\in\bar{\gk}$. The extended algebra 
$\bar{\gH}=\gH\otimes_{\gk}\bar{\gk}$ splits over $\bar{\gk}$ and, as 
before, one can find an $\gs\gl_2$-triple $(h'=2\lambda^{-1}h,x,y)$ in it. 
Recall also, that if $\varrho$ \;is a representation of $\gH$, or of 
$\bar{\gH}$, then eigenvalues of $\varrho(h')$ are integer and 
multiplicities of opposite eigenvalues are equal. So, eigenvalues of $h$ 
are of the form $\pm(m/2)\lambda$ \; with $\lambda^2=- \kappa, \;m\in \Z$. 
Since the element $h$ is semisimple the operator $\varrho(h)$ is semisimple 
as well (see \cite{Hum}). Therefore the representation space $U$ of $\varrho$ 
splits into a direct sum of 1-dimensional and 2-dimensional 
$\varrho(h)$-invariant subspaces in such a way that 1-dimensional subspaces 
belong to $\ker\,\varrho(h)$ \;while each of 2-dimensional ones is 
annihilated by the operator $\varrho(h)^2-(1/4)m^2\kappa$ for a suitable 
integer $m\neq 0$. We shall call them \emph{eigenlines} and 
\emph{$m$-eigenplanes}, respectively. Obviously, if 
$\gH$  splits, then any eigenplane splits into two eigenlines generated 
by eigenvectors of eigenvalues $\pm(m/2)\lambda$. In the non-split case 
eigenplanes are irreducible with respect to  $\varrho(h)$. 
 
Now we shall associate a d--pair with a simple 3-dimensional subalgebra $\gH$
of a Lie algebra $\gG$. First, we recall the following elementary fact. Let $\gA$ be a Lie 
algebra, $x\in \gA$  and $\gA_{\mu}$ the root space of the 
operator $\textrm{ad}\,x$ \;corresponding to the eigenvalue $\mu$. Then 
\begin{equation}\label{comut}
[\gA_{\mu},\gA_{\nu}]\subset\gA_{\mu+\nu}.
\end{equation}

Let $h\in\gH$ be as above and  $A=\textrm{ad}_{\gG}h$. Put 
$$
\gG_0=\ker\,A, \quad\gG_m=\ker(A^2-(m^2/4)\kappa\;\textrm{id}_{\gG})
$$ 
Then $\gG=\bigoplus_{m\geq 0}\gG_m$ and commutation relations
\begin{equation}\label{g-com}
[\gG_k, \gG_l]\subset \gG_{k+l}\oplus \gG_{k-l}.
\end{equation}
take place. If $\gH$ splits, this directly follows from (\ref{comut}). Indeed, 
in this case  $\gG_m, \;m>0$, \;splits into a direct sum of root spaces 
corresponding to eigenvalues $\pm(m/2)\lambda$ (notice that $\gG_k=\gG_{-k}$). 
If $\gH$ does not split one obtains the result by extending scalars from 
$\gk$ to $\bar{\gk}$. In fact, the extended subalgebra $\bar{\gH}$ of the 
extended algebra $\bar{\gG}$ splits and hence the extended analogues 
$\bar{\gG}_m$'s of subspaces $\gG_m$'s commute according to (\ref{g-com}), 
while $\gG_m\subset \bar{\gG}_m$. 

Relations (\ref{g-com}) show that
\begin{equation}\label{1stD-pair}
\gs=\bigoplus_{m\geq 0}\gG_{2m}, \quad  W=\bigoplus_{m\geq 0}\gG_{2m+1}
\end{equation}
is a d--pair in $\gG$, which will be called the \emph{first d-pair associated with} 
$\gH$ if $W\neq\{0\}$. 
Since $\gH\subset\gs$, this d--pair is trivial iff $\gG=\gs\oplus_{\rho} W$. In 
particular, it is nontrivial and simplifying if  $\gG$ is semisimple.

If $W=\{0\}$, i.e., $\gG=\oplus_{m\geq 0}\gG_{2m}$, then
\begin{equation}\label{2dD-pair}
\gs=\bigoplus_{m\geq 0}\gG_{4m}, \quad  W=\bigoplus_{m\geq 0}\gG_{4m+2}
\end{equation}
is the \emph{second d-pair associated with} $\gH$. Since $h\in\gG_0\subset\gs$ 
and $x,y\in\gG_2\subset W$, this d--pair is nontrivial and, obviously, simplifying 
if $\gG$ is simple.

\subsection{Solution of the disassembling problem
for algebraically closed fields}

D--pairs (\ref{1stD-pair}) and (\ref{2dD-pair}) allow us  to solve
the disassembling problem for Lie algebras over algebraically closed fields.
In this subsection we keep the notation of the previous one and assume the ground 
field $\gk$ to be algebraically closed.

\begin{proc}\label{split-by-3d}
Let $\gG$ be a Lie algebra possessing a simple $3$-dimensional subalgebra 
and $\rho$ a representation of $\gG$ in $V$. Then a nontrivial associated
with $\gH$ d--pair admits a $\rho$--extension.
\end{proc}
\begin{proof}
We identify $\gG$ (resp., $V$) with subalgebra $\gG\oplus_{\rho}\{0\}$ 
(resp., $\{0\}\oplus_{\rho}V$) in the algebra $\gG\oplus_{\rho}V$. Put 
$$
B=\rho(h), \quad V_0=\ker\,B, \quad V_m=\ker(B^2-(m^2/4)\kappa\;\textrm{id}_V).
$$ 
Obviously, $V=\oplus_{m\geq 0}V_m$. Since $\gG$ is simple, the first d--pair 
$(\gs, W)$ associated with $\gH$ is nontrivial if $W\neq\{0\}$. Then
\begin{equation}\label{1st-ext}
\gs_{\rho}\df\gs\oplus\left(\oplus_{m\geq 0}V_{2m}\right), 
\quad W_{\rho}\df W\oplus_{m\geq}V_{2m+1}
\end{equation}
is the required extension. Indeed, this directly follows from commutation relations
\begin{equation}\label{V-com}
[\gG_k, V_l]\subset V_{k+l}\oplus V_{k-l},
\end{equation}
which can be proved by the same arguments as for (\ref{g-com}). Note that
this part of the proof does not require algebraic closure of $\gk$.

If the first d--pair associated with $\gH$ is trivial, we consider finer
decompositions of $\gG$ and $V$ using the fact that $\gH$ splits if $\gk$ is
algebraically closed. Namely, put
\begin{equation}
\gG_m^{\prime}=\ker\;(A-\frac{m}{2}\lambda\id_{\gG}), \quad
V_m^{\prime}=\ker\;(B-\frac{m}{2}\lambda\id_V). 
\end{equation}
Then, obviously, $\gG_m=\gG_m^{\prime}\oplus\gG_{-m}^{\prime}, 
\;V_m=V_m^{\prime}\oplus V_{-m}^{\prime}$ and (see \ref{comut})
\begin{equation}\label{l-com}
[L_k, L_l]\subset L_{k+l}.
\end{equation}
where $L_s$ stands for one of subspaces $\gG_s^{\prime}, \,V_s^{\prime}$.
Now it immediately follows from relations (\ref{l-com})  that subspaces
\begin{equation}\label{fine-V} 
 \mathcal{V}_0=\bigoplus_{k\in\Z}(V_{4k}^{\prime}\oplus V_{4k+1}^{\prime}),
 \quad \mathcal{V}_1=\bigoplus_{k\in\Z}(V_{4k+2}^{\prime}\oplus V_{4k+3}^{\prime}).
\end{equation}
provides a $\rho$-extension of the second  d--pair associated with $\gH$.
\end{proof}

An important consequence of proposition\,\ref{split-by-3d} is
\begin{cor}\label{nontrivS_1}
Let $\gH$ be a simple 3-dimensional  subalgebra of an algebra Lie $\gG$
over an arbitrary ground field and $(\gs, W)$ the associated with $\gH$ d-pair. 
Then $\mathcal{S}_1(\gs, W, \rho)$ is nontrivial. 
\end{cor}
\begin{proof}
Immediately from proposition\,\ref{ext-ext}, (1).
\end{proof}
\begin{thm}\label{C-dis}
Any finite-dimensional Lie algebra over an algebraically closed field of 
characteristic zero can be completely disassembled. 
\end{thm}
\begin{proof}
By Morozov lemma, any simple Lie algebra $\gG$ over an algebraically closed field $\gk$ of 
characteristic  zero possesses a 3-dimensional subalgebra $\gH$ isomorphic to 
$\gs\gl(2,{\gk})$ (see \cite{Jac}, \cite{Hum}). If the algebra $\gG$ in proposition
\,\ref{split-by-3d} is simple, then the $\rho$--extension of one of the d--pairs
$(\gs, W)$ associated with $\gH$ is simplifying. Indeed, the Stripping Lemma 
applied to this extended d--pair leads to an algebra of the form $\gs\oplus_{\rho'}V'$
(see the proof of proposition\,\ref{split-by-3d})   whose semisimple part coincides
with that of $\gs$. Hence $l(\gs)<l(\gG)$, since $\gs$ is a proper subalgebra of $\gG$.
\end{proof}

\subsection{Simplest algebras.}

\noindent Simple Lie algebras,  which can not be 
directly disassembled by the above methods,  will be discussed in this section. 
\begin{defi}
A simple Lie algebra is called \emph{simplest} if all its proper 
subalgebras are abelian. 
\end{defi}
This definition is justified by the following
\begin{proc}\label{simplest-in}
A simple Lie algebra $\gG$ over a field $\gk$ of characteristic zero contains either 
a simplest  subalgebra, or an subalgebra isomorphic to $\gs\gl(2,\gk)$. 
\end{proc}
\begin{proof}
If $\gG$ is not simplest, then it contains a proper non-abelian subalgebra 
$\gH$. If the semisimple part $\gH$ is nontrivial, then $\gH$ contains a proper
simple subalgebra, and one gets the desired result by obvious induction arguments. 
If, on the contrary, the semisimple part is trivial, then $\gH$ is a nonabelian
solvable algebra and, therefore, contains a nontrivial nilpotent element 
$g$. Hence the endomorphism $ad_{\gG}g$ of $|\gG|$ has a nontrivial 
nilpotent part. In other words, $g$, considered as an element of $\gG$, has 
a nontrivial nilpotent part $g_n$ which, according to a well-known property 
of semisimple algebras, belongs to $\gG$. Thus $\gG$ possesses a nontrivial 
nilpotent element. By Morozov's lemma, such an element is contained in a 
3-dimensional subalgebra of $\gG$ isomorphic to $\gs\gl(2,\gk)$. 
\end{proof}

\begin{cor}\label{semisimple}
All elements of a simplest algebra are semisimple. 
\end{cor}
\begin{proof}
Assume that $\gG$ has a nonsemisimple element. Then the nilpotent part
of such an element is nontrivial and, by a well-known property of semisimple
Lie algebras, belongs to $\gG$. Hence $\gG$ has a nontrivial nilpotent
element. In its turn this element is contained in a $3$-dimensional 
subalgebra of $\gG$ which is isomorphic to $\gs\gl(2,k)$ (see the proof 
of the above proposition). But $\gs\gl(2,k)$ and, therefore, $\gG$ contains  
$2$-dimensional nonabelian subalgebras in contradiction with the fact that
$\gG$ is a simplest algebra. 
\end{proof}

Existence and diversity of simplest algebras depend exclusively on 
arithmetic properties of the ground field $\gk$. For instance, there are no 
simplest algebras over algebraically closed fields. Indeed, any simple 
algebra over such a field $\gk$ contains a $3$-dimensional simple 
subalgebra isomorphic to $\gs\gl(2,\gk)$, which, in its turn, contain   
proper $2$-dimensional non-abelian subalgebras. On the contrary, there is 
only one (up to isomorphism) simplest algebra over $\R$, namely, 
$\gs\go(3,\R)$ (see proposition\,\ref{R-simp} below). 

Now we shall collect some elementary properties of simplest algebras. Denote
by $C_x$ the centralizer of an element $x\in\gG$. 
. 

\begin{proc}\label{solid property}
Let $\gG$ be a simplest Lie algebra and $0\neq x\in \gG$. Then 
\begin{enumerate}
\item $C_x$ is abelian;
\item if\, $y,z\in\gG, \,y\neq 0$, and
$[x,y]=[z,y]=0$, then $[x,z]=0$; 
\item if\, $0\neq y\in\gG$, then either $C_x=C_y$, or $C_x\cap
C_y=\{0\}$; 
\item $C_x$ is a Cartan subalgebra of $\gG$, i.e., $y\in C_x$ if\, $[x,y]\in 
C_x$;
\item all nonzero elements of $\gG$ are regular;
\item $[\gG,C_x]\cap C_x=\{0\}$.
\end{enumerate}
\end{proc}
\begin{proof}

(1) Since $0\neq x\in C_x$, the center of $C_x$ is nontrivial. But  the center
of $\gG$ is trivial. Hence $C_x$ does not coincide with $\gG$, i.e., is a 
proper subalgebra of $\gG$. As such it is abelian. 

\noindent(2) Obviously, $x$ and $z$ belong to $C_y$, which, by (1), is abelian. 

\noindent(3) If $0\neq y\in C_x$, then, according to (2), any $z\in C_y$ 
belongs to $C_x$. 

\noindent(4)  $C_x$ is an ideal in its normalizer $N_x$. Since $\gG$ is simple 
$N_x$ is a proper subalgebra of $\gG$ and, as such, must be abelian. 
So, $N_x\subset C_x$ and hence $N_x=C_x$. 

\noindent(5) Directly from (4). 

\noindent(6) We have to prove that $[y,C_x]\cap C_x=\{0\}, \,\forall y\in 
\gG$. Assuming the contrary consider an element $z\in C_x$ such that $0\neq 
[y,z]\in C_x$. In view of (3) $C_x=C_z\,\Rightarrow\,[y,z]\in C_z$. By (4), 
this implies  that $y\in C_z$ and, therefore, $[y,z]=0$ in contradiction 
with the assumption. 
\end{proof}

To proceed on we need some information about operators of the adjoint 
representation of a simplest algebra. 

\medskip
\subsection{On the adjoint representation of simplest algebras.}

First, we mention without proof the following elementary facts.
\begin{lem}\label{sk-sym}
Let $V$ be a finite-dimensional vector space and $b(\cdot,\cdot)$ a 
nondegenerate, symmetric bilinear form on $V$. If $A:V\rightarrow V$ is a 
linear, skew-symmetric with respect to $b$ operator, i.e., 
$b(Au,v)+b(u,Av)=0, \,u,v\in V$, then the minimal polynomial of $A$ is of 
the form $t^r\varphi(t^2), \varphi(0)\neq0, \,r\geq 0$. If $A$ is 
semisimple and $\varphi=\varphi_1^{n_1}\cdot,\dots,\cdot \varphi_m^{n_m}$ 
is the canonical factorization of the polynomial $\varphi$ into irreducible 
and relatively prime factors, then $r=0,1$ and $n_1=\dots =n_m=1$.  
\end{lem}
\begin{cor}\label{sk-sym-sim}
The assertion of the above lemma is valid for operators of the adjoint 
representation of a simplest algebra and  in this case $r=1$. 
\end{cor}
\begin{proof}
Recall that operators of the adjoint representation of a Lie algebra $\gG$ 
are skew-symmetric with respect to the Killing form, which is nondegenerate 
for a semisimple $\gG$. Moreover, according to corollary \ref{semisimple}, 
these operators are semisimple. So, it suffices to take the Killing form 
for $b$ and to observe that the kernel of an adjoint representation operator
is nontrivial. 
\end{proof}
 
It is not difficult to see that if $g(\tau)$ is an irreducible polynomial, 
then the polynomial $h(t)=g(t^2)$ is either irreducible, or 
$h(t)=\psi(t)\psi(-t)$ with irreducible and relatively prime $\psi(t)$ and 
$\psi(-t)$. Hence, by lemma \ref{sk-sym}, the minimal polynomial 
$f(t)$ of a semisimple skew-adjoint operator $A$ is of the form 
\begin{equation}\label{s-factor}
F(t)=t^{\epsilon}f_1(t^2)\cdot\dots\cdot f_k(t^2)\psi_1(t)\psi_1(-t)\cdots\psi_l(t)\psi_l(-t), 
\quad \epsilon=0,1,
\end{equation}  
with relatively prime and irreducible factors. Under the hypothesis of 
lemma \ref{sk-sym} with $\epsilon=1$ we have the following direct sum 
decomposition 
\begin{equation}\label{sum-dir} 
V=\ker\,A\oplus \ker\,f_1(A^2)\oplus\ldots\oplus \ker\,f_k(A^2)\oplus 
\ker\,g_1(A^2) \oplus\ldots\oplus \ker\,g_l(A^2) 
\end{equation}  
with $g_i(t^2)=\psi_i(t)\psi_i(-t)$.
\begin{lem}\label{sk-sym}
Subspaces in decomposition (\ref{sum-dir}) are mutually orthogonal with 
respect to the form $b$.
\end{lem}
\begin{proof}
This is a direct consequence of the fact that 
$b(\ker\,\varphi(A^2),\ker\,\phi(A))=0$, if polynomials $\varphi(t^2)$ and 
$\phi(t)$ are relatively prime. To prove this assertion, consider the identity
$$
\varphi(t^2)\alpha(t)+\phi(t)\beta(t)=1
$$
with $\alpha(t), \,\beta(t)$ being some polynomials . Let $u\in 
\ker\,\varphi(A^2), \,v\in\ker\,\phi(A)$ and 
$\alpha(t)=\alpha_0(t^2)+t\alpha_1(t^2)$. Then 
\begin{align*}
0=b([(\alpha_0(A^2)-A\alpha_1(A^2))\varphi(A^2)]u,v)=
b(u,[(\alpha_0(A^2)+A\alpha_1(A^2))\varphi(A^2)]v)=\\b(u,[\alpha(A)\varphi(A^2)]v)=
b(u,v-\phi(A)\beta(A)v)=b(u,v).
\end{align*}
\end{proof} 

Finally, we shall fix some properties of operators of the adjoint representation of a simplest 
Lie algebra $\gG$ (over $\gk$). Below $(\cdot,\cdot)$ stands for the 
Killing form on $\gG$ and  $\gk_f$ for the splitting field of a 
polynomial $f\in \gk[t]$. 
\begin{proc}\label{simpl-ad}
Let $\gG$ be a simplest Lie algebra, $0\neq x\in\gG$ and $A=ad\,x$. Then
\begin{enumerate}
\item the minimal polynomial $F(t)$ of $A$ has the form (\ref{s-factor}) with
$\epsilon=1$ and decomposition (\ref{sum-dir}) with Killing orthogonal 
summands for $V=|\gG|$ holds;
\item $C_x=\ker\,A$; 
\item nonzero roots of $F(t)$ do not belong to $\gk$;
\item the lattice (in $\gk_F$) generated by roots of $F(t)$ coincides with that
of $f_i(t)$ and with that of $\varphi_j(t), \,i=1,\ldots,k, 
\;j=1,\ldots,l$; 
\item $\gk_F=\gk_{f_i}=\gk_{\varphi_j}, \,i=1,\ldots,k, \;j=1,\ldots,l$;
\end{enumerate}
\end{proc} 
\begin{proof}
(1) Directly from corollary\,\ref{sk-sym-sim} and lemma\,\ref{sk-sym}.

(2) Directly from proposition\,\ref{solid property}.

(3) Assuming the contrary we observe that $x$ and an eigenvector 
$y\in \gG$  corresponding to a nonzero eigenvalue \, of $A$ span a 
$2$-dimensional nonabelian subalgebra of 
$\gG$. Being simplest  $\gG$ must coincide with this subalgebra in 
contradiction with simplicity of $\gG$.

(4) Let $h(t)$ be one of polynomials $f_i$'s, $\varphi_j$'s and 
$\lambda_1,\ldots,\lambda_m\in\gk_h$ be its roots. Consider the subalgebra 
$\gH_h$ of $\gG$ generated by $W=\ker\,h(A)$ and $C_x$. Obviously, $\gH$ is nonabelian 
and hence $\gG=\gH_h$. On the other hand, in view of (\ref{comut}) applied to 
the $\gk_F$-extension of $\gG$, eigenvalues of $A\mid_{\mid\gH\mid}$ belong 
to the lattice $L_h$ in $\gk_F$ generated by $\lambda_1,\ldots,\lambda_m$. 
Lattices $L_{f_i}$'s and $L_{\varphi_j}$'s must coincide, since, otherwise, one
of the subalgebras would be proper in $\gG$.

(5) Directly from (4).
\end{proof}

\subsection{Complete disassembling of real Lie algebras.} 

Now we are ready to prove that any Lie algebra over $\R$ can be 
completely disassembled. The stripping procedure in this case is based 
on the following fact.
\begin{proc}\label{R-simp}
Simplest Lie algebras over $\R$ are isomorphic to $\gs\go(3,\R)$.
\end{proc} 
\begin{proof}
Let $\gG$ be a simplest real Lie algebra. Note that the minimal polynomial 
$F(t)$ of $A=ad_x, \,x\in\gG$, is of the form 
$F(t)=t(t^2+\lambda_1^2)\ldots(t^2+\lambda_k^2), 
\,\lambda_1,\ldots,\lambda_k\in\R, \,k\geq 1$. Indeed, according to 
proposition \ref{simpl-ad} (3), all nonzero roots of $A$ do not belong to 
$\R$.  Let $C$ be a Cartan subalgebra of $\gG$. Then, according to 
proposition \ref{solid property}, $\{ad_y\}_{y\in C}$ is a family of 
commuting semisimple operators. Recall that such a family possesses a 
\emph{primitive} element, i.e., a such one that any its invariant subspace 
is also invariant with respect to all operators of the family. Let $A=ad_x, 
\,x\in C,$ be primitive for the family $\{ad_y\}_{y\in C}$ and 
$f(t)=t^2+\lambda^2$ be one of irreducible factors of its minimal 
polynomial $F(t)$. If $0\neq y\in\ker\,f(A)$, then $P=\mathrm{span}<y,Ay>$ 
is, obviously, $A$-invariant and hence $ad_z$-invariant for any $z\in C$. 
Equivalently, $[C,P]\subset P$. Obviously, $\dim\,P=2$. Moreover, 
$A([y,Ay])=[Ay,Ay]+[y,A^2y]=0$, since $A^2y=-\lambda^2y$. But, according to 
proposition\,\ref{solid property}, $C=C_x=\ker\,A$. This shows that 
$[y,Ay]\in C\Leftrightarrow [P,P]\subset C$. Hence $\gH=C+P$ is a 
nonabelian subalgebra of $\gG$.

Observe now that $[P,P]\neq 0$. Indeed, otherwise,  elements $x,y$ and $Ay$ 
would span a 3-dimensional  subalgebra of $\gG$, which is nonabelian, since 
$0\neq Ay=[x,y]$, and solvable. But being simplest $\gG$ must coincide with 
this subalgebra in contradiction with simplicity of $\gG$. Moreover, 
$\dim\,[P,P]=1$, since $[P,P]=\mathrm{span}\,[y,Ay]\neq 0$. Therefore 
$[P,P]$ and $P$ span a nonabelian 3-dimensional subalgebra $\gH_0$ of $\gH$. 
So, $\gG=\gH_0$, since $\gG$ is simplest. But any 3-dimensional simple Lie 
algebra over $\R$ is isomorphic either to $\gs\go(3,\R)$, or to  $\gs\gl(2,\R)$
(Bianchi's classification), and the latter is not simplest one.
\end{proof}
\begin{thm}\label{R-dis}
Any finite-dimensional Lie algebra over $\R$ can be completely 
disassembled. 
\end{thm}
\begin{proof}
As we have seen earlier the problem reduces
to existence of simplifying d--pairs for algebras $\gG\oplus_{\rho}V$ with a simple 
$\gG$ and an irreducible $\rho:\gG\to\End V$. We shall construct such a pair with help
of a simple 3-dimensional subalgebra $\gH$ of $\gG$. By propositions\,\ref{simplest-in}
and \ref{R-simp}, $\gG$ contains either a subalgebra isomorphic to $\gs\gl(2,\R)$
or a subalgebra isomorphic to $\gs\go(3,\R)$. In the first  case
proposition\,\ref{split-by-3d} proves existence of a  simplifying d--pair. Hence  we
have to analyze the situation when $\gH$ is isomorphic to $\gs\go(3,\R)$. If
in this case the first  d-pair associated with $\gH$ is nontrivial, then d-pair 
(\ref{1st-ext}) in the proof of proposition\,\ref{split-by-3d} solves the problem.

So, we shall assume that the d-pair $(\gs, W)$ associated with $\gH$ is of the
second type. In particular, in the notation of proposition\,\ref{split-by-3d}, we 
have $\gG=\oplus_{k\geq 0}\gG_{2k}$. Put also
$$
V_{even}=\bigoplus_{k\geq 0}V_{2k} \quad \mathrm{and} \quad
V_{odd}=\bigoplus_{k\geq 0}V_{2k+1}.
$$
Commutation relations (\ref{V-com}) show that subspaces $V_{even}$ and 
$V_{odd}$ are $\rho$-invariant. Since $\rho$ is irreducible, one of these 
subspaces is trivial.

First, assume that $V_{odd}$ is trivial. Once again, relations (\ref{V-com})
show that 
$$
\boldsymbol{V}_0=\bigoplus_{k\geq 0}V_{4k}, \quad\quad  
\boldsymbol{V}_1=\bigoplus_{k\geq 0}V_{4k+2}
$$
is a $\rho$--extension of $(\gs, W)$.

Finally, assume that $V_{even}$ is trivial. First of all, we 
note that, by corollary\,\ref{nontrivS_1}, $\mathcal{S}_1(\gs, W, \rho)$ is nontrivial. 
Moreover, proposition\,\ref{AlmExt} reduces the problem to the case when the 
$\dF_2$-graded algebra $\mathcal{S}(\gs, W, \rho)$ is isomorphic to $\C$. In this 
case $\mathcal{S}_1(\gs, W, \rho)$ is 1-dimensional and contains an operator $J$
such that $J^2=-1$. So, $J$ supplies $V$ with a $\C$-vector space structure
which will be denoted by $V_{\C}$.

Below we shall adopt the notation used in the proof of  proposition \ref{split-by-3d}. 
By construction, operators 
$\rho(z), \,z\in\gs$, commute with $J$, i.e., they are $\C$-linear in $V_{\C}$. 
In particular, $B=\rho(h)$ is such an operator , since $h\in\gs$.
Obviously, eigenvalues of  $B$ are $\frac{m}{2}\sqrt{-1}, \,m\in\Z$, and
\begin{equation}
V_{\C}=\oplus_{m\in\Z}\V_m, \quad \V_m=\ker(B-\frac{m}{2}\sqrt{-1}\id_V)=
\ker(B-\frac{m}{2}J)
\end{equation}
($\V_m$'s are subspaces of  $V_{\C}$). 
Also, observe  that in the considered case $\V_m$ may be nontrivial only for
odd $m$ and consider subspaces
$$
\boldsymbol{W}_{\boldsymbol{1}}=\bigoplus_{m\in\Z}\V_{4m+1} \quad\mathrm{and}
\quad \boldsymbol{W}_{\boldsymbol{2}}=\bigoplus_{m\in\Z}\V_{4m-1}
$$
of $V_{\C}$. Then $V_{\C}=\boldsymbol{W}_1\oplus\boldsymbol{W}_2$ and
$\dim\boldsymbol{W}_1=\dim\boldsymbol{W}_2$, since $V_m=\V_m\oplus\V_{-m}$.
We shall prove that $\boldsymbol{W}_1$ and $\boldsymbol{W}_2$
are $\rho$--invariant. If so, one of these subspaces must be trivial, since $\rho$ is 
irreducible. But $\dim\boldsymbol{W}_1=\dim\boldsymbol{W}_2$ and hence the other 
subspace must be trivial too. This, however, is impossible, since  $V$ is nontrivial.

Let $v\in\V_m$ and $w\in \gG_{4k+2}\subset W$. Then 
$$
Bv=\frac{m}{2}\sqrt{-1}\,v= 
\frac{m}{2}Jv, \quad [h,w]\in W, \quad [h,[h,w]]=-(2k+1)^2w,
$$ 
and, therefore, $$
\rho(w)J+J\rho(w)=0=\rho([h,w])J+J\rho([h,w]).
$$ 
Now we have
\begin{eqnarray}\label{happy1}
B(\rho(w)v)=\rho(w)(Bv)+[B,\rho(w)]v=\frac{m}{2}\rho(w)(Jv)+[\rho(h),\rho(w)]v=
\nonumber\\-\frac{m}{2}J(\rho(w)v)+\rho([h,w])v=
 -\frac{m}{2}\sqrt{-1}\,\rho(w)v+\rho([h,w])v
 \end{eqnarray}
 and
 \begin{eqnarray}\label{happy2}
B(\rho([h,w])v)=\rho([h,w])(Bv)+[B,\rho([h,w])]v=\frac{m}{2}\rho([h,w])(Jv)+
\nonumber \\  +[\rho(h),\rho([h,w])]v=-\frac{m}{2}J(\rho([h,w])v)+\rho([h,[h,w]])v=
\nonumber \\ -\frac{m}{2}\sqrt{-1}\,\rho([h,w])v-(2k+1)^2\rho(w)v
\end{eqnarray}
It follows from (\ref{happy1}) and (\ref{happy2}) that vectors $\rho(w)v$ and 
$\rho([h,w])v$ span a 2-dimensional subspace $\Pi$ in $V_{\C}$, which is 
invariant with respect to $B$. As it is easy to see, eigenvalues of $B\!\mid_{\Pi}$ are 
$\frac{-m\pm(4k+2)}{2}\sqrt{-1}$. This shows that $\Pi$ is spanned by eigenvectors of 
$B\!\mid_{\Pi}$ corresponding to these eigenvalues, i.e., that vectors 
$\rho(w)v$ and $\rho([h,w])v$ belong to $\V_{-m+4k+2}\oplus\V_{-(m+4k+2)}$.
Since $m=4s\pm 1$, the residue of $m \;\mathrm{mod} \;4$ coincides with that 
of $-m\pm(4k+2)$. Therefore, vectors $\rho(w)v$ and $\rho([h,w])v$ belong to 
the same subspace $\boldsymbol{W}_{\boldsymbol{i}}$ as $v$.

If $z\in\gG_{4k}\subset\gs$, then $[h,z]\in\gs, \,[h,[h,z]]=-4k^2z$ and operators
$\rho(z)$ and $\rho([h,z])$ commute with $J$. The same arguments as above
show that vectors $\rho(z)v$ and $\rho([h,z])v$ belong to 
$\V_{m+4k}\oplus\V_{(m-4k)}$, i.e., to the same subspace 
$\boldsymbol{W}_{\boldsymbol{i}}$ as $v$. This proves that subspaces
$\boldsymbol{W}_{\boldsymbol{i}}$'s are $\rho$--invariant.
\end{proof}
\begin{rmk}
The final part of the proof of proposition\,\ref{R-dis} shows that 
$\mathcal{S}(\gs, W, \rho)$ can not be a division algebra, if $S_0(\rho)$
is isomorphic to $\R$ and $\V_{even}$ is trivial.
\end{rmk}

\medskip
\noindent\textit{On the disassembling problem for arbitrary fields.} 


\noindent It seems rather plausible that any finite-dimensional Lie algebra 
over a field of  characteristic zero can be completely disassembled. In view 
of proposition\,\ref{simplest-in} the disassembling problem is reduced to 
simplest algebras and their finite-dimensional representations.  This approach 
presumes a description of simplest algebras over arbitrary ground fields 
$\gk$ and involutions of them. This problem doesn't appear to be extremely 
challenging not to attemp to resolve it.  Probably, the circle of ideas that one can 
find in the last chapter of \cite{Jac} would be sufficient for the solution.

Derived algebras $[\mathcal{A}_{\mathrm{Lie}}, \mathcal{A}_{\mathrm{Lie}}]$  
of Lie algebras $\mathcal{A}_{\mathrm{Lie}}$  associated with 
division algebras $\mathcal{A}$ over $\gk$ give examples of simplest algebras. 
In this connection it should be also stressed
that simplest Lie algebras and their representations have all merits to be 
studied in its own right. For instance, they appear to be  
natural substitutes for $\gs\gl_2$-triples in the perspective of developing 
analogues of root space decompositions for simple Lie algebras 
over arbitrary fields. 

An alternative approach to the disassembling problem could be a direct
description of suitable d--pairs in simple algebras over a given field  
based in its turn on a description of these algebras as, for instance, it is 
done in \cite{Jac}.   However, in order to become practical such a description 
must be duly extended to representations of these algebras. Moreover, a 
serious deficiency of this approach is that it, apart of being rather boring, 
does not reveal the true nature of the phenomenon.

The conjecture that all Lie algebras can be assembled from lions and the 
fact that lions are, in fact, Lie algebras over $\Z$ lead to suspect that
Lie algebras over a field $\gk$ are obtained as specifications  to $\gk$ 
of some universal assemblage schemes. In the rest of this paper some
facts supporting and clarifying this idea will be given.


\section{Matching lieons and first level Lie algebras}\label{1st-level}


A Lie algebra is of the \emph{first level} if it can be completely 
disassembled in one step. In other words, first level Lie algebras are ones that 
can be assembled from a number of mutually compatible lieons. In this section we 
geometrically characterize compatible pairs of lieons and, on this basis, construct 
examples of first level Lie algebras.  


\subsection{Compatible $\pitchfork$-lieons.} 
 
Let $\gG$ be a an $n$-dimensional $\pitchfork$-lieon and  $V=|\gG|$. 
Denote by $C=C_{\gG}$ the center of $\gG$ and put $l=l_{\gG}=[\gG,\gG]$. 
Then $\mathrm{dim}\,C=n-2, 
\,\mathrm{dim}\,l =1$ and $l\subset C$. A basis $e_1,\dots,e_n$ of $V$ such 
that $e_3\in l$ and $e_i\in C$, if $i>2$, will be called \emph{normal} for 
$\gG$. The only nontrivial product in this basis is $[e_1,e_2]=\alpha e_3, 
\,\alpha\neq 0$. The associated Poisson bivector on $V^*$ in the 
corresponding coordinates is $P=P_{\gG}=\alpha x_3\xi_1\xi_2$. This shows 
that, up to proportionality, $\gG$ is uniquely defined by the pair $(C,l)$. 

Consider now two $\pitchfork$-lieons $\gG_1$ and $\gG_2$ on $V$, i.e., 
$|\gG_1|=|\gG_2|=V$, and put $C_i=C_{\gG_i}, \,l_i=l_{\gG_i}, 
\,P_i=P_{\gG_i}\,i=1,2, \,C_{12}=C_1\cap C_2$. Obviously, 
$n-4\leq\mathrm{dim}\,C_{12}\leq n-2$. 

Below we use formula (\ref{Schouten in coordinates}) for computations of the
occurring  Schouten brackets.
\begin{lem}\label{a1}
If $\mathrm{dim}\,C_{12}=n-4$, then $\gG_1$ and $\gG_2$ are compatible iff
\, $l_i\subset C_{12}, i=1,2$.
\end{lem} 
\begin{proof}
We have to examine the following four cases.

A$_1$\,: \,$l_i$ \emph{does not belong to} $C_{12}, \,i=1,2$. In this case 
a basis $e_5,\dots,e_n$ in $C_{12}$  can be completed by some vectors 
$e_1\in l_1, e_2\in C_1, e_3\in l_2, e_4\in C_2$ up to a basis in $V$. The 
only nonzero product $[e_i,e_j]_1$ in $\gG_1$ is $[e_3,e_4]_1=\alpha_1e_1$, 
and $[e_1,e_2]_2=\alpha_2e_3$ in $\gG_2$ with some
$\alpha_1, \alpha_2\in \gk$. Hence $P_1=\alpha_1x_1\xi_3\xi_4$ and 
$P_2=\alpha_2x_3\xi_1\xi_2$ and a direct computation by using formula 
shows that $\ls P_1,P_2\rs\neq 0$, i.e., that $\gG_1$ and $\gG_2$ are 
not compatible. 

A$_2$\,: \,\emph{One of subspaces} $l_i$\emph{'s, say,} $l_1$, 
\emph{belongs to} $C_{12}$ \emph{and the other,} $l_2$, \emph{does not}. In 
this case we complete a basis $e_5\in l_1,e_6,\dots,e_n$ in $C_{12}$ by 
some vectors $e_1,e_2\in C_1, \,e_3\in l_2,  \,e_4\in C_2$ up to a basis in 
$V$. By similar reasons as above, $P_1=\alpha_1x_5\xi_3\xi_4, 
\,P_2=\alpha_2x_3\xi_1\xi_2$ in the corresponding coordinates. Since $\ls 
P_1,P_2\rs\neq 0$, $\gG_1$ and $\gG_2$ are not compatible. 

A$_3$\,: \,$l_i\subset C_{12}, \,i=1,2,$ \emph{and} $l_1\neq l_2$. In this 
case we consider a basis $e_5\in l_1,e_6\in l_2,\dots,e_n$ in $C_{12}$ and 
complete it by independent $\mathrm{mod}\,C_{12}$ vectors $e_1,e_2\in C_1, 
\,e_3,e_4\in C_2$ up to a basis in $V$. In such a basis 
$P_1=\alpha_1x_5\xi_3\xi_4, \,P_2=\alpha_2x_6\xi_1\xi_2$ and 
$\ls P_1,P_2\rs=0$. So, $\gG_1$ and $\gG_2$ are compatible and it is 
easy to see that $\gG_1+\gG_2$ is isomorphic to 
$\pitchfork\oplus\pitchfork\oplus\gamma_{n-6}$. 

A$_4$\,: \;$l_1=l_2\subset C_{12}$. A basis $e_5\in l_1=l_2,e_6,\dots,e_n$ 
in $C_{12}$ can be completed by independent $\mathrm{mod}\,C_{12}$ vectors 
$e_1,e_2\in C_1, \,e_3,e_4\in C_2$ up to a basis in $V$. Then 
$P_1=\alpha_1x_5\xi_3\xi_4, \,P_2=\alpha_2x_5\xi_1\xi_2$ and $\ls 
P_1,P_2\rs=0$, i.e., $\gG_1$ and $\gG_2$ are compatible. The Poisson 
bivector corresponding to $\gG_1+\gG_2$ is proportional to 
$x_5(\xi_1\xi_2+\xi_3\xi_4)$. 
\end{proof}

\begin{lem}\label{b1}
If $\mathrm{dim}\,C_{12}=n-3$, then $\gG_1$ and $\gG_2$ are compatible.
\end{lem} 
\begin{proof}
Put $C=C_1+C_2$. Then $\mathrm{dim}\,C=n-1$. We have four 
qualitatively different situations as before. 

B$_1$\,: \;$l_1$ \emph{and} $l_2$ \emph{do not belong to} 
$C_{12}\,\Leftrightarrow\,C=l_1\oplus l_2\oplus C_{12}$. Consider a basis 
$e_1,\dots,e_n$ in $V$ with $e_i\in l_i, \,i=1,2, \,e_i\in C_{12}, \,i>3$. 
Then, in the corresponding coordinates, $P_1=\alpha_1x_1\xi_2\xi_3, 
\,P_2=\alpha_2x_2\xi_1\xi_3$. 

B$_2$\,: \;\emph{One of subspaces} $l_i$\emph{'s, say,} $l_1$, 
\emph{belongs to} $C_{12}$ \emph{and the other}, $l_2$, \emph{does not}. 
Then $V=\langle v\rangle\oplus C$, if $v\in V\setminus C$, and $C=\langle 
w\rangle\oplus C_{12}$, if $w\in C\setminus C_{12}$. So, vectors $e_1=w, 
\,0\leq e_2\in l_1, \,0\leq e_3\in l_2, \,e_4=v, 
\,e_5,\dots,e_n$ form a basis in $V$ assuming that $e_2,e_5\dots,e_n$ is a 
basis in $C_{12}$. In the corresponding coordinates we have 
$P_1=\alpha_1x_2\xi_3\xi_4, \,P_2=\alpha_2x_3\xi_1\xi_4$. 

B$_3$\,: \;$l_i\in C_{12}, \,i=1,2,$ \emph{and} $l_1\neq l_2$. In this case 
there is a basis $e_1\dots,e_n$ in $V$ such that $e_i\in C_i\setminus 
C_{12}, \,e_{i+2}\in l_i, \,i=1,2, \,e_5\in V\setminus C$ and $e_i\in 
C_{12}$ for $i>5$. Then, in the corresponding coordinates, 
$P_1=\alpha_1x_3\xi_2\xi_5, \,P_2=\alpha_2x_4\xi_1\xi_5$. 

B$_4$\,: \;$l_i\in C_{12}, \,i=1,2,$ and $l_1=l_2$. Similarly to the 
preceding case there is a basis $e_1\dots,e_n$ in $V$ such that $e_i\in 
C_i\setminus C_{12}, \,i=1,2, \, e_3\in l_1=l_2, \,e_4\in V\setminus C$ and 
$e_i\in C_{12}$ for $i>4$. Then, in the corresponding coordinates, 
$P_1=\alpha_1x_3\xi_2\xi_4, \,P_2=\alpha_2x_3\xi_1\xi_4$. 

Now a simple computation shows that $\ls P_1,P_2\rs=0$ in any of these cases. 
\end{proof}

\begin{lem}\label{c1}
If $\mathrm{dim}\,C_{12}=n-2$, then $\gG_1$ and $\gG_2$ are compatible. 
\end{lem} 
\begin{proof}
In this case $C_1=C_2=C_{12}$ and $V=\langle v_1\rangle\oplus\langle 
v_2\rangle\oplus C_{12}$ for any independent $\mathrm{mod}\,C_{12}$ vectors 
$v_1, v_2$. Here we have two possibilities: 

D$_1$\,: \;$l_1\neq l_2$. Consider a basis $e_1,\dots,e_n$ in $V$ with 
$e_i\in l_i, \,e_{i+2}=v_i,  \,i=1,2, \,e_j\in C_{12}$ for $j>4$. Then, as 
above, $P_1=\alpha_1x_1\xi_3\xi_4, 
\,P_2=\alpha_2x_2\xi_3\xi_4$ and $\ls P_1,P_2\rs=0$. 

D$_2$\,: \;$l_1=l_2$. In this case lieons $\gG_1$ and $\gG_2$ are 
proportional. 
\end{proof}

Using lemmas \ref{a1}, \ref{b1} and \ref{c1} it is not difficult to 
construct various families of mutually compatible $\pitchfork_n$-structures on a 
vector space $V$. Two such constructions are described below. 

A (finite) family 
$\{C_i\}$ of $(n-2)$-dimensional subspaces of $V$ will be called  \emph{tight} if 
$\mathrm{dim}\,C_i\cap C_j>n-4$. Tight families are easily described. 

\begin{lem}\label{tight}
A family $\{C_i\}$ of $(n-2)$-dimensional subspaces of $V$ is \emph{tight} 
if either all $C_i$'s are contained in a common 
$(n-1)$-dimensional subspace \emph{(``co-pencil")}, or  all $C_i$'s have a 
common $(n-3)$-dimensional subspace \emph{(``pencil")}. $\square$
\end{lem}

A finite family $\{\gG_i\}$ of $\pitchfork$-lieons on $V$ will be called 
\emph{tight} if the family $\{C_i\}$ of their centers is tight. It follows 
from lemmas \ref{b1} and \ref{c1} that $\pitchfork$-lieons belonging to a tight 
family are mutually compatible. This observation together with lemma \ref{tight} prove 

\begin{proc}\label{pencil}
Let $\{C_i\}$ be a co-pencil (resp., pencil) of $(n-2)$-dimensional 
subspaces of $V$. Assign to each $C_i$ a 1-dimensional subspace 
$l_i\subset C_i$. Then $\pitchfork$-lieons characterized by pairs $(C_i, l_i)$ 
are mutually compatible so that their linear combinations are first 
level Lie algebras. $\square$
\end{proc} 

Also we have.

\begin{proc}\label{pencil1}
Let $\{\gG_1,\dots,\gG_m\}$ be a family of $\pitchfork$-lieons 
characterized by pairs $(C_1,l_1),\dots,(C_m,l_m)$. If 
$\mathrm{span}\,(l_1,\dots,l_m)\subset\bigcap_{i=1}^m C_i$, then $\gG_i$'s 
are mutually compatible so that their linear combinations are first level Lie 
algebras. $\square$ 
\end{proc}
\begin{proof}
If $\mathrm{dim}\,C_i\cap C_j>n-4$, then $\gG_i$ and $\gG_j$ are compatible 
by lemmas \ref{b1} and \ref{c1}. Otherwise, they are compatible by lemma 
\ref{a1}.
\end{proof}

These constructions illustrate the diversity of combinations 
of $\pitchfork$-lieons that produce first level Lie algebras. 

\medskip
\subsection{Compatible $\pitchfork$- and $\between$- lieons.} 

A $\between$-lieon  $\gG$ on $V$  is  
up to proportionality characterized by its center $C=C_{\gG}$ and the 
derived algebra $\Delta=\Delta_{\gG}=[\gG,\gG]$. Since both $\Delta$ and 
$C$ are abelian, we identify them with the supporting them subspaces of $V$. 
Obviously, $\mathrm{dim}\,C=n-2, \,\mathrm{dim}\,\Delta=1, \,C\cap 
\Delta=\{0\}$. So, up to a scalar factor $\gG$ is completely determined 
by the pair $(C,\Delta)$ of subspaces of $V$ and vice-versa.

Consider now a $\pitchfork$-lieon $\gG_{\pitchfork}$ and  a $\between$-lieon
$\gG_{\between}$ on $V$ and the 
characterizing them pairs $(C_{\between},\Delta)$ and $(C_{\pitchfork},l)$. 
Put $C_{12}=C_{\between}\cap C_{\pitchfork}$. Then 
$n-4\leq\mathrm{dim}\,C_{12}\leq n-2$.
 
\begin{lem}\label{a2} If $\mathrm{dim}\,C_{12}=n-4$, then $\gG_{\between}$ and 
$\gG_{\pitchfork}$ are incompatible.
\end{lem}
\begin{proof} 
First, note that  $\gG_{\between}$ and $\gG_{\pitchfork}$ are compatible 
iff the factorized lieons $\gG_{\between}/C_{12}$ and 
$\gG_{\pitchfork}/C_{12}$ are compatible. So, we can assume that $n=4$ and 
$C_{12}=0 \Leftrightarrow \;V=C_{\between}\oplus C_{\pitchfork}$ with 
$\mathrm{dim}\,C_{\between}=\mathrm{dim}\,C_{\pitchfork}=2$. Projections 
defined by this splitting of $V$ send the line $\Delta$ to subspaces 
$\Delta^{\between}\subset C_{\between}$ and $\Delta^{\pitchfork}\subset 
C_{\pitchfork}$, respectively. Since $\Delta$ does not belong to 
$C_{\between}, \,\mathrm{dim}\,\Delta^{\pitchfork}=1$. There may occur one 
of the following three cases. 

I$_1$ : $\mathrm{dim}\,\Delta^{\between}=1$ and $C_{\pitchfork}=l\oplus 
\Delta^{\pitchfork}$. Let $e_1\in\Delta^{\between}, ,\ 
e_2\in\Delta^{\pitchfork}$ be such that $e_1+e_2$ generates $\Delta$. If 
$e_3$ generates $l$ and $e_1,e_4$ generate $C_{\between}$, then 
$e_1,\dots,e_4$ is a basis in $V$ and, in the corresponding coordinates, 
Poisson bivectors associated with $\gG_{\between}$ and $\gG_{\pitchfork}$ 
are proportional to $P_1=(x_1+x_2)\xi_2\xi_3$ and $P_2=x_3\xi_1\xi_4$, 
respectively. Now a computation shows that $\ls P_1,P_2\rs\neq 0$. 

I$_2$ : $\mathrm{dim}\,\Delta^{\between}=0$ and $C_{\pitchfork}=l\oplus 
\Delta^{\pitchfork}$. Consider a basis 
$e_1,\dots,e_4$ in $V$ with $e_1,e_2\in C_{\between}, \,e_3\in l, \,e_4\in 
\Delta^{\pitchfork}$. In the corresponding coordinates  $P_1$ and $P_2$ 
are proportional to $x_4\xi_3\xi_4$ and $x_3\xi_1\xi_2$, respectively, and 
 one finds that $\ls P_1,P_2\rs\neq 0$. 

I$_3$ : $\mathrm{dim}\,\Delta^{\between}=0$ and $l=\Delta^{\pitchfork}$. 
Similarly, in a basis $e_1,\dots,e_4$ in $V$ with $e_1,e_2\in C_{\between}, 
\,e_3\in l, \,e_4\in \C_{\pitchfork}$  $P_1$ and $P_2$ 
are proportional to $x_3\xi_3\xi_4$ and $x_3\xi_1\xi_2$, respectively, and a 
computation show that  $\ls P_1,P_2\rs\neq 0$. 
\end{proof}

Put $C=C_{\between}+C_{\pitchfork}$ and note that $\mathrm{dim}\,C=n-1$ iff 
$\mathrm{dim}\,C_{12}=n-3$. In this case there are two possibilities : $\Delta\cap 
C=\{0\}$ and $\Delta\subset C$. 

\begin{lem}\label{b2} If $\mathrm{dim}\,C_{12}=n-3$ and $\Delta\cap 
C=\{0\}$, then $\gG_{\between}$ and $\gG_{\pitchfork}$ are incompatible.
\end{lem}
\begin{proof} 
In this case $V=\Delta\oplus C$. If $l$ does not belong to $C_{12}$, then, 
as in the preceding lemma, the factorization by $C_{12}$ reduces the problem 
to $n=3$. If $n=3$, then 
$\mathrm{dim}\,C_{\between}=\mathrm{dim}\,C_{\pitchfork}=1, 
\,l=C_{\pitchfork}$ and $V=\Delta\oplus C_{\between}\oplus l$. In a basis 
$e_1,e_2,e_3$ of $V$ with $e_1\in l, \,e_2\in C_{\pitchfork}, 
\,e_3\in\Delta$ we have $P_1\sim x_3\xi_1\xi_3, \,P_2\sim x_1\xi_2\xi_3$ 
and find that $\ls P_1,P_2\rs\neq 0$. 

If $l\subset C_{12}$, then $C_{12}=l\oplus C', \mathrm{dim}\,C'=n-4$, and 
the factorization by $C'$ reduces the situation to $n=4$. In this 
particular case $\mathrm{dim}\,C_{\between}=\mathrm{dim}\,C_{\pitchfork}=2$ 
and $C_{\between}\cap C_{\pitchfork}=l$. In 
a basis $e_1,\dots,e_4$  such that $e_1\in C_{\pitchfork}, e_2\in 
C_{\between}, e_3\in\Delta, e_4\in l$ we 
have $P_1\sim x_3\xi_1\xi_3 , P_2\sim x_4\xi_2\xi_3$ with $\ls P_1,P_2\rs\neq 0$. 
\end{proof}

\begin{lem}\label{c2} If $\mathrm{dim}\,C_{12}=n-3$ and $\Delta\subset C$, then 
$\gG_{\between}$ and $\gG_{\pitchfork}$ are compatible.
\end{lem}
\begin{proof} 
In this case $C=\Delta\oplus C_{\between}$ and $V=W\oplus C$ for a 
$1$-dimensional subspace $W$ of $V$. If $l$ does not belong to $C_{12}$, the 
factorization of $V$ by $C_{12}$ reduces the situation to a 
$3$-dimensional one in which $C=C_{\pitchfork}\oplus C_{\between}$, \, 
$\mathrm{dim}\,C_{\between}=\mathrm{dim}\,C_{\pitchfork}=1, 
\,l=C_{\pitchfork}$ and $V=C_{\pitchfork}\oplus C_{\between}\oplus W$. 
Consider a basis $e_1,e_2,e_3$ in $V$ with $e_1\in C_{\pitchfork}, \,e_2\in 
C_{\between}, e_3\in W$. Since $\Delta\subset C_{\pitchfork}\oplus 
C_{\between}$ and $\Delta\cap C_{\between}=\{0\}, \,\Delta$ is generated by 
a vector of the form $e_1+\lambda e_2, \lambda\in \gk,$. In the 
corresponding to such a basis coordinates we have $P_1\sim x_1\xi_2\xi_3, 
P_2\sim (x_1+\lambda x_2)\xi_1\xi_3$ with $\ls P_1,P_2\rs=0$. 

If $l\subset C_{12}$, then $C_{12}=l\oplus C', \mathrm{dim}\,C'=n-4$, and 
the factorization of $V$ by $C'$ reduces the situation to $n=4$. In this 
case $\mathrm{dim}\,C_{\between}= \mathrm{dim}\,C_{\pitchfork}=2$ and 
$\mathrm{dim}\,C_{\between}\cap \mathrm{dim}\,C_{\pitchfork}=l$. Consider a 
basis $e_1\dots,e_4$ in $V$ such that $e_1\in C_{\pitchfork}, e_2\in 
C_{\between}, e_3\in l, e_4\in W$. Since $\Delta\cap C_{\between}=\{0\}$,  
$\Delta\subset C=C_{\between}+C_{\pitchfork}$ is generated by a vector 
of the form  $e_1+\lambda e_2+\mu e_3$. In the 
corresponding coordinates, we have $P_1\sim x_3\xi_2\xi_4, P_2\sim (x_1+\lambda 
x_2+\mu x_3)\xi_1\xi_4$ and $\ls P_1,P_2\rs=0$. 
\end{proof}
\begin{lem}\label{c2} 
If $\mathrm{dim}\,C_{12}=n-2$, then $\gG_{\between}$ and $\gG_{\pitchfork}$ 
are compatible. 
\end{lem}
\begin{proof} 
In this case $C_{\between}=C_{\pitchfork}=C$ and $V=\Delta\oplus W\oplus C, 
\,\mathrm{dim}\,W=1$. In a basis $e_1\dots,e_n$ in $V$ such that $e_1\in 
\Delta, \,e_2\in W, \,e_3\in l\subset C, \,e_4,\dots,e_n\in C$, we have
 $P_1\sim x_3\xi_1\xi_2, P_2\sim x_1\xi_1\xi_2$ and see that
$\ls P_1,P_2\rs=0$.  
\end{proof}

A summary of the above lemmas is
\begin{proc}\label{puni}
A $\between$-lieon and  a 
$\pitchfork$-lieon on $V$ are compatible iff the intersection of their centers is not 
generic, i.e., of dimension grater than $n-4$. 
\end{proc}

The following assertion immediately results from propositions \ref{pencil} 
and \ref{puni}  
 
\begin{cor}
Let $\gG_1,\dots,\gG_m$ be $\pitchfork$-lieons and $\gG$ a 
$\between$-lieon. If centers of all these lieons are contained in a 
common hyperplane in $V$, then  $\gG+\alpha_1\gG_1+\dots+\alpha_m\gG_m, 
\,\alpha_1,\dots,\alpha_m\in \gk$, is a non-unimodular first level Lie 
algebra. 
\end{cor}

\medskip
\subsection{Compatible $\between$-lieons.} 
 
Consider two $\between$-lieons $\gG_i, \,i=1,2,$ on a vector space 
$V$ and the characterizing them pairs $(\Delta_i,C_i), \,i=1,2$. Recall 
that $\mathrm{dim}\,\Delta_i=1, \,\mathrm{dim}\,C_i=n-2$ and $\Delta_i\cap 
C_i=\{0\}$. Put $C_{12}=C_1\cap C_2$. Obviously, $n-4\leq 
\mathrm{dim}\,C_{12}\leq n-2$ and compatibility of $\gG_i$'s is 
equivalent to that of factorized $\between$-lieons $\gG_i/C_{12}$'s. 

\begin{lem}\label{a3} 
If $\mathrm{dim}\,C_{12}=n-4$, then $\gG_1$ and $\gG_2$ are compatible iff
$\Delta_1\subset C_2$ and $\Delta_2\subset C_1$. 
\end{lem}
\begin{proof} 
Passing to the factorized structures $\gG_i/C_{12}$'s we can assume that 
$\mathrm{dim}\,V=4$. In this particular case 
$\mathrm{dim}\,C_1=\mathrm{dim}\,C_2=2$ and $V=C_1\oplus C_2$. Let $p_i : 
V\rightarrow C_i$ be a natural projection and $L=<\Delta_1,\Delta_2>$ be 
the span of $\Delta_1$ and $\Delta_2$. Examine now various situations  
occurring in this context.

K$_1$ :  $\mathrm{dim}\,L=2$ and $L\cap C_i=\{0\}, \,i=1,2$. Then 
$p_i|_L:L\rightarrow C_i$ is an isomorphism, $i=1,2,$ and hence there is a 
basis  $e_1\dots,e_4$ in $V$ such that $e_1, e_2\in C_1, \,e_3, e_4\in C_2$ 
and $e_1+e_3\in \Delta_1, e_2+e_4\in \Delta_2$. In the corresponding 
coordinates we have $P_1\sim (x_1+x_3)\xi_3\xi_4, P_2\sim (x_2+x_4)\xi_1\xi_2$ 
and a computation shows that $\ls P_1,P_2\rs\neq 0$. 

K$_2$ :  $\mathrm{dim}\,L=2, \,\mathrm{dim}\,L\cap C_1=1$ and $L\cap 
C_2=\{0\}$. Then $p_1|_L$ is an isomorphism. So, if $0\neq\varepsilon_i\in 
\Delta_i, \,i=1,2$, then $e_1=p_1(\varepsilon_1), e_2=p_1(\varepsilon_2)$ 
is a basis in $C_1$. Also $e_3=p_2(\varepsilon_1)\neq 0$, since 
$\Delta_1\cap C_1=\{0\}$, and $p_2(\varepsilon_1)$ and $p_2(\varepsilon_2)$ 
are proportional. If $e_4\in C_2$ is not proportional to $e_3$, then 
$e_1,\dots,e_4$ is a basis in $V$. By construction $\varepsilon_1=e_1+e_3$ 
and $\varepsilon_2=e_2+\lambda e_3$. So, in the corresponding coordinates, 
$P_1\sim (x_1+x_3)\xi_3\xi_4, P_2\sim (x_2+\lambda x_3)\xi_1\xi_2$ and 
we can see that $\ls P_1,P_2\rs\neq 0$. 

K$_3$ :  $\mathrm{dim}\,L=2, \,\mathrm{dim}\,L\cap C_i=1, \,i=1,2$. If 
$\varepsilon_1, \,\varepsilon_2$ are as above, then 
$e_3=p_2(\varepsilon_1)\neq 0, \,e_1=p_1(\varepsilon_2)\neq 0$ and 
$p_2(\varepsilon_2)=\lambda e_3, \,p_1(\varepsilon_1)=\mu e_1$ for some 
$\lambda, \mu\in\gk$. By construction $\varepsilon_1=\mu e_1+e_3, 
\,\varepsilon_2=e_1+\lambda e_3$. Complete vectors $e_1, e_3$ to a basis in 
$V$ by vectors $e_2\in C_1, e_4\in C_2$. Then  $P_1\sim (\mu 
x_1+x_3)\xi_3\xi_4, P_2\sim (x_1+\lambda x_3)\xi_1\xi_2$ in the corresponding 
coordinates, and $\ls P_1,P_2\rs\sim -\lambda(\mu 
x_1+x_3)\xi_1\xi_2\xi_4-\mu(x_1+\lambda x_3)\xi_2\xi_3\xi_4$. Now a
computation shows that $\ls P_1,P_2\rs=0$ iff $\mu=\lambda=0$. Geometrically, 
this condition tells that $\Delta_1\subset C_2, \,\Delta_2\subset C_1$, or, 
equivalently, that $\gG_1+\gG_2$ is isomorphic to $\between\oplus\between$ 
for $n=4$ and to  $\between\oplus\between\oplus\gamma_{n-4}$ in the general 
case. 

K$_4$ :  $\mathrm{dim}\,L=1 \Leftrightarrow \Delta_1=\Delta_2$. In this 
case one easily constructs a basis $e_1,\dots,e_4$ in $V$ with $e_1, e_2\in 
C_1$ and $e_3, e_4\in C_2$ and $e_1+e_3\in \Delta_1=\Delta_2$. As earlier we 
see that $P_1\sim (x_1+x_3)\xi_3\xi_4, \,P_2\sim (x_1+x_3)\xi_1\xi_2$ and $\ls 
P_1,P_2\rs\neq 0$. 
\end{proof}

\begin{lem}\label{b3} 
If $\mathrm{dim}\,C_{12}=n-3$, then $\gG_1$ and $\gG_2$ are compatible 
either if $\Delta_1=\Delta_2 \,\mathrm{mod} \, C_{12}$, or  if $\Delta_i\subset C_1+C_2, \,i=1,2$. 
\end{lem}
\begin{proof} 
The factorization $\mathrm{mod}\,C_{12}$ reduces the problem, as above,  to 
$n=3$. In this case $\mathrm{dim}\,C_1=\mathrm{dim\,C_2}=1$ and $C_1\cap 
C_2=\{0\}$. Equivalently, if $C=C_1+C_2$, then $\mathrm{dim}\,C=2$. Here two 
possibilities occur: 

J$_1$ : One of $\Delta_i$'s, say, $\Delta_1$, does not belong to $C$. In a 
basis $e_1, e_2, e_3$ in $V$ with $e_i\in C_i, \,e_3\in \Delta_1$ we have 
$P_1\sim x_3\xi_2\xi_3, \,P_2\sim\sum_{i=1}^{3}\alpha_ix_i\xi_1\xi_3$ and $\ls 
P_1,P_2\rs\sim (\alpha_1x_1+\alpha_2x_2)\xi_1\xi_2\xi_3 $. Now a 
computation shows that $\ls P_1,P_2\rs=0$ iff 
$\alpha_1=\alpha_2=0 \Leftrightarrow \Delta_1=\Delta_2$. 
For arbitrary $n$ the last condition means that $\Delta_1\subset 
\Delta_2\oplus C_{12}$. 

J$_2$ : $\Delta_i\subset C, \,i=1,2$. Let $0\neq e_i\in C_i, \,i=1,2$. Then 
$e_1+\lambda e_2\in \Delta_2$ and $\mu e_1+e_2\in \Delta_1$ for some 
$ \lambda, \mu\in\gk$. Complete $e_1, e_2$ to a basis in $V$ by a vector 
$e_3$. Then $P_2\sim (x_1+\lambda x_2)\xi_1\xi_3$,  
$P_1\sim (\mu x_1+x_2)\xi_2\xi_3$ and $\ls P_1,P_2\rs=0$. 
\end{proof}
\begin{rmk}
The results of this section show that compatible configurations of lieons
can be described in a manner which does not refer explicitly to a concrete
ground field  $\gk$. Namely, this description operates with the characterizing 
pairs and the relative position of composing them elements. Concreteness
of  $\gk$ is exclusively confined to coefficients of linear combinations 
of basic ``abstract" lions from which first level Lie algebras over  $\gk$
are made.
\end{rmk}
It is not difficult to extract from the proof of theorems \ref{C-dis} and 
\ref{R-dis} that there is a number $\nu(n)$ such 
that any $n$-dimensional Lie algebra can be assembled from  not more than 
$\nu(n)$ lieons. On the other hand, the results of this section show that 
even first level Lie algebras can be assembled from an unlimited number of 
$\pitchfork$- and $\between$-lieons intertwined one another in a 
chaotic manner. This makes nontrivial the problem of recognizing isomorphic 
Lie algebras on the basis of their $a$-schemes.  By this reason more regular 
assembling procedures are of interest. One of them, in a sense simplest, will
be discussed the sections dedicated to coaxial algebras.

\section{Canonical disassemblings of classical Lie algebras}\label{classical}

In this section we shall describe \emph{canonical}, in a sense, complete
disassemblings of classical Lie algebras.  This will be done in a way which 
simultaneously covers Lie algebras over $\R$ and $\C$. More exactly, classical 
Lie algebras are among symmetry algebras of some bilinear and volume forms, 
and this interpretation make it possible to completely disassemble them over an 
arbitrary ground field $\gk$ of characteristic zero. The techniques we use here
are mainly based on the Schouten bracket formalism and the stripping procedure.

\medskip
\subsection{Disassembling of g-orthogonal algebras.}
\label{SO-dis} 

\noindent Let $g=\sum_1^n a_ix_i^2, \,0\neq a_i\in\gk,$ be a nondegenerate 
quadratic form on a $\gk$-vector space $V$. The Lie algebra $\gs\go(g)$ of 
(infinitesimal) symmetries of $g$ is composed of linear vector fields $X$ 
on $V$ such that $X(g)=0$. Obviously, 
$e_{ij}=a_ix_i\partial_j-a_jx_j\partial_i\in \gs\go(g)$ and 
$[e_{ij},e_{jk}]=a_je_{ik}$. Moreover, fields $\{e_{ij}\}_{i<j}$  form a 
base of $\gs\go(g)$. For instance,  this is a standard base of $\gs\go(p,q), \,q=n-p$,
if $\gk=\R$ and  $a_1=\cdots a_p=1, \,a_{p+1}=\cdots=a_n=-1$. 

Let $x_{ij}$ be the linear function on $|\gs\go(g)|^*$ corresponding to 
$e_{ij}$. Obviously, $x_{ij}=-x_{ji}$ and  $\{x_{ij}\}_{i<j}$ is a 
cartesian chart on $|\gs\go(g)|^*$. The only nonzero Poisson brackets, which 
involve functions $x_{ij}$'s, are $\{x_{ij},x_{jk}\}=a_jx_{ik}$. The Poisson 
bivector   
$$
P=\sum_{i<j,\alpha}a_{\alpha}x_{ij}\xi_{i\alpha}\wedge\xi_{\alpha j} \quad 
\mathrm{with} \quad \xi_{ij}=\frac{\partial}{\partial x_{ij}} 
$$ 
represents the associated with $\gs\go(g)$ Poisson structure on 
$|\gs\go(g)|^*$ in terms of coordinates $\{x_{ij}\}$. Observe that
\begin{equation}\label{dec-ort}
P=\sum_{\alpha}a_{\alpha}P_{\alpha}\quad\mbox{with}\quad 
P_{\alpha}=\sum_{i<j}x_{ij}\xi_{i\alpha}\wedge\xi_{\alpha j}
\end{equation}
Since $P$ is a Poisson bivector for arbitrary $a_{\alpha}$'s, this shows
that $P_1,\dots, P_n$ are mutually compatible Poisson bivectors. 

The same result may be obtained by observing that
$$
2P_{\alpha}=\ls P,X_{\alpha}\rs=\ls P_{\alpha},X_{\alpha}\rs \;\mathrm{with} \; X_{\alpha}=
\sum_sx_{\alpha s}\xi_{\alpha s},
\quad\mathrm{and} \quad \ls P_{\alpha},X_{\beta}\rs=0, \;\forall \alpha\neq\beta .
$$
Indeed, $\ls P, P_{\alpha}\rs=\frac{1}{2}\partial_P^2(X_{\alpha})=0$ and
$$
\ls P_{\alpha}, P_{\beta}\rs=\frac{1}{2}\ls P_{\alpha}, \ls P, X_{\beta}\rs\rs=
\pm\frac{1}{2}\ls\ls P_{\alpha}, P\rs, X_{\beta}\rs
\pm\frac{1}{2}\ls P,\ls P_{\alpha}, X_{\beta}\rs\rs=0.
$$
So, finally, we have
$$
P_{\alpha}=\sum_{i<j}P_{\alpha,ij}\quad\mbox{with}\quad P_{\alpha,ij}=
x_{ij}\xi_{i\alpha}\wedge\xi_{\alpha j} .
$$ 
For a fixed $\alpha$ Poisson bivectors $P_{\alpha,ij}$'s are, obviously,  compatible 
each other. Each of them is associated with an algebra isomorphic to $\pitchfork_m, 
\,m=n(n-1)/2$. This shows that the algebra $\gs\go(g)$ can be assembled in two steps 
from $n(n-1)(n-2)/2$  $\pitchfork$-lieons. 

By translating the above results in terms of Lie brackets one easily finds 
that
$$
[\cdot,\cdot]=[\cdot,\cdot]_1+\cdots+[\cdot,\cdot]_n
$$ 
where $[\cdot,\cdot]$ stands for the Lie bracket in $\gs\go(g)$ and 
the structure $[\cdot,\cdot]_{\alpha}$ is defined by relations 
$$
[e_{i\alpha},e_{\alpha j}]_{\alpha}=a_{\alpha}e_{ij} \quad \mathrm{and}
\quad [e_{ij},e_{kl}]_{\alpha}=0, \;\mathrm{if} \;\alpha\notin \{i,j\}\cap\{k,l\}.
$$  
In its turn, $[\cdot,\cdot]_{\alpha}=\sum_{i<j}[\cdot,\cdot]_{\alpha,ij}$ 
where the only nontrivial product $[\cdot,\cdot]_{\alpha,ij}$ involving 
base vectors $e_{kl}$'s is $[e_{i\alpha},e_{\alpha 
j}]_{\alpha,ij}=a_{\alpha}e_{ij}$. 
\begin{rmk}\label{ort-any-fields}
The Poisson bivectors $P_{\alpha}=\sum_{i<j}x_{ij}\xi_{i\alpha}\wedge\xi_{\alpha j}$
may be interpreted as bivectors over the ring $\Z[x_{ij}]_{1\leq i < j\leq n}$. Any formal
linear combination of these bivectors with coefficients in a field $\gk$ is naturally 
interpreted as a linear bivector over the polinomial algebra $\gk[x_{ij}]_{1\leq i < j\leq n}$,
i.e., as a Lie algebra over $\gk$. In this sense $P_{\alpha}$'s are universal building blocks
for $g$-orthogonal algebras. Fir instance,  if $\gk=\R$, then
$$
P_1+\dots+P_s-P_{s+1}-\dots-P_n, \quad s=n-r.
$$
is the Poisson bivector associated with $\gs\go(r,s)$.
\end{rmk}

\medskip
\subsection{Disassembling of symplectic Lie algebras $\gs\gp(2n)$.}
\label{SP-dis}  

\noindent Let $\beta(v,w)$ be a nondegenerate skew- 
symmetric form on a $\gk$--vector space $V$. Then the dimension of $V$  
is even, say, $2n$, and there exists  a (canonical) basis  
$\{e_1,\ldots,e_n,e_1^{\prime},\ldots,e_n^{\prime}\}$ in $V$ such that 
$$
\beta(e_i,e_j^{\prime})=\delta_{ij}, \quad\beta(e_i,e_j)= 
\beta(e_i^{\prime},e_j^{\prime})=0, \quad i,j=1,\ldots,n.
$$ 
The symplectic Lie algebra  
$\gs\gp(\beta)$ consists of operators 
$A\in \mathrm{End}\,V$ such that 
$$
\beta(Av,w)+\beta(v,Aw)=0, \;v,w\in V .
$$ 
The algebra $\gs\gp(\beta)$ can be completely disassembled essentially by the 
same method as  for orthogonal algebras. It will be described
below in a form, which is better adapted to the symplectic situation. 

Let $(p_1,\ldots,p_n,q_1,\ldots,q_n)$ be coordinates in $V$ with respect to 
the above basis and $\omega=\sum_idp_i\wedge dq_i$. Then the algebra  
$\gs\gp(\beta)$ may be interpreted as the algebra of linear vector fields 
$X$ on $V$ such that $L_X(\omega)=0$. They are 
hamiltonian (with respect to $\omega$) fields $X_f$ corresponding  to
quadratic in $p$'s and $q$'s hamiltonians $f=f(p,q)$.
So, in this interpretation hamiltonian fields corresponding to monomials 
$p_ip_j, q_iq_j, p_iq_j, \,i,j=1,\ldots,n$, form a base of $\gs\gp(\beta)$ and the 
Lie product in $\gs\gp(\beta)$ is interpreted as commutator of vector fields. 
Alternatively, the identification $f\leftrightarrow X_f, 
\;\{f,g\}\leftrightarrow[X_f,X_g]=X_{\{f,g\}}$ allows us to interpret 
$\gs\gp(\beta)$ as the Lie algebra of quadratic polynomials 
$\gk_2[p,q]=\gk_2[p_1,\ldots p_n,q_1,\ldots,q_n]$ in $p$'s and $q$'s with 
respect to the Poisson bracket $\{\cdot,\cdot\}$ determined by the Poisson 
bivector $\Pi=\sum_i\partial_{p_i}\wedge\partial_{q_i}$. In other words, we 
model the algebra $\gs\gp(2n,\gk)$ as the vector space $\gk_2[p,q]$ 
supplied with the bracket $[\cdot,\cdot]=\{\cdot,\cdot\}|_{\gk_2[p,q]}$. 
Observe that $\Pi=\Pi_1+\ldots+\Pi_n$ with 
$\Pi_i=\partial_{p_i}\wedge\partial_{q_i}$ and denote by 
$\{\cdot,\cdot\}^i$ the bracket associated with the Poisson bivector 
$\Pi_i$. Then $[\cdot,\cdot]=[\cdot,\cdot]_1+\cdots+[\cdot,\cdot]_n$ with  
$[\cdot,\cdot]_i=\{\cdot,\cdot\}^i|_{\gk_2[p,q]}$. Since bivectors 
$\Pi_i$'s are compatible each other, the brackets $[\cdot,\cdot]_i$'s are 
mutually compatible as too and we get the disassembling 
$$
(\gk_2[p,q],\,[\cdot,\cdot])=(\gk_2[p,q],\,[\cdot,\cdot]_1)+
\ldots+(\gk_2[p,q],\,[\cdot,\cdot]_n).  
$$
Obviously, Lie algebras $\gs\gp_i(2n,\gk)=(\gk_2[p,q],\,[\cdot,\cdot]_i), 
\,i=1,\ldots,n$, are isomorphic one to another. So, it suffices to 
completely disassemble one of them, say, $\gs\gp_1(2n,\gk)$. To this end, 
observe that the Levi-Malcev decomposition of $\gs\gp_1(2n,\gk)$ is
$$
\gs\gp_1(2n,\gk)=\langle p_1^2,p_1q_1,q_1^2\rangle\oplus
\langle p_1p_i,p_1q_i,q_1p_i,q_1q_i,p_ip_j,p_iq_j,q_iq_j\rangle_{1<i,j\leq n}
$$
where $\langle a,\ldots,b\rangle$ denotes the subalgebra of 
$\gs\gp_1(2n,\gk)$ spanned by $a,\ldots,b$.

The semisimple part $\mathfrak{s}=\langle p_1^2,p_1q_1,q_1^2\rangle$ of 
$\gs\gp_1(2n,\gk)$ is isomorphic to $\gs\gl(2,\gk)$. The radical $\mathfrak{r}$
of it is 
$$
\mathfrak{r}=
\langle p_1p_i,p_1q_i,q_1p_i,q_1q_i,p_ip_j,p_iq_j,q_iq_j\rangle_{1<i,j\leq n}
$$
and
$$
\mathfrak{c}=\langle p_ip_j,p_iq_j,q_iq_j\rangle_{1<i,j\leq n}
$$ 
is the center of $\mathfrak{r}$. Note that 
$ [\mathfrak{r},\mathfrak{r}]\subset \mathfrak{c}$. 

According to 
proposition \ref{str}, the algebra $\gs\gp_1(2n,\gk)$ is assembled from 
$\mathfrak{s}\oplus_{\rho}|\mathfrak{r}|$ and 
$\gamma_m\oplus\,\mathfrak{r}$ for a suitable $m$ where $\rho$ stands for a 
natural representation of $\gs$ in the vector space $|\mathfrak{r}|$  
supporting  the ideal $\mathfrak{r}$. So, it remains to disassemble each of 
these two algebras. 

The algebra $\mathfrak{r}$ contains the following Heisenberg subalgebras: 
\begin{eqnarray}
\mathfrak{h}_{ij}^{pp}=\langle p_1p_i,q_1p_j,p_ip_j\rangle, \quad\quad
\mathfrak{h}_{ij}^{pq}=\langle p_1p_i,q_1q_j,p_iq_j\rangle \\
\qquad\mathfrak{h}_{ij}^{qp}=\langle p_1q_i,q_1p_j,p_jq_i\rangle, 
\quad\quad \mathfrak{h}_{ij}^{qq}=\langle p_1q_i,q_1q_j,q_iq_j\rangle 
\end{eqnarray}

\noindent Each subalgebra $\mathfrak{h}_{ij}^{ab}$
from this list naturally extends to the unique $\pitchfork$--lieon $r$--dumennsional
$\pitchfork_{ij}^{ab}$ on $|\mathfrak{r}|, \,r=\dim\,\mathfrak{r}$,  whose 
center contains all quadratic monomials, which do not figure in the
definition of  $\mathfrak{h}_{ij}^{ab}$. It is easy to check (see subsection\,6.1)
that lieons $\pitchfork_{ij}^{ab}$'s  are 
compatible each other and hence completely disassemble $\mathfrak{r}$. 

Aiming to disassemble the algebra 
$\mathfrak{s}\oplus_{\rho}|\mathfrak{r}|$ consider the following 
$d$-pair in it:
\begin{equation}\label{long-d-pair}
(\,\langle p_1q_1,V_p\rangle, \langle p_1^2,q_1^2,V_q,\mathfrak{c}\rangle\,)
\;\mathrm{with} \; V_p=\langle p_1p_i,p_1q_i\rangle_{1<i\leq n},
V_q=\langle q_1p_i,q_1q_i\rangle_{1<i\leq n}
\end{equation}
Easily verified inclusions
$$
\begin{array}{rcc}
[V_p,V_p]_1=[V_q,V_q]_1=0,& [V_p,V_q]_1\subset\mathfrak{c}, &[\langle p_1^2\rangle,V_p]_1=
[\langle q_1^2\rangle,V_q ]_1=0, \\ 
\quad\quad [\langle p_1^2\rangle,V_q]_1\subset V_p,& [\langle q_1^2\rangle,V_p]_1\subset V_q,&
[\langle p_1^2,q_1^2\rangle,\langle p_1^2,q_1^2\rangle]_1\subset \langle p_1q_1\rangle, \\

&[\langle
p_1^2,q_1^2\rangle , 
V_p]_1\subset V_q,   & 
[\langle p_1^2,q_1^2\rangle,V_q]_1\subset V_p,
\end{array}
$$
prove that (\ref{long-d-pair}) is a d-pair.

Nontrivial relations among quadratic monomials in the dressing algebra 
of this $d$-pair are $[p_1^2,q_1^2]_1=4p_1q_1, \,[p_1^2, q_1p_i]_1=2p_1p_i, 
\,[p_1^2, q_1q_i]_1=2p_1q_i, \,1<i\leq n,$. So, the triples 
$(p_1^2,q_1^2,p_1q_1), (p_1^2,q_1p_i,p_1p_i), \,(p_1^2,q_1q_i,p_1q_i), 
\,1<i\leq n,$ span subalgebras of the dressing algebra, which are isomorphic 
to $\pitchfork$. As in the case of subalgebras $\mathfrak{h}_{ij}^{ab}$, these
subalgebras naturally extend to some $\pitchfork$-lieons. These extensions 
are mutually compatible and, therefore, disassemble the dressing algebra. 

Now it remains to disassemble the algebra 
$$
(\,\langle p_1q_1,V_p\rangle \oplus_{\varrho}\langle p_1^2,q_1^2,V_q,\mathfrak{c}\rangle
$$
where $\varrho$ is a natural representation of the subalgebra $\,\langle 
p_1q_1,V_p\rangle$ in the vector space 
$\langle p_1^2,q_1^2,V_q,\mathfrak{c}\rangle$. This algebras is, in fact, 
the semidirect product
$$
\mathfrak{a}=\langle p_1q_1\rangle\oplus_{\varsigma}(V_p\oplus_{\rho}
\langle p_1^2,q_1^2,V_q,\mathfrak{c}\rangle)
$$
where the action $\varsigma$ of $\langle p_1q_1\rangle$ on $V_p$ and  
$\langle p_1^2,q_1^2,V_q,\mathfrak{c}\rangle$ is induced by the bracket 
$[\cdot,\cdot]_1$. It is easy to see that nontrivial relations among 
elements of the basis of the ideal 
$\mathfrak{i}=V_p\oplus_{\rho}\langle p_1^2,q_1^2,V_q,\mathfrak{c}\rangle
\subset\mathfrak{a}$ are 
$$
[q_1^2,p_1r]_1=2q_1r\in V_q, \;[p_1r,q_1s]_1=rs\in \mathfrak{c} \quad
\mathrm{with} \quad r,s=p_i,q_j, \,0<i,j\leq n.
$$
Now we see that the triples $(q_1^2,p_1r,2q_1r), \,(p_1r,q_1s,rs)$  
span Heisenberg subalgebras in $\mathfrak{i}$. By the same arguments 
as before, their natural extensions disassemble the algebra 
$\mathfrak{i}$  into a number of $\pitchfork$--lieons.

The final step is to disassemble the algebra
$$
\langle p_1q_1\rangle\oplus_{\varrho}|\mathfrak{i}|=\langle p_1q_1\rangle
\oplus_{\varrho}\langle p_1^2,q_1^2,V_p,V_q,\mathfrak{c}\rangle.
$$
Observe that $|\mathfrak{i}|$ is the direct sum of 
$\varrho$-invariant subspaces 
$$
\langle p_1^2,q_1^2\rangle,\quad\langle p_1r,q_1r\rangle \;\mathrm{with} \;r=p_i,q_i, \;0<i\leq n,
\quad \mathrm{and} \quad|\mathfrak{c}|.
$$ 
The action $\varrho$ on $|\mathfrak{c}|$ is trivial and each of  
subalgebras $\langle 
p_1q_1\rangle\oplus_{\varrho}(\langle p_1^2,q_1^2\rangle\oplus|\mathfrak{c}|), 
\,\langle 
p_1q_1\rangle\oplus_{\varrho}(\langle p_1r,q_1r\rangle\oplus|\mathfrak{c}|)$ 
is isomorphic to $2\between_{m}=2\pitchfork_m$ for a suitable $m$ (see 
formula (\ref{2h=2b}) and subsection\,5.2). 
This disassembles the algebra $\langle 
p_1q_1\rangle\oplus_{\varrho}|\mathfrak{i}|$ into
$2n-1$ $\pitchfork$--lieons.

\medskip
\subsection{Disassembling of $\gG\gl(n,\gk), \gs\gl(n,\gk), \gu(n)$ and 
$\gs\gu(n)$.}\label{gl-sl-u-dis} 

First, we shall construct a 3-step disassembling of $\gG\gl(n,\gk)$. It is
convenient to interpret this algebra as the algebra of linear vector fields
on an $n$--dimensional vector space $V$. Vector fields 
$e_{ij}=x_i\frac{\partial}{\partial x_j}, \,1\leq i,j\leq n$, form a natural basis of it. 
The nontrivial Lie products in this algebra are
$$
[e_{i\alpha},e_{\alpha j}]=e_{ij}, \;\mathrm{if} \;i\neq j,  \quad\mathrm{and}\quad 
[e_{i\alpha},e_{\alpha i}]=e_{ii}-e_{\alpha\alpha},  \;\mathrm{if}  \;i\neq \alpha.
$$
Let $z_{ij}$'s be the corresponding to $e_{ij}$'s coordinates on $|\gG\gl(n,\gk)|^*$. 
Then the associated with $\gG\gl(n,\gk)$ Poisson bivector is
\begin{equation}\label{gl-pois}
P=\sum_{1\leq i,j,\alpha\leq n}z_{ij}\xi_{i\alpha}\xi_{\alpha j} \quad \mathrm{with} \quad
\xi_{ij}=\frac{\partial}{\partial x_{ij}}.
\end{equation}
In the basis 
$$
\{e_{ij}^t=t_i t_j e_{ij}\}, \,0\neq t_i\in\gk, \,i=1,\dots,n,
$$ 
we, obviously, have 
$P=\sum t_{\alpha}^2z_{ij}^t \xi_{i\alpha}^t\xi_{\alpha j}^t$ where
$z_{ij}^t$ and $\xi_{ij}^t$ stand for coordinates and partial derivatives with respect
to this basis. The isomorphism identifying the second basis with the first one transforms
$P$ into the Poisson bivector
\begin{equation}\label{1-gl-dis}
P_t=\sum_{1\leq\alpha\leq n} t_{\alpha}^2P_{\alpha} \quad\mathrm{with} \quad P_{\alpha}=
\sum_{(i,j)\neq (\alpha,\alpha)}z_{ij}\xi_{i\alpha}\xi_{\alpha j}, \quad t=(t_1,\dots,t_n).
\end{equation}
This implies that $P_{\alpha}$'s are mutually compatible Poisson bivectors, and, in particular,
that $P=P_1+\dots+P_n$. In its turn, $P_{\alpha}$ disassembles as
\begin{equation}\label{2-gl-dis}
P_{\alpha}=\sum_{i,i\neq\alpha}(z_{i\alpha}\xi_{i\alpha}\xi_{\alpha\alpha}+
z_{\alpha i}\xi_{\alpha\alpha}\xi_{\alpha i})+\sum_{i,j,i\neq\alpha, j\neq \alpha}
(z_{ij}\xi_{i\alpha}\xi_{\alpha j})
\end{equation}
The first of Poisson bivectors in the left hand side of (\ref{2-gl-dis}) corresponds to
an algebra of the form $\Gamma_A$, while the second to a dressing algebra. They 
both can be simply disassembled into a number of $n^2$--dimensional 
$\pitchfork$-lieons (see (\ref{g-dec})).

However, Poisson bivectors $P_{\alpha}$'s in (\ref{1-gl-dis}) do not restrict to the 
subalgebra $\gs\gl(n,\gk)$ of $\gG\gl(n,\gk)$ and hence can not be used to 
disassemble it. With this purpose we consider another basis in $\gG\gl(n,\gk)$:
\begin{equation}\label{basis-gl}
e_{ij}^0=e_{ij}-e_{ji}, \,1\leq i < j\leq n, \;e_{ij}^1=e_{ij}+e_{ji}, \,1\leq i\leq j\leq n,
\end{equation}
which respects the matrix transposition. We also put
$e_{ij}^0=-e_{ji}^0, \,e_{ij}^1=e_{ji}^1$. The nontrivial Lie products of elements
of basis (\ref{basis-gl}) are
\begin{eqnarray}\label{com-for-gl}
[e_{ij}^{\sigma},e_{jk}^{\tau}]=e_{ik}^{\sigma+\tau}, \mathrm{if} \;i\neq k,
\;i\neq j, \;i\neq k;  & \nonumber \\  
\,[e_{ij}^{\sigma},e_{ji}^{\tau}]=2(e_{ii}^1-e_{jj}^1), 
\mathrm{if} \;i\neq j, \;{\sigma}\neq\tau ; 
&\, [e_{ij}^{\sigma},e_{jj}^1]=2e_{ij}^{\sigma+1}, \;\mathrm{if} \;i\neq j.
\end{eqnarray}
Here we interpret  upper indices of vectors $e_{ij}^{\epsilon}$ as elements of $\dF_2$.
The corresponding d--pair in $\gG\gl(\gk,n)$ is $(\gs=\langle e_{ij}^0\rangle, 
W=\langle e_{ij}^1\rangle)$. Obviously, $\gs$ is isomorphic to $\gs\go(g)$ for
$g=\sum_{i=1}^n x_i^2$, $\gs\subset \gs\gl(n,\gk)\subset \gG\gl(n,\gk)$ and 
$$
W_0\df W\cap\gs\gl(n,\gk)=\langle e_{ij}^1, e_{ii}^1-e_{jj}^1\rangle_{1\leq i\neq j\leq n}.
$$
In particular, $(\gs,W_0)$ is a d--pair in $\gs\gl(n,\gk)$. Denote by $\gA_{\beta}$
(resp., $\gA_{\beta_0}$) the dressing algebra (see subsection\,5.3) corresponding to the d--pair
$(\gs,W)$ (resp.,  in $(\gs,W_0)$), and by $\rho$ (resp., $\rho_0$) the corresponding
representation of $\gs$ in $W$ (resp., $W_0$). So, by construction, we have
\begin{equation}\label{dis-gl-sl}
 \gG\gl(n,\gk)=\gs\oplus_{\rho}W+\gA_{\beta},  \qquad
 \gs\gl(n,\gk)=\gs\oplus_{\rho_0}W_0+\gA_{\beta_0}.
\end{equation}
Moreover, we have
\begin{lem}\label{dis-u-su}
For $\gk=\R$ the algebras
$$
\gu(n,\gk)=\gs\oplus_{\rho}W+\gA_{-\beta},  \qquad
 \gs\gu(n,\gk)=\gs\oplus_{\rho_0}W_0+\gA_{-\beta_0}.
$$
are isomorphic to the unitary  and special unitary  Lie 
algebras, respectively.
\end{lem}
\begin{proof}
A direct check on the basis of relations (\ref{com-for-gl}).
\end{proof}
\begin{rmk}\label{field-as-glue}
The isomorphism class of algebras $\gG\gl_{\lambda}(n,\gk)\df\gs\oplus_{\rho}W+
\gA_{\lambda\beta}$ and $\gs\gl_{\lambda}(n,\gk)\df\gs\oplus_{\rho_0}W_0+
\gA_{\lambda\beta_0} \,\lambda\in\gk$, depend on the quadratic residue of $\lambda$. 
Namely, $\gG\gl_{\lambda}$ and  $\gG\gl_{\lambda^{\prime}}$ (resp., $\gs\gl_{\lambda}$ 
and  $\gs\gl_{\lambda^{\prime}}$) are isomorphic iff 
$\lambda^{\prime}=\lambda\mu^2, \,\mu\in\gk$. 
\end{rmk}

Since a dressing algebra can be simply disassembled into a number
of $\pitchfork$--lieons, we shall focus on the algebras $\gs\oplus_{\rho}W$ and
$\gs\oplus_{\rho_0}W_0$. In virtue of (\ref{dis-gl-sl}) and lemma\,\ref{dis-u-su}
a complete disassembling of this algebra automatically gives complete disassemblings
of algebras $ \gG\gl(n,\gk), \,\gs\gl(n,\gk), \,\gu(n,\gk)$ and $\gs\gu(n,\gk)$. 

Denote by  $z_{ij}$ (resp., $w_{pq}$) linear functions on $|\gs|$ (resp., $|W|$) 
corresponding to $e_{ij}^0$ (resp., $e_{pq}^1$). They together form a cartesian 
chart on  $|\gs\oplus_{\rho}W|=|\gs|\oplus|W|$. In this chart the Poisson bivector 
of the algebra $\gs\oplus_{\rho}W$ reads
\begin{equation}\label{gl-2dis}
Q=\sum_{\alpha,  i<j}z_{ij}\xi_{i\alpha}\xi_{\alpha j}+
\sum_{p\neq\alpha\neq q}w_{pq}\xi_{p\alpha}\eta_{\alpha q}+
2\sum_{p\neq q}w_{pq}\xi_{pq}\eta_{qq}
\end{equation}
with $\xi_{ij}=\frac{\partial}{\partial z_{ij}}, \eta_{pq}=\frac{\partial}{\partial w_{pq}}$.
By the same arguments as in subsection\,7.1, we see that if
$$
 Q_{\alpha}=
\sum_{i<j}z_{ij}\xi_{i\alpha}\xi_{\alpha j}+
\sum_{p,q, q\neq\alpha}w_{pq}\xi_{p\alpha}\eta_{\alpha q}+
2\sum_{p}w_{p\alpha}\xi_{p\alpha}\eta_{\alpha\alpha},
$$
then $Q=Q_1+\dots+Q_n$ is a simple disassembling of $Q$. 

In order to disassemble $Q_{\alpha}$ note that  single summand in the
left hand side of (\ref{gl-2dis}) are $n^2$--dimensional $\pitchfork$--lieons, 
and the only incompatible pairs of the are
$$
w_{pq}\xi_{p\alpha}\eta_{\alpha q}  \quad\mathrm{and} \quad 
w_{q\alpha}\xi_{q\alpha}\eta_{\alpha\alpha}, \quad p\neq q, \,p\neq \alpha, \,q\neq \alpha.
$$
This shows that
\begin{equation}\label{3gl-dis}
Q_{\alpha}^1=\sum_{i<j}z_{ij}\xi_{i\alpha}\xi_{\alpha j},\quad
Q_{\alpha}^2=\sum_{p,q, q\neq\alpha}w_{pq}\xi_{p\alpha}\eta_{\alpha q}, \quad
Q_{\alpha}^3=\sum_{p}w_{p\alpha}\xi_{p\alpha}\eta_{\alpha\alpha}
\end{equation}
are Poisson bivectors and $\ls Q_{\alpha}^1,Q_{\alpha}^2\rs=
\ls Q_{\alpha}^1,Q_{\alpha}^3\rs=0$. Moreover,  by using formula (\ref{Schouten 
in coordinates}) we easily find that $\ls Q_{\alpha}^2,Q_{\alpha}^3\rs=0$. Hence
$Q_{\alpha}=Q_{\alpha}^1+Q_{\alpha}^2+Q_{\alpha}^3$ is a simple disassembling
of $Q_{\alpha}$. Finally, it follows from formula (\ref{3gl-dis}) that Poisson bivectors 
$Q_{\alpha}^i$ are assembled from mutually compatible $\pitchfork$--lieons. 
This way we get a complete  common disassembling  of $\gG\gl(n,\gk)$ and 
$\gu(n,\gk)$ in 4 steps.

In order to apply a similar aproach to the algebra $\gs\oplus_{\rho_0}W_0$ 
we have to choose a base in $W_0$. A such one is 
$$
e_{ij}^1, \,1\leq i< j\leq n, \qquad e_i=\frac{1}{2}(e_{ii}-e_{11})=
x_i\partial_i-x_1\partial_1, \,1<i\leq n.
$$
Denote by $w_{ij}$ and $w_i$ the corresponding linear functions on 
$W_0^*$, respectively. Together with functions $z_{ij}$'s 
they form a cartesian chart on $|\gs\oplus_{\rho_0}W_0|^*=|\gs|^*\oplus|W_0|^*$.  
Also, put $\eta_{ij}=\frac{\partial}{\partial_{w_{ij}}}, 
\,\eta_i=\frac{\partial}{\partial_{w_{i}}}$.
As it follows from (\ref{com-for-gl}), in this chart the Poisson bivector of 
$\gs\oplus_{\rho_0}W_0$  reads
\begin{eqnarray}\label{pois-for-sl}
Q^0=\sum_{j,  i<k}z_{ik}\xi_{ij}\xi_{jk}+
\sum_{\i\neq j,j\neq k,k\neq i }w_{ik}\xi_{ij}\eta_{jk}+
2\sum_{1< i,1< j, i<j }(w_i-w_j)\xi_{ij}\eta_{ji} \nonumber \\
-2\sum_{1<i}w_i\xi_{1i}\eta_{i1}+
\sum_{1<j,1<j,\i\neq j}w_{ij}\xi_{ij}\eta_{j}+
2\sum_{1<i}w_{1i}\xi_{1i}\eta_i+
\sum_{1<i,1<j,i\neq j}w_{1i}\xi_{1i}\eta_j.
\end{eqnarray}
Now apply the trick we have used in subsection 7.1 to $Q^0$.

The expression of $Q^0$ in the basis
$$
t_it_je_{ij}^0, \,1\leq i< j\leq n, \qquad t_it_je_{ij}^1, \,1\leq i< j\leq n, 
\qquad t_i^2e_i, \,1<i\leq n.
$$
of  $\gs\oplus_{\rho_0}W_0$ is of the form $Q^0=t_1^2Q^0_1+\dots+t_n^2Q^0_n$ 
with $Q^0_j$ non depending on $t_i$'s. This proves that $Q_j^0$'s are mutually
compatible Poisson bivectors. In particular, $Q^0=Q^0_1+\dots+Q^0_n$ is a 
simple disassembling of $Q^0$. Exact expressions of bivectors $Q^0_j$'s are 
easily obtained from (\ref{pois-for-sl}). Namely, we have
\begin{equation}\label{pois-for-dis1-sl}
Q^0_1=\sum_{i<k}z_{ik}\xi_{i1}\xi_{1k}+
\sum_{i\neq 1,k\neq 1}w_{ik}\xi_{i1}\eta_{1k}
-2\sum_{1<i}w_i\xi_{1i}\eta_{i1}
\end{equation}
and for $j>1$
\begin{eqnarray}\label{pois-for-dis2-sl}
Q^0_j=\sum_{i<k,  j\notin\{i,k\}}z_{ik}\xi_{ij}\xi_{jk}+
\sum_{\ i\neq k, j\notin\{i,k\} }w_{ik}\xi_{ij}\eta_{jk}+
2\sum_{i,1< i\neq j }w_i\xi_{ij}\eta_{ji}+ \nonumber \\
\sum_{i,1<i\neq j}w_{ij}\xi_{ij}\eta_{j}+
2w_{1j}\xi_{1j}\eta_j+
\sum_{i,1<i\neq j}w_{1i}\xi_{1i}\eta_j.
\end{eqnarray}
Each single term of summations (\ref{pois-for-dis1-sl}) and  (\ref{pois-for-dis2-sl})
is an $(n^2-1)$--dimensional $\pitchfork$--lieon. It is easy to verify by a direct 
check or by using lemmas\,\ref{a1}, \ref{b1} and \ref{c1} that
\begin{itemize}
\item all $\pitchfork$--lieons in  (\ref{pois-for-dis1-sl}) are mutually compatible;
\item incompatible pairs of  $\pitchfork$--lieons in (\ref{pois-for-dis2-sl}) are 
\begin{eqnarray}\label{comp-fork-pairs}
z_{1k}\xi_{1j}\xi_{jk} \;\mathrm{and}  \;w_{1k}\xi_{1k}\eta_j, \;k>1; \quad
w_{1k}\xi_{1j}\eta_{jk} \;\mathrm{and}  \;w_{kj}\xi_{kj}\eta_j, \;k>1;\nonumber \\
w_{1k}\xi_{kj}\eta_{j1} \;\mathrm{and}  \;w_{1j}\xi_{1j}\eta_j, \;k>1;
\end{eqnarray}
\end{itemize}
This shows that 
\begin{itemize}
\item $Q_1^0$ is simply disassembled into a number of $\pitchfork$--lieons;
\item $Q_j^0, \,j>1$, is simply disassembled into a number of $\pitchfork$--lieons
and structures
\begin{eqnarray}
Q_j^{\prime}=\sum_{k>1, k\neq j}z_{1k}\xi_{1j}\xi_{jk}+\sum_{k>1, k\neq j}w_{1k}\xi_{1j}\eta_{jk}+
2w_{1j}\xi_{1j}\eta_j \\
Q_j^{\prime\prime}=\sum_{k>1, k\neq j}w_{1k}\xi_{1k}\eta_j+
\sum_{k>1, k\neq j}w_{kj}\xi_{kj}\eta_j+\sum_{k>1, k\neq j}w_{1k}\xi_{kj}\eta_{j1} 
\end{eqnarray}
\end{itemize}
Indeed, by (\ref{comp-fork-pairs}), $Q_j^{\prime}$ and $Q_j^{\prime\prime}$
are composed of mutually compatible $\pitchfork$--lieons and, therefore,
are Poisson bivectors. Moreover, a direct computation by using formula 
(\ref{Schouten in coordinates}) proves that $\ls Q_j^{\prime},Q_j^{\prime\prime}\rs=0$.

Finally, by the above said both $Q_j^{\prime}$ and $Q_j^{\prime\prime}$ can be simply
disassembled into a number of $\pitchfork$--lieons. Thus we have explicitly 
described a 4-step complete disassembling of considered simple Lie algebras.
\begin{rmk}\label{can-dis-repr}
Identify  the algebra $\gs\go(n,\gk)$ with $\gs\go(g)$, $g$ being the standard `scalar' 
product on $V=\gk^n$. Then the above discussed representation $\rho$ of the algebra 
$\gs\go(n,\gk)$ is identified with the canonical representation of this algebra in $S^2V^*$.
It is easy to see that the disassembling procedure discussed in this subsection can be
applied to all canonical representations of $\gs\go(n,\gk)$ in tensor powers of $V$.
\end{rmk}

By concluding this section we note that the discussed in it simple algebras belonging 
to the same series are assembled from the same `universal elements' with coefficients
belonging to a given ground field (see remarks \ref{ort-any-fields} and \ref{field-as-glue}). 
The same can be said about their representations (see remark\,\ref{can-dis-repr}). These
observations lead to a natural conjecture. To state it we, first, recall that the type of a 
central simple Lie algebra $\gG$ over a field is the type of a splitting  simple algebra 
obtained from $\gG$ by an extension of scalars. Second, we say that two Lie algebras
are assembled from the same elements if the corresponding a--schemas are 
equivalent and the corresponding terms of them are proportional.
\begin{conj}
All central simple Lie algebras of the same type can be assembled from the same 
``universal'' elements which are Lie algebras structures over $\Z$.
\end{conj}
An approach, which appear more boring than difficult, is to apply the techniques of
this section to the known description of simple Lie algebras over arbitrary fields of
characteristic zero (see, for instance, chapter\,X of Jacobson's book \cite{Jac}).

\section{Coaxial Lie algebras}\label{coaxial}

 \emph{Coaxial}  Lie algebras form a natural subclass of first level Lie
algebras, which is, in a sense, attached to a chosen basis of the supporting 
vector space $V$ (see below). Informally speaking, these are ``molecules"
which can be directly, i.e., in one step, ``synthesized'' from lions which play the
role of "atoms" in this context.  Description of these ``molecules" is a combinatorial
problem which is solved in this and the subsequent  sections. Coaxial algebras provide 
us with necessary ingredients for synthesis of more complicated ``molecules". 
For instance, such ones have already appeared in the procedure of disassembling of 
classical Lie algebras in the previous section. 
 
The central point in this section is a description of some ``maximal" families of mutually
compatible lions, called \emph{clusters}. It is  rather instructive to see how lions in clusters
are self-organized in structural groups, surrounded by \emph{casings} and tied one another
by means of \emph{connectives}. All this looks like a kind of chemistry, and from this point
of view the study of coaxial algebras may be thought as the first step toward the general
``chemistry" of Lie algebras.

Another interesting aspect of coaxial Lie algebras is that they are easily deformable. As such
they furnish the deformation theory with numerous examples, which, in particular, may serve
as an useful ``experimental" material. Finally, it should be mentioned that infinite-dimensional 
version of coaxial Lie algebras  brings in light a new class of Lie algebras, which will be
illustrated by some examples.

\subsection{Coaxial algebras: definitions}
 
 Fix  a basis $\mathcal{B}=\{e_1,\dots,e_m\}$  in $V$. Suppose that numbers 
 $1\leq i,j,k\leq n$  differ each other and denote by $\lfloor 
 i,j|k\rceil$ \,(resp., by $\lfloor i|j\rceil$) the $n$--dimensional
 $\pitchfork$--lieon (resp., the $\between$--lieon) for which 
 $[e_i,e_j]=-[e_j,e_i]=e_k$ (resp., $[e_i,e_j]=-[e_j,e_i]=e_j$) are the 
 only nontrivial of Lie products involving base vectors. Call them 
 $\mathcal{B}$-\emph{base}, or, simply, \emph{base} $\pitchfork$- and 
 $\between$-lieonss, respectively. We also shall use the notation like
 $\lfloor A,B|C\rceil$ or $\lfloor A|B\rceil$ instead of $\lfloor i,j|k\rceil$
 and $\lfloor i|j\rceil$, respectively, if $A=e_i, B=e_j$ and $C=e_k$.
 \begin{defi}
 A linear combination (over $\gk$) of 
 some mutually compatible  $\mathcal{B}$-\emph{base} lieons is 
 called a $\mathcal{B}$-\emph{coaxial}, or, simply, \emph{coaxial} (Lie 
 algebra) structure. Such a structure will be called 
 \emph{$\pitchfork$-coaxial} (resp., \emph{$\between$-coaxial}) if it is
 composed only of base $\pitchfork$--lieons  (resp., of 
 $base \between$--lieons). 
 \end{defi}
 A coaxial Lie algebra $\gG$ may be presented in the form
  \begin{equation}\label{Lie-base-comp}
  \gG=\sum\alpha_{(i,j|k)}\lfloor i,j|k\rceil\,+
 \,\sum\beta_{(p|q)}\lfloor p|q\rceil, \quad \alpha_{(i,j|k)}, \;\beta_{(p|q)}
 \in\gG,
\end{equation}
where figuring in it base structures with nonzero coefficients are compatible
each other.
 \begin{figure}[h]
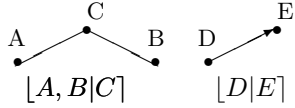
 
\figthree
\caption{Base lieons.}
\end{figure}
 Vectors $e_i,e_j$ and $e_k$ (resp., $e_i,e_j$) are called \emph{vertices} 
 of $\lfloor i,j|k\rceil$ (resp., of $\lfloor i|j\rceil$). Vectors 
 $e_i,e_j$ are \emph{ends} of $\lfloor i,j|k\rceil$, while $e_k$ is its 
 \emph{center}. The \emph{origin} and the \emph{end} of $\lfloor i|j\rceil$ 
 are $e_i$ and $e_j$, respectively. In the sequel we do not distinguish 
 between $\lfloor i,j|k\rceil$ and $\lfloor i,j|k\rceil=-\lfloor 
 j,i|k\rceil$, since they have identical compatibility properties. 

 Poisson bivectors on $V^*$ corresponding to $\lfloor i,j|k\rceil$ and 
 $\lfloor i|j\rceil$  are $\langle i,j|k\rangle=x_k\xi_i\xi_j$ and $\langle 
 i|j\rangle=x_j\xi_i\xi_j$, respectively. They will be called \emph{base bivectors.} 
 Obviously, the coordinate 
 expression of the Poisson bivector $P_{ \gG}$ associated with a Lie 
 algebra structure $ \gG$ on $V$ is a linear combination of base bivectors:
  \begin{equation}\label{Pois-base-comp}
 P_{\gG}=\sum\alpha_{(i,j|k)}\langle
 i,j|k\rangle\,+\,\sum\beta_{(p|q)}\langle p|q\rangle, \quad
 \alpha_{(i,j|k)}, \;\beta_{(p|q)}\in\gk.
 \end{equation}

 These bivectors are not generally mutually compatible. For instance, such
are base bivectors occuring in coordinate expressions of Poisson bivectors of
 classical Lie algebras which were discussed in section\,7.
 The Poisson bivector 
 associated with a $\pitchfork-,\between$--coaxial algebra will be also 
 called $\pitchfork-,\between$--coaxial, respectively. 
 
Base lieons are conveniently, up to proportionality, represented as diagrams in 
Fig. 3. In the sequel such diagrams will be systematically used in construction of 
families of mutually compatible base lieons.

 \medskip
 \noindent \textit{Compatibility of base structures.}
 
 \noindent We shall say that two base lieons 
 are \emph{trivially compatible} if they either 
 coincide up to the sign, or have no common vertices at 
 all.
 \begin{proc}\label{123}
 1) Two base $\pitchfork$-lieons are nontrivially compatible if and only if  they 
 have either a common center vertex, or a common end vertex at least.\\ 
 2) Two base $\between$-lieons are  incompatible if and only if  the origin
 of one of them is the end of the other one and they have no other common 
 vertices.\\ 
 3) A $\between$-lieon is nontrivially compatible with a 
$\pitchfork$-lieon if and only if its origin coincides with one of the 
ends of this $\pitchfork$-lieon.
 \end{proc}
 \begin{proof}
 A direct consequence of lemmas \ref{a1}-\ref{b3}.
 \end{proof}
 A graphical interpretation of proposition \ref{123} is given in Fig. 4.
 \begin{figure}[htb]
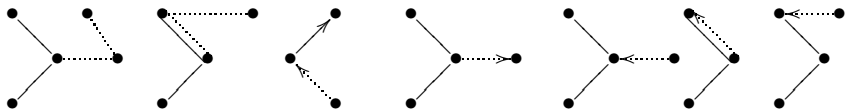

$$
\xy
(-6,-6)*{\bullet}="A"; (-6,6)*{\bullet}="B"; (0,0)*{\bullet}="C"; 
"A";"C" **\dir{-}; "B";"C" **\dir{-}; 
(8,0)*{\bullet}="D"; "D";"C" **\dir{.}; (4,6)*{\bullet}="E"; 
{\ar@{.} "E"; "D"}; 
(14.4,6.4)*{}="BA"; (14,6)*{\bullet}="BE"; (13.6,5.6)*{}="BB"; 
(20.4,0.4)*{}="BC"; (20,0)*{\bullet}="BF"; (19.6,-0.4)*{}="BD"; 
"BB";"BD" **\dir{-}; "BA";"BC" **\dir{.}; (14,-6)*{\bullet}="BG";
"BG";"BF"; **\dir{-}; (26,6)*{\bullet}="BK"; "BE";"BK"; **\dir{.};
(31,0)*{\bullet}="A1";(37,6)*{\bullet}="B1";(37,-6)*{\bullet}="C1";
{\ar@{->} "A1"; "B1"}; 
{\ar@{.>} "C1"; "A1"}; 
(47,6)*{\bullet}="E1"; (53,0)*{\bullet}="E2";
(47,-6)*{\bullet}="G1"; (61,0)*{\bullet}="G2";
{\ar@{-} "E1"; "E2"}; {\ar@{-} "G1"; "E2"}; {\ar@{.>} "E2"; "G2"};
(68,6)*{\bullet}="B2"; (74,0)*{\bullet}="A2"; 
(68,-6)*{\bullet}="C2"; (82,0)*{\bullet}="D2";
{\ar@{-} "A2"; "B2"}; {\ar@{-} "C2"; "A2"}; {\ar@{.>} "D2"; "A2"};
(84.4,6.4)*{}="A3A"; (84,6)*{\bullet}="A3"; (83.6,5.6)*{}="A3B";
(84,-6)*{\bullet}="B3"; (90.4,0.4)*{}="C3A"; (90,0)*{\bullet}="C3";
(89.6,-0.4)*{}="C3B"; "C3";"B3"; **\dir{-};
"A3B";"C3B"; **\dir{-}; {\ar@{.>} "C3A"; "A3A"};
(96,6)*{\bullet}="A4"; (96,-6)*{\bullet}="B4"; (102,0)*{\bullet}="C4";
(104,6)*{\bullet}="D4"; "C4";"B4"; **\dir{-};"A4";"C4"; **\dir{-};
{\ar@{.>} "D4"; "A4"};
\endxy
$$
\caption{Incompatible configurations of 2 base lieons}\label{incompatible}
\end{figure}
 
 \begin{ex}
 Poisson bivectors $P_{\alpha}$'s that disassemble the algebra $\gs\go(g)$ 
 (see subsection \ref{SO-dis}) are coaxial with respect to the base 
 $\{e_{ij}\}$. The corresponding to $P_{\alpha}$ coaxial Lie algebra is 
 isomorphic to the direct sum of $(n-1)(n-2)/2$ copies of the Heisenberg 
 algebra and an abelian one. This algebra is, obviously, coaxial. So, 
 $\gs\go(g)$ is assembled from $n$ coaxial Lie algebras of this kind.
 Poisson bivectors $Q_j^{\prime}$ and $Q_j^{\prime\prime}$ which 
 appear in the disassembling procedure in subsection\,\ref{gl-sl-u-dis} are 
 also $\pitchfork$--coaxial.
 \end{ex}

\subsection{Clusters.}

Lie algebra  structures (resp., Poisson bivectors), which are compatible 
with two proportional ones, are, obviously, the same. So, it is convenient to
work with classes of proportional base lieons or correspomding Poisson bivectors.
A class of proportional base   $\pitchfork$-lieons (resp., $\between$-lieons)
will be called a \emph{tee} (resp., a \emph{dee}). We shall use keep the notation  
$\lfloor i,j|k\rceil$ (resp., $\lfloor i|j\rceil$) for the corresponding tee (resp., dee).
The \emph{ends and center of a tee} are those of the corresponding base
$\pitchfork$-lieon, and similarly, for the  \emph{end and origin of a dee}.

 A family of tees (resp., dees) will be called a \emph{tee} (resp., 
 \emph{dee})   \emph{family}. Union of a tee and of a dee families
 will be called a  \emph{base family}.
 A graph $\Upsilon_{\Phi}$ is 
 naturally associated with a base family $\Phi$. Namely, let 
 $S(\Phi)=\{e_{i_1},\dots,e_{i_m}\}\subset \{e_1,\dots,e_n\}$ be the 
 totality of vertices of all tees and dees composing $\Phi$ and 
 $I(\Phi)=\{i_1,\dots,i_m\}$.  Vertices $v_1,\dots, v_m$ of 
 $\Upsilon_{\Phi}$ are in one-to-one correspondence with base vectors 
 $e_{i_1},\dots, e_{i_m}$. Vertices $v_k$ and $v_l$ of $\Upsilon_{\Phi}$ 
 are connected by an unique edge iff $e_{i_k}, e_{i_l}$ are either vertices 
 of a dee, or the center and one of the ends of a 
 tee belonging to $\Phi$. For example, the graph 
 corresponding to $\Phi=(\lfloor i,j|k\rceil, \lfloor i|k\rceil, \lfloor 
 l,k|j\rceil)$ has four vertices and three edges. Base vectors belonging
 to $S(\Phi)$ will be called \emph{vertices} of $\Phi$. Obviously,
 $\Phi=\Phi^{\between}\cup\Phi^{\pitchfork}$ where $\Phi^{\between}$
 (resp., $\Phi^{\pitchfork}$) is the set of all dees (resp., tees) belonging 
 to $\Phi$.

 Base families $\Phi$ 
 and $\Phi^{\prime}$ are \emph{equivalent} if there exists a one-to-one 
 correspondence $S(\Phi)\leftrightarrow S(\Phi^{\prime})$ which
 induces a one-to-one correspondence between tees and dees belonging 
 to $\Phi$ and those belonging to $\Phi^{\prime}$.
  
 A base family  will be called  \emph{compatible} if composing it lieons 
 are mutually compatible. A compatible family, denoted by $\Phi_{\gG}$, is 
 naturally associated with a coaxial Lie algebra $\gG$. Namely, it consists 
 of tees and dees corressponding to nonzero terms in expression
 (\ref{Lie-base-comp}) for $\gG$.
 Denote by $\Phi_{\gG}^{\between}$ 
 (resp.,$\Phi_{\gG}^{\pitchfork}$) the $\between$-family (resp., 
 $\pitchfork$--family) composed of dees (resp., tees) belonging to 
 $\Phi_{\gG}$. Finally, two 
 compatible families will be called \emph{compatible} if composing them 
tees and dees are mutually compatible.  A Lie algebra $\gG$ will be 
called \emph{associated} with a base family $\Phi$ if $\Phi=\Phi_{\gG}$.
 
 The associated with $\Phi_{\gG}$ graph will be denoted by 
 $\Upsilon_{\gG}$, i.e., $\Upsilon_{\gG}=\Upsilon_{\Phi_{\gG}}$. Let 
 $\Upsilon$ be a connected component of $\Upsilon_{\gG}, 
 \{\,e_{i_1},\dots,e_{i_s}\}$ the vertices of $\Upsilon$ and 
 $I=\{i_1,\dots,i_s\}$. A Lie algebra structure on the subspace of $V$ 
 generated by $e_{i_1},\dots,e_{i_s}$ is defined as the linear combination 
  \begin{equation}\label{Lie-base-comp-comp}
  \gG_I=\sum_{i,j,k\in I}\alpha_{(i,j|k)}\lfloor i,j|k\rceil\,+
 \,\sum_{p,q,\in I}\beta_{(p|q)}\lfloor p|q\rceil, \quad \alpha_{(i,j|k)}, \;\beta_{(p|q)}
 \in\gk,
\end{equation}
assuming that $\gG$ is given by (\ref{Lie-base-comp}). Obviously,
 \begin{equation}\label{componentwise}
 \gG=\gG_{I_1}\oplus\dots\oplus\gG_{I_m}\oplus\gamma_l, 
\qquad l=n-(\mathrm{dim}\,\gG_{I_1}+\dots+\mathrm{dim}\,\gG_{I_m}). 
 \end{equation}
 where $I_j$'s are sets of indices corresponding to connected 
 components of $\Upsilon_{\gG}$. 
 
 A compatible base family is \emph{maximal} if it is not contained in a larger 
compatible one, which have the same set of vertices. A maximal compatible 
family will be called a  \emph{cluster} if the 
 corresponding to it graph is connected. The number of vertices of the 
 graph associated with a cluster is called the \emph{dimension} of it. 
 Obviously, a compatible family with the connected graph is contained in a 
 cluster. So, in view of decomposition (\ref{componentwise}), the problem 
 of description of coaxial algebras is reduced to description of clusters. 
 Similarly are defined maximal  compatible tee-- and  dee--families and,
 accordingly, \emph{tee-clusters} and  \emph{dee-clusters}.
 It should be stressed that a \emph{tee-cluster}, or a 
 \emph{dee-cluster} is not necessarily a cluster (see below). 
 
 According to the above said we see that 
 \begin{center}
 {\it the problem of description of coaxial Lie algebras \\ is reduced  to 
 description of clusters.}
 \end{center}
 By this reason in the sequel we shall concentrate on study of clusters.
 
 The following terminology will be useful in our further analysis of the structure 
of clusters. We shall say that a tee/dee $\vartheta$ \emph{blocks} (alternatively, 
is \emph{blocking}) a tee/dee $\theta$ if it is incompatible with $\theta$. So, 
we have the following 
\newline\newline
\noindent
{\bf Blocking rule:}
{\it Let $\Phi$ be a (tee/dee-)cluster and vertices of $\theta$ belong to $S(\Phi)$. Then
$\theta$ belongs to $\Phi$, if $\Phi$ does not contain tees/dees that block $\theta$.}

Base families are conveniently represented by means of their diagrams (see, for
instance, Fig. 4). The use of such diagrams makes much more clear proofs of 
various assertions about compatibility of base families appearing in the forthcoming
analysis of the structure of clusters. It turned out impossible to accompany
these proofs, due to their numerosity, by the corresponding pictures. So, the 
reader is strongly suggested to do that on his own.

\subsection{Structure of dee-clusters}\label{dee-clusters} 
 
Dee-clusters are easily classified. With this purpose we shall introduce the
following compatible dee-families.

\textbf{Double}. This is a compatible dee-family of the form
$\{\lfloor p|q\rceil, \lfloor q|p\rceil\}$.

\noindent\emph{Basic property of doubles}: If a base lieon is compatible with a 
double $D$, then it is a $\pit$--lieon whose ends are vertices of $D$ 
(proposition\,\ref{123}).

In particular, if $D$ belongs to a compatible dee-family $\Phi$, then 
a dee $\vartheta\in \Phi\setminus D$ has no common vertices with $D$.
Vertices of $D$ will be called \emph{double vertices} of $\Phi$. Therefore if $\nu$ 
is not a double vertex of $\Phi$, then $\nu$ is either the common origin, 
or the common end of all dees $\vartheta\in \Phi$ that have  $\nu$ as one of 
its vertices. Accordingly, vertices of a compatible dee-family are subdivided into 
\emph{double, initial} and \emph{end} vertices, respectively.

\textbf{Spider}.  Let $I_0=\{i_1,\dots,i_k\}$ and $I_1=\{j_1,\dots,j_l\}$ be nonempty 
subsets of $\{1,\dots,n\}$ such that $I_0\cap I_1=\emptyset$. The dee-familly
$$
Sp(I_0,I_1)\df\{\lfloor p|q\rceil \,| \,p\in I_0, \,q\in I_1\}
$$
will be called a \emph{$(k,l)$-spider}. Vertices
$e_{i_1},\dots,e_{i_k}$ (resp., $e_{j_1},\dots,e_{j_l}$) will be called 
\emph{initial} (resp., \emph{end}) vertices of the spider $Sp(I_0,I_1)$.
Obviously, all $(k,l)$-spiders are equivalent. We shall use the notation 
 $Sp_k^l$ when referring  to a $(k,l)$-spider.
\begin{proc}\label{doubles}
\begin{enumerate}
\item If a double $D$ is contained in a dee-cluster $\Phi$, then $n=2$ and 
$D=\Phi$.
\item If $n>2$, then a dee-cluster is a (m,n-m)--spider, $1\leq m<n$.
\end{enumerate}
\end{proc}
\begin{proof}
(1) If a dee $\delta$ has a common vertex with a double $D$ and $\delta\notin D$,
then, according to proposition\,\ref{123}, 2),  $\delta$ is incompatible with $D$.

(2) Let $\Phi$ be a dee-cluster and $\lfloor p_i|q_i\rceil\in \Phi, \,i=1,2$. By assertion 
(1), $\Phi$ does not contain doubles. So, by proposition\,\ref{123}, 2), $p_i$  (resp., $q_i$)
is not the end (resp., the origin) of a dee belonging to $\Phi$. Therefore the
dees $\lfloor p_1|q_2\rceil$ and $\lfloor p_2|q_1\rceil$ are compatible with all dees
from $\Phi$, and, by maximality property of $\Phi$, they must belong to $\Phi$.
\end{proof}
Put 
$$
C_{m,n}^{\between}=Sp(I_0,I_1) \quad\mbox{with}  \quad I_0=\{1,\dots,m\}, \,I_1=
\{m+1,\dots,n\},  \;1\leq m< n. 
$$   
 \begin{proc}
 \begin{enumerate}
 \item $C_{1,2}^{\between}=\{\lfloor 1|2\rceil, \lfloor 2|1\rceil\}$ is the 
 unique $2$-dimensional dee--cluster and, at the same time, the 
 unique $2$-dimensional cluster. 
 \item For $n>2$ there exists an unique $n$-dimensional cluster $C_{m,n}$ 
 containing $C_{m,n}^{\between}$. Namely, \\
 $\qquad\qquad\qquad C_{1,n}=C_{1,n}^{\between}\cup\{\lfloor 1,k|l\rceil\;|\; 2\leq k,l\leq n\}$;\\
 $\qquad\qquad\qquad C_{2,n}=C_{2,n}^{\between}\cup\{\lfloor 1,2|k\rceil|\;3\leq k\leq n\}$;\\
$\qquad\qquad\qquad C_{m,n}=C_{m,n}^{\between}$, \; $m\geq 3$. 
\end{enumerate}
 \end{proc}
 \begin{proof}
 The first assertion is obvious. 
 
 If $n>2$, then, as it follows from proposition\,\ref{123}, 2), $C_{m,n}$ contains 
all tees that are compatible with  $C_{m,n}^{\between}$. In particular, there are no 
such ones, if $m\geq 3$. Moreover, by proposition\,\ref{123}, 1), these tees are 
mutually compatible.
 \end{proof}
As it is easy to see, a $\between$-coaxial Lie algebra $\gG$,
which is associated with $C_{m,n}^{\between}, \;n>\!2,$ is of the form
$\gG=\gA\oplus_{\rho}W$ with $\rho$ being a representation of an $m$-dimensional
abelian algebra $\gA$ in $W$ whose operators have a common diagonalizing basis.  
The corresponding Poisson bivector  is
\begin{equation}\label{Pmn}
 P_{m,n}^{\between}(\alpha)=\sum_{1\leq i\leq m,
 m<j\leq n}\alpha_{ij}x_{j}\xi_{i}\xi_{j}, \quad \alpha=\{\alpha_{ij}\}, \;\alpha_{ij}\in\gk.
\end{equation}                     
 
 Poisson bivectors representing coaxial Lie algebras associated with 
 clusters $C_{m,n}$ are
 \begin{eqnarray}
 \label{p1}P_{1,n}(\alpha,\beta)=P_{1,n}^{\between}(\alpha)+
 \sum_{2\leq k,l\leq n}\beta_{kl}x_l\xi_1\xi_k, 
 \quad \beta=\{\beta_{kl}\}_{2\leq k,l\leq n;}\\
 \label{p2}P_{2,n}(\alpha,\tau)= P_{2,n}^{\between}(\alpha)+
 \sum_{3\leq k\leq n}\tau_{k}x_k\xi_1\xi_2, \quad \tau=\{\tau_k\}_{3\leq k\leq n};\\
 \label{p3}P_{m,n}(\alpha)= P_{m,n}^{\between}(\alpha), \quad m\geq 3.
 \end{eqnarray} 
 
The Lie algebra corresponding to Poisson bivector  (\ref{p1}) is isomorphic to an
algebra $\Gamma_A, \,A:W \rightarrow W, \;A=\ad\,e_1,$ (see subsection 
\ref{ReductionSolvable}). Here the operator $A$ can be arbitrary.

The Lie algebra  $\gG$ corresponding to Poisson bivector  (\ref{p2}) has an 
abelian ideal $\mathcal{I}, |\mathcal{I}|=\langle e_3,\dots,e_n\rangle,$ of 
codimension 2 such that $[\gG,\gG]\subset\mathcal{I}$. In particular, $\gG$ is 
solvable and its derived series consists of two nontrivial terms.

\subsection{Structural groups of tee-clusters} 
\label{TeeClusters}
 
Tee--clusters, unlike dee--clusters,  are much more diversified and have a quite 
complex structure. Their description is a specific combinatorial problem 
whose solution requires, like in chemistry, determination of basic structural elements.

First of all, it is useful to distinguish vertices of a compatible tee-family $\Phi$. Namely, 
a vertex $\nu\in S(\Phi)$ will be called an \emph{end} (resp., a \emph{center}) 
vertex of $\Phi$ if it is an end (resp., the center) vertex of any tee $\vartheta\in\Phi$
such that $\nu$ is one of vertices of $\vartheta$. Otherwise, $\nu$ will be 
called a \emph{mixing} vertex. Vertices of the graph $\Upsilon_{\Phi}$ will be called 
accordingly. 

The role of various type vertices is illustrated by the following proposition. 
\begin{proc}\label{Lie-ecm}
Let $\gG$ be a coaxial Lie algebra such that $\Phi_{\gG}=\Phi_{\gG}^{\pitchfork}$. 
Then  
 \begin{enumerate} 
\item The subspace of $|\gG |$ generated by all center vertices of $\Phi_{\gG}$ 
supports a central ideal $\mathcal{I}$ of $\gG$.
\item Let $\gG=\sum_{\vartheta\in\Phi_{\gG}}a_{\vartheta}\vartheta$ and 
$\Psi\subset\Phi_{\gG}$ be the family formed by the tees $\vartheta\in\Phi_{\gG}$ 
whose centers are not centers of $\Phi_{\gG}$. Then for a suitable $m$ the algebra 
$\gG/\mathcal{I}\oplus\gamma_m $ is isomorphic 
to the algebra $\sum_{\vartheta\in\Psi}a_{\vartheta}\vartheta$;
\item $|[\gG,\gG]|$ belongs to the subspace generated by all mixing and 
center vertices of $\Phi$. 
\end{enumerate}
\end{proc}
\begin{proof}
Straightforwardly from proposition\,\ref{123} and the definitions. 
\end{proof}

Now we pass to describe  \emph{structural groups}, which are some special
compatible tee-families. They are building blocks of which  tee-clusters are made.
This description is accompanied by \emph{basic properties} of structural groups.
They are direct consequences of proposition\,\ref{123} and by this reason the 
proofs are omitted.  \newline

\textbf{Triangle}. A compatible tee-family of the form
$$
\Delta_{ijk}=\{\lfloor i,j|k\rceil, \lfloor j,k|l\rceil,  
\lfloor k,l|i\rceil\}, \,1\leq i,k,l,\leq n.
$$
will be called a \emph{triangle}. 
It is easy to see (proposition\,\ref{123}) that triangles are 
$3$-dimensional clusters and vice versa. Lie algebras associated with a 
triangle are of the form $\gs\go(g)$ with $g$ being a $3$-dimensional 
nondegenerate quadratic form, and hence are simple.

\emph{Basic property of triangles}: The ends of a tee 
nontrivially compatible with a triangle are vertices of this triangle. All 
such tees structures are compatible each other. 
\newline\newline
\textbf{Hedgehog}. Consider a disjoint, i.e., without common vertices, 
family of triangles $\blacktriangle_1,\dots,\blacktriangle_p$  and base 
vectors $\varepsilon_1=e_{i(1)},\dots,\varepsilon_q=e_{i(q)}$ which are 
not vertices of these triangles. Enlarge the family 
$\blacktriangle_1\cup\dots\cup\blacktriangle_p$ by adding to it all 
tees whose centers are in 
$(\varepsilon_1\dots,\varepsilon_q)$ and ends in one of triangles 
$\blacktriangle_i$'s. The so-obtained family is compatible and 
is called a \emph{hedgehog} of type $(p,q)$, or a $(p,q)$-hedgehog. 
Vertices $\varepsilon_i$'s of the hedgehog are its \emph{thorns} and 
$\blacktriangle_i$'s are its \emph{base triangles}. A triangle can be 
viewed as a $(1,0)$-hedgehog. 

A $(p,q)$-hedgehog is a $(3p+q)$-dimensional cluster.  
A Lie algebras associated with a $(p,q)$--hedgehog is a central 
extension of the direct sum of $p, \,3$-dimensional 
simple algebras (see proposition\,\ref{Lie-ecm}). 

\emph{Basic property of hedgehogs}: If a tee $\vartheta\notin\Phi$ 
is nontrivially compatible with a hedgehog $\Phi$, then either its center 
is one of thorns of $\Phi$, while its ends do not belong to $S(\Phi)$, or 
its ends belong to one of triangles $\blacktriangle_i$'s, while its center
is not a thorn of $\Phi$.  In particular, if $\Phi$ is contained in a compatible
tee family $\Psi$, then the thorns of $\Phi$ are center points of $\Psi$.

\textbf{Twain}. A family of the form $\wedge_{ij}^k=\{\lfloor i,k|j, 
\rceil,\,\lfloor j,k|i\rceil\}$ will be called a \emph{twain}. The vertex 
$e_k$ is the \emph{top} of $\wedge_{ij}^k$, while $e_i$ and $e_j$ 
form its \emph{bottom}. The unique triangle containing $\wedge_{ij}^k$ is 
$\Delta_{ijk}$. 

\emph{Basic property of twains}: If a tee $\vartheta$ is compatible with 
a twain $\wedge$ and the center of $\vartheta$ is the top of $\wedge$, 
then the ends of $\vartheta$ belong to the bottom of $\wedge$, i.e., 
$\wedge\cup\{\vartheta\}$ is a triangle. Consequently, if a twain $\wedge$ belongs 
to a compatible family $\Phi$, then either its top is an end vertex of 
$\Phi$, or $\wedge\subset\Delta\subset\Phi$ where $\Delta$ is the unique 
triangle which contains $\wedge$. 

An associated with a twain Lie algebra is the algebra of infinitesimal symmetries 
of a plain "metric" $adx^2+bdy^2, \,a,b,\in\Bbbk $, or, in other words, the Killing 
algebra of this metric. 

\textbf{Trey.} An $m$-\emph{trey} is a family equivalent to
$$
\top_m=\{\wedge_{2,3}^1, \lfloor 2,3|4\rceil,\dots,\lfloor 2,3|m+4\rceil\}, \quad m>0.
$$
 Note that $\wedge_{2,3}^1$ is the unique twain contained in $\top_m$ and 
 $e_4,\dots,e_{m+4}$ are center vertices of $\top_m$.  Bottom vertices of this 
 twain will be called \emph{side vertices} of this $m$-trey. Center 
vertices of an $m$-trey are a center vertices of any containing it compatible tee-family.
A $1$-trey will be called simply a \emph{trey}. 

\emph{Basic property of treys}: If a tee $\vartheta$  is compatible with an $m$-trey 
$\top_m$ and the center of $\vartheta$ is a side vertex of $\top_m$ , then 
$\vartheta$ belongs to the twain contained in $\top_m$. 

A Lie algebra associated with an $m$-trey is an 
$m$-dimensional central extension of a Lie algebra associated with the 
contained in it twain. \newline

\textbf{Pyramid}.  An $m$-dimensional \emph{pyramid}, or, shortly, 
$m$-\emph{pyramid}, is a tee-family equivalent to 
$$ 
\bigtriangledown^m_1\,=\,\bigcup_{2\leq i,j\leq m+1}\wedge^1_{i,j} 
$$
The common end of all composing a pyramid tees is the 
\emph{top} of it, while other vertices form its \emph{bottom}. 
A 2-pyramid is a twain.

Notice that a pyramid is a tee-cluster if $m>2$, but not a cluster. Indeed,  
$\{\lfloor 1\,|\,i\rceil, \;2\leq i\leq m+1\}$ are all compatible with $\bigtriangledown^m_1$ 
dees. They are also mutually compatible. Hence 
$$
C_1^m\,=\,\bigtriangledown^m_1\bigcup C_{1,m+1}^{\between}
$$ 
is the unique containing $\bigtriangledown^m_1$ cluster of the same 
dimension called a \emph{dressed pyramid}. 
\newline\newline
\textbf{Multiplex}. An $m$-\emph{plex} is a family composed of $m$ tees, 
which have one common end and one common center vertex.  For instance, 
$$
\bot_m=\bigcup_{3\leq i\leq m+2}\lfloor 2,i|1\rceil.
$$
is a such one. 
The common end (resp., center) of these tees will be called the 
\emph{origin} (resp., \emph{center}) of the $m$-plex. All other  
vertices of it are its  \emph{ends}. 

\emph{Basic property of $m$-plexes}: If a family $\Phi$ contains an 
$m$-plex, $m\geq 3$, then the origin of this $m$-plex is an end vertex of 
$\Phi$. 

\textbf{Multiped}. An $(p,q)$-\emph{ped}, $p\geq 2$, is a family equivalent to
\begin{equation}\label{ship}
\sqcap_p^q=\bigcup_{1\leq i,j\leq p, p+1\leq k\leq p+q}\lfloor i,j|k\rceil.
\end{equation}
A $(p,q)$-multiped has $p+q$ vertices $q$ of which are center and $p$ 
end vertices of it. Multipeds with one center and 
$m$ ends will be called $m$-\emph{peds}. Among them \emph{tripods} 
(=$3$-peds) are of a special interest. A $(p,q)$-multiped is the union of $q$ 
$p$-peds which have common ends. All tees composing an $m$-ped 
that have a common end vertex form an $(m-1)$-plex. 
Multipeds with more than three ends are clusters. 

\emph{Basic property of multipeds}: If a compatible tee-family $\Phi$ contains a 
$(p,q)$-ped $\sqcap$, $p\geq 3$, then any center vertex of $\sqcap$ is a center 
vertex of $\Phi$. 
\newline\newline
\textbf{Hybrid}. A $(p,q|r)$-\emph{hybrid} is union of a $(p,r)$-multiped and a 
$(q,r)$-hedgehog, which have common center vertices and no other common 
ones. Here $p\geq 2, \,q\geq 1, \,r\geq 1$. A $(p,q|r)$-hybrid is a cluster 
and hence a tee-cluster iff $p>3$. This directly follows from the 
basic property of hedgehogs and proposition\,\ref{123}, 3). 
\newline\newline
\textbf{Cross}. Two tees with the common center and mutually
different end vertices form a  \emph{cross}. For instance, 
$(\lfloor 2,3|1\rceil, \lfloor 4,5|1\rceil)$ is a cross. The common 
center of these tees is called the \emph{center} of the cross.

\emph{Basic property of crosses}: The center  of a cross belonging to a 
compatible tee-family $\Phi$ is a center vertex of $\Phi$. 
\newline\newline
\textbf{Catena}. An \emph{open m-catena} is a family equivalent to
\begin{equation}\label{catena}
\wr_m=\bigcup_{2\leq l\leq m+1}\lfloor 1,l+1|l\rceil, \quad m\geq 2.
\end{equation}
An open $m$-catena has two end and one center vertices. The common end of 
tees belonging to a catena is the  \emph{initial vertex} of it ($e_1$ for catena 
(\ref{catena}). The other end vertex of it is its  \emph{final} vertex ($e_{m+1}$
for catena (\ref{catena}).  There is only one tee belonging to an open catena whose
end vertices coincide with end vertices of this catena. 

A \emph{closed (m,k)-catena} is a tee-family equivalent to 
\begin{equation}\label{cl-catena}
\wr_m^k=\wr_{m-1}\cup\lfloor 1,m\mid k-1\rceil, \quad m\geq 2, \quad 3\leq k\leq m.
\end{equation}
For instance, $\wr_2^1$ is a twain. If $k\geq 2$ a closed catena has only one end 
vertex and one center vertex. For catena (\ref{cl-catena}) these vertices are
$e_1$ and $e_2$, respectively.

Diagrams of described above structural groups together with their icons are
presented in Fig. 5.

\scalebox{1} 
{
\begin{pspicture}(0,-6.5464373)(11.040375,6.5264373)
\definecolor{color12471b}{rgb}{0.8,0.8,0.8}
\psdots[dotsize=0.152,fillstyle=solid,dotstyle=o](9.784375,-2.5695624)
\psline[linewidth=0.02cm,doubleline=true,doublesep=0.12,doublecolor=white](9.784375,-1.8295625)(9.784375,-2.5695624)
\psline[linewidth=0.024cm](9.384375,-1.6295625)(10.184375,-2.0295625)
\psline[linewidth=0.024cm](9.384375,-2.0295625)(10.184375,-1.6295625)
\psdots[dotsize=0.152,fillstyle=solid,dotstyle=o](9.784375,-1.8295625)
\psdots[dotsize=0.152](2.264375,4.1304374)
\psdots[dotsize=0.152](0.324375,4.1104374)
\psdots[dotsize=0.152](1.304375,5.5504375)
\psline[linewidth=0.024cm](2.324375,4.1504374)(1.324375,5.6104374)
\psline[linewidth=0.024cm,linestyle=dashed,dash=0.16cm 0.16cm](1.264375,5.6304374)(0.264375,4.1904373)
\psline[linewidth=0.04cm,linestyle=dotted,dotsep=0.06cm](1.324375,5.4904375)(0.364375,4.1304374)
\psline[linewidth=0.024cm,linestyle=dashed,dash=0.16cm 0.16cm](0.344375,4.1704373)(2.224375,4.1904373)
\psline[linewidth=0.024cm](0.344375,4.0504375)(2.284375,4.0704374)
\psline[linewidth=0.04cm,linestyle=dotted,dotsep=0.06cm](2.184375,4.1904373)(1.204375,5.5704374)
\psdots[dotsize=0.152](4.824375,5.5144377)
\psdots[dotsize=0.152](2.884375,5.5344377)
\psdots[dotsize=0.152](3.864375,4.0944376)
\psline[linewidth=0.024cm](4.884375,5.4944377)(3.884375,4.0344377)
\psline[linewidth=0.024cm,linestyle=dashed,dash=0.16cm 0.16cm](3.824375,4.0144377)(2.824375,5.4544377)
\psline[linewidth=0.024cm,linestyle=dashed,dash=0.16cm 0.16cm](2.904375,5.4744377)(4.784375,5.4544377)
\psline[linewidth=0.024cm](2.904375,5.5944376)(4.844375,5.5744376)
\psdots[dotsize=0.152](7.724375,4.9544377)
\psdots[dotsize=0.152](5.784375,4.9744377)
\psdots[dotsize=0.152](6.764375,3.5344374)
\psline[linewidth=0.024cm](7.784375,4.9344373)(6.784375,3.4744375)
\psline[linewidth=0.024cm,linestyle=dashed,dash=0.16cm 0.16cm](6.724375,3.4544375)(5.724375,4.8944373)
\psline[linewidth=0.024cm,linestyle=dashed,dash=0.16cm 0.16cm](5.804375,4.9144373)(7.684375,4.8944373)
\psline[linewidth=0.024cm](5.804375,5.0344377)(7.744375,5.0144377)
\psdots[dotsize=0.152](6.744375,6.3744373)
\psline[linewidth=0.04cm,linestyle=dotted,dotsep=0.06cm](7.764375,4.9744377)(6.764375,6.4344373)
\psline[linewidth=0.04cm,linestyle=dotted,dotsep=0.06cm](6.704375,6.4544377)(5.704375,5.0144377)
\psdots[dotsize=0.152](9.764375,4.9904375)
\psdots[dotsize=0.152](9.764375,6.4304376)
\psdots[dotsize=0.152](8.564375,3.6104374)
\psdots[dotsize=0.152](10.944375,3.6104374)
\psline[linewidth=0.04cm,linestyle=dotted,dotsep=0.06cm](8.644375,3.5904374)(9.764375,4.9504375)
\psline[linewidth=0.04cm,linestyle=dotted,dotsep=0.06cm](9.764375,4.9504375)(10.864375,3.5904374)
\psline[linewidth=0.024cm,linestyle=dashed,dash=0.16cm 0.16cm](10.984375,3.6704376)(9.824375,5.0304375)
\psline[linewidth=0.024cm,linestyle=dashed,dash=0.16cm 0.16cm](9.824375,5.0304375)(9.844375,6.4504375)
\psline[linewidth=0.024cm](9.684375,6.4704375)(9.684375,5.0304375)
\psline[linewidth=0.024cm](9.684375,5.0104375)(8.504375,3.6704376)
\psdots[dotsize=0.152](1.744375,-5.8095627)
\psdots[dotsize=0.152](0.164375,-4.9695625)
\psdots[dotsize=0.152](1.384375,-4.5895624)
\psdots[dotsize=0.152](2.544375,-4.5695624)
\psline[linewidth=0.024cm](0.124375,-5.0095625)(1.704375,-5.8295627)
\psline[linewidth=0.024cm](1.824375,-5.7895627)(2.604375,-4.6095624)
\psline[linewidth=0.024cm,linestyle=dashed,dash=0.16cm 0.16cm](0.204375,-4.8895626)(1.684375,-5.6895623)
\psline[linewidth=0.024cm,linestyle=dashed,dash=0.16cm 0.16cm](1.664375,-5.6895623)(1.324375,-4.6295624)
\psline[linewidth=0.04cm,linestyle=dotted,dotsep=0.06cm](1.424375,-4.5495625)(1.804375,-5.7295623)
\psline[linewidth=0.04cm,linestyle=dotted,dotsep=0.06cm](1.764375,-5.6295624)(2.464375,-4.5295625)
\psdots[dotsize=0.152](1.744375,-3.6495626)
\psline[linewidth=0.04cm,linestyle=dotted,dotsep=0.06cm](1.444375,-4.5895624)(1.804375,-3.6295626)
\psline[linewidth=0.04cm,linestyle=dotted,dotsep=0.06cm](1.664375,-3.6295626)(2.484375,-4.6095624)
\psline[linewidth=0.024cm](2.584375,-4.5095625)(1.804375,-3.6095624)
\psdots[dotsize=0.152](0.164375,-3.6095624)
\psline[linewidth=0.024cm](0.104375,-3.5895624)(0.104375,-4.9095626)
\psline[linewidth=0.024cm](0.184375,-3.5295625)(2.584375,-4.5095625)
\psline[linewidth=0.04cm,linestyle=dotted,dotsep=0.06cm](0.204375,-3.5695624)(1.424375,-4.5495625)
\psline[linewidth=0.04cm,linestyle=dotted,dotsep=0.06cm](0.144375,-3.6695626)(2.504375,-4.6095624)
\psline[linewidth=0.024cm,linestyle=dashed,dash=0.16cm 0.16cm](0.244375,-3.5695624)(0.224375,-4.9095626)
\psline[linewidth=0.024cm,linestyle=dashed,dash=0.16cm 0.16cm](0.104375,-3.6295626)(1.324375,-4.6095624)
\psline[linewidth=0.024cm](0.224375,-4.8495626)(0.844375,-4.3095627)
\psline[linewidth=0.024cm](1.684375,-3.6095624)(1.324375,-3.9095626)
\psline[linewidth=0.024cm,linestyle=dashed,dash=0.16cm 0.16cm](0.284375,-4.9295626)(0.904375,-4.3895626)
\psline[linewidth=0.024cm,linestyle=dashed,dash=0.16cm 0.16cm](1.404375,-3.9495625)(1.764375,-3.6695626)
\psline[linewidth=0.024cm](0.984375,-4.1695623)(1.104375,-4.0695624)
\psline[linewidth=0.024cm](1.104375,-4.2295623)(1.204375,-4.1495624)
\psline[linewidth=0.024cm,linestyle=dashed,dash=0.16cm 0.16cm](1.704375,-3.6695626)(1.324375,-4.5295625)
\psline[linewidth=0.024cm](6.584375,0.0704375)(7.484375,0.0704375)
\psline[linewidth=0.024cm](7.484375,0.0704375)(6.984375,-0.4295625)
\psline[linewidth=0.024cm](6.984375,-0.4295625)(6.584375,0.0704375)
\psline[linewidth=0.024cm](6.584375,0.0704375)(6.984375,-1.2295625)
\psline[linewidth=0.024cm](6.984375,-1.2295625)(7.484375,0.0704375)
\psline[linewidth=0.024cm](6.984375,-1.2295625)(6.984375,-0.4295625)
\psline[linewidth=0.024cm](6.684375,-0.0295625)(7.384375,-0.0295625)
\psline[linewidth=0.024cm](6.984375,-0.5695625)(7.384375,-0.1495625)
\psline[linewidth=0.024cm](6.984375,-0.5695625)(6.664375,-0.1695625)
\psdots[dotsize=0.152](10.284375,0.1704375)
\psdots[dotsize=0.152](9.784375,-0.2295625)
\psdots[dotsize=0.152](9.784375,-1.1295625)
\psline[linewidth=0.024cm](9.784375,-1.1295625)(9.784375,-0.2295625)
\psline[linewidth=0.024cm,linestyle=dashed,dash=0.16cm 0.16cm](9.844375,-1.1095625)(9.844375,-0.2095625)
\psline[linewidth=0.04cm,linestyle=dotted,dotsep=0.06cm](9.724375,-0.2095625)(9.724375,-1.1295625)
\psline[linewidth=0.024cm](9.764375,-0.2095625)(10.284375,0.1904375)
\psline[linewidth=0.024cm,linestyle=dashed,dash=0.16cm 0.16cm](9.784375,-0.2095625)(10.584375,-0.3295625)
\psdots[dotsize=0.152](10.584375,-0.3295625)
\psdots[dotsize=0.152](8.984375,-0.2295625)
\psline[linewidth=0.04cm,linestyle=dotted,dotsep=0.06cm](9.784375,-0.2295625)(8.984375,-0.2295625)
\psdots[dotsize=0.08](9.184375,-0.0295625)
\psdots[dotsize=0.08](9.584375,0.1704375)
\psdots[dotsize=0.08](9.984375,0.1704375)
\pspolygon[linewidth=0.024,fillstyle=solid,fillcolor=color12471b](0.384375,-0.2295625)(0.784375,-0.8295625)(1.584375,-0.2295625)(1.584375,-0.2295625)
\psline[linewidth=0.08cm]{cc-cc}(0.784375,-1.5695626)(1.584375,-0.2295625)
\psline[linewidth=0.08cm]{cc-cc}(0.784375,-1.6095625)(0.384375,-0.1895625)
\psline[linewidth=0.08cm](0.384375,-0.2295625)(1.184375,0.3704375)
\psline[linewidth=0.08cm]{-cc}(1.184375,0.3704375)(1.604375,-0.2895625)
\psline[linewidth=0.08cm]{cc-cc}(0.784375,-0.8495625)(1.204375,0.4104375)
\pstriangle[linewidth=0.024,dimen=outer,fillstyle=solid,fillcolor=color12471b](1.264375,1.9704375)(0.8,0.6)
\rput{-180.0}(7.64875,4.580875){\pstriangle[linewidth=0.024,dimen=outer](3.824375,1.9904375)(0.8,0.6)}
\psline[linewidth=0.024cm](3.544375,2.4304376)(4.104375,2.4304376)
\psline[linewidth=0.024cm](6.384375,2.3104374)(6.784375,2.9104376)
\psline[linewidth=0.024cm](6.784375,2.9104376)(7.184375,2.3104374)
\psline[linewidth=0.024cm](6.384375,2.3104374)(7.184375,2.3104374)
\psline[linewidth=0.024cm](7.184375,2.3104374)(6.784375,1.7104375)
\psline[linewidth=0.024cm](6.784375,1.7104375)(6.384375,2.3104374)
\psline[linewidth=0.024cm](6.484375,2.1504376)(7.064375,2.1504376)
\psline[linewidth=0.08cm]{cc-cc}(9.804375,2.2104375)(9.804375,2.7504375)
\psline[linewidth=0.08cm]{cc-cc}(9.804375,2.2104375)(9.404375,1.8104374)
\psline[linewidth=0.08cm]{cc-cc}(9.804375,2.2104375)(10.204375,1.8104374)
\psline[linewidth=0.08cm]{cc-cc}(0.784375,-1.5895625)(0.804375,-0.8095625)
\psdots[dotsize=0.152](3.984375,-0.6295625)
\psdots[dotsize=0.152](4.584375,-0.0295625)
\psdots[dotsize=0.152](4.584375,-1.2295625)
\psdots[dotsize=0.152](3.384375,-1.2295625)
\psdots[dotsize=0.152](3.384375,-0.0295625)
\psline[linewidth=0.024cm,linestyle=dashed,dash=0.16cm 0.16cm](4.584375,-1.2295625)(3.984375,-0.6295625)
\psline[linewidth=0.024cm,linestyle=dashed,dash=0.16cm 0.16cm](3.984375,-0.6295625)(4.584375,-0.0295625)
\psline[linewidth=0.024cm](3.384375,-1.2295625)(3.984375,-0.6295625)
\psline[linewidth=0.024cm](3.984375,-0.6295625)(3.384375,-0.0295625)
\psdots[dotsize=0.152](3.384375,-4.8295627)
\psdots[dotsize=0.152](3.784375,-4.2295623)
\psdots[dotsize=0.152](4.584375,-4.2295623)
\psdots[dotsize=0.152](5.184375,-4.8295627)
\psdots[dotsize=0.152](4.184375,-5.6295624)
\psline[linewidth=0.024cm](3.384375,-4.8295627)(3.784375,-4.2295623)
\psline[linewidth=0.024cm,linestyle=dashed,dash=0.16cm 0.16cm](4.184375,-5.6295624)(5.184375,-4.8295627)
\psline[linewidth=0.04cm,linestyle=dotted,dotsep=0.06cm](4.184375,-5.6295624)(4.584375,-4.2295623)
\psline[linewidth=0.024cm](4.184375,-5.6295624)(3.784375,-4.2295623)
\psline[linewidth=0.04cm,linestyle=dotted,dotsep=0.06cm](4.584375,-4.2295623)(3.784375,-4.2295623)
\psline[linewidth=0.024cm,linestyle=dashed,dash=0.16cm 0.16cm](5.184375,-4.8295627)(4.584375,-4.2295623)
\psellipse[linewidth=0.024,dimen=outer](6.984375,-1.8295625)(0.412,0.212)
\psline[linewidth=0.024cm](6.984375,-2.8178437)(7.384375,-1.8295625)
\psline[linewidth=0.024cm](6.584375,-1.8295625)(6.984375,-2.8178437)
\usefont{T1}{ptm}{m}{it}
\rput(0.361875,2.2804375){icons}
\usefont{T1}{ptm}{m}{it}
\rput(1.3404688,1.2704375){\small triangle}
\usefont{T1}{ptm}{m}{it}
\rput(3.7554688,1.2704375){\small twain}
\usefont{T1}{ptm}{m}{it}
\rput(6.851875,1.2704375){\small trey}
\usefont{T1}{ptm}{m}{it}
\rput(9.819375,1.2704375){\small tripod}
\usefont{T1}{ptm}{m}{it}
\rput(1.6148437,-6.3295627){\small (3,2)-ped}
\usefont{T1}{ptm}{m}{it}
\rput(8.361875,-2.3195624){icons}
\usefont{T1}{ptm}{m}{it}
\rput(7.2548437,-3.1295626){\small (3-) pyramid}
\usefont{T1}{ptm}{m}{it}
\rput(10.035781,-3.1295626){\small multiplex}
\usefont{T1}{ptm}{m}{it}
\rput(1.2514062,-2.1295626){\small (1,3)-hedgehog}
\usefont{T1}{ptm}{m}{it}
\rput(3.9546876,-2.1195624){cross}
\usefont{T1}{ptm}{m}{it}
\rput(4.373594,-6.3195624){open (3-) catena}
\end{pspicture} }
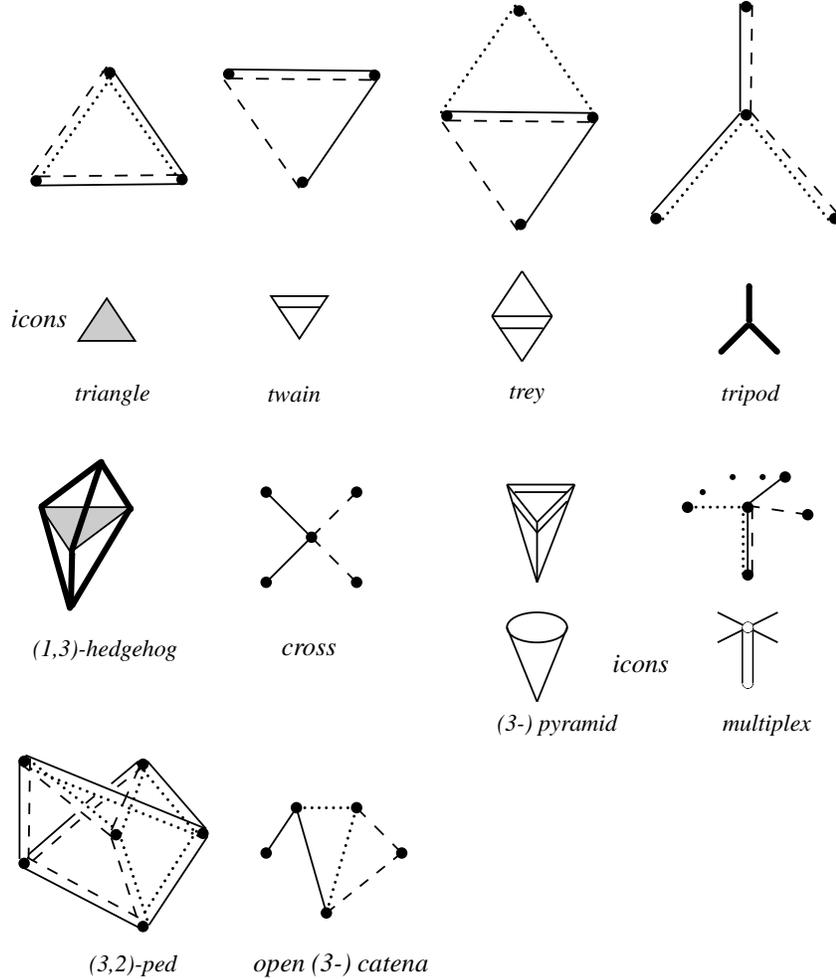
\begin{center} Fig. 5. Structural groups and their icons. \end{center}

\subsection{Types of vertices of a compatible tee-familiy and casings}
\label{TypeVertex} 

The following terminology will be useful in our further analysis of tee-clusters.
We shall say that a tee/dee $\vartheta$ \emph{blocks} (alternatively, is 
\emph{blocking}) a tee/dee $\theta$ if it is incompatible with $\theta$. So, 
we have the following 
\newline\newline
\noindent
{\bf Blocking rule:}
{\it If $\Phi$ is a (tee/dee-)cluster and vertices of $\theta$ belong to $S(\Phi)$, then
$\theta$ belongs to $\Phi$, if there are no blocking $\theta$ elements in $\Phi$.}
\newline\newline
Tees of a compatible tee-family 
$\Phi$ are naturally subdivided into six classes: 
$$
 ece-,\quad ee-,\quad ec-,\quad e-,\quad c-\quad \mathrm{and} \quad 0-\mathrm{tees}
$$  
according to their vertices which are mixing (in $\Phi$). For instance, 
a tee $\theta\in \Phi$ is an $ec$-tee if one of its end vertices and the center 
vertex are mixing, while the remaining third one is another end 
vertex. End vertices of an $ee$--tee are mixing but the center is not. 
All vertices of a $0$-tee are not mixing, etc. 

In what follows we shall determine the structure of tee-clusters 
by analyzing "neighborhoods" of tees of each of these  types. Schematically,  
a tee-cluster consists of "neighborhoods", called \emph{casings}, of the above 
described structural groups which are tied together by means of tees called 
\emph{connectives}. Exact meaning of these terms is explained below.

\medskip
\noindent \textit{Hedgehogs and $\textbf{ece}$-tees.} 

\noindent First, we have
\begin{lem}\label{ece-str}
Let $\Phi$ be a compatible tee-family. Then any 
$ece$-structure in $\Phi$ belongs to an unique  contained in $\Phi$ 
triangle. 
\end{lem} 
\begin{proof}
Let $\theta=\lfloor i,j|k\rceil\in \Phi$ be an ${ece}$-structure. Since 
$e_i$ (resp., $e_j$) is mixing, there is a tee 
$\theta_1\in \Phi$ (resp., $\theta_2\in \Phi$) whose center vertex is $e_i$ 
(resp., $e_j$). Since $\theta_1$ (resp., 
$\theta_2)$ is compatible with $\theta$, it  must be of the form 
$\theta_1=\lfloor j,l|i\rceil$ (resp., 
$\theta_2=\lfloor i,m|j\rceil$). Moreover, compatibility of $\theta_1$ and 
$\theta_2$ implies $l=m$. If $l=m\neq k$, then $\theta_1, \theta_2$ and 
$\theta$ form a trey contained in $\Phi$. In this case the center vertex 
$e_k$ of this trey is a center vertex of $\Phi$ by one of basic properties 
of treys in contradiction with the hypothesis of the lemma. Hence $l=m=k$ and  
$\theta_1, \theta_2$ together with $\theta$ form a triangle. 
\end{proof}

\begin{cor}
If all vertices of a compatible tee-family $\Phi$ 
are mixing, then $\Phi$ is a disjoint, i.e., without common vertices, union 
of triangles. In particular, if all vertices of a  tee-cluster are  mixing, then it 
is a triangle. 
\end{cor}
\begin{proof}
If all vertices of a compatible tee-family $\Phi$ are mixing, then all belonging 
to it tees are $ece$--tees. Therefore, $\Phi$ is a union of triangles. On the 
other hand, if two triangles have one or two common vertices, then, obviously, 
they contain incompatible tees.
\end{proof}

Let $\Delta\subset \Phi$ be a triangle. It is easy to see that any 
nontrivially compatible with $\Delta$ tee, whose center is a center 
vertex of $\Phi$, is also compatible with all tees belonging to $\Phi$. 
Moreover, all such tees are, obviously, compatible each other. They all
form the \emph{casing} of  $\Delta$ (in $\Phi$). So, if $\Phi$ is a tee-cluster
this casing belongs to $\Phi$. Hence all triangles in a 
tee-cluster $\Phi$ together with their casings form a hedgehog 
contained in $\Phi$. Denote it by  $\Phi_{\textbf{h}}$. Emhasize that the 
thorns of $\Phi_{\textbf{h}}$ are center vertices of $\Phi$ and vice versa, 
and the tees forming this casing are ${ee}$-tees. 

By summing up  these observations and lemma\;\ref{ece-str} we get
\begin{proc}\label{max-hedgehog}
Let $\Phi$ be a tee-cluster. Then $\Phi_{\textbf{h}}$ is not empty 
if and only if $\Phi$ has at least one center vertex and at least one 
${ece}$-tee. 
\end{proc}

\medskip
\noindent \textit{Pendent twains and $\textbf{ee}$-tees.} 

\noindent  
Let $\Phi$ be a compatible tee-family. A twain $\wedge\subset \Phi$ will be 
called \emph{pendent}, or, shortly,  \emph{p-twain}, if it belongs to a trey 
$\top\subset \Phi$ and the top of $\wedge$  is an end vertex of $\Phi$. 

\begin{lem}\label{ee-str}
\begin{enumerate}
\item The ends of an $ee$-tee are either side vertices of a $p$--twain, or 
vertices of a triangle belonging to $\Phi$. The center of such a tee 
is a center vertex of $\Phi$.
\item If the ends of a tee $\theta$ are bottom vertices of a pendent twain 
$\wedge\subset \Phi$ and the center of $\theta$ is a center vertex of $\Phi$, 
then $\theta$ is compatible with $\Phi$.
\item Let $\top_1\neq\top_2$ be two treys such that $S(\top_1)\cap S(\top_2)\neq\emptyset$. 
If belonging to them tees are compatible each other, then either $\top_1\cap\top_2$
is a twain, or $\top_1\cap\top_2=\emptyset$. In the last case $\top_1$ and $\top_2$ 
have a common center, or a common end vertex, or both.
\end{enumerate}
\end{lem} 
\begin{proof}
(1) The proof literally repeats that of lemma\,\ref{ece-str} with the exception that in 
this case $k$ may differ from $l=m$. If $l=m\neq k$, then tees $\theta, \theta_1$ 
and $\theta_2$ (see the proof of  lemma\,\ref{ece-str}) form a contained in 
$\Phi$ trey. The second alternative takes place 
if the triangle containing the twain $\{\theta_1, \theta_2\}$ 
belongs to $\Phi$. 

(2)A tee $\vartheta\in \Phi$ might be incompatible with $\theta$ only if its center 
is one of end vertices of  $\theta$. Therefore, the center
of $\vartheta$ is a side point of a trey $\top\subset\Phi$, which contains 
$\wedge$. Then, by basic property of treys,  $\vartheta\in\wedge$ . Hence  
$\vartheta$ and  $\theta$ have a common end vertex and
as such are compatible. 

(3) A direct check by paying attention to the basic property of treys. 
\end{proof}

All tees described in assertion (2) of the above lemma 
constitute the \emph{ee-casing} of the pendent twain $\wedge$. 
By this assertion, this casing  belongs to $\Phi$ if $\Phi$ is a 
tee-cluster. 

Now consider a tee such that its end vertices are end and bottom vertices
of a p-twain $\wedge\subset \Phi$, while its center is a center vertex of $\Phi$.
All such tees form the 
\emph{e-casing} of  $\wedge$. We shall see below (lemma\,\ref{e-str}) that 
the \emph{e-}casing of  a p-twain $\wedge$ is compatible with $\Phi$ and 
hence belongs to $\Phi$ if $\Phi$ is a  tee-cluster. Obviously, 
all tees composing an \emph{e}-casing (resp., an 
\emph{ee}-casing) are \emph{e}-tees (resp., \emph{ee}-tees). A 
trey together with its \emph{e}-casing will be called \emph{completed}. 
A completed trey is obtained from an (1,q)-hedgehog by removing 
one tee from its base triangle. So, it is not a tee-cluster.
\begin{ex}
\textbf{1-trey cluster.}  Let $E,C,B_1,B_2$ be the end, center and side 
vertices of a trey $\top$, respectively. Add to them a new vertex $D$ and consider 
tees $\lfloor E,B_i|D\rceil,  \,\lfloor E,B_i|C\rceil,  \,i=1,2,$ and $\lfloor 
E,D|C\rceil$. These 5 tees together with $\top$  form a tee-cluster, 
called 1-trey cluster and denoted $1\top$. By adding to it the dee $\lfloor E|D\rceil$ 
one gets a cluster. 
\end{ex}

\medskip
\noindent \textit{Standing twains, $\textbf{ec}$-tees and pyramids.}  

\noindent A twain $\wedge\subset \Phi$ which does belong neither to a 
triangle, nor to a trey 
contained in $\Phi$ will be called \emph{standing}, or, shortly,  an
\emph{s-twain}. Obviously, the top of an s-twain is an end vertex of 
$\Phi$. Standing twains can be characterized as those which have no common
vertices with \emph{ee}-tees. They naturally appear in connection with 
\emph{ec}-tees. 

To proceed on we need the following terminology. A tee 
$\theta\in\Phi$ (resp., a twain, or a pyramid belonging to $\Phi$)  one 
end of which (resp., the top) is an end vertex $E$ of $\Phi$ will be called 
\emph{rooted} at $E$.  Let $\wedge\subset\Phi$ be a twain. A tee 
$\theta\in\Phi, \,\theta\notin\wedge$, which is rooted at the same vertex as  
$\wedge$, will be called a  \emph{side} (resp., \emph{lateral}) tee for
$\wedge$ if the center (resp., the second end) of $\theta$ is one of bottom 
vertices of $\wedge$.
\begin{lem}\label{s-twain-str}
Let $\Phi$ be a tee-cluster and $\wedge\subset \Phi$ be rooted at 
$E$ s-twain. Then
\begin{enumerate}
\item if $\Phi$ 
possesses a center vertex, then at least one of side tees of $\wedge$ 
belongs to $\Phi$;
\item all nontrivially compatible with $\wedge$ tees are 
rooted at $E$ and hence are compatible each other;
\item if  $\theta\in\Phi$ is a lateral 
tee of $\wedge$ and the center of $\theta$ is not a center vertex of $\Phi$, then 
$\theta$ belongs to an s-twain $\wedge'\subset\Phi$.
\item the set of bottom vertices of all rooted at $E$ s-twains form the 
bottom of a rooted at $E$ pyramid $P\subset\Phi$. The pyramid $P$
contains any belonging to $\Phi$ pyramid, which is rooted at $E$, and any
lateral tee of it does not belong to $\Phi$;
\item two compatible pyramids rooted at different vertices have no common 
vertices. 
\end{enumerate}
\end{lem}
\begin{proof}
(1) Let $B_1, \,B_2$ be bottom vertices of $\wedge$ and $C$ a center vertex
of $\Phi$. The tee $\vartheta=\lfloor B_1,B_2|C\rceil$ does not
belong to $\Phi$, since $\wedge$ is an s-twain. On the other hand, since $\Phi$
is a tee-cluster, there is a tee $\theta\in\Phi$ which is incompatible with  $\vartheta$.
Since $C$ is a center vertex of $\Phi$, the center of $\theta$ is either
$B_1$, or  $B_2$. But being compatible with $\wedge$ the tee $\theta$ is rooted 
at $E$ and hence is a side tee for $\wedge$.

(2) It can be easily checked that a tee which is nontrivially compatible with a twain 
$\wedge$ and not rooted at its top vertex is of the form 
$\vartheta=\lfloor B_1,B_2|C\rceil$ with $B_1, B_2$ being bottom vertices of $\wedge$.
Obviously, this is impossible for an s-twain. 

(3) Let $B_1,B_2\in S(\Phi)$ be bottom vertices of $\wedge$ and 
$\theta=\lfloor E,B_2|C\rceil$. It suffices to show that $\vartheta=\lfloor 
E,C|B_2\rceil$ belongs to $\Phi$, i.e. there is no incompatible with
$\vartheta$ tee $\varrho\in\Phi$. Such a  tee must have either its 
center at $C$, or one 
of its ends at $B_2$. In the first case one of the ends of  $\varrho$ must 
be either $B_2$, or $E$. Indeed, otherwise, $\varrho$ and $\vartheta$ would 
form a cross whose center $C$ will be, by the basic property of crosses, a 
center vertex of $\Phi$.  But a tee of the form $\varrho=\lfloor D,B_2|C\rceil, \,C\neq E$, 
is compatible with $\theta^{\prime}=\lfloor E,B_1|B_2\rceil\in \wedge$ iff $D=B_1$.
This is impossible, since $\varrho=\lfloor B_1,B_2|C\rceil$ together with
$\wedge$ form a trey in contradiction with the hypothesis that $\wedge$ is
an s-twain. On the other hand, $\varrho$ can not have an end at $E$,
since any rooted at $E$ tee is compatible with $\vartheta$.
 
Finally, if one end of $\varrho$ is $B_2$, then $B_1$ is its second end, 
since $\varrho$ is compatible with $\wedge$ but not compatible with $\theta$.
But we have already seen that this is impossible.

(4) Let  $\wedge_1,\dots,\wedge_m$ be rooted at $E$ s-twains that belong to 
$\Phi$ and $\mathcal{P}$ the mentioned in the statement pyramid. If 
$\theta\in \mathcal{P}$, then $\theta$ is either a lateral tee for
one of twains $\wedge_i$'s, or belongs to one of them. In the first case
$\theta$ belongs to $\Phi$ in virtu os assertion (3). This proves that 
$\mathcal{P}\subset\Phi$.

The remaining part of (4) also directly follows from assertion (3).

(5) Directly from proposition \ref{123}. 
\end{proof} 

The pyramid figuring in assertion (4) of the above lemma will be denoted 
by $P_E=P_E(\Phi)$.  Its \emph{casing} is composed of all tees of the 
form $\lfloor E,B|C\rceil$ with $B$ being a bottom vertex of $P_E$ and
$C$ a center vertex of $\Phi$.
It will be shown that the casing of $P_E$ belongs to $\Phi$ (see lemma\,\ref{e-str} below). 

As in the case of twains we shall call a rooted at $E$ 
tee a \emph{side} (resp., \emph{lateral}) tee of 
$P_E$ if its center (resp., second end) belongs to the bottom of  $P_E$, 
while the remaining third vertex of it is not a vertex of $P_E$. 

\subsection{Connectives and $\textbf{nec}$-tees.}\label{con-nec-tees}  

Now we shall describe the situations when 
\emph{ec}-tees appear in a \emph{non-subbordinate} manner, i.e., 
not as elements of p-, or s-twains. Such  a tee will be 
called a \emph{nec-tee}. 

\begin{lem}\label{nec-str} 
Let $\Phi$ be a  tee-cluster and $\theta=\lfloor E,A|C\rceil\in\Phi$ 
a rooted at $E$ \emph{nec}-tee. Then it holds:
\begin{enumerate} 
\item $\theta$ is included into a 3-catena of the form 
$$
\wr=\{\theta_{or}=\lfloor E,C|B\rceil,\,\theta, \,\theta_{end}=\lfloor 
E,D|A\rceil\} \quad \mbox{with}  \quad A\neq B  \quad\mbox{and}  \quad C\neq D;
$$ 
\item $D\neq B$; 
\item If $P_E\neq \emptyset$, then any 
$\varrho=\lfloor E,C|B'\rceil$ with $B'\in S(P_E)$ belongs to $\Phi$. In 
particular, any such a tee can be taken for $\theta_{or}$. 

\item  If $P_E\neq \emptyset$, then 
$\theta$ is a side tee of $P_E$. 

\item $\Phi$ contains a tee of the form $\varrho=\lfloor 
A,D|Q\rceil$ with $Q\neq C$, i.e., $\theta$ is contained in the compatible 
family $\Psi(\theta)=\{\varrho\}\cup\wr\subset\Phi$. If $Q=B$, then $B$ is 
a center vertex. 

\item If $D$ is a mixing vertex, then $A$ is a bottom vertex of a p-twain, and,  
therefore, $\theta$ is a lateral tee of it. 

\item If $P_E=\emptyset$, then the center $B$ of  $\theta_{or}$ is a center 
vertex of $\Phi$. 

\item If $C$ is a vertex of a tee  $\vartheta$, which is 
compatible with $\wr$, then  $\vartheta$ is rooted at $E$. 

\item Let $C'$ be the center of a \emph{nec}-tee  $\vartheta$, which is 
rooted at $E$. If $C\neq C'$, then 
$\{\lfloor E,C|C'\rceil, \,\lfloor E,C'|C\rceil\}$ is an s-twain in $\Phi$. 
In particular, $P_E\neq \emptyset$. 

\item If $A$ is the center vertex of a tee, which is compatible 
with $\Psi(\theta)$ (see assertion (5)), then this tee coincides with 
$\theta_{end}$. 
\end{enumerate}
\end{lem} 
\begin{proof}
(1) Since $C$ is a mixing vertex, there is a compatible with $\theta$ 
tee $\vartheta$ one of whose ends is $C$. Another end vertex of such a
$\vartheta$ must be either $A$, or $E$. The first of these alternative is 
impossible. Indeed, since $A$ is a mixing vertex there is a tee $\varrho\in\Phi$
with the center at $A$. The only such tee, which is compatible with
$\theta$ and $\vartheta$, is $\varrho=\lfloor E,C|A\rceil\}$. But  
$\{\theta, \varrho\}\subset\Phi$ is a containing $\theta$ twain in 
contradiction with the hypothesis that $\theta$ is a  \emph{nec}-tee. So, 
we may take $\vartheta=\lfloor E,C|A\rceil$ for $\theta_{or}$.

Next, since $C$ is a mixing vertex, there is a compatible with $\theta$ 
tee $\vartheta'$ whose center is $A$. Any such tee is
of the form $\lfloor E,D|A\rceil$. Finally, any of equalities $A=B$, or $C=D$
implies that $\theta$ belongs to a twain contained in $\Phi$.  But this is
impossible, since $\vartheta'$ is a \emph{nec}-tee.

(2) Assume that $B=D$. In this case any tee, which is nontrivially
compatible with the family $\{\theta_{or},\,\theta,\,\theta_{end}\}\subset\Phi$, 
is rooted at $E$. By this reason, such a tee and, in particular, any  
$\varrho\in\Phi$ is compatible with the tee $\vartheta=\lfloor E,C|A\rceil$. This 
tee forms a twain together with $\theta$ and belongs to $\Phi$, since $\Phi$ is 
a cluster.  But this contradicts the hypothesis that $\theta$ is a 
\emph{nec}-tee.

(3) A tee $\varrho$, which is incompatible with $\vartheta=\lfloor E,C|B'\rceil$,
must have either its center at $C$, or one of its ends at $B^{\prime}$. If, 
additionally, $\varrho\in\Phi$, then $\varrho$ is compatible with $P_E\cap\wr$.
But all such tees are rooted at $E$ and hence are
compatible with $\vartheta$. This proves that  all tees belonging to $\Phi$ 
are compatible with  $\vartheta$ and hence $\vartheta\in\Phi$.            

(4) According to (3) we can assume that $B$ is a bottom vertex of $P_E$. 
It ifollows that a tee, which has the center (resp., one of its ends) at $B$ 
(resp., at $C$) and is compatible with $\wr$ and $P_E$, is rooted at $E$. As
previously, this shows that $\vartheta=\lfloor E,B|C\rceil$ is compatible with
$\Phi$ and hence belongs to $\Phi$. So, the twain $\{\lfloor E,C|B\rceil,\,\lfloor 
E,B|C\rceil\})$ belongs to $\Phi$ and, therefore, to $P_E$ 
(lemma\, \ref{s-twain-str}, (3)). This proves that $C$ is a bottom vertex of $P_E$,

(5) Since $\theta$ is a  \emph{nec}-tee, the tee $\vartheta=\lfloor E,C|A\rceil$, which forms 
a twain with $\theta$, does not belong to $\Phi$. Therefore, $\Phi$, being a tee-cluster,
contains a tee $\varrho$, which is  incompatible with $\vartheta$. But any tee, which is
incompatible with $\vartheta$ and at the same time is compatible with  the catena 
$\wr\subset\Phi$, is of the form $\lfloor A,D|Q\rceil$. Moreover, $Q\neq C$, since, 
otherwise, $\varrho$ and $\theta_{or}$ would be incompatible. Finally, if $Q=B$, 
then, in view of assertion (1), $\varrho$ and $\theta_{or}$ form a cross. So, $B$ is 
a center vertex of $\Phi$ by the basic property of crosses. 

(6) In this case the tee $\lfloor A,D|Q\rceil$ from (5) is an  
\emph{ee}-tee. So, by lemma\,\ref{ee-str}, $A$ and $D$ are bottom 
vertices of a rooted at $E$ p-twain. 

(7) The tee $\vartheta=\lfloor E,C|B\rceil$ does not belong to $\Phi$. Indeed, 
otherwise, the twain $\{\lfloor E,C|B\rceil,\,\lfloor E,B|C\rceil\})$ would belong to $\Phi$
in contradiction with  the hypothesis that $P_E=\emptyset$. Hence there is a
tee $\bar{\vartheta}\in\Phi$ which is incompatible with  $\vartheta$. On the other hand, 
$\bar{\vartheta}$, as a tee from $\Phi$,  is compatible with $\wr$. It remain to
observe that any tee which satisfies these conditions has its center at $B$ and form
a cross with $\theta_{or}$. Hence $B$ is a center vertex of $\Phi$ by the basic 
property of crosses.

(8) Immediately from assertion (1).

(9) A tee from $\Phi$, which is incompatible with one of tees 
$\lfloor E,C|C'\rceil, \,\lfloor E,C'|C\rceil\}$, has one of its vertices either at $C$, 
or at $C'$. But this is impossible, since, according to assertion (8), any such 
tee is rooted at $E$.

(10) Obvious. 
\end{proof}

The fact that any \emph{nec}-structure $\theta$ is included in a family of 
the form  $\Psi(\theta)$ (assertion (5) of the preceding lemma) will be often
used in the sequel. It is worth noticing that 
$\Psi(\theta)$ is not unique and  tees composing  
$\Psi(\theta)$ have different center vertices except, possibly, $B$ and $Q$ 
which may coincide. 

\medskip
\noindent\textit{Pendent \emph{nec}-tees}

\noindent  Below we shall keep the notation of  lemma\,\ref{nec-str} 
for vertices of $\Psi(\theta)$. A \emph{nec}-tee will be called \emph{pendent} 
if it is not a side tee of a contained in $\Phi$ pyramid. 
\begin{cor}\label{nec-cor}
Let $\Phi$ be a tee-cluster, $E\in S(\Phi)$ an end vertex of it and 
$\theta$ a rooted at $E$ \emph{nec}-tee. Then 
\begin{enumerate}
\item $\theta$ is pendent if and only if $P_E=\emptyset$.

\item $\theta$ is a side tee of $P_E$, if $P_E\neq\emptyset$ or
a pendent one, if $P_E=\emptyset$. In both cases it may at the 
same time be a lateral tee of a p-twain. 

\item If a side tee of  $P_E$ is also lateral for a 
p-twain in $\Phi$, then it belongs to  $\Phi$. 

\item If $\theta$ is pendent, then the center of any other rooted at $E$ 
pendent tee coincides with the center of $\theta$. 

\item Any tee $\lfloor E,C|B'\rceil$, which is lateral 
for a p-twain in $\Phi$, belongs to  $\Phi$. 

\item If the center vertex of $\theta$ is a vertex of $\vartheta\in\Phi$, 
then $\vartheta$ is rooted at $E$. 
\end{enumerate}
\end{cor}
\begin{proof}
(1) Directly from the definition and lemma\,\ref{nec-str}, (4).

(2) Obviously from assertion (1) and  lemma\,\ref{nec-str}, (6). 

(3) Let $\vartheta=\lfloor E,Q|R\rceil$ be lateral for a p-twain $\wedge\subset\Phi$. 
It could be blocked either by a tee with the center at $Q$, or by 
a tee with an end at $R$. In the first case $Q$ is a side vertex of the 
trey, which contains $\wedge$. But, by the basic property of treys,  
$\vartheta\in \wedge$ and hence does not block $\vartheta$. 
In the second case $R\in S(P_E)$. Since any nontrivially compatible 
with $P_E$ tee is rooted at $E$, it does not block $\vartheta$. So, 
by the blocking rule, $\vartheta$ belongs to $\Phi$.

(4) Let $\theta=\lfloor E,A|C\rceil$ and $\theta'=\lfloor E,A'|C'\rceil$ be  
pendent tees and $C\neq C'$. If one of vertices of a tee $\vartheta$ is the 
center $C$ of $\theta$ and $\vartheta$ is compatible with $\Psi(\theta)$, 
then, as it is easy to see,  $\vartheta$ is rooted at $E$ and similarly for 
$\theta'$. By this reason, tees forming the rooted at $E$ twain $\wedge$ 
with bottom vertices $C$ and $C'$ can not be blocked by a tee from $\Phi$. 
So, by the blocking rule $\wedge\subset\Phi$. Obviously, $\wedge$ is an s-twain. 
This shows that $P_E\neq\emptyset$ in contradiction with the made assumption. 

(5) To prove this assertion it suffices to substitute $\Psi(\theta)$ 
for $P_E$ in the proof of assertion (3). 

(6) Immediately from assertions (1) and (8) of lemma\,\ref{nec-str}.
\end{proof}

According to assertion (4) of this corollary, all pendent \emph{nec}-tees rooted 
at $E$ have the common center vertex. Denote it by $C_E^{pn}$ and consider 
the multiplex $\bot_E^{pt}$ with the origin $E$ and the center $C_E^{pn}$ 
whose ends are bottom vertices of all rooted at $E$ p-twains. It will be 
called the \emph{twain multiplex at E}. Assertions (4) and (5) of the 
corollary\,\ref{nec-cor} show that $\bot_E^{pt}$ belongs to $\Phi$. When $\Phi$ does 
not have rooted at $E$ s-twains, i.e., $P_E=\emptyset$, then the``rod" with ends $E$ 
and $C_E^{pn}$ may be viewed as the ``collapsed" $P_E$.  The 
\emph{casing} of $\bot_E^{pt}$ is composed of all tees with 
ends at $E$ and $C_E^{pn}$ and whose centers are center vertices of 
$\Phi$. If $\Phi$ is a  tee-cluster, then this casing belongs to $\Phi$ 
(see lemma\,\ref{e-str} below).  

A \emph{nec}-tee  will be called a \emph{pyt-connective} (resp., 
\emph{pt-connective}) at $E$ if it is lateral for a rooted at $E$ p-twain 
and simultaneously a side tee for $P_E$ (resp., a rooted at $E$ 
\emph{nec}--pendent tee if $P_E=\emptyset$). Informally speaking, these 
connectives ``consolidate" the system of rooted at $E$ p-twains around the 
``central pillar" $P_E$ (resp., the ``rod" $\overline{EC_E^{pn}}$) into a 
``rigid structure". Both \emph{pyt}- and \emph{pt}-connectives are 
\emph{internal} in the sense that they join p- and s-twains rooted at the 
same vertex. \emph{External} connectives will be discussed below. 

\medskip
\subsection{C- and e-tees.}\label{c-str-all} 

Let $\Phi$ be a compatible tee-family. Denote by $\Phi_E$ the set of all 
rooted at $E$ \emph{ec}-tees. A \emph{c}-tee with end vertices 
$E$ and $D$ connects families $\Phi_E$ and $\Phi_D$. The following lemma 
allow us to subdivide \emph{c}-tees into two classes. 
\begin{lem}\label{two-c-str}
Let $C,D\in S(\Phi)$ be end vertices of $\Phi$ and 
$$
\wr_2=\{\lfloor E,D|C\rceil,\,\lfloor E,C|B\rceil\}\subset\Phi.
$$
Then it holds:
\begin{enumerate}
\item $\theta=\lfloor E,D|C\rceil\in\wr_2$ is a \emph{c}-tee and, conversely, 
any  \emph{c}-tee is contained in a 2-catena $\wr_2\subset\Phi$.

\item If $C$ is a vertex of a tee $\varrho$ which is 
nontrivially compatible with $\wr_2$, then $\varrho$ is of one of the 
following types: $\lfloor E,A|C\rceil$ (type I at $E$), or $\lfloor 
D,C|Q\rceil$ (type II at $E$), or $\lfloor E,C|B'\rceil$ (type III at $E$). 
Tees of type I and II at $E$ are incompatible. 

\item If $\Phi$ is a tee-cluster and $\varrho\in\Phi$ is of type I 
at $E$, then any $\theta'=\lfloor E,E'|C\rceil$ with $E'$ being an end 
vertex belongs to $\Phi$.

\item Let $\varrho\in\Phi$ be of type I at $E$ and $C$ be a vertex of a 
tee $\vartheta\in\Phi$. Then $\vartheta$ is rooted at 
$E$. 
\end{enumerate}
\end{lem}  
\begin{proof}
(1)-(2) Obviously.

(3) A tee $\vartheta\in\Phi$ which could block $\theta'$ must have  
one of its ends at $C$. But being compatible with $\wr_2$ 
$\vartheta$ is either of type II, or of type III at $E$. The first alternative is 
impossible, since tees of types  I and II are incompatible.
On the other hand, a tee of type III does not block $\theta'$. 

(4) An obvious consequence of compatibility of $\vartheta$ with 
$\left(\wr_2\cup\{\varrho\}\right)$. 
\end{proof}

Let $\theta\in\left(\wr_2\cup\{\varrho\}\right)\subset\Phi$ be as in 
lemma\,\ref{two-c-str}. The \emph{c}-tee $\theta$ will be called a 
\emph{hook} at $E$ (resp., a \emph{bridge}) if $\varrho$ is of type I at 
$E$ (resp., of type II). We shall say that the bridge $\lfloor E,D|Z\rceil$ 
\emph{connects} vertices $E$ and $D$. 

\medskip
\noindent\textit{Hooks.}

\noindent Now we shall describe ``environments" of hooks.
\begin{lem}\label{hooks}
Let $\Phi$ be a tee-cluster. Then it holds:
\begin{enumerate}
\item Let $\theta=\lfloor E,D|C\rceil\in\Phi$ be a hook at $E$ and $C$ a
vertex of $\vartheta\in\Phi$. Then $\vartheta$ is rooted at E. 

\item Let $\theta_i=\lfloor E,E_i|C_i\rceil\in\Phi, \,i=1,2,$ be hooks at 
$E$. Then either $C_1=C_2$, or $\wedge=\{\lfloor 
E,C_1|C_2\rceil\},\,\{\lfloor E,C_2|C_1\rceil\}$ is an s-twain. 

\item Let $\theta_1=\lfloor E,E_1|C_1\rceil$ be a hook at $E$ and 
$\theta_2=\lfloor E,E_2|C_2\rceil$ be a pending \emph{nec}--tee at $E$. 
Then either $C_1=C_2$ or 
$\wedge=\{\lfloor E,C_1|C_2\rceil\},\,\{\lfloor E,C_2|C_1\rceil\}$ is an s-twain.

\item Let $E$ and $E'$ be end vertices and $C$ be a bottom vertex of $P_E$, 
or the center vertex of a pending \emph{nec}--tee rooted at $E$, or the center 
vertex of a hook. Then $\theta=\lfloor E,E'|C\rceil$ is a hook at $E$. 

\item Let $\lfloor E,E'|C\rceil$ be a hook at $E$ and $S_1,\,S_2$ be bottom 
vertices of a rooted at E p-twain $\wedge$. Then $\theta_i=\lfloor 
E,S_i|C\rceil\in\Phi, \,i=1,2$. 
\end{enumerate} 
\end{lem}
\begin{proof}
(1) A particular case of lemma\,\ref{two-c-str}, (4). 

(2) Assume that $C_1\neq C_2$. If $C_i, \,i=1,2,$ is a vertex of a 
tee $\vartheta$, then, in virtue of assertion (1), 
$\vartheta$ is rooted at $E_i$ and, therefore, does not block  tees 
that compose $\wedge$. i.e., $\wedge$ belongs to $\Phi$. 
Moreover, tees of the form $\lfloor C_1,C_2|Q\rceil$ are 
incompatible with $\theta_1$ and $\theta_2$ and hence do not belong to 
$\Phi$, i.e., $\wedge$ is an s-twain. 

(3) Observe that one of vertices of a tee $\vartheta\in\Phi$ which 
could block a  belonging to $\wedge$ tee must be either 
$C_1$ or $C_2$. But any such structure is rooted at $E$ 
(corollary\,\ref{nec-cor}, (6),  and lemma\,\ref{two-c-str}, (4)) and 
hence  can not block them. 

(4) Observe that if a tee $\vartheta\in\Phi$ blocks 
$\theta$, then $C$ must be one of its vertices.  But all such tees are 
rooted at $E$ (corollary\,\ref{nec-cor}, (6),  and lemma\,\ref{two-c-str}, 
(4)) and hence  can not block $\theta$.  

(5) By the basic property of treys and lemma\,\ref{two-c-str}, 
(4), $\Phi$ does not contain blocking $\theta_i$ tees.
\end{proof}
According to assertion (2) of lemma\,\ref{hooks} all hooks at $E$ have a common 
center vertex if $P_E= \emptyset$. Denote it by $C_E^{hk}$. Denote also 
by $\bot_E^{hk}$ the multiplex constituted by all tees 
$\lfloor E, E'|C_E^{hk}\rceil$ with $E'$ running all different from $E$ end 
vertices of $\Phi$. It will be called the \emph{hook multiplex} at E.
\begin{cor}\label{hook-str}
Let $\Phi$ be a tee-cluster and $E$ an end vertex of it. Then
\begin{enumerate}
\item all hooks at $E$ are side tees of $P_E$ if $P_E\neq\emptyset$.

\item If $\Phi$ contains  at least one pendent tee and one 
hook at $E$, then $C_E^{pn}=C_E^{hk}$. 

\item $\bot_E^{hk}\subset\Phi$.
\end{enumerate}
\end{cor}
\begin{proof}

Assertion (1) directly follows from  assertions (2) and 
(4) of lemma\,\ref{hooks}, and assertion (2) from corollary\,\ref{nec-cor}, (1)
and lemma\,\ref{hooks} (2), (3).
Assertion (3) is a particular case of assertion (4) of this lemma.   
\end{proof}

In view of corollary\,\ref{hook-str}, (2), there would be no confusion to use 
the common notation $C_E$ for centers $C_E^{hk}$ and  $C_E^{pn}$. 

The \emph{casing }of $\bot_E^{hk}$ is composed of all 
tees $\lfloor E,C_E^{hk}|C\rceil$ with $C$ running all 
center vertices of $\Phi$. Obviously, it coincides with the casing of $\bot_E^{pn}$ 
assuming that $\bot_E^{pt}\neq\emptyset$.  

\medskip
\noindent \textit{Bridges.}  

\noindent Bridges join different families 
$\Phi_E$'s and in this sense are external connectives. On the other hand, 
they are naturally combined with internal connectives as will be shown 
below. First, we need the following

\begin{lem}\label{pend-str}
Let $\Phi$ be a tee-cluster and $\theta=\lfloor 
E,A|C\rceil\in\Phi$ be a \emph{nec}-tee rooted at $E$. Then $A$ is 
either a bottom vertex of a rooted at $E$ p-twain, or the center vertex of 
a bridge one end vertex of which is $E$. 
\end{lem}
\begin{proof}
Recall that $\theta\in\Psi(\theta)$ (see lemma\,\ref{nec-str}, (5)). If, in 
the notation of lemma\,\ref{nec-str}, the vertex $D$ is mixing, then the 
first alternative occurs. If $D$ is an end vertex, then the 
\emph{c}-tee $\lfloor E,D|A\rceil\in\Psi(\theta)$ is a bridge, since 
the unique tee with the center at $A$ which is 
compatible with $\Psi(\theta)$ is, obviously, $\theta$. 
\end{proof}

The following lemma shows the role of bridges in the structure of 
tee-clusters. 
\begin{lem}\label{bridge-str}
Let $\Phi$ be a tee-cluster and $\varrho=\lfloor 
E,D|Z\rceil\in\Phi$ be a bridge. Then it holds:
\begin{enumerate}
\item If $\vartheta\in\Phi$, $\varrho\neq\vartheta$, and $Z$ is a vertex of 
$\vartheta$, then $Z$ is an end vertex of $\vartheta$. 

\item If one of end vertices of a side structure $\theta$ of $P_E$ is $Z$,
then $\theta$ belongs to $\Phi$. 

\item The tee $\theta=\lfloor E,Z|C_E\rceil$ belongs to 
$\Phi$. 
\end{enumerate}
\end{lem}
\begin{proof}

(1) Directly from the definition of a bridge. 

(2)-(3). Let $\theta=\lfloor E,Z|B\rceil$ where $B$ is either a bottom 
vertex of $P_E$, or $B=C_E$. $\theta$ can be blocked by a 
tee $\vartheta\in\Phi$, which either has  the center at $Z$
or one one of its ends at  $B$. Assertion (1) of this lemma excludes 
the first of these possibilities. Next, $B$ is either the center of a 
\emph{nec}-structure, or the center of a hook rooted at $E$. In each of 
these cases a  tee one of whose vertices is $B$ is 
rooted at $E$ (corollary\,\ref{nec-cor}, (6), and lemma\,\ref{hooks}, 
(1)). But such a tee does not block $\theta$. 
\end{proof} 

A \emph{pyb-connective} (resp., \emph{pb-connective}) at an end vertex $E$ 
is a side structure of $P_E$ (resp., $\lfloor E,Z|C_E\rceil$), one of whose 
ends is the center $Z$ of a bridge connecting $E$ with another end vertex. 
All \emph{pb}-connectives at $E$ form a multiplex, denoted by 
$\bot_E^{pb}$. If $\Phi$ is a tee-cluster, then, by  
lemma\,\ref{bridge-str},(3), this multiplex belongs to $\Phi$. Also, denote by  
$\bot_E^{pyt}$ (resp., $\bot_E^{pyb}$) the family of all 
\emph{pyt}-connectives  (resp., \emph{pyb}-connectives) at $E$. If $\Phi$ 
is a tee-cluster, then both $\bot_E^{pyt}$ and $\bot_E^{pyb}$ 
belong to $\Phi$. Each of these families is the union of multiplexes with 
common origin and ends whose centers run bottom vertices of $P_E$. The 
\emph{casing} of a bridge $\lfloor E_1,E_2|Z\rceil$ is composed of all 
tees of the form $\lfloor E_i, Z|C\rceil, \,i=1,2,$ with 
$C$ being a center vertex of $\Phi$. In contrast to all previously 
introduced casings, it is naturally subdivided into two parts that are 
composed of rooted at $E_1$ and of rooted at $E_2$ tees, 
respectively. Lemma\,\ref{e-str} below tells that this casing 
belongs to $\Phi$ as well. 

\medskip
\noindent\textit{Casings and e-tees.} 

\noindent Now we shall describe 
\emph{e}-structures. 
\begin{lem}\label{e-str}Let $\Phi$ be a tee-cluster and $E$ an end 
vertex of it. Then  $\theta=\lfloor E,Z|C\rceil\in\Phi$ is a (rooted at 
$E$) \emph{e}-tee if and only if $Z$ is the center either of an  
\emph{ec}-, or of an \emph{c}-tee rooted at $E$. 
\end{lem} 
\begin{proof}
First, note that if $\theta$ is an  \emph{e}-tee, then $Z$ is 
the center vertex of a tee $\varrho\in\Phi$. Being 
compatible with $\theta$ the tee $\varrho$ is rooted at $E$. Moreover, $\theta$ and 
$\varrho$ form an open 2-catena $\wr_2$ whose initial vertex is $E$. This 
shows that $\varrho$ is either an  \emph{ec}-, or an \emph{c}-tee 
rooted at $E$. 

Conversely, assume that $Z$ is the center vertex of an  \emph{ec}-, or of 
an \emph{c}-tee $\varrho\in\Phi$ which is rooted at $E$. Since $C$ is a 
center vertex of $\Phi$,  $\theta$ can be 
blocked only by a tee $\vartheta\in\Phi$ whose center 
vertex is $Z$. But from the previous description of rooted at $E$ \emph{ec}-, 
and \emph{c}-tees we see that all such tees are rooted at 
$E$ too. Hence they can not block $\theta$. 
\end{proof}
Now we can state that if $\Phi$ is a tee-cluster, then all 
previously considered casings belong to $\Phi$ and any \emph{e}-tee 
belongs to one of these casings.

\medskip
\noindent\textit{0-tees.} 

\noindent These are tees whose end and center vertices are 
end and center vertices of  $\Phi$, respectively. If $\Phi$ is tee-cluster, 
then they, obviously, belongs to it and constitute a 
multiped $\Phi_{mp}\subset\Phi$. If $\Phi$ has at least two ends 
vertices and one \emph{ece}-structure, then both $\Phi_h$ and 
$\Phi_{mp}$ are not empty. In this case they form the hybrid 
$\Phi_{hyb}=\left(\Phi_h\cup\Phi_{mp}\right)\subset\Phi$.

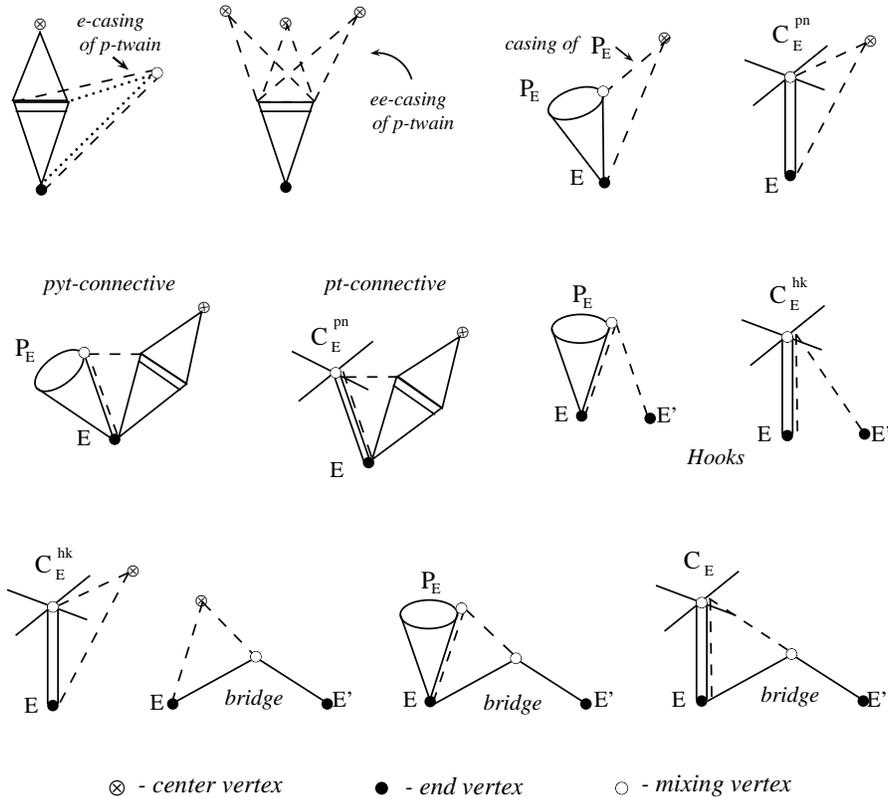
\scalebox{0.95} 
{
\begin{pspicture}(0,-6.0210633)(12.492,6.039077)
\psline[linewidth=0.024cm,linestyle=dashed,dash=0.16cm 0.16cm](5.972,-3.783251)(6.412,-2.4832509)
\psline[linewidth=0.024cm](5.952,-3.823251)(7.152,-3.163251)
\psline[linewidth=0.024cm,linestyle=dashed,dash=0.16cm 0.16cm](0.5123489,3.3761523)(2.1716511,5.017346)
\psline[linewidth=0.04cm,linestyle=dotted,dotsep=0.06cm](0.46157473,3.4469297)(2.0824254,5.0465684)
\psline[linewidth=0.023999996cm,doubleline=true,doublesep=0.12,doublecolor=white](10.972,4.956749)(10.972,3.616749)
\psdots[dotsize=0.16](10.972,3.596749)
\psline[linewidth=0.024cm,linestyle=dashed,dash=0.16cm 0.16cm](11.052,4.996749)(12.032,5.456749)
\psline[linewidth=0.024cm,linestyle=dashed,dash=0.16cm 0.16cm](12.032,5.456749)(11.072,3.636749)
\psdots[dotsize=0.16,dotangle=-90.0,dotstyle=otimes](12.092,5.476749)
\psline[linewidth=0.024cm](10.452,4.576749)(11.492,5.416749)
\psline[linewidth=0.024cm](10.332,5.216749)(11.452,4.796749)
\psdots[dotsize=0.16,fillstyle=solid,dotstyle=o](10.972,4.976749)
\usefont{T1}{ptm}{m}{n}
\rput{0.9395907}(0.058443736,-0.17528975){\rput(10.71119,3.4636457){\small E}}
\usefont{T1}{ptm}{m}{n}
\rput{0.9395907}(0.09300857,-0.17725097){\rput(10.841815,5.5736456){C}}
\usefont{T1}{ptm}{m}{n}
\rput(11.146844,5.741749){\tiny pn}
\usefont{T1}{ptm}{m}{n}
\rput(11.0745,5.421749){\tiny E}
\pstriangle[linewidth=0.024,dimen=outer](0.472,4.636749)(0.8,1.0)
\rput{-180.0}(0.984,8.073498){\pstriangle[linewidth=0.024,dimen=outer](0.492,3.436749)(0.8,1.2)}
\psline[linewidth=0.024cm](0.132,4.496749)(0.852,4.496749)
\psdots[dotsize=0.16,dotangle=-90.0,dotstyle=otimes](0.472,5.716749)
\usefont{T1}{ptm}{m}{it}
\rput(1.5004375,5.751749){\footnotesize e-casing }
\usefont{T1}{ptm}{m}{it}
\rput(1.6009063,5.411749){\footnotesize of p-twain}
\psdots[dotsize=0.16,dotangle=70.32086](0.492,3.396749)
\psline[linewidth=0.024cm,linestyle=dashed,dash=0.16cm 0.16cm](0.101220064,4.631051)(2.1027799,5.082447)
\psdots[dotsize=0.16,fillstyle=solid,dotstyle=o](2.112,5.036749)
\psline[linewidth=0.04cm,linestyle=dotted,dotsep=0.06cm](2.035763,5.005274)(0.9082369,4.648224)
\psline[linewidth=0.032cm,arrowsize=0.05291667cm 2.0,arrowlength=1.4,arrowinset=0.4]{->}(1.492,5.256749)(1.712,5.076749)
\pstriangle[linewidth=0.024,linestyle=dashed,dash=0.16cm 0.16cm,dimen=outer](3.9320002,4.616749)(0.8,1.12)
\rput{-180.0}(7.824,8.073498){\pstriangle[linewidth=0.024,dimen=outer](3.9120002,3.436749)(0.8,1.2)}
\psline[linewidth=0.024cm](3.552,4.496749)(4.272,4.496749)
\psline[linewidth=0.024cm,linestyle=dashed,dash=0.16cm 0.16cm](4.312,4.6367493)(4.912,5.836749)
\psline[linewidth=0.024cm,linestyle=dashed,dash=0.16cm 0.16cm](4.912,5.836749)(3.512,4.6367493)
\usefont{T1}{ptm}{m}{it}
\rput(5.6704373,4.651749){\footnotesize ee-casing }
\psline[linewidth=0.024cm,linestyle=dashed,dash=0.16cm 0.16cm](4.312,4.6367493)(3.112,5.836749)
\psline[linewidth=0.024cm,linestyle=dashed,dash=0.16cm 0.16cm](3.112,5.836749)(3.512,4.6367493)
\psdots[dotsize=0.16,dotangle=-90.0,dotstyle=otimes](3.912,5.736749)
\psdots[dotsize=0.16,dotangle=-90.0,dotstyle=otimes](4.952,5.876749)
\usefont{T1}{ptm}{m}{it}
\rput(5.7009063,4.311749){\footnotesize of p-twain}
\rput{-58.291996}(-1.6308925,6.6888905){\psarc[linewidth=0.024,arrowsize=0.05291667cm 2.0,arrowlength=1.4,arrowinset=0.4]{->}(5.182,4.806749){0.51}{73.92642}{158.0744}}
\psdots[dotsize=0.16,dotangle=3.17983,dotstyle=otimes](3.072,5.896749)
\psdots[dotsize=0.16](3.912,3.436749)
\psline[linewidth=0.024cm](8.360697,3.4751618)(7.5905952,4.4786277)
\psline[linewidth=0.024cm](8.360697,3.4751618)(8.346695,4.7399955)
\rput{19.069193}(1.9431888,-2.350497){\psellipse[linewidth=0.024,dimen=outer](7.968645,4.6093116)(0.412,0.212)}
\usefont{T1}{ptm}{m}{n}
\rput(7.282781,4.826749){P}
\usefont{T1}{ptm}{m}{n}
\rput(7.3945,4.701749){\tiny E}
\psline[linewidth=0.024cm,linestyle=dashed,dash=0.16cm 0.16cm](8.352,4.736749)(9.292,5.536749)
\psline[linewidth=0.024cm,linestyle=dashed,dash=0.16cm 0.16cm](9.212,5.476749)(8.372,3.436749)
\usefont{T1}{ptm}{m}{it}
\rput(7.493094,5.431749){\footnotesize casing of}
\usefont{T1}{ptm}{m}{n}
\rput(8.282782,5.426749){P}
\usefont{T1}{ptm}{m}{n}
\rput(8.3945,5.301749){\tiny E}
\psline[linewidth=0.024cm,arrowsize=0.05291667cm 2.0,arrowlength=1.4,arrowinset=0.4]{->}(8.527454,5.3971896)(8.7765465,5.1963086)
\psdots[dotsize=0.16,dotangle=-90.0,dotstyle=otimes](9.212,5.516749)
\usefont{T1}{ptm}{m}{n}
\rput(7.9885626,3.576749){\small E}
\psdots[dotsize=0.16](8.372,3.496749)
\usefont{T1}{ptm}{m}{n}
\rput(1.0885625,-0.08325096){\small E}
\psline[linewidth=0.023999996cm,doubleline=true,doublesep=0.12,doublecolor=white](0.652,-2.4632509)(0.652,-3.803251)
\psdots[dotsize=0.16](0.652,-3.823251)
\psline[linewidth=0.024cm,linestyle=dashed,dash=0.16cm 0.16cm](0.732,-2.423251)(1.712,-1.963251)
\psline[linewidth=0.024cm,linestyle=dashed,dash=0.16cm 0.16cm](1.712,-1.963251)(0.752,-3.783251)
\psdots[dotsize=0.16,dotangle=-90.0,dotstyle=otimes](1.772,-1.943251)
\psline[linewidth=0.024cm](0.132,-2.843251)(1.172,-2.003251)
\psline[linewidth=0.024cm](0.012,-2.203251)(1.132,-2.623251)
\psdots[dotsize=0.16,fillstyle=solid,dotstyle=o](0.652,-2.443251)
\usefont{T1}{ptm}{m}{n}
\rput{0.9395907}(-0.06232828,-0.006383091){\rput(0.35118994,-3.8163543){\small E}}
\usefont{T1}{ptm}{m}{n}
\rput{0.9395907}(-0.030053819,-0.0090190815){\rput(0.52181494,-1.8463544){C}}
\usefont{T1}{ptm}{m}{n}
\rput(0.827,-1.6782509){\tiny hk}
\usefont{T1}{ptm}{m}{n}
\rput(0.7545,-1.998251){\tiny E}
\psline[linewidth=0.024cm,linestyle=dashed,dash=0.16cm 0.16cm](2.312,-3.803251)(2.752,-2.363251)
\psline[linewidth=0.024cm,linestyle=dashed,dash=0.16cm 0.16cm](2.752,-2.383251)(3.512,-3.163251)
\psline[linewidth=0.024cm](2.312,-3.803251)(3.512,-3.143251)
\psline[linewidth=0.024cm](3.512,-3.163251)(4.512,-3.803251)
\usefont{T1}{ptm}{m}{it}
\rput(3.4705937,-3.7432508){\small bridge}
\usefont{T1}{ptm}{m}{n}
\rput{0.9395907}(-0.06038616,-0.077924564){\rput(4.7146273,-3.7363544){\small E'}}
\usefont{T1}{ptm}{m}{n}
\rput{0.9395907}(-0.0614357,-0.03523857){\rput(2.1111898,-3.7763543){\small E}}
\psdots[dotsize=0.16,dotangle=-90.0,dotstyle=otimes](2.732,-2.363251)
\psdots[dotsize=0.16](2.332,-3.803251)
\psdots[dotsize=0.16,fillstyle=solid,dotstyle=o](3.492,-3.143251)
\psdots[dotsize=0.16](4.492,-3.803251)
\psline[linewidth=0.024cm](5.9286594,-3.7473743)(5.5145626,-2.5521657)
\psline[linewidth=0.024cm](5.9286594,-3.7473743)(6.314507,-2.5427492)
\rput{0.67443496}(-0.029575959,-0.06979552){\psellipse[linewidth=0.024,dimen=outer](5.9145346,-2.5474575)(0.412,0.212)}
\usefont{T1}{ptm}{m}{n}
\rput{0.27316308}(-0.009727639,-0.028191924){\rput(5.9024253,-2.0636768){P}}
\usefont{T1}{ptm}{m}{n}
\rput{0.27316308}(-0.0103265215,-0.028721401){\rput(6.0147457,-2.1881504){\tiny E}}
\psdots[dotsize=0.16,dotangle=-36.82971](5.933575,-3.768413)
\usefont{T1}{ptm}{m}{n}
\rput{0.9395907}(-0.060303785,-0.09328288){\rput(5.65119,-3.7363544){\small E}}
\psline[linewidth=0.024cm](7.152,-3.183251)(8.152,-3.823251)
\psdots[dotsize=0.16,fillstyle=solid,dotstyle=o](7.132,-3.163251)
\usefont{T1}{ptm}{m}{it}
\rput(7.090594,-3.803251){\small bridge}
\psline[linewidth=0.024cm,linestyle=dashed,dash=0.16cm 0.16cm](6.412,-2.4632509)(7.092,-3.103251)
\usefont{T1}{ptm}{m}{n}
\rput{0.9395907}(-0.059894037,-0.13794203){\rput(8.374627,-3.7363544){\small E'}}
\psdots[dotsize=0.16](8.112,-3.803251)
\psline[linewidth=0.023999996cm,doubleline=true,doublesep=0.12,doublecolor=white](9.732,-2.403251)(9.732,-3.7432508)
\psdots[dotsize=0.16](9.732,-3.763251)
\psline[linewidth=0.024cm](9.212,-2.783251)(10.252,-1.943251)
\psline[linewidth=0.024cm](9.092,-2.143251)(10.212,-2.563251)
\psdots[dotsize=0.16,fillstyle=solid,dotstyle=o](9.732,-2.383251)
\usefont{T1}{ptm}{m}{n}
\rput{0.9395907}(-0.027849032,-0.15790682){\rput(9.601815,-1.7863543){C}}
\usefont{T1}{ptm}{m}{n}
\rput(9.8345,-1.938251){\tiny E}
\psline[linewidth=0.024cm,linestyle=dashed,dash=0.16cm 0.16cm](9.872,-3.703251)(9.852,-2.343251)
\psline[linewidth=0.024cm](9.812,-3.763251)(11.012,-3.103251)
\psline[linewidth=0.024cm](11.012,-3.123251)(12.012,-3.763251)
\usefont{T1}{ptm}{m}{it}
\rput(10.970593,-3.703251){\small bridge}
\usefont{T1}{ptm}{m}{n}
\rput{0.9395907}(-0.059380397,-0.20058323){\rput(12.194628,-3.7363544){\small E'}}
\psdots[dotsize=0.16,fillstyle=solid,dotstyle=o](10.992,-3.103251)
\psdots[dotsize=0.16](11.992,-3.763251)
\psline[linewidth=0.024cm,linestyle=dashed,dash=0.16cm 0.16cm](9.832,-2.323251)(10.932,-3.043251)
\usefont{T1}{ptm}{m}{n}
\rput{0.9395907}(-0.05979822,-0.15494017){\rput(9.41119,-3.7363544){\small E}}
\psline[linewidth=0.024cm](1.4726223,-0.08489519)(0.42413107,0.6226826)
\psline[linewidth=0.024cm](1.4926223,-0.0648952)(1.0783759,1.1302617)
\rput{37.551476}(0.6816754,-0.27224696){\psellipse[linewidth=0.024,dimen=outer](0.7412535,0.8664721)(0.412,0.212)}
\psdots[dotsize=0.16,dotangle=-7.689285](1.5167608,-0.09752419)
\rput{-33.74507}(-0.30587023,1.5931488){\pstriangle[linewidth=0.024,dimen=outer](2.4734254,0.800812)(0.8,1.0)}
\rput{-213.32481}(3.6738877,-0.2989726){\pstriangle[linewidth=0.024,dimen=outer](1.881685,-0.19969015)(0.8,1.2)}
\psline[linewidth=0.024cm](1.8511575,1.0036348)(2.4405773,0.59013206)
\psdots[dotsize=0.16,dotangle=-125.051285,dotstyle=otimes](2.770152,1.7471085)
\psline[linewidth=0.024cm,linestyle=dashed,dash=0.16cm 0.16cm](1.552,-0.023250965)(1.112,1.1967491)
\psline[linewidth=0.024cm,linestyle=dashed,dash=0.16cm 0.16cm](1.8925388,1.095745)(1.1514612,1.097753)
\usefont{T1}{ptm}{m}{n}
\rput(0.22278126,1.186749){P}
\usefont{T1}{ptm}{m}{n}
\rput(0.3345,1.061749){\tiny E}
\usefont{T1}{ptm}{m}{it}
\rput(1.453875,2.056749){\small pyt-connective}
\psline[linewidth=0.024cm](8.048659,0.25262564)(7.6345625,1.4478341)
\psline[linewidth=0.024cm](8.048659,0.25262564)(8.434507,1.4572508)
\rput{0.67443496}(0.017654266,-0.09447258){\psellipse[linewidth=0.024,dimen=outer](8.034534,1.4525425)(0.412,0.212)}
\usefont{T1}{ptm}{m}{n}
\rput{0.27316308}(0.009366763,-0.03825373){\rput(8.022426,1.936323){P}}
\usefont{T1}{ptm}{m}{n}
\rput{0.27316308}(0.0087678805,-0.038783208){\rput(8.134746,1.8118497){\tiny E}}
\psdots[dotsize=0.16,dotangle=-36.82971](8.053575,0.23158689)
\usefont{T1}{ptm}{m}{n}
\rput{0.9395907}(0.0055741337,-0.12750925){\rput(7.77119,0.2636457){\small E}}
\usefont{T1}{ptm}{m}{n}
\rput{0.9395907}(0.0054891896,-0.15183732){\rput(9.254627,0.2436457){\small E'}}
\psdots[dotsize=0.16,dotangle=-90.0](9.012,0.19674903)
\usefont{T1}{ptm}{m}{it}
\rput(9.93825,-0.34325096){\small Hooks}
\psline[linewidth=0.024cm,linestyle=dashed,dash=0.16cm 0.16cm](8.092,0.21674904)(8.532,1.516749)
\psline[linewidth=0.024cm,linestyle=dashed,dash=0.16cm 0.16cm](8.972,0.23674904)(8.532,1.536749)
\psline[linewidth=0.023999996cm,doubleline=true,doublesep=0.12,doublecolor=white](10.932,1.316749)(10.932,-0.023250965)
\psdots[dotsize=0.16](10.932,-0.043250963)
\psline[linewidth=0.024cm](10.412,0.93674904)(11.452,1.776749)
\psline[linewidth=0.024cm](10.292,1.5767491)(11.412,1.156749)
\psdots[dotsize=0.16,fillstyle=solid,dotstyle=o](10.932,1.3367491)
\usefont{T1}{ptm}{m}{n}
\rput{0.9395907}(0.001364499,-0.17411783){\rput(10.61119,-0.016354306){\small E}}
\usefont{T1}{ptm}{m}{n}
\rput{0.9395907}(0.033313684,-0.17708448){\rput(10.801815,1.9336457){C}}
\usefont{T1}{ptm}{m}{n}
\rput(11.107,2.101749){\tiny hk}
\usefont{T1}{ptm}{m}{n}
\rput(11.0345,1.781749){\tiny E}
\psline[linewidth=0.024cm,linestyle=dashed,dash=0.16cm 0.16cm](11.072,0.016749036)(11.052,1.376749)
\psline[linewidth=0.024cm,linestyle=dashed,dash=0.16cm 0.16cm](11.052,1.376749)(11.972,0.036749035)
\usefont{T1}{ptm}{m}{n}
\rput{0.9395907}(0.001629033,-0.20106693){\rput(12.254627,-0.016354306){\small E'}}
\psdots[dotsize=0.16,dotangle=-90.0](12.012,-0.023250965)
\psline[linewidth=0.024cm](5.0926223,-0.4048952)(4.6783757,0.7902617)
\rput{-35.47871}(0.5699898,3.7034116){\pstriangle[linewidth=0.024,dimen=outer](6.0734253,0.460812)(0.8,1.0)}
\rput{-214.90776}(9.95023,-2.9678724){\pstriangle[linewidth=0.024,dimen=outer](5.441685,-0.51969016)(0.8,1.2)}
\psline[linewidth=0.024cm](5.4111576,0.66363484)(6.000577,0.25013208)
\psdots[dotsize=0.16,dotangle=-125.051285,dotstyle=otimes](6.390152,1.4071084)
\psline[linewidth=0.024cm,linestyle=dashed,dash=0.16cm 0.16cm](5.112,-0.34325096)(4.712,0.85674906)
\psline[linewidth=0.024cm,linestyle=dashed,dash=0.16cm 0.16cm](5.492538,0.7780061)(4.751462,0.775492)
\usefont{T1}{ptm}{m}{n}
\rput(4.6285625,-0.52325094){\small E}
\psline[linewidth=0.024cm](4.092,0.43674904)(5.132,1.276749)
\psline[linewidth=0.024cm](3.9976184,1.135331)(5.0663815,0.598167)
\psline[linewidth=0.024cm](4.5873613,0.7951585)(5.0166388,-0.44166043)
\psdots[dotsize=0.16](5.072,-0.42325097)
\usefont{T1}{ptm}{m}{n}
\rput{0.9395907}(0.022939488,-0.07188592){\rput(4.381815,1.3536457){C}}
\usefont{T1}{ptm}{m}{n}
\rput(4.686844,1.521749){\tiny pn}
\usefont{T1}{ptm}{m}{n}
\rput(4.6145,1.2017491){\tiny E}
\psdots[dotsize=0.16,fillstyle=solid,dotstyle=o](4.612,0.836749)
\usefont{T1}{ptm}{m}{it}
\rput(5.313875,2.056749){\small pt-connective}
\psdots[dotsize=0.216,dotstyle=otimes](1.552,-4.983251)
\usefont{T1}{ptm}{m}{it}
\rput(2.8560624,-4.953251){- center vertex}
\psdots[dotsize=0.18](5.252,-4.983251)
\usefont{T1}{ptm}{m}{it}
\rput(6.4060626,-4.953251){- end vertex}
\psdots[dotsize=0.18,fillstyle=solid,dotstyle=o](8.612,-4.983251)
\usefont{T1}{ptm}{m}{it}
\rput(9.936063,-4.953251){- mixing vertex}
\psdots[dotsize=0.16,fillstyle=solid,dotstyle=o](6.372,-2.4632509)
\psdots[dotsize=0.16,fillstyle=solid,dotstyle=o](8.472,1.536749)
\psdots[dotsize=0.16,fillstyle=solid,dotstyle=o](1.092,1.116749)
\psdots[dotsize=0.16,fillstyle=solid,dotstyle=o](8.352,4.776749)
\usefont{T1}{ptm}{m}{n}
\rput(6.142625,-5.793251){Dashed and dotted lines show tees of casings and connectives}
\end{pspicture} 
 }
 \begin {center} Fig. 6. Casings and connectives. \end{center}
 
\noindent A graphical summary of various kinds of casings and connectives is given in Fig. 6.   


\subsection{The card of a tee-cluster.}\label{tee-cards}  

The above analysis 
revealed basic structural units of which all tee-clusters are made. 
On this ground we can now describe all tee-clusters.  

First of all, observe that a tee-cluster $\Phi$ is naturally 
divided into two parts, $\Phi_h$ and $\Phi_{end}$. The family $\Phi_{end}$ 
is composed of all tees $\theta\in\Phi$ one
of whose vertices is an end vertex of  $\Phi$. It is easy to see that 
$\Phi_{end}=\Phi\setminus\Phi_h$ and $\Phi_{mp}\subset\Phi_{end}$. So, 
$\Phi=\Phi_h\cup\Phi_{end}$, $\Phi_h\cap\Phi_{end}=\emptyset$ and  
$S(\Phi_h)\cap S(\Phi_{end})$ consists of all center vertices of $\Phi$. 
As motivated by proposition\,\ref{coax-str} below, $\Phi_h$ (resp., 
$\Phi_{end}$) will be called the \emph{semi-simple} (resp., 
\emph{solvable}) \emph{part} of $\Phi$. 

The data characterizing  a  compatible tee-family $\Phi$ are: $n_c$=(the 
number of center vertices),  $n_e$=(the number of end vertices), 
$n_{tr}$=(the number of triangles), $t_E$=(the number of p-twains rooted at 
the end vertex $E$),  $p_E$=(the dimension of $P_E$) and $b_{E,D}$=(the 
number of bridges connecting end vertices $E$ and $D$). The numbers $t_E$ 
and $p_E$  will be called the \emph{twain} and \emph{pyramid numbers} at 
$E$, respectively, and $b_{E,D}$ the \emph{bridge number} at $(E,D)$. Each 
of these numbers is a nonnegative integer. Since the dimension of a true 
pyramid is greater then $1$, the value $p_E=1$ requires a comment. Namely, 
it is interpreted as existence of the vertex $C_E\in S(\Phi)$, the common 
center of all connectives and hooks rooted at $E$. Informally speaking, 
$p_E=1$ refers to the "collapsed" pyramid $P_E$, that is the 
``rod" $\overline{E\,C_E}$.

All above numbers forms the \emph{card} of $\Phi$, which will be 
denoted by $\mathcal{C}(\Phi)$. More precisely, numerate 
end vertices of $\Phi$ and put $t_i=t_E, \,p_i=p_E$ if $E$ is the $i$-th 
end vertex and, similarly, $b_{ij}=b_{E,D}$ if $D$ is the $j$-th end 
vertex. Thereby we have the \emph{twain vector} 
$\textbf{t}=(t_1,...,t_{n_e})$, the \emph{pyramid vector} 
$\textbf{p}=(p_1,...,p_{n_e})$ and the \emph{bridge matrix}  
$\textbf{B}=\|b_{ij}\|$. A renumbering of end vertices corresponds to a 
simultaneous permutation of components of these vectors and the matrix. 
So, the triple $(\textbf{t},\textbf{p},\textbf{B})$ will be considered as a 
representative of the corresponding equivalence class  
$[\textbf{t},\textbf{p},\textbf{B}]$ modulo these permutations. Thus 
\begin{equation}\label{id-card}
\mathcal{C}(\Phi)=(n_c,n_e,n_{tr},[\textbf{t},\textbf{p},\textbf{B}]).
\end{equation}
\begin{proc}\label{ID-CARD}
Two tee-clusters are equivalent if and only if their  
cards are equal. 
\end{proc}
\begin{proof}
First, observe that if the layout of center and end vertices, triangles, 
p-twains, pyramids and bridges of a tee-cluster $\Phi$ is known, 
then $\Phi$ is automatically and uniquely restored just by adding to these 
data all possible connectives, hooks and the casing. 

So, it suffices to show that  if $\Phi$ and $\Phi^{\prime}$ are 
tee-clusters with equal cards, then there exists a 
one-to-one correspondence $\zeta:S(\Phi)\rightarrow S(\Phi^{\prime})$ which identifies  
center and end vertices of these clusters as well as their triangles, p-twains, 
pyramids (including ``collapsed") and bridges. We shall construct such a map 
gradually by starting from a map $\zeta_1:S(\Phi_h)\rightarrow 
S(\Phi_h^{\prime})$ which establishes an equivalence of 
$(n_{tr},\,n_c)$-hedgehogs $\Phi_h$ and $\Phi_h^{\prime}$. After that we 
shall extend $\zeta_1$ to a biunique correspondence $\zeta_2$ of end vertices of  
$\Phi$ and $\Phi^{\prime}$ in such a way that 
$t_E(\Phi)=t_{\zeta_2(E)}(\Phi^{\prime}), 
p_E(\Phi)=p_{\zeta_2(E)}(\Phi^{\prime})$ and 
$b_{E,D}(\Phi)=B_{\zeta_2(E),\zeta_2(D)}(\Phi^{\prime})$ for all end 
vertices of $\Phi$. This is , obviously, possible. Since $p_E=p_{\zeta_2(E)}$, 
there is a bijection between bottom vertices of $P_E$ and $P_{\zeta_2(E)}$. 
This way $\zeta_2$ is extended to pyramid, and we  shall proceed on 
similarly for p-twains and bridges.
\end{proof}

\noindent\textit{Abstract cards.} 

\noindent Proposition\,\ref{ID-CARD} reduces the 
classification of tee-clusters to description of their  
cards. Namely, an \emph{abstract} card is an ordered set of the form 
$(k,l,m, [\textbf{t},\textbf{p},\textbf{B}])$ where $k,l,m\in \mathbb{N}_0$, 
$\textbf{t},\,\textbf{p}\in\mathbb{N}_0^l$ and $\textbf{B}$ is a symmetric 
$l\times l$-matrix with entries in $\mathbb{N}_0$ and zero diagonal 
elements. As earlier, $[\textbf{t},\emph{\textbf{p}},\textbf{B}]$ stands for 
the orbit of the triple  $(\textbf{t},\emph{\textbf{p}},\textbf{B})$ under 
a natural action of the symmetric group $S_l$. The number 
$$
\mathrm{card}\,\mathcal{J}=k+l+3m+\sum_{i=1}^{l}p_i+
2\sum_{i=1}^{l}t_i+\sum_{1\leq i<j\leq l}b_{ij}
$$
will be called the \emph{dimension} of $\mathcal{J}$. If 
$\mathcal{J}=\mathcal{C}(\Phi)$, then 
$\mathrm{card}\,\mathcal{J}=\mathrm{dim}\,\Phi$. 

Obviously, the card of a tee-cluster is an abstract card 
but the converse is not true. So, the problem is to find exact conditions 
that distinguish cards of tee-clusters among other abstract 
cards. To this end the notion of a \emph{realization} of an abstract 
card will be useful. Namely, choose among base vectors the 
following disjoint groups: 

\begin{eqnarray}\label{abs-verticies}
\{C_1,\dots,C_k \}, \quad  \{E_1,\dots,E_l\}, & \quad 
\{T_{i1},T_{i2},T_{i3}\}, \;1\leq i\leq m, \nonumber \\
\{W_{j1}^i,W_{j2}^i\}, \;1\leq i\leq l, 
\;1\leq j\leq t_i, & \, \{P_{ir}\}, \;1\leq i\leq l, \,1\leq r\leq p_i,\\ 
\quad \{C_{ij}^s\}, \;1\leq i<j\leq l,& 1\leq s\leq b_{ij} \nonumber 
\end{eqnarray}   

Vectors $C_i$'s (resp., $E_i$'s) will be  called \emph{declared} center 
(resp., end) vertices.  Similarly, $T_{iq}, \,1\leq q\leq3,$ are vertices 
of the $i$-th  \emph{declared} triangle,  $\{W_{j1}^i,W_{j2}^i\}$ are 
bottom vertices of the $j$-th \emph{declared} p-twain  rooted at $E_i$,  
$P_{ir}, \,1\leq r\leq p_i,$ are bottom vertices of the  \emph{declared} 
pyramid   rooted at $E_i$ and $C_{ij}^s$ is the center vertex of the  $s$-th 
\emph{declared} bridge connecting $E_i$ and $E_j$. Then we 
shall consider the corresponding declared triangles, p-twains, pyramids and 
bridges by adding to them all possible \emph{declared} connectives and the 
casing. For instance, an declared \emph{pyb}-connective rooted at an 
declared end vertex $E_i$ is of the form $\lfloor 
E_i,W_{jq}^i|P_{ir}\rceil$, etc. The so-constructed family, which is, 
obviously, compatible, will be called a \emph{realization} of $\mathcal{J}$ and  
denoted $\Phi_{\mathcal{J}}$. Two realizations of a given abstract card 
are, obviously, equivalent. Also, if $\mathcal{J}=\mathcal{C}(\Phi)$, 
then, as it is easy to see, $\Phi_{\mathcal{J}}=\Phi$. Hence the above 
problem can be reformulated as: 
\begin{center}
\emph{ what are abstract cards $\mathcal{J}$ 
such that $\Phi_{\mathcal{J}}$  a tee-cluster}.
\end{center}

The following necessary conditions are on the surface. First, the graph of 
$\Phi_{\mathcal{J}}$ must be connected and, second, 
$\mathcal{J}$ must be equal to $\mathcal{C}(\Phi_{\mathcal{J}})$, i.e., that
the \emph{declared 
parameters}  must coincide with \emph{actual} ones.  For instance, if the 
declared parameters are $k=0, \,l=m=p_1=t_1=1$, i.e., 
$\mathcal{J}=(0,1,1,[(1),(1),(0)])$, then $\Phi_{\mathcal{J}}$ consists of 
two rooted at the same vertex twains and one disjoint from them triangle. 
So, the graph of $\Phi_{\mathcal{J}}$ is disconnected and, moreover, none 
of these two twains can be distinguished as a p-twain, i.e., the declared 
value $t_1=1$ differs from the actual. If the declared parameters are $k=0, 
l=m=1, p_1=t_1=0,$ then $\Phi_{\mathcal{J}}$ is a triangle. So, in this case, 
the realization does not have the declared end vertex $E_1$. 

These necessary conditions (resp., an satisfying them abstract card) will be 
called \emph{consistency conditions} (resp., a \emph{consistent card}). 

In subsections \ref{TypeVertex} - \ref{c-str-all} we have established the role 
of various kinds of tees in the construction of a tee-cluster. 
Now it is convenient to bring  together the obtained results in order to ease  
further discussion of consistency conditions. \newline

\begin{center}
\begin{tabular}{| l | | c |}
\multicolumn{2}{ c }{DISTRIBUTION OF TEES IN A TEE-CLUSTER} \\
\hline
{\bf type} & {\bf } \\
\hline\hline
{\it ece-} & belongs to a triangle \\
\hline 
{\it ec-} & belongs to a p-, or s-twain, or is a connective    \\
\hline
{\it ee-} & belongs to a hedgehog but not to a traingle,\\  &or to the 
\emph{ee}--casing of a p-twain\\
\hline
{\it c-} & is a bridge \\
\hline
{\it e-} & belongs to a casing different from \emph{ee}--type\\
\hline
 0 - & belongs to a multiped \\
\hline
\end{tabular}
\end{center}

\medskip
\noindent This table will be referred as the \emph{DT-table}.

\medskip
\noindent\textit{Join operations.} 

\noindent In order to explicitly describe consistent 
cards we need the following four operations with 
tee-clusters. Below $\Phi$ stands for a tee-cluster. 
\newline\newline
\emph{Joining a triangle}. Assumption: $n_c\neq 0$. Include the 
$(n_{tr},n_c)$-hedgehog $\Phi_h$ into an $(n_{tr}+1,n_c)$-hedgehog 
$\bar{\Phi}_h$ by adding three new vertices to $S(\Phi)$. 
The new tee-cluster is $\bar{\Phi}_h\cup \Phi_{end}$.
\newline\newline
\emph{Joining  a p-twain}. Assumption: $n_c\neq 0, n_e\neq 0$. Let $E$ be 
an end vertex of $\Phi$ and $B_1,B_2\notin S(\Phi)$. First, add to $\Phi$ the twain 
$\bigwedge=\{\lfloor E,B_1|B_2\rceil,\,\lfloor E,D_2|B_1\rceil\}$ and all 
tees of the form $\lfloor B_1,B_2|C\rceil$ with $C$ 
being a center vertex of $\Phi$. Then add all the connectives and 
casings to the so-obtained family. The resulting compatible tee-family will be 
denoted by $\Phi_{\bigwedge,E}$, or, simply, $\Phi_{\bigwedge}$.  
\newline\newline
\emph{Joining  a pyramid}. Assumption: $p_E\neq 0$. Let $E$ be an end 
vertex of $\Phi$ and $B_1\dots,B_r$ base vectors not belonging 
to $S(\Phi)$. We assume that $r\geq 1$, if at least one of families 
$P_E,  \,\bot_E^{pn},  \,\bot_E^{pt}$ is nonempty (equivalently, $p_E >0$), and $r>1$
otherwise. Consider the pyramid $\mathbf{\nabla}$ whose top vertex is 
$E$ and bottom vertices are that of $P_E$, if $P_E\neq\emptyset$, or $C_E$,
if $P_E=\emptyset$ and $C_E$ exists, and
$B_1\dots,B_r$. Now we get a new compatible tee-family by
adding all new connectives, hooks and casings to $\Phi\cup\mathbf{\nabla}$.

If $p_E=0$, then the ``collapsed pyramid", i.e., a new vertex interpreted as 
$C_E$, can be created, assuming that $n_{c}>0$ and $\Phi$ contains either
or both a rooted at $E$ trey and a bridge with one end at $E$. Namely,
first, we add the tees $\lceil E,A|C_E\rceil$ and $\lceil E,C_E|E^{\prime}\rceil$ to $\Phi$ 
where $A$ is the center of a bridge, or a side vertex of a rooted at $R$ trey, and
$E^{\prime}$ is a center vertex of $\Phi$. Then we complete the so-obtained  tee-family
by adding to it all possible new connectives, hooks and casings.
\newline\newline

\emph{Joining a bridge}. Assumption: $n_e\geq 2$. First, add to $\Phi$ the tee 
$\theta=\lfloor E_1,E_2|C\rceil$ where $E_1$ and $E_2$ are end vertices of $\Phi$ 
and $C\notin S(\Phi)$. Then add to $\Phi\cup\{\theta\}$ all tees of the form
$\lfloor E_i,C|A\rceil$ where $A$ is  a bottom vertex of $P_{E_i}$, or $C_{E_i}$, or
a center vertex of $\Phi$. 

Observe that these join operations commute, preserve both end and center 
vertices of the original tee-cluster and do not create new ones. 

An end (resp., center) vertex of a compatible tee-family $\Psi$ remaining be
such in any containing $\Psi$ compatible tee-family will be 
called \emph{stable}. The following assertion is obvious.
\begin{lem}\label{stable} We have:
\begin{enumerate}
\item An end vertex $E$ of $\Psi$ is stable if there are at least three rooted at 
$E$ tees with mutually different second ends. 
\item The center of a cross or a tripod belonging to  $\Psi$ is a stable center vertex
of $\Psi$. $\square$
\end{enumerate}
\end{lem}
\begin{proc}\label{join-oper}
If the ends of a tee-cluster $\Phi$ are stable, then the result of 
any of the above joining procedures is a  tee-cluster. 
\end{proc}
\begin{proof}
For triangles the assertion is obvious. In order to prove that  
$\Phi_{\bigwedge}$ is a tee-cluster we have to show that any 
compatible with $\Phi_{\bigwedge}$ tee $\theta$ at least one of 
whose vertices is $B_i, \,i=1,2$, and others are in $S(\Phi_{\bigwedge})$ 
belongs to $\Phi_{\bigwedge}$. But such a tee,  due to stability of end vertices 
of $\Phi$, is either of the form 
$\lfloor B_1,B_2|C\rceil$ with $C$ being a center vertex of $\Phi$
or of the form $\lfloor E,B_i|Z\rceil$ where $E$ is the top of $\wedge$ and $Z$ 
is a center/mixing vertex of $\Phi$. In the first case $\theta$ belongs 
to $\Phi_{\bigwedge}$ by construction as well as in the case when $Z$ is a
center vertex $\Phi$. If $Z$ is mixing, then it may be a bottom vertex of a 
p-twain, or of a pyramid $P_D$, or $C_D$, or the center of a bridge in $\Phi$ 
as it follows from the description of mixing vertices of a tee-cluster.The first and 
the fourth of these possibilities are manifestly impossible. For the rest,  
compatibility with $\Phi$ conditions imply that 
$\theta$ must be rooted at $D$ and, therefore, that $E=D$. In other words, $Z$
is a bottom vertex of $P_E$, or $C_E$ and hence, by construction, 
$\lfloor E,B_i|Z\rceil\in\Phi$.

Similar arguments together with   
DT-table prove the remaining two assertions.    
\end{proof} 
\begin{cor}\label{generic ID}
Let $\mathcal{J}=(k,l,m,[\textbf{t},\textbf{p,\textbf{B}}])$ be an abstract 
card. If $k\geq 1$ and $l\geq 4$, then $\Phi_{\mathcal{J}}$ is a 
tee-cluster. 
\end{cor}
\begin{proof}
Consider the contained in $\Phi_{\mathcal{J}}$ $(m,l|k)$-hybrid $\Psi$. It 
is a tee-cluster. Since $k\geq 1$, any center vertex of $\Psi$ contains a tripod
and hence is stable. Also, since $l\geq 4$, at least three 0-tees rooted at an 
end vertex of $\Psi$ have different second end vertices. By this reason end
vertices of $\Psi$ are stable too. Now it remains to observe that 
$\Phi_{\mathcal{J}}$ is obtained from $\Psi$ by a series of join
operations and apply  proposition\,\ref{join-oper}
\end{proof}

\medskip
\subsection{Exceptional cards.}\label{except-cards} 

Corollary\,\ref{generic ID} shows that 
nontrivial consistency conditions may occur only if $k=0$ (case I), or if 
$k>0, l<4$ (case I\!I). Consider them separately by anticipating the 
following evident facts (see subsection\,\ref{c-str-all} for the notation):
\begin{lem} \label{end-center}
Let $\Psi$ be a compatible tee-family, $E,\,E^{\prime}$  end 
vertices of $\Psi$ and $\theta\in\Psi$. Then
\begin{enumerate}
\item if the center of $\theta$ is in $S(\Psi_E)$, then $\theta$ is rooted in $E$;
\item if  the ends of $\theta$ are in $S(\Psi_E)$ and $S(\Psi_{E^{\prime}})$, 
respectively, then they coincide with $E$ and $E^{\prime}$. 
\end{enumerate}
\end{lem}

\textbf{Case I:} \,$\mathbf{k=0.}$ If $\Phi$ is a tee-cluster, 
then $n_c=0$ implies that $n_{tr}\leq 1, \,\textbf{t}=0, ,p_i\neq 1, \,\forall 
i$. Indeed, all triangles of $\Phi$ belong to the hedgehog $\Phi_h$ which 
has  at least one thorn, if $n_{tr}\geq 2$. Also, existence of p-twains and
``collapsed pyramids" in $\Phi$ presumes (see lemma\,\ref{nec-str},\,(6),\,(7)
existence of center vertices in $\Phi$.
So, the consistency conditions in this case are: 
$k=0\Rightarrow m\leq 1, \,\textbf{t}=0, p_i\neq 1, \,\forall i$. 

If, moreover, $n_{tr}=1$, then   $\Phi_{end}=\emptyset$, since 
$S(\Phi_h)\cap S(\Phi_{end})$ consists of center vertices of $\Phi$. In 
other words,  $n_c=0\Rightarrow n_e=0$ and the corresponding 
consistency condition is: $k=0, \,m=1\Rightarrow l=0$ and hence 
$\textbf{t}=\textbf{p}=0, \,\textbf{B}=0$, i.e., 
$\mathcal{J}=(0,0,1,[\textbf{0},\textbf{0},\textbf{0}])$ and 
$\Phi_{\mathcal{J}}$ is a triangle. 

If, on the contrary, $n_{tr}=0$, then $n_e\neq 0$ (see DT-table). To analyze 
this case denote by $n_{e,0}$ (resp., $l_0$) the number of end vertices 
$E$ of $\Phi$ for which $P_E=\emptyset$ (resp., the number of components 
$p_i$ of $\textbf{p}$ in an abstract card, which are equal to zero). Consider 
cases $n_{e,0}=0$ and $n_{e,0}\neq 0$ separately. 

If $n_{e,0}=0$ and $n_e=1$, then, obviously, $\Phi=P_E$ 
where $E$ is the unique end vertex of $\Phi$. But $P_E$ is a 
tee-cluster iff its dimension is greater than 2. So, the corresponding 
card is $(0,1,0,[(0),(p),\mathbf{0}])$ with $p\geq 3$. 

If $l>1$,  then the family $\mathbf{St}_l=\Phi_{\mathcal{J}}$ for 
$\mathcal{J}=(0,l,0,[(0), (2,\dots,2),\mathbf{0})$ is a tee-cluster. Indeed, 
by definition, $\mathbf{St}_l$ consists of mutually disjoint s-twains
$\wedge_1,\dots,\wedge_l$ rooted at some vertices $E_1,\dots,E_l$,
and all hooks of the form $\lfloor E_i,E_j|B_i\rceil\in\Phi, \, i\neq j$,
with $B_i$ being a bottom vertex of $\wedge_i$ (see also DT-table).
$E_1,\dots,E_l$ are end vertices of $\mathbf{St}_l$. By lemma\,\ref{stable}, 
they are stable if $l>1$. So, if a tee $\theta$ is compatible with $\mathbf{St}_l$ and
its vertices belong to $S(\mathbf{St}_l)$, then the center of $\theta$ is
a bottom vertex a twain $\wedge_i$. It implies that $\theta$ is a rooted at $E_i$
hook and hence belong to $\mathbf{St}_l$.

Now, by appropriately  joining pyramids and bridges to the 
tee-cluster $\mathbf{St}_l$, we can construct a 
tee-cluster with arbitrary card of the form $(0,l, 
0,[(0),(p_1,\dots,p_l),\mathbf{B})$,  $l>1, \,p_i\geq 2, \forall 
i$ (see proposition\,\ref{join-oper}).   

Assume now that $l_0\neq 0, \,\mathcal{J}=(0,l,0,[(0), 
(p_1,\dots,p_l),\mathbf{B}])$ and the end vertices $E_1,\dots,E_{l}$ of 
$\Phi_{\mathcal{J}}$ are numbered in such a way that $p_i=0$, if 
$i\leq l_0$, and $p_i\geq 2$, if $i>l_0$. If $l-l_0=1$, then $\theta=\lfloor 
B,E_l,|E_1\rceil$ with $B$ being a bottom vertex of $P_{E_l}$ is compatible 
with  $\Phi_{\mathcal{J}}$ but does not belong to $\Phi_{\mathcal{J}}$. So, 
$\Phi_{\mathcal{J}}$ is not a tee-cluster. If $l-l_0=2$, then
$l\geq 3$ and $\lfloor E_{l-1},E_l|E_1\rceil\notin\Phi_{\mathcal{J}}$ 
is compatible with  $\Phi_{\mathcal{J}}$. So, $\Phi_{\mathcal{J}}$ is not a 
tee-cluster in this case too. On the contrary, all end vertices of 
$\Phi_{\mathcal{J}}$ are stable if $l-l_0\geq 3$. Moreover, similar 
arguments as above show that in this case $\Phi_{\mathcal{J}}$ a 
tee-cluster. In other words, $l-l_0\geq 3$ is the consistency 
condition in the case when $k=m=0, l_0\neq 0$. 

Thus cards of tee-clusters without center vertices are:
\begin{eqnarray}\label{nc-0-cards}
(0,0,1,[(0),(0),\mathbf{0}]) &\;(\mathrm{triangle}), \; \nonumber\\
(0,1,0,[(0),(p),\mathbf{0}]), &\; p\geq 3 \;(p-\mathrm{pyramid}) \nonumber \\
(0,l,0,[(0),(p_1\dots,p_l),\mathbf{B}]),& \quad l>1, \quad p_i\geq 2, \quad 1\leq i\leq l; \\
(0,l,0,[(0),(0,\dots,0,p_{l_0+1}, \dots,p_l), \mathbf{B}]), &\quad l-l_0\geq 3, 
\quad p_i\geq 2, \quad i\geq l_0.\nonumber
\end{eqnarray}

Denote by $\mathbf{O}_s\mathbf{St}_r$ a tee-cluster whose 
card is $(0,r+s,0,[(0),(0,\dots,0,2\dots,2),\mathbf{0}])$ with 
$l=r+s$ and $l_0=s$ and by $\mathbf{Pr_k}$ the $k$-dimensional pyramid. 
Then all tee-clusters from the above list, except the triangle, 
are obtained from $\mathbf{Pr_3}, \;\mathbf{St_l}, \,l>1,$ and  
$\mathbf{O}_s\mathbf{St}_r, \,s>0, \,r>2$,  by   
joining to them pyramids and bridges. \newline

\textbf{Case I\!I:} \,$\mathbf{k>0, \,0\leq l\leq 3}.$  This case is 
subdivided into four subcases, $\mathrm{I\!I}_0,...,\mathrm{I\!I}_3$, according 
to the value of $l$. 

$\mathrm{I\!I}_0.$ If $l=0\Leftrightarrow\Phi_{end}=\emptyset$, then 
$\Phi=\Phi_h$ is an $(m,k)$-hedgehog.

$\mathrm{I\!I}_1.$  If $\mathcal{J}=(k,1,0,[(1),(1),\mathbf{0}])$,
then $\Phi_{\mathcal{J}}$  is a tee-cluster. This is easily verified
by a direct check using basic properties of treys. Denote the class 
of equivalent to it tee-clusters by $\mathbf{Pr_1}\mathbf{Pt}^k$. 
Ends and center vertices of such a  cluster 
are stable. Realizations of cards  $(k,1,0,[(r),(s), 
\mathbf{0}]), \;r\geq 1, \;s\geq 1,$ are obtained by  joining 
p-twains and pyramids to $\mathbf{Pr_1}\mathbf{Pt}^k$. By  
proposition\,\ref{join-oper} they are tee-clusters. So, within the 
considered case it remains to check consistence of abstract
cards with  $(t)=(0)$ and with $(p)=(0)$. 

In the first of these cases $\mathcal{J}=(k,1,m,[(0),(p),\mathbf{0}])$ and 
$(\Phi_{\mathcal{J}})_{end}$ is the $p$-pyramid together with tees of the 
form $\lfloor E,B|C_i\rceil, \,i=1,\dots,k,$ where $E=E_1$ and $B$ is a 
bottom vertex of $P_E$ (see (\ref{abs-verticies})). If $m=0$, then 
$\Phi_{\mathcal{J}}=(\Phi_{\mathcal{J}})_{end}$ belongs to the rooted at $E$ 
pyramid whose bottom vertices are those of $P_E$ and also declared center 
vertices $C_i$. Hence it is not a tee-cluster. Also, $\Phi_{\mathcal{J}}$ is not
a cluster if $m>0, p=2$. Indeed, in this case the tee
$\lfloor B_1,B_2|E\rceil\notin\Phi_{\mathcal{J}}$ with $B_i$'s being the bottom 
vertices of the twain $P_E$ is compatible with $\Phi_{\mathcal{J}}$. On the
contrary, if $m>0, p\geq 2$, then end and center vertices of $\Phi_{\mathcal{J}}$
are stable and it is a tee-cluster.

In the second case $\mathcal{J}=(k,1,m,[(t),(0),\mathbf{0}])$ and 
$\Phi_{\mathcal{J}}$ is a tee-cluster iff $t\geq 2$. It is easily follows
from the basic property of treys.

Thus the list 
of cards of tee-clusters in the considered case is: 
\begin{eqnarray}\label{list-k1}
(k,1,m,[(t),(p),\mathbf{0}]), \quad k>0, & \; m\geq 0, \; t\geq 1, \; p\geq 1;\nonumber\\
(k,1,m,[(0),(p),\mathbf{0}]), & \; k>0, \; m> 0, \; p\geq 3; \\
(k,1,m,[(t),(0),\mathbf{0}]), & \; k>0, \; m\geq 0, \; t\geq 2. \nonumber
\end{eqnarray}
 
$\mathrm{I\!I}_2.$  Let $\mathcal{J}=(k,2,0,[(0),(0),\mathbf{B}])$  with the
$2\times 2$ bridge matrix $\mathbf{B}=\|b_{ij}\|, \,b_{12}=2$. Put
$\mathbf{Br}_k=\Phi_{\mathcal{J}}$.  By definition, $\mathbf{Br}_k$ contains
a $(2,k)-multiped$. The casing of each of two bridges that are contained in 
$\mathbf{Br}_k$ consists of $2k$ tees. By using lemmas\,\ref{end-center}
and \ref{stable},
it is not difficult to verify that $\mathbf{Br}_k$ is a tee-cluster with stable 
end and center vertices. Now, by using the join operations and 
proposition\,\ref{join-oper}, we see that realizations of abstract cards of 
the form  $(k,2,m,[(t_1,t_2),(p_1,p_2),\mathbf{B}]$ with $b_{12}\geq 
2$ are tee-clusters.

Similarly, a direct check shows that a realization of 
$(k,2,0,[(1,1),(0,0),\mathbf{0}])$ is a tee-cluster with stable end and center
vertices. Hence, by proposition\,\ref{join-oper}, realizations of absrtact cards
of the form $(k,2,m|[(t_1,t_2),(p_1,p_2),\mathbf{B}])$ such that $t_1t_2\geq 1$ are
tee-clusters. So, within the considered case we have to analyze abstract cards 
with $b_{12}\leq 1, t_1t_2=0$.

First, assume that $\mathbf{t}=(t_1,0), \,t_1\geq 1$. If 
$\mathcal{J}=(k,2,0,[(t_1,0),(p_1,p_2),\mathbf{B}])$ with $p_2<2, 
\,b_{12}\leq1$, then $\Phi_{\mathcal{J}}$ is not a tee-cluster. Indeed, 
observe that in this case $C_{E_2}$ does not exist so that 
$p_2\neq 1$ and hence $p_2=0$. If $b_{12}=0$, then the tee
$\lfloor E_1, B|E_2\rceil$ with $B$ being a bottom vertex of a rooted at 
$E_1$ p-twain in $\Phi_{\mathcal{J}}$ is compatible with  $\Phi_{\mathcal{J}}$
but does not belong to it.  If $b_{12}=1$, then such is the tee
$\lfloor E_1, C|E_2\rceil$ with $C$ being the center of the unique 
connecting $E_1$ and $E_2$ bridge in $\Phi_{\mathcal{J}}$.
This proves that $p_2\geq 2$.

On the other hand, if $\mathcal{J}=(k,2,0,[(1,0),(0,2),\mathbf{B}]), 
b_{12}\leq 1,$, then $\Phi_{\mathcal{J}}$ is a tee-cluster with stable 
end and center vertices. Now proposition\,\ref{join-oper} shows 
that realizations of  abstract cards $(k,2,m,[(t_1,0),(p_1,p_2),\mathbf{B}])$
with $k\geq 1, \,t_1\geq 1, \,p_2\geq 2$ are tee-clusters. 

Now it remains to 
examine abstract cards with $\mathbf{t}=(0,0), b_{12}\leq 1$. 
First of all, observe that a realization of the abstract card
$(k,2,0,[(0,0),(2,2),\mathbf{0}])$ is a  tee-cluster 
with stable end and center vertices. By the same arguments as earlier this 
implies that realizations of cards $(k,2,m,[(0,0),(p_1,p_2),\mathbf{B}])$ with 
$p_i\geq 2, i=1,2,$ are tee-clusters too. So, the next step is to examine 
realization of abstract cards $\mathcal{J}=(k,2,m,[(0,0),(p_1,p_2),\mathbf{B}])$
with $p_1\geq 2, \,p_2\leq 1$. Since $t_2=0$, equality $p_2=1$, i.e., 
existence of the vertex $C_{E_2}$ in $\Phi_{\mathcal{J}}$, is possible, iff 
$b_{12}=1$. But, as earlier, the  tee 
$\lfloor E_1, C|E_2\rceil\notin\Phi_{\mathcal{J}}$ with $C$
being the center of the unique bridge in $\Phi_{\mathcal{J}}$ is compatible
with $\Phi_{\mathcal{J}}$. So, in this case $\Phi_{\mathcal{J}}$ is not a tee-cluster.
If $p_2=0$, then the tee $\lfloor E_1, B|E_2\rceil\notin\Phi_{\mathcal{J}}$ 
with $B$ being a bottom vertex of $P_{E_1}$ is compatible with 
$\Phi_{\mathcal{J}}$, so that $\Phi_{\mathcal{J}}$ is not a tee-cluster 
in this case too.

The same arguments show that realizations of abstract cards with 
$p_i\leq 1, \,i=1,2,$ and $b_{12}=1$ are not tee-clusters. If $b_{12}=0$,
then $p_1=p_2=0$, as we have observed  earlier. In this case $m=0$ is,
obviously, impossible, while realization of the abstract card
$(k,2,m,[(0,0),(0,0),\mathbf{0}]), \,m>0,$ is a $(2,k|m)$-hybrid. 
A peculiarity of this cluster is that its ends are not stable. 

Thus cards of  tee-clusters in the considered case are: 
\begin{eqnarray}\label{list-k2}
(k,2,m,[(t_1,t_2),(p_1,p_2),\mathbf{B}]), &\quad b_{12}\geq 2; \nonumber\\
(k,2,m,[(t_1,t_2),(p_1,p_2),\mathbf{B}]),  & t_1t_2\neq 0;\nonumber\\
(k,2,m,[(t_1,0),(p_1,p_2),\mathbf{B}]), &\quad t_1\geq 1, 
\quad p_2\geq 2, \quad b_{12}\leq 1; \quad \\
(k,2,m,[(0,0),(p_1,p_2),\mathbf{B}]), & p_1\geq 2, \; p_2\geq 2, \; b_{12}\leq 1;\nonumber\\
\quad  (k,2,m,[(0,0),(0,0),\mathbf{0}]), &\quad m>0. \nonumber
\end{eqnarray}
   
$\mathrm{I\!I}_3.$ The $(k,3)$-ped  $\Phi_{mp}$ contained in a 
tee-cluster $\Phi$ with $n_c=k>0, n_e=3$ is not a  tee-cluster. 
Center vertices of $\Phi_{mp}$ are, obviously, stable, while the end 
ones are not. Indeed, if $E_i, \,i=1,2,3,$ are end vertices of  $\Phi$ 
and, therefore, of $\Phi_{mp}$, then tees 
$\theta_{rs|t}=\lfloor E_r,E_s|E_t\rceil, \,\{r,s,t\}=\{1,2,3\}$, being compatible 
with $ \Phi_{mp}$ do not belong to it. Moreover, they do not belong to $\Phi$,
since, otherwise, $E_i$'s the would not be end vertices.
So, $\Phi$ must contain tees that block $\theta_{rs|t}$'s. But a tee
blocking $\theta_{rs|t}$ must be rooted at $E_t$ and  have a mixing second end. 
According to DT-table this happens  iff $\Phi$ contains at least one of the following
structure group rooted at $E_t$: a pyramid (possibly "collapsed"),  a p-twain, 
a bridge. Moreover, a bridge 
connecting $E_s$ and $E_t$ simultaneously blocks  $\theta_{rs|t}$ and 
$\theta_{rt|s}$. The following abstract cards describe all minimal combinations of 
these groups,  which simultaneously block all $\theta_{rs|t}$'s: 
\begin{eqnarray*}
(k,3,0,[(0,0,0),(0,0,0),\mathbf{B}]), \;& \;b_{12}=b_{13}=1, \;b_{23}=0;\\
(k,3,0,[(\varepsilon,0,0,),(1-\varepsilon,0,0),\mathbf{B}]), \;& 
\;b_{12}=b_{13}=0, \;b_{23}=1, \;\varepsilon=0,1;\\
(k,3,0,[(\varepsilon_1,\varepsilon_2,\varepsilon_3),
(1-\varepsilon_1,1-\varepsilon_2,1-\varepsilon_3),\mathbf{0}]), 
&  \;\varepsilon_1,\varepsilon_2,\varepsilon_3=0,1;
\end{eqnarray*}

A simple direct check shows that realizations of these abstract cards
are, in fact, tee-clusters with stable center and end 
vertices. Now proposition\,\ref{join-oper} allows to obtain the full list of 
tee-clusters in the considered case in which is assumed that
$\mathbf{t}=(t_1,t_2,t_3),\mathbf{p}=(p_1,p_2,p_3)$ and 
$\mathbf{B}=|\!|b_{ij}|\!|, \,1\leq i,j\leq 3$.

\begin{equation}\label{list-k3}
(k,3,m,[(\mathbf{t},\mathbf{p},\mathbf{B}]) \quad \mathrm{with}\quad \left \{
\begin{array}{l}
(1) \quad b_{12}>0, \; b_{13}>0, \; b_{23}=0.\\
(2) \quad b_{23}>0, \; t_1\geq \varepsilon, \; p_1\geq 1-\varepsilon,  \;\varepsilon=0,1.\\
(3) \quad t_i\geq\varepsilon_i,  \;p_i\geq 1-\varepsilon_i, \;\varepsilon_i=0,1.
\end{array}
\right. 
\end{equation}

\subsection{Classification of tee-clusters. Generators}\label{tee-classification} 

By summing up the results of subsequences \ref{tee-cards} 
(corollary\;\ref{generic ID}) and \ref{except-cards}
we get the following description of  tee-clusters.
\begin{thm}
Tee-clusters are in one-to-one correspondence with their cards. An abstract card 
$(k,l,m,[\mathbf{t},\mathbf{p},\mathbf{B}])$ is the card of a tee-cluster if and only if
$k>0, \,l>3$, or it belongs to one of lists (\ref{list-k1}), (\ref{list-k2}), (\ref{list-k3}).
\end{thm}

An alternative way to describe tee clusters is as follows. Let $\Phi$ be a tee-cluster, 
$\mathcal{C}(\Phi)=(k,l,m,[\mathbf{t},\mathbf{p},\mathbf{B}])$. Denote by 
$\langle\Phi\rangle$, or, equivalently, by 
$\langle k,l,m|\mathbf{t},\mathbf{p},\mathbf{B}\rangle$ 
the set of equivalence classes of tee-clusters that are obtained by 
successively applying to $\Phi$ join operations. We shall say that  they are
equivalence classes of tee-clusters tee-clusters \emph{generated} by $\Phi$.  
The totality $\mathfrak{T}_{\pitchfork}$ of equivalence classes of all tee-clusters 
can be described by indicating a \emph{base} of 
it, i.e., a system of tee-clusters $\Phi_{\alpha}$ such that 
$\mathfrak{T}_{\pitchfork}=\cup\langle\Phi_{\alpha}\rangle$.   
One of such bases, which is easily extracted from the above description of 
tee-clusters (see lists (\ref{list-k1}), (\ref{list-k2}), (\ref{list-k3})), is the following. 
In the list below is assumed that $k>0$ and $m>0$. 

\begin{eqnarray*}\label{fork-list}
\mathbf{Tr}: \; (0,0,1,[(0,0,0),(0,0,0),\mathbf{0}]), &\;(\mathrm{triangle}); \\
\mathbf{Pr}_3: \; (0,1,0,[(0),(3),\mathbf{0}]), &\;(3-\mathrm{pyramid});\\
\mathbf{St}_r: \; (0,r,0,[(0,\dots,0),(2\dots,2),\mathbf{0}]), & r >1 \; (\mathrm{r-twain});\\
{\mathbf{O}_s\mathbf{St}_r}: \; (0,r+s,0,[\underbrace{(0,\dots,0)}_{r+s \;times},
(\underbrace{(0,\dots,0)}_{s \;times},\underbrace{(2,\dots,2)}_{r \;times}),\mathbf{0}]), &
 r\geq 3, \,s>0 \; (\mathrm{rooted \;r-twain});\\
----------------&\\
\mathbf{Hg}_1^k: \; (k,0,1,[\emptyset,\emptyset,\emptyset]), & 
\;((k,1)-\mathrm{hedgehog});\\
----------------&\\
\mathbf{Pr}_1\mathbf{Pt}_1^k: \;(k,1,0,[(1),(1),\mathbf{0}]), &\;(k-\mathrm{trey \;cluster});\\
\mathbf{S}^k\mathbf{Pr}_3: \;(k,1,1,[(0),(3),\mathbf{0}]), 
\; &(\mathrm{suspended} \;3-\mathrm{pyramid}); \\
\mathbf{Pt}_2^k: \;(k,1,0,[(2),(0),\mathbf{0}]), &\; 
(\mathrm{double} \;k-\mathrm{trey \;cluster});\\
----------------&\\
\mathbf{Br}_k:\;(k,2,0,[(0,0),(0,0),\mathbf{B}]), & b_{12}=2, \\
(k,2,0,[(1,1),(0,0),\mathbf{0}]), &\\
(k,2,0,[(1,0),(0,2),\mathbf{0}]), & \\
(k,2,0,[(0,0),(2,2),\mathbf{0}]), &\\
\quad\quad  (k,2,m,[(0,0),(0,0),\mathbf{0}]),&  \quad ((2.k|m)-\mathrm{hybrid});\\
----------------&\\
(k,3,0,[(0,0,0),(0,0,0),\mathbf{B}]), & b_{12}=b_{13}=1, \;b_{23}=0;\\
(k,3,0,[(\varepsilon,0,0,),(1-\varepsilon,0,0),\mathbf{B}]), & 
b_{12}=b_{13}=0, \;b_{23}=1, \;\varepsilon=0,1;\\
(k,3,0,[(\varepsilon_1,\varepsilon_2,\varepsilon_3),
(1-\varepsilon_1,1-\varepsilon_2,1-\varepsilon_3),\mathbf{0}]), &
\quad \varepsilon_1,\varepsilon_2,\varepsilon_3=0,1;\\
----------------&\\
\mathbf{Mp}_k^l: \;(k,l,0,[(0,\dots,0),(0,\dots,0),\mathbf{0}]), & \quad l>3.
\end{eqnarray*}
\begin{ex}
\emph{(4- and 5-dimensional  tee-clusters.)} The following lists of  4- and 
5-dimensional  tee-clusters are easily extracted from the above description.
\newline

\emph{4-dimensional} tee-clusters: $\mathbf{Hg}_1^1$ 
\;\emph{((1,1)-hedgehog)}, \; $\mathbf{Pr}_3$ 
\;\emph{(3-pyramid)}.  \\

\emph{5-dimensional} tee-clusters: $\mathbf{Hg}_2^1$ 
\;\emph{((2,1)-hedgehog)}, \; $\mathbf{Pr}_4$ \;\emph{(4-pyramid)}, \\
$\qquad\mathbf{Pt}_1^1\mathbf{O}$ \;\emph{(1-trey cluster)}, \; $\mathbf{Br}_2$ 
\;\emph{(2-bridge cluster)}, \;$\mathbf{Mp}_4^1$ \emph{(4-ped)}. 
\end{ex}
\subsection{Coaxial Lie algebras associated with tee-clusters.} 

Now we are ready to answer the question: what are  coaxial Lie algebra
associated with tee-clusters.
The following observation is useful to this end.
\begin{lem}\label{coax-sub-id}
Let  $\gG$ be a Lie algebra associated with a compatible tee-family 
$\Phi$, $S_i\subset S(\Phi)$, \,i=1,2, and $V_i$ the subspace of $|\gG|$
spanned by $S_i$.  Then 
\begin{enumerate} 
\item the subspace $[V_1,V_2]=\mathrm{Span}\{[v_1,v_2]\,|\,v_i\in V_i, \,i=1,2\}$ 
of $|\gG|$ is spanned by center vertices 
of all tees $\lfloor E_1,E_2|C\rceil\in\Phi$ such that $E_i\in S_i$.
\item the subspace $V$ of $|\gG|$ spanned by a subset $S\subset S(\Phi)$
is a subalgebra (resp., an ideal) of $\Phi$, if the center vertex of any tee 
$\theta\in\Phi$ with ends in $S$ (resp., with one end in $S$) also belongs to $S$.
\item $|[\gG,\gG]|$ belongs to the subspace spanned by center and mixing
vertices of $\Phi$.
\end{enumerate}                    
\end{lem}
\begin{proof}
The first assertion is obvious, while the others are immediate consequences of it.
\end{proof}

\begin{proc}\label{coax-str}
Let $\gG$ be a coaxial Lie algebra such that $\Phi=\Phi_{\gG}$ is a 
tee-cluster. Then
\begin{enumerate} 
\item the subspace  of $|\gG|$ spanned by $S(\Phi_h)$ (resp., by 
$S(\Phi_{end})$) supports an ideal $\gH$ (resp., $\gR$)  of $\gG$.
\item $[\gH,\gR]=0$;
\item  $\gu=\gH\cap\gR$ is a central ideal of $\gG$ whose support 
is the subspace spanned by center vertices of $\Phi$;  
\item the quotient algebra $\gG/\gR=\gH/\gu$ is the direct sum of 
3-dimensional simple Lie algebras, associated with triangles contained in $\Phi$; 
\item the second derived ideal $\gR^{(2)}=[\gR^{(1)},\gR^{(1)}]$ of $\gR$ is abelian;
$\gR^{(1)}=[\gR,\gR]$ is abelian if $\Phi$ does not contain p-twains.
\item  $\gR$ is the radical of $\gG$. 
\end{enumerate} 
\end{proc} 
\begin{proof}
Assertions (1)-(4 ) directly follow from lemma\,\ref{coax-sub-id}. By the third assertion
of this lemma, $|[\gR,\gR]|$ belongs to the subspace  spanned by center and mixing
vertices of $\Phi_{end}$. Therefore, $|[\gR^{(1)},[\gR^{(1)}]|$ 
belongs to the subspace spanned by centers of \emph{ee}-tees contained in
$\Phi_{end}$. But, according to DT-table, \emph{ee}-tees form the \emph{ee}-casing
of p-twains contained in $\Phi$, and hence their centers are center vertices of  $\Phi$. 
This shows that  $[\gR^{(1)},[\gR^{(1)}]\subset\gu$. In particular, 
$[\gR^{(1)},[\gR^{(1)}]=\{0\}$, if $\Phi$ does not contain p-twains. This proves assertion
(5). Finally, the last assertion directly follows from $(4)$ and $(5)$ ones.
\end{proof}
\begin{cor}
Let $\gG$ be a coaxial Lie algebra and $\gR$ the radical of it. Then $\gR^{(3)}=0$ 
and the semisimple part of $\gG$ is isomorphic to a direct sum of 3-dimensional 
simple Lie algebras.
\end{cor}
\begin{proof}
This is a direct consequence of proposition\,\ref{coax-str} and the fact that 
$\Phi_{\gG}$ is contained in a tee-cluster.
\end{proof}

\section{Generic clusters}\label{clusters} 
\bigskip

Let $\Phi$ be a compatible family of base structures.
Then $\Phi=\Phi_0\cup\Phi_1$ where $\Phi_0$ is a tee-family and
$\Phi_1$ is a dee-family. A cluster will be called \emph{generic}, if 
$\Phi_0\neq \emptyset, \Phi_1\neq\emptyset$.  
The previous analysis of dee- and tee-clusters naturally extends to generic 
ones whose structure will be described in this section.

\subsection{Framed tee-clasters}\label{fr-tee-clusters}

A natural question is: what are generic clusters one can construct by adding
some dees to a tee-cluster. To answer it we need the following 
\begin{lem}\label{diode-tee}
Let $\Phi$ be a compatible family and $\varrho\in\Phi_1$. Then 
\begin{enumerate}
\item if the origin of $\varrho$ belongs to $S(\Phi_0)$, then it is an end 
vertex of $\Phi_0$;
\item if the end of $\varrho$ belongs to $S(\Phi_0)$, then the origin of $\varrho$ 
also belongs to $S(\Phi_0)$;
\item if $\Phi_0$ is a tee-cluster, then the end of $\varrho$ can not 
be a center vertex of $\Phi_0$;
\item if $\Phi_0$ is a tee-cluster, then the end of $\varrho$ is neither a 
bottom vertex of a p-twain in $\Phi_0$, nor a center vertex of a bridge.
\end{enumerate}
\end{lem}
\begin{proof}
First two assertions are direct consequences of proposition\,\ref{123} and 
definitions. The third one follows from the fact that a dee whose end is 
the center of a cross (resp., a tripod) is not compatible with this cross 
(resp., the tripod). Indeed, if  $C$ is a center vertex of a tee-cluster 
$\Psi$, then  $\Psi$ contains a cross, or a tripod with the center at 
$C$ as one can see from the description of tee-clusters.
Concerning the last assertion, observe that a dee whose end is a side 
vertex of a trey is not compatible with this trey. But a p-twain is contained 
in a trey belonging to $\Phi_0$. Finally, the center of
a bridge is  manifestly impossible by proposition\,\ref{123}.
\end{proof}

Diagrams of doubles, rotators and spiders are illustrated in Fig. 6.

\begin{figure}[htb]
$$
\xy
(1,-6)*{\bullet};    (1,6)*{\bullet};     (0.4,-5.6)*{}="p1";   (1.6,-6)*{}="p2";   
(0.4,6)*{}="q1";   (1.6,5.6)*{}="q2";   {\ar@{->} "q1"; "p1"};     {\ar@{.>} "p2"; "q2"};
(1,11)*{{\text {\it \tiny Double and its icon}}}; (12,-3)*{\circ}; (12,6)*{\circ}; (11.4,-2.6)*{}="p11";   
(12.6,-2.6)*{}="p12";   (11.4,5.6)*{}="q11";   (12.6,5.6)*{}="q12"; 
"p11";"q11"; **\dir{-}; "p12";"q12"; **\dir{-};
(43,11)*{{\text {\it \tiny (2,3)-rotator}}};(37,-6)*{\bullet};    (37,6)*{\bullet};     
(36.4,-5.6)*{}="r1";   (37.6,-6)*{}="r2";   
(36.4,6)*{}="s1";   (37.6,5.6)*{}="s2";   {\ar@{->} "s1"; "r1"};     {\ar@{.>} "r2"; "s2"};
(43,1.6)*{*}="r11"; (43,-3.6)*{*}="r21"; (43,6.8)*{*}="r31"; "r11";"s2" **\dir{--};  
"r11";"r2" **\dir{--};  "r21";"s2" **\dir{.}; "r21";"r2" **\dir{.}; "r31";"s2" **\dir{-};    
"r31";"r2" **\dir{-};  (49,-6)*{\bullet}; (49,6)*{\bullet};  (48.4,-5.6)*{}="t1";   (49.6,-6)*{}="t2";   
(48.4,6)*{}="u1";   (49.6,5.6)*{}="u2";   {\ar@{->} "u1"; "t1"};     {\ar@{.>} "t2"; "u2"};
 "r31";"t1" **\dir{-};  "r31";"u1" **\dir{-};  "r21";"t1" **\dir{.}; "r21";"u1" **\dir{.}; 
 "r11";"t1" **\dir{--}; "r11";"u1" **\dir{--};  
 (75,11)*{{\text {\it \tiny (3,2)-spider}}}; (75,-6)*{\bullet}="a1";   (77.5,-0.5)*{\bullet}="a2"; 
 (83,1)*{\bullet}="a3";  
 (70,1.5)*{\bullet}="a4"; (79,7)*{\bullet}="a5";    {\ar@{->} "a2"; "a1"};  
  {\ar@{.>} "a4"; "a1"};   {\ar@{-->} "a3"; "a1"};  
  {\ar@{-->} "a4"; "a5"};   {\ar@{.>} "a2"; "a5"};{\ar@{->} "a3"; "a5"}; 
\endxy
$$
\end{figure}
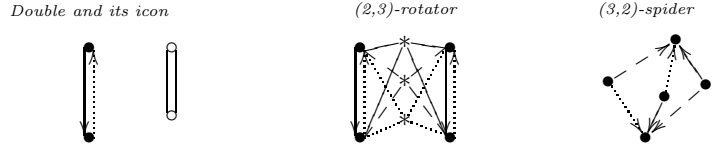
\begin{center}Fig. 7. Doubles, rotators and spiders. \end{center}
\medskip

Let $\Psi$ be a  tee-cluster. Add to it all dees of the form 
$\lfloor E|B\rceil$ with $E$ being an end vertex of $\Psi$ and $B$ a 
bottom vertex of $P_E$  ($B=C_E$, if $P_E=\emptyset$). The so-obtained 
family, denoted by $\bar{\Psi}$, is, obviously, compatible. It will be 
called a \emph{framed}  tee-cluster. It will be shown (proposition\,\ref{framed cluster}) 
that a framed tee-cluster, which is different from a $(k,2|m)$-hybrid, is also a cluster.                    

We stress that the framed twain is a 3-dimensional cluster. So, a framed $k$-pyramid
is a cluster, if $k\geq2$. The \emph{framing} of a framed $k$-pyramid $\bar{P}$ is the
$(1,k)$-spider, which is formed by all dees $\lfloor E|B\rceil$ where $E$ and $B$ are 
top  and bottom vertices of $P$, respectively.

Framed tee-clusters do not exhaust generic ones as the following 
two examples show. 

\textbf{Rotator.} An $m$-\emph{rotator}, $m\geq 0$, is a family equivalent 
to 
$$
\{\lfloor e_1|e_2\rceil, \lfloor e_2|e_1\rceil, \lfloor e_1,e_2|e_3\rceil,\dots,\lfloor e_1,e_2|e_{m+2}\rceil\}.
$$
It is easily verified that an $m$-rotator is a generic cluster if 
$m>0$. The double $\{\lfloor e_1|e_2\rceil, \lfloor e_2|e_1\rceil\}$ is its 
\emph{axis}, $e_1, e_2$ are its \emph{ends} and $ e_i$'s, $i=3,\dots,m+2$, 
are its \emph{thorns}. 

\emph{Basic properties of rotators}. No dee is nontrivially compatible 
with a rotator. A tee $\theta=\lfloor E_1,E_2|C\rceil$ is nontrivially 
compatible with a rotator $\mathbf{r}$ if either $E_1,E_2$ are ends of 
$\mathbf{r}$, or $C$ is a thorn of $\mathbf{r}$, which is the unique common 
vertex of $\theta$ and $\mathbf{r}$. 

\textbf{D-bridge.} This is a family equivalent to $\{\lfloor 
e_1|e_3\rfloor, \lfloor e_2|e_3\rceil, \lfloor e_1,e_2|e_3\rceil\}$. It is, 
obviously, a (generic) cluster. We shall refer to $e_1, e_2$ as the ends of 
this d-bridge, to $e_3$ as its center and to $\lfloor e_1,e_2|e_3\rceil\}$ as its axis.
We shall say also that a d-bridge \emph{connects} its ends. $\mathbf{Bd}$ will 
be standard notation of a d-bridge.

\emph{Basic properties of d-bridges}. If one vertex of a  tee/dee, which is 
compatible with a d-bridge $\mathbf{Bd}$, is the center vertex of $\mathbf{Bd}$, 
then this tee/dee belongs to $\mathbf{Bd}$. 

\textbf{Raft.} Let $\mathcal{R}=\{\mathbf{Bd}_1,\dots,\mathbf{Bd}_k\}$ be a compatible 
family of d-bridges whose end vertices belongs to a subset $\{E_1,\dots,E_l\}$ of base 
vectors. Obviously, $2\leq l\leq 2k$. The family $\mathcal{R}$ is compatible iff center 
vertices of d-bridges $\mathbf{Bd}_i$'s differ each other. 

The union $\mathbf{Bd}_1\cup\dots\cup\mathbf{Bd}_k$ 
is called a $(k,l)$-\emph{raft}.  If $l$ differs from 2 and $2k$, then there are more
than one nonequivalent $(k,l)$-rafts. By basic property of d-bridges, a raft is a cluster
if its graph is connected. Such a cluster will be called a \emph{raft cluster}.  

Let $\mathcal{R}$ be as above and  $\{C_1,\dots,C_m\}\cap S(\mathcal{R})=\emptyset$. 
A \emph{suspended} from $\{C_1,\dots,C_m\}$ \emph{raft} (schortly, s-raft) is the union 
of  $\mathcal{R}$ and the $(l,m)$-ped whose center vertices are $C_i$'s and the end 
vertices are those of $\mathcal{R}$. As it is easy to see, a suspended raft is a cluster, 
if $l>2$. Diagrams of d-bridges and rafts are shown in Fig. 8.
                           
\scalebox{1} 
{
\begin{pspicture}(0,-1.2047187)(10.344375,1.1647187)
\psline[linewidth=0.024cm](5.7464943,-0.46303362)(6.1754313,0.25919443)
\psline[linewidth=0.024cm](6.4922557,-0.4830336)(6.0633187,0.23919444)
\psdots[dotsize=0.152](6.129375,-0.41128126)
\psdots[dotsize=0.152](6.529375,-0.41128126)
\psdots[dotsize=0.152](5.729375,-0.41128126)
\psline[linewidth=0.024cm](5.709375,-0.47128126)(6.549375,-0.47128126)
\psline[linewidth=0.024cm](5.709375,-0.35128126)(6.509375,-0.35128126)
\psdots[dotsize=0.152,dotangle=59.293606](5.909375,-0.07128125)
\psdots[dotsize=0.152,dotangle=59.293606](6.113631,0.27263686)
\psline[linewidth=0.024cm](5.649375,-0.39128125)(6.05183,0.2860793)
\psdots[dotsize=0.15048584,dotangle=-59.293606](6.329375,-0.07128125)
\psline[linewidth=0.024cm](6.589375,-0.37128124)(6.169375,0.32871875)
\psline[linewidth=0.024cm](1.329375,0.84871876)(2.089375,0.08871875)
\psdots[dotsize=0.152](0.349375,0.10871875)
\psdots[dotsize=0.152](1.349375,0.90871876)
\psdots[dotsize=0.152](2.149375,0.10871875)
\psline[linewidth=0.024cm,linestyle=dashed,dash=0.16cm 0.16cm,arrowsize=0.05291667cm 2.0,arrowlength=1.4,arrowinset=0.4]{->}(2.209375,0.12871875)(1.429375,0.92871875)
\psline[linewidth=0.024cm,linestyle=dashed,dash=0.16cm 0.16cm,arrowsize=0.05291667cm 2.0,arrowlength=1.4,arrowinset=0.4]{->}(0.349375,0.18871875)(1.289375,0.90871876)
\psline[linewidth=0.024cm](0.369375,0.068718754)(1.329375,0.8287187)
\psdots[dotsize=0.152](1.349375,-0.49128124)
\psdots[dotsize=0.152](1.749375,-0.49128124)
\psdots[dotsize=0.152](0.949375,-0.49128124)
\psdots[dotsize=0.152](0.949375,-0.49128124)
\psline[linewidth=0.024cm](0.929375,-0.5512813)(1.769375,-0.5512813)
\psline[linewidth=0.024cm](0.929375,-0.43128124)(1.729375,-0.43128124)
\psdots[dotsize=0.152,dotangle=-39.45292](3.569375,0.18871875)
\psdots[dotsize=0.152,dotangle=-39.45292](3.2605162,0.44289634)
\psdots[dotsize=0.152,dotangle=-39.45292](3.878234,-0.06545883)
\psdots[dotsize=0.152,dotangle=-39.45292](3.878234,-0.06545883)
\psline[linewidth=0.024cm](3.85555,-0.12449653)(3.2069466,0.4092764)
\psline[linewidth=0.024cm](3.9318035,-0.031838883)(3.3140857,0.47651628)
\psellipse[linewidth=0.024,dimen=outer](4.429375,0.02871875)(0.572,0.312)
\psellipse[linewidth=0.024,dimen=outer](4.419375,-0.13128126)(0.562,0.312)
\psdots[dotsize=0.152](4.389375,0.24871875)
\psdots[dotsize=0.152](4.409375,-0.35128126)
\psdots[dotsize=0.152](4.949375,-0.07128125)
\psellipse[linewidth=0.024,dimen=outer](8.709375,-0.07128125)(0.572,0.312)
\psellipse[linewidth=0.024,dimen=outer](8.699375,-0.23128125)(0.562,0.312)
\psdots[dotsize=0.152](8.669375,0.14871874)
\psdots[dotsize=0.152](8.689375,-0.45128125)
\psdots[dotsize=0.152](9.229375,-0.17128125)
\psline[linewidth=0.024cm]{cc-}(8.769375,1.0687188)(8.089375,-0.13128126)
\psline[linewidth=0.024cm]{cc-}(8.769375,1.0687188)(9.289375,-0.111281246)
\psdots[dotsize=0.152](8.769375,1.0687188)
\psdots[dotsize=0.152](8.169375,-0.15128125)
\usefont{T1}{ptm}{m}{it}
\rput(0.29421875,-0.48128125){icon}
\usefont{T1}{ptm}{m}{it}
\rput(5.219531,-0.95128125){\small (3,3)-rafts}
\usefont{T1}{ptm}{m}{it}
\rput(8.931719,-0.93128127){\small suspendent (2,2)-raft}
\usefont{T1}{ptm}{m}{it}
\rput(1.4992187,-0.9912813){\small d-bridge}

\end{pspicture} 
}
\begin{center} Fig. 8.D-bridges and rafts. \end{center}


\subsection{Types of dees in a generic cluster}

Now we pass to a systematic study of generic clusters by starting with the 
necessary terminology. Let $\Phi$ be a compatible family. A vertex $e_i\in 
S(\Phi)$ will be called an \emph{end vertex} of $\Phi$, if $e_i$ is either or
both the origin of a dee $\varrho\in\Phi_1$ and an end vertex of $\Phi_0$. 
Assertion (1) of Lemma\,\ref{diode-tee} guarantees correctness of this 
terminology. A center vertex of $\Phi_0$, which is not a vertex of 
$S(\Phi_1)$, will be called a \emph{t-center vertex} of $\Phi$.  Similarly, 
the end of a dee in $\Phi_1$, which does not belong to $S(\Phi_0)$, will be 
called a \emph{d-center vertex} of $\Phi$. 
\begin{lem}\label{cluster-vertex} Let $\Phi$ be a generic cluster. Then it holds:
\begin{enumerate}
\item  If a double belongs to $\Phi_1$, then it is the axis of the 
belonging to $\Phi$ rotator whose thorns are t-center vertices of 
$\Phi$. 
\item If $E$ and $D$ are an end and an d-center vertices of $\Phi$, respectively, 
then the dee $\lfloor E|D\rceil$ belongs to $\Phi$.
\item If $E_1, E_2$ are end vertices of  $\Phi$, while $C$ is an t-center of  it, then 
the tee $\lfloor E_1,E_2|C\rceil$ belongs to $\Phi$.
\end{enumerate}
\end{lem}
\begin{proof}
A direct consequence of proposition\,\ref{123}.
\end{proof}

Denote by $\Phi_1^{db}$ (resp., $\Phi_1^{br}$) the union of all doubles 
(resp., dees in d-bridges) that belong to $\Phi$. A dee $\varrho\in\Phi$ will be 
called \emph{immersed} if its vertices belong to $S(\Phi_0)$ and
$\varrho\notin\Phi_1^{db}\cup\Phi^{br}$. A \emph{spike} (resp., a 
\emph{poker}) is a family equivalent to $\{\lfloor e_i|e_j\rceil, \lfloor 
e_i,e_j|e_k\rceil\}$ (resp., $\{\lfloor e_i|e_j\rceil, \lfloor 
e_i,e_k|e_j\rceil\}$). Obviously, an immersed dee belongs to at least one 
spike or poker contained in $\Phi$.  Denote by $\Phi_1^{im}$ the set of all 
immersed dees in $\Phi$ and put 
$\Phi_1^{sp}=\Phi_1\setminus(\Phi_1^{im}\cup\Phi_1^{db}\cup\Phi_1^{br})$. 
So, $\Phi_1$ is the disjoint union 
$$
\Phi_1=\Phi_1^{im}\cup\Phi_1^{db}\cup\Phi_1^{sp}\cup\Phi_1^{br}.
$$ 
Obviously, the end of a dee in $\Phi_1^{sp}$ is a d-center vertex, and, 
as it follows from lemma\,\ref{cluster-vertex}\,(2), $\Phi_1^{sp}$ is a 
$(k,l)$-spider (see subsection\,\ref{dee-clusters} ) where $k$ (resp., $l$) 
is the number of end (resp., d-center) vertices of $\Phi$.    

\begin{lem}\label{cluster-d-center} 
Let $\Phi$ be a generic cluster with $k$ end vertices. 
Then $k\geq 3$, if $\Phi$ has at least one d-center vertex.
\end{lem}
\begin{proof}
First, assume that $k=2$. Let $E_1,E_2$ be end vertices of $\Phi$. If $D$ 
is a d-vertex of $\Phi$, then $\lfloor E_i|D\rceil, \,i=1,2,$ are the only 
dees in $\Phi$ that have $D$ as one of their vertices. On the other hand, 
$\theta=\lfloor E_1,E_2|D\rceil$ is the only tee, which is compatible with these 
dees and has $D$ as one of its vertices. This shows that $\theta$ is compatible 
with $\Phi_1$ and, therefore, belongs to $\Phi$ in contradiction with the 
assumption that $D$ is a d-center vertex.

Now assume that $k=1$. Let $E$ be the end vertex  and $D$  
one of d-center vertices of $\Phi$. So, by lemma\,\ref{diode-tee}, 
$\lfloor E|D\rceil$ is the unique dee in $\Phi$ with the end at $D$. 
If $C$ is an t-center vertex of $\Phi$, then, as it is easy to check, the tee 
$\lfloor E,D|C\rceil$ is compatible with $\Phi$ and hence belongs to 
$\Phi$. This is, however, impossible, since $D$ is a d-center vertex. 
So, $\Phi$ has no t-center vertices. Consider now a rooted at $E$ tee 
$\theta=\lfloor E,P|Q\rceil\in \Phi$. Then the tee $\vartheta=\lfloor 
E,D|Q\rceil$ is compatible with $\Phi$. Indeed, $\vartheta$ can not be 
blocked by a dee $\sigma\in \Phi$, since its origin is at $E$ by the 
assumption $k=1$. On the other hand,  $\vartheta$ can be blocked only 
by a tee $\varrho\in \Phi$ one whose ends is at $Q$ and which is not rooted 
at $E$, i.e., $\varrho=\lfloor P,Q|R\rceil\in \Phi$, $R\in S(\Phi)$. Since $k=1$, 
$P$ is not the end vertex of a dee in $\Phi$ (lemma\,\ref{diode-tee}). 
Moreover, $P$ is not an end 
vertex of $\Phi$, since $P\neq E$. This implies that there is a tee 
$\rho\in \Phi$ with the center at $P$. But the unique tee, which is 
compatible with $\theta$ and $\varrho$, is 
$\rho=\lfloor E,Q|P\rceil$. Since $\{\theta,\varrho,\rho\}$ is a trey, the 
center vertex $R$ of it is an t-center vertex of $\Phi$. But, as it was
already proved, this is impossible.  
\end{proof}

\begin{cor}\label{spider in}
 Let $\Phi$ be a cluster. Then it holds:
\begin{enumerate}
\item If $\Phi_1^{sp}\neq\emptyset$, then the number $k$ of end vertices of 
$\Phi$ is greater than 2 and $\Phi_1^{sp}$ is a $(k,l)$-spider with $l$ being 
the number of d-center vertices of $\Phi$. 
\item If the center of a $(2,1)$-spider $\Sigma\subset\Phi_1$ is not 
a d-center vertex of $\Phi$,  then $\Sigma$ belongs to an unique 
d-bridge contained in $\Phi$. 
\end{enumerate}  
\end{cor}
\begin{proof}
The first assertion is a direct consequence of 
lemma\,\ref{cluster-vertex},\,(2) and lemma\,\ref{cluster-d-center}. Next, 
the center of $\Sigma$ is a vertex of a tee $\theta\in\Phi$. But being
compatible with $\Sigma$ the tee $\theta$ 
has common ends with $\Sigma$. 
\end{proof}

The following lemma clarifies the status of immersed dees.
\begin{lem}\label{immersed}
Let $\Phi$ be a cluster, $E$ be an end vertex of $\Phi$ and 
$\varrho=\lfloor E|D\rceil\in\Phi_1^{im}$. Then
\begin{enumerate}
\item $\varrho$ belongs to 
a contained in $\Phi$ poker. 
\item If $A\in S(\Phi), \,A\neq E$, is an end vertex 
of $\Phi$, or a bottom vertex of a rooted at $E$ twain in $\Phi_0$, or 
the center of a bridge in $\Phi_0$ with one end at $E$, then $\lfloor 
E,A|D\rceil\in\Phi$. 
\item If $\varrho_i=\lfloor E|D_i\rceil\in\Phi_1^{im}, \,i=1,\dots,m$, then 
$D_1,\dots,D_m$ are bottom vertices of the (possibly, ``collapsed") pyramid 
$P_E\subset\Phi_0$.
\item If $\varrho$ is the unique dee in $\Phi_1^{im}$, which is rooted at $E$,
then $D=C_E$ and vice versa. 
\end{enumerate}  
\end{lem}
\begin{proof}
(1). Since a dee in $\Phi_1^{im}$ belongs either to a spike, or to a poker, 
only the former case is to be examined. Let $\theta=\lfloor E,D|B\rceil\in\Phi$  
complements $\varrho$ to a spike. Since $\varrho\notin\Phi_1^{db}$, the 
dee $\varrho '=\lfloor D|E\rceil$ is blocked. As it is easy to see, 
$\varrho'$ can be blocked either by a tee forming a poker with $\varrho$, or 
by a rooted at $E$ tee whose second end $A$ is different from $D$, or  by 
a dee with the origin at $E$. In the second of these cases let 
$\vartheta=\lfloor E,A|C\rceil\in\Phi$ 
be a blocking $\varrho '$ tee. Then a direct check shows that the tee 
$\lfloor E,A|D\rceil$, which forms a poker with $\varrho$, is compatible with
any tee/dee, which is compatible with tees $\varrho, \,\theta$ and $\vartheta$.
Hence $\lfloor E,A|D\rceil$ belongs to $\Phi$. In the third case let
$\sigma=\lfloor E|D'\rceil\in\Phi$ be a dee, which blocks $\varrho'$,
and  $\rho=\lfloor E,D'|D\rceil$. Observe now two facts. First, any tee whose
vertices are $E, D, D'$, which is incompatible with the family 
$\Sigma=\{\varrho, \,\theta, \,\sigma\}$, is rooted at $E$ and hence does not
block $\rho$. Second, a dee, which is compatible with $\Sigma$ but not 
compatible with $\rho$, is of the form $\sigma'=\lfloor E'|D'\rceil, \,E\neq E'$. So, 
the tee $\rho$, which forms a poker with $\varrho$, can be blocked only by such 
a dee belonging to
$\Phi$. This proves the assertion assuming that $\Phi$ does not contain dees
of this kind. If, on the contrary, $\{\sigma, \,\sigma'\}\subset \Phi$, then, according
to corollary\,\ref{spider in}, either $D$ is a d-center vertex or $\{\sigma, \,\sigma'\}$
belongs to a d-bridge in $\Phi$. In each of 
these cases $\Phi$ has more than one end vertices (see 
lemma\,\ref{cluster-d-center}). If $E'\neq E$ is a such one, then the tee 
$\lfloor E,E'|D\rceil$, which forms a poker with $\varrho$, is not blocked
and hence belongs to $\Phi$.

(2) By using basic property of twains and lemma\,\ref{cluster-d-center}, (2), one
easily verifies that $\Phi$ does not contain tees/dees that block $\lfloor E,A|D\rceil$. 

(3) If $\lfloor E|D\rceil, \lfloor E|D'\rceil\in \Phi_1^{im}$, then, by the same 
reasons as previously,
is easy to verify, the tee $\lfloor E,D'|D\rceil$ is not blocked and hence 
belongs  to $\Phi$.  This shows that that the rooted at $E$ pyramid with 
bottom vertices $D_1,\dots,D_m$ belongs to $\Phi$. Moreover, if $D$ is a 
bottom vertex of a rooted at $E$ s-twain in $\Phi_0$, then, obviously, 
$\lfloor E|D\rceil\in\Phi$. These two facts proves the assertion. 

(4) A direct check by using proposition\,\ref{123} and assertion (1).
\end{proof}

The dee described in lemma\,\ref{immersed},\,(4), will be called a \emph{single}
(rooted at $E$). A single may be thought as the framing of a ``collapsed pyramid".
According to lemma\,\ref{immersed}  a rooted at $E$ single belongs to a nonzero number of 
rooted at $E$ pokers, which belong to $\Phi$.

\begin{cor}\label{immersed-class}
Let $\Phi$ be a cluster and $E_1,\dots,E_m$ end vertices of it.
All s-twains $\wedge\in\Phi_0$ rooted at $E_i$ form a (possibly empty) pyramid
$P_{E_i}$. The framing $\Phi_{E_i}^{im}$ of $P_{E_i}$ belongs to $\Phi$.
$\Phi_1^{im}$ is disjoint union of framings $\Phi_{E_i}^{im}$ and singles that
correspond to end vertices $E_j$ such that $p_{E_j}=1$.
\end{cor} 

The following lemma describes how a poker $\Pi\in\Phi$ is attached to the rest of $\Phi$ .
\begin{lem}\label{poker-in}
Let $\Pi=\lfloor E|B\rceil,\lfloor E,C|B\rceil\in\Phi$ and $\lfloor E|B\rceil\in\Phi_1^{im}$. 
Then one of the following possibilities takes place:
\begin{enumerate}
\item $\Pi\subset \bar{P}_E$.
\item $C$ is an end vertex of $\Phi$.
\item $C$ is a side vertex of a rooted at $E$ trey.
\item $C$ is the center vertex of a bridge $\lfloor E,E'|C\rceil\in\Phi$. If 
$\lfloor C,E'|B'\rceil\in\Phi$, then either $\lfloor E'|B'\rceil\in\Phi$, or $B'$ is a center vertex
of $\Phi$. \end{enumerate}
\end{lem}
\begin{proof}
If $\Pi\not\subset \bar{P}_E$, then the dee $\varrho=\lfloor E|C\rceil$ is blocked. But 
all other dees, which are compatible with $\Pi$, are of the form $\lfloor C|B'\rceil$. So,  
$C$ is an end vertex of $\Phi$ if $\varrho$ is blocked by such a dee. Let now $\theta$ 
be a compatible with  $\Pi$ tee, which 
blocks $\varrho$.  A simple prove by exhaustion shows that $\theta$ must be of the form  
$\lfloor C,E'|B'\rceil$ with vertices $E', B'$ not belonging to $\{E, B, C\}$. The vertex $C$ 
is either an end or a mixing vertex of $\Phi_0$. In the first case $C$ is an end vertex of 
$\Phi$ too. Indeed, otherwise, $C$ would be the end vertex of a dee  belonging to $\Phi$.
But any such dee is incompatible with $\Pi\cup\{\theta\}$.
So, it remains to analize the second of these possibilities only.

In this case $\Phi_0$ contains a tee $\rho$ whose center vertex is $C$. The only such tee, 
which is compatible with $\Pi\cup\{\theta\}$, is $\rho=\lfloor E,E'|C\rceil$. Note that a dee is
incompatible with $\Pi\cup\{\theta, \rho\}$, if $E'$ is the end vertex of it. By this reason, $E'$ is an end 
vertex of $\Phi$, if it is an end vertex of $\Phi_0$. So, in this case $\rho$ is  the bridge from
assertion (4). Moreover, the second part of this assertion follows the fact that the dee
$\lfloor E'|B'\rceil$ can be blocked either by the tee $\lfloor E,C|B'\rceil$ or by a tee which
forms a cross with $\lfloor E',C|B'\rceil$.

Finally, if $E'$ is a mixing vertex of $\Phi_0$, then there is a tee $\sigma\in\Phi_0$ with the
center at $E'$. But the only such one, which is compatible with  $\Pi\cup\{\theta, \rho\}$, is
$\lfloor E,C|E'\rceil$. This tee completes $\{\lfloor E,C|B\rceil, \,\lfloor C,E'|B'\rceil\}$ up
to a trey.
\end{proof}
\begin{cor}\label{status-poker}
A poker in a cluster belongs to a framed pyramid or to a d-bridge, or the belonging to it tee is a connective.            
\end{cor}
\begin{proof}
If in the notation of lemma\,\ref{poker-in} $C$ is and end vertex, then $\Phi$ may possess
a dee of the form $\lfloor C|B'\rceil$. If $B=B'$, then $\Pi$ belongs to a d-bridge. Otherwise,
the considered poker is as in  lemma\,\ref{poker-in}.
\end{proof}

\subsection{Join operations and cards of clusters}

The join operations defined earlier for tee-clusters can be applied also to arbitrary
clusters. Their definitions remain literally the same except the pyramid join operation 
where pyramids we must be replaced by framed pyramids. Additionally, we have three
new join operations we are passing to describe. Below $\Phi$ stands for a cluster.

\emph{Joining a rotator.} Add two new vertices $A_1$ and $A_2$ to $S(\Phi)$ and 
consider the rotator $\mathbf{r}$ whose axis is $\lfloor A_1|A_2\rceil$ and thorns 
are t-center vertices of $\Phi$. Then, by the basic property of rotators, 
$\Phi\cup\mathbf{r}$ is a cluster.

\emph{Joining a d-bridge.} There are two versions of this procedure. \\
1) Consider the d-bridge 
$\bar{\frown}=\{\lfloor E_1|C\rceil, \lfloor E_2|C\rceil, \lfloor E_1,E_2|C\rceil\}$
with $E_1, E_2$ being some end vertices of $\Phi$ and $C\notin S(\Phi)$.
Then,  by the basic property of d-bridges,$\Phi\cup\bar{\frown}$ is a cluster.\\
2) Let $\bar{\frown}$ be as above but $S(\Phi)\cap S(\bar{\frown})=\{E_1\}$ with 
$E_1$ being an end vertex of $\Phi$. By adding to $\Phi\cup \bar{\frown}$ new
connectives and all dees $\{\lfloor E_2|C\rceil$ with $C$ running through d-center
vertices of $\Phi$, we get a cluster

\emph{Joining a spider.} Add a new vertex $C$ to $S(\Phi)$ and consider the
spider $\mathbf{S}$ composed of all dees $ \{\lfloor E|C\rceil$ with $E$ running 
through end vertices of $\Phi$. Then $\Phi\cup\mathbf{S}$ is a cluster and $C$ is
a d-center vertex of it, if the number of end vertices of $\Phi$ is not less than 3.

These new join operations preserve original t-center and end vertices and do 
not create new ones except the second join bridge operation. In this case
one new end vertex comes out. These operations also commute 
each other as well as with old join operationsl. Moreover, we have
\begin{proc}\label{tee-generated}
Let $\Phi$ be a cluster such that $\Phi_0$ is a tee-cluster. Then $\Phi$ is obtained
from the framed tee-cluster $\bar{\Phi}_0$ by applying to it rotator, d-bidge and spider
join operations.
\end{proc}
\begin{proof}
First of all, observe that $\bar{\Phi}_0=\Phi_0\cup\Phi_1^{im}$. So, $\Phi$ is obtained
from the framed tee-cluster $\bar{\Phi}_0$ by adding to it the spider $\Phi_1^{sp}$, all 
d-bridges corresponding to the forming $\Phi_1^{br}$ $(2,1)$-spiders 
(corollary\,\ref{spider in},\;(2)) and all rotators whose axes are the composing $\Phi_1^{db}$ 
doubles and thorns are t-center vertices of $\Phi$. 
\end{proof}

By summing up the previous results of this section and that of the preceding one 
we see that clusters are made of \emph{basic groups}, 
their \emph{casings} and \emph{connectives}. Namely, denote by $n_t, n_d$ and $n_e$ 
numbers of t-center, d-center and end vertices of $\Phi$, respectively, and by $n_{tr}, n_{r}$ 
numbers of triangles and doubles in $\Phi$.  The list of basic groups is as follows:

\begin{tabbing}
\qquad $\bullet \quad(n_{tr}, n_t)$-hedgehog \= \quad $\Phi_h$, 
\qquad \qquad \qquad \= $\bullet  \quad n_t$-treys, \\
\qquad $\bullet \quad(n_e,n_t)$-multiped \> \quad $\Phi_{mp}$,  \> $\bullet$ \quad framed pyramids, \\
\qquad $\bullet \quad(n_r,n_t)$-rotator \> \quad $\Phi_{rt}$,  \> $\bullet$ \quad bridges \\
\qquad $\bullet \quad(n_e,n_d)$-spider \> \quad $\Phi_{sp}$, \> $\bullet$ \quad d-bridges. 
\end{tabbing}
Here $\Phi_h=(\Phi_0)_h$. We also emphasize that the end and center vertices of 
$\Phi_{mp}$ are those of $\Phi$, and, similarly, for $\Phi_{rt}$ and $\Phi_{sp}$.
Also, associate with $\Phi$ the s-raft $\Phi_{rf}\subset\Phi$ composed of all d-bridges 
belonging to $\Phi$, which are  suspended from center vertices of  $\Phi$.

$\Phi_{FW}=\Phi_{h}\cup\Phi_{mp}\cup\Phi_{rt}\cup\Phi_{sp}$ is the 
\emph{framework} of $\Phi$. This part of $\Phi$ is well defined by numbers 
$n_t, n_d, n_e, n_{tr}$ and $n_r$, while basic groups in the right column of this
list are attached to end vertices of $\Phi$ and their numbers may vary almost arbitrarily 
when the end vertices of $\Phi$ remain fixed. Let $E_1\dots,E_{n_e}$ be end vertices of
$\Phi$. Put $\mathbf{t}=(t_1\dots,t_{n_e})$ (resp., $\mathbf{p}=(p_1\dots,p_{n_e})$) with
$t_i$ (resp., $p_i$) being the number of $n_t$-treys rooted at $E_i$ (resp., the dimension of 
$\bar{P}_{E_i}$). Here  $\bar{P}_{E_i}$ is the framed pyramid $P_{E_i}\in\Phi_0$. Denote 
by $b_{ij}$ (resp., $d_{ij}$) the number of bridges (resp., d-bridges) connecting $E_i$ and 
$E_j$. By putting $b_{ii}=d_{ii}=0$ we have symmetric matrices $\mathbf{B}=|\!|b_{ij}|\!|$ and
$\mathbf{D}=|\!|d_{ij}|\!|$. The \emph{card} of a cluster $\Phi$ is
$$
\mathrm{Card}(\Phi)=(n_t, n_e, n_d, n_{tr}, n_r, [\mathbf{t}, \mathbf{p}, \mathbf{B}, \mathbf{D}]) . 
$$    
Here the meaning of the bracket $[\dots ]$ is as for tee-clusters (see subsection\,\ref{tee-cards}).
The \emph{tee-part} (resp.,  \emph{dee-part}) of $\mathrm{Card}(\Phi)$ is 
$\mathcal{C}_t(\Phi)=(n_t, n_e, n_{tr}, [\mathbf{t}, \mathbf{p}, \mathbf{B}])$ (resp, 
$\mathcal{C}_d(\Phi)=(n_t, n_e, n_d, n_r, [\mathbf{p}, \mathbf{D}])$. Obviously, 
$\mathcal{C}_t(\Phi)=\mathcal{C}(\Phi_0)$ and
$$
\dim\Phi=n_t+n_e+n_d+3n_{tr}+2n_r+\sum_{i=1}^{n_e}(p_i+2t_i)+
\sum_{1\leq i<j\leq n_e}(b_{ij}+d_{ij})
$$
\begin{proc}\label{CARD}
Two clusters are equivalent if and only if their cards are equal. 
\end{proc}
\begin{proof}
If all basic groups composing a cluster $\Phi$ are known, then completing it casings and 
connectives are uniquely restored. The proof is essentially the same as that 
of proposition\,\ref{ID-CARD}, and we omit the details.
\end{proof}

\subsection{Tee-generated-clusters and similar.}\label{t-generated}

A cluster $\Phi$ such that $\Psi=\Phi_0$ is a tee-cluster will be called \emph{tee-generated}, 
or, more precisely, $\Psi$-\emph{generated}. Framed tee-clusters belongs to this class but
do not exhaust it. Namely, we have
\begin{proc}\label{framed cluster}
Let $\Phi$ be a tee-generated cluster.Then either $\Phi=\overline{\Phi_0}\cup\Phi_{sp}$, 
or $\Phi_0$ is a $(n_t,2\,|\,n_{tr})$-hybrid and $\Phi=\Phi_0\cup\{\theta\}$ where $\theta$ is the 
double whose ends are end vertices of $\Phi_0$.
\end{proc}
\begin{proof}
First, note that $\Phi$ does not contain d-bridges. Indeed, the axis of a d-bride 
$\mathbf{Bd}\in\Phi$ belongs to $\Phi_0$. On the other hand, by the basic property of
d-bridges, the center $C$ of this axis, which is also the center of $\mathbf{Bd}$, is a
center vertex of $\Phi_0$. The axis is the unique tee in  $\Phi_0$ that passes
through $C$. But there are no tee-clusters with such a center vertex (see, for instance, 
the list of tee-clusters in subsection \ref{tee-classification}).

Second, if $\Phi$ contains rotators $\mathbf{r_1},\dots,\mathbf{r_l}$, then the ends 
$E_{i1}, E_{i2}$ of the axis of $\mathbf{r_i}, \,i=1,\dots,l,$ are end vertices of $\Phi_0$. 
If $l>1$, then the tee $\lfloor E_{i1},E_{j1}|C\rceil$ with $C$ being a center vertex of 
$\Phi_0$ and $i\neq j$ is compatible with $\Phi_0$ but does not belong to it.  This 
proves that  $\Phi_0$ is not a tee-cluster if $l>1$. If $l=1$, then the tee 
$\lfloor E_{11}, E|C\rceil$ with $C$ being a center vertex of $\Phi_0$ is compatible with 
$\Phi_0$ but does not belong to it if $E$ is an end vertex of $\Phi_0$, which is different
from $E_{11}$ and  $E_{12}$. So, $\Phi_0$ is not a tee-cluster, if $n_e(\Phi_0)>2$. 
The last condition is, obviously, equivalent to $n_e(\Phi)>0$. On the contrary, if 
$n_e(\Phi)=0$ and $l=n_r(\Phi)=1$, $\Phi_0$ is a tee cluster, namely, a 
$(n_c,2\,|\,n_{tr})$-hybrid. In this case $\Phi_{sp}=\Phi_1^{im}=\emptyset$. 

Thus if $n_r=0$, then $\Phi_1=\Phi_1^{im}\cup\Phi_1^{sp}\Leftrightarrow\Phi=
\Phi_0\cup\Phi_1^{im}\cup\Phi_1^{sp}=\overline{\Phi_0}\cup\Phi_{sp}$. To 
conclude the proof it remains to note that $\overline{\Phi_0}$ is a cluster,
if $\Phi_0$ is a tee-cluster and $n_r=0$.
\end{proof}
A direct consequence of the above proposition is:
\begin{cor}\label{double clusters}
A tee-cluster $\Psi$ is also a cluster if and only if it is not a $(k,2|m)$-hybrid 
and $\mathbf{p}(\Psi)=0.  \quad\square$
\end{cor}
 If $\Phi$ is a framed tee-cluster, then 
$\Phi=\Phi_0\cup\Phi_1^{im}$. The converse is not true as the following examples show.

Consider the compatible family $\Psi_{k,l}$ composed of dees 
$\lfloor e_i|e_{k+i}\rceil, \;1\leq i\leq k,$ and tees 
$\lfloor e_i,e_j|e_{k+i}\rceil$ where $1\leq i\leq k,$ and $j\in\{1,\dots,k\}\cup\{2k+1,\dots,2k+l\}$. 
We also assume that $k\geq 1$ and $l\geq 1$, if $k=1$. 

\begin{lem}\label{psi-clusters}
If $(k,l)\neq (1,1)$, then the family $\Psi_{k,l}$ is included in a unique cluster, denoted 
$\bar{\Psi}_{k,l}$ such that $S(\bar{\Psi}_{k,l})=S(\Psi_{k,l})$. More exactly, we have:
\begin{enumerate}
\item  $\bar{\Psi}_{k,l}=\Psi_{k,l}$, i.e., $\Psi_{k,l}$ is a cluster, if $k\geq 3$.
\item  $\bar{\Psi}_{1,l}$ is an $(l+1)$-dimensional framed pyramid, and $e_1$ is its top vertex.
\item $\bar{\Psi}_{2,0}\setminus\Psi_{2,0}=\{\lfloor e_1|e_{4}\rceil, \,\lfloor e_2|e_{3}\rceil\}$,
i.e., $\bar{\Psi}_{2,0}$ is composed of two d-bridges, which have common end vertices 
$e_1$ and $e_2$. 
\item If $l\geq 1$, then 
$\bar{\Psi}_{2,l}\setminus\Psi_{2,l}=\{\lfloor e_1, e_2|e_{4+i}\rceil\}_{1 \leq i\leq l}$, i.e.,
$\bar{\Psi}_{2,l}$ is the system of $l$ bridges connecting $e_1$ and $e_2$, which are 
``suspended" on the corresponding connectives to singles $\lfloor e_1, e_3\rceil$ and 
$\lfloor e_2, e_4\rceil$. 
\end{enumerate}
\end{lem}
\begin{proof}
If $k\geq 3$, then all end vertices of $\Psi_{k,l}$ are stable. This proves the first assertion.
The remaining ones are by a simple direct check. 
\end{proof}
It follows from this lemma that $\Phi=\bar{\Psi}_{k,l}$ is a cluster of the form 
$\Phi=\Phi_0\cup\Phi_1^{im}$.  If $k\geq 2$ and $(k,l)\neq (2,0)$, then
 $\Phi_0$ is not a tee-cluster,
 
 In the sequel we shall use the notation  $\bar{\Psi}_{k,l}$ also for a cluster, which is equivalent 
 to the above described model. Let $p_i\geq 1, \,i=1,\dots,k,$ and $\mathbf{p'}=(p_1,\dots,p_k)$.
 The cluster denoted by $\bar{\Psi}_{k,l}(\mathbf{p'})=\bar{\Psi}_{k,l}(p_1,\dots,p_k)$ is obtained 
 from $\bar{\Psi}_{k,l}$ by applying to the latter $(p_i-1)$-times the pyramid join operation at 
 $e_i$ for all $i=1,\dots, k$.
 
 The ``suspended" version of $\bar{\Psi}_{k,l}(\mathbf{p}')$  is obtained by adding to it $m$ 
 center vertices and the corresponding casings. It will be denoted by 
 $\hat{\Psi}_{k,l}(\mathbf{p}';m)$, or simply $\hat{\Psi}_{k,l}(m)$, if $\mathbf{p}'=(1,\dots,1)$. 
 For instance, $\hat{\Psi}_{1,l}(m), \,l>1,$ is a suspended framed pyramid.

\subsection{Description of clusters}\label{Construction clusters}

As in the case of tee-clusters proposition\,\ref{CARD} reduces classification of generic 
clusters to description of their cards. We subdivide this problem into four separate cases:
\begin{enumerate}
\item clasters $\Phi$ with $n_d>0\Leftrightarrow \Phi_{sp}\neq\emptyset$.
\item clasters $\Phi$ with $n_d=0, \;n_{tr}+n_r>0\Leftrightarrow \Phi_{sp}=\emptyset, 
\;\Phi_h\cup\Phi_{rt}\neq\emptyset$.
\item clasters $\Phi$ with $n_d=n_{tr}=n_r=0, \;n_t>0\Leftrightarrow \Phi_{sp}\cup 
\;\Phi_h\cup\Phi_{rt}=\emptyset$.
\item clasters $\Phi$ with $n_d=n_{tr}=n_r=n_t=0\Leftrightarrow \Phi_{FW}=
\emptyset$.
\end{enumerate}
According to this subdivision, clusters will be called of \emph{types} I,...,IV, respectively.
\newline

 {\bf Type I.} Recall that  $n_e\geq 3$, if $n_d>0$ (lemma\,\ref{cluster-d-center}), i.e., 
 $\Phi_{sp}$ has at least 3 end vertices,  which, at the same time,  are end vertices 
of $\Phi$. If $n_t>0$, then $\Phi_{mp}\neq\emptyset$. Any center vertex of the 
multiped $\Phi_{mp}$ is the center vertex of a contained in $\Phi_{mp}$ tripod and, 
therefore, is stable. By this reason, $\tilde{\Phi}=\Phi_{mp}\cup\Phi_{sp}$ is a cluster,
and $\Phi$ is obtained from $\tilde{\Phi}$ by means of suitable join operations. So,
in this case, $n_{tr}, n_r$ and $[\mathbf{t},\mathbf{p}, \mathbf{B},\mathbf{D}]$ may be 
arbitrary. 

If $n_t=0$, then, obviously,   $n_{tr}=n_r=0, \,\mathbf{t}=0$. In this case, $\Phi$ 
is obtained from $\tilde{\Phi}=\Phi_{sp}$ by means of bridge, d-bridge and (framed) 
pyramid join operations. Since $n_t=0$, a bridge connecting end points $E_i$ and 
$E_j$ may exist only if $p_i$ and $p_j$ are nonzero. Since $p_i, p_j, b_{ij}$ are 
nonnegative, this condition is equivalent to $p_ i p_j b_{ij}\geq b_{ij}$. 
Thus cards of clusters of type I are completely described by the following relations:
$$\mathbf{\mbox{\bf{\underline{Type I}\,:}}\quad n_d>0.}$$
\begin{equation}\label{list-type I}
\begin{array}{l}
\mathrm{I_{+}}:\;n_t>0, \,n_e\geq 3.\\
\mathrm{I_0}:\; n_t=n_{tr}=n_r=0, \,n_e\geq3,  \,p_ i p_j b_{ij}\geq b_{ij}, 
\,1\leq i,j \leq n_e.
\end{array} 
\end{equation}

 {\bf Type II.} In this case $n_t>0$ if $\dim\,\Phi>3$. If $n_e>3$, then $\Phi_{mp}$ is a cluster
 with stable ends as well as $\hat{\Phi}=\Phi_h\cup\Phi_{rt}\cup\Phi_{mp}$. Moreover, 
 $\Phi$ and $\hat{\Phi}$ have the same center and end vertices. This shows that 
 $\Phi$ is obtained from $\hat{\Phi}$ by means
 of p-twain, (framed) pyramid, bridge  and d-bridge join operations, i.e., in this case 
 $[\mathbf{t},\mathbf{p}, \mathbf{B},\mathbf{D}]$ may be arbitrary. So, it remains to 
 analyze the case $n_e\leq 3$. We subdivide it into three subcases according to the 
 number $n_{rf}$ of end vertices of $\Phi_{rf}$. Obviously, in the considered case 
 $n_{rf}=0, 2, 3$.
 
 If $n_{rf}=3$, then the s-raft $\Phi_{rf}$ contains two d-bridges, which have one common 
 end vertex, say, $E_1$. In other words, ends of $\Phi_{rf}$ coincide with ends of $\Phi$, and, 
 moreover, $\Psi=\Phi_h\cup\Phi_{rt}\cup\Phi_{rf}\subset\Phi$ is a cluster with the same 
 center and end vertices as $\Phi$. By this reason   $\Phi$ is obtained from $\Psi$ by means 
 of p-twain, (framed) pyramid and bridge join operations. This shows that in the considered 
 case $[\mathbf{t},\mathbf{p}, \mathbf{B}]$ may be arbitrary.
 
 In the case $n_{rf}=2, \,n_e=3$ numerate end vertices $E_1, E_2, E_3$  of $\Phi$ so 
 that $E_1, E_2$ be common ends of the belonging to $\Phi$ d-bridges. So, end vertices 
 $E_1, E_2$ are automatically stable, while $E_3$ is stable if and only if there is either a 
 dee belonging to $\Phi$ with the origin at $E_3$, or
 a rooted at $E_3$ tee $\theta\in\Phi$ whose second end differs from 
 $E_1, \,E_2$. A direct item-by-item examination shows that this occurs only if one of the
 following three conditions holds:
 \begin{enumerate}
 \item $p_3>0\Leftrightarrow$ the framing of the pyramid $p_{E_3}$ is nonempty.
 \item $p_3=0, \,t_3>0$.
 \item $p_3=t_3=0, \,\mathbf{B}\neq 0$.
 \end{enumerate}
 By using suitable join operations we easily find that there are no more limitations
 on $\mathrm{Card}(\Phi)$. 
 
 Let $n_{rf}=2, \,n_e=2$ and $\mathbf{Br}$ be the union of d-bridges in $\Phi$. In this case
 $\tilde{\Phi}=\mathbf{Br}\cup\Phi_{mp}\cup\Phi_h\cup\Phi_{rt}$ is a 
 cluster such that $S(\tilde{\Phi})=S(\Phi)$. Hence $\Phi$ is obtained from $\tilde{\Phi}$
 by means of suitable join operations. This shows that $[\mathbf{t},\mathbf{p}, \mathbf{B}]$
 is arbitrary in the considered case.
 
 If $n_{rf}=0$, then $\Phi_1=\Phi_1^{im}\cup\Phi_1^{db}$. Denote by $n_e^+$ the number of 
 end vertices of $\Phi$ for which $p_E\neq 0$ and by $\mathbf{p}^+$ the arithmetic vector 
 obtained from the vector $\mathbf{p}$ by canceling all its zero components. In the 
 considered case $0\leq n_e^+\leq 3$. We shall label the occurring subcases by couple
 $(n_e, n_e^+)$. Below  $E_1, E_2, E_3$ stand for end vertices of $\Phi$.
 
 $(n_e, n_e^+)=(3,3)$: In this case $\Phi$ contains the cluster 
 $\hat{\Psi}_{3,0}(\mathbf{p}^+;n_t)$ (see subsection\,\ref{t-generated}).
 Then $\tilde{\Phi}=\hat{\Psi}_{3,0}(\mathbf{p^+};n_t)\cup\Phi_{h}\cup\Phi_{rt}$ is a cluster 
 as well and $S(\Phi)=S(\tilde{\Phi})$. So, $\Phi$ is obtained from $\tilde{\Phi}$ by means
 of s-twain and bridge join operations. Hence in this case $p_i>0, \,i=1, 2, 3,$
 and $[\mathbf{t}, \mathbf{B}]$ is arbitrary. 
 
 $(n_e, n_e^+)=(3, 2)$: Assume that the end vertices $E_1, E_2, E_3$ are enumerated 
 in such a way that $p_3=0$ and consequently $p_1p_2>0$. Then  $E_1, E_2$ are stable 
 end vertices, while $E_3$ is such one if either $t_3>0$, or $\Phi$ contains a bridge with
 one of its ends at $E_3$ (equivalently, $b_{13}+b_{23}>0$). Indeed, otherwise
 $\lfloor E_1, E_2|E_3\rceil\notin \Phi$ is compatible with $ \Phi $. Moreover, any of
 this conditions implies that $\Phi$ is a cluster.
 
$(n_e, n_e^+)=(3,1)$: Similarly, assume that $p_1>0, \,p_2=p_3=0$.  
Tees $\lfloor E_1, E_2|E_3\rceil$ and $\lfloor E_1, E_3|E_2\rceil$ deprive $E_2$ and 
$E_3$ of their status of end vertices of $\Phi$. Therefore, they are blocked. It is easy 
to see that in the considered situation the tees are among tees belonging to
treys and bridges of $\Phi$. An item-by-item examination shows that it occurs only
 in one of the following situations:
  \begin{enumerate}
 \item $t_2t_3>0$.
 \item $t_2>0, \,t_3=0, b_{13}>0$.
 \item $t_2=t_3=0, \,b_{12}b_{13}>0\Leftrightarrow$ there are bridges in $\Phi$ 
 connecting $E_1$ with $E_2$ and $E_3$.
 \item $t_2=t_3=b_{12}=b_{13}=0, \,b_{23}>0\Leftrightarrow$ there is a bridge in $\Phi$ 
 connecting $E_2$ and $E_3$.
 \end{enumerate}
 Alternatively, this list is equivalent to the following list of inequalities
 $$
 (1) \;t_2t_3>0,\quad(2)\;t_2b_{13}>0 \quad(3)\;b_{12}b_{13}>0 \quad(4)\;b_{23}>0.
 $$
Moreover, these conditions, as it is easily verified, guarantee that $\Phi$ is
a cluster.
 
$(n_e, n_e^+)=(3,0)$:  Since $n_e=3$, the multiped $\Phi_{mp}$ has stable center vertices.
By this reason, $\tilde{\Phi}=\Phi\setminus(\Phi_{h}\cup\Phi_{rt})$ is a cluster as well. On 
the other hand,  $\tilde{\Phi}$ does not contain dees and, therefore, is a tee-cluster. The 
clusters that simultaneously  are tee-clusters (``double clusters") are described in 
corollary\,\ref{double clusters}, and their cards are easily extracted from the lists in 
section\,\ref{coaxial}. So, in the considered case all clusters are obtained from ``double 
clusters" with $n_e=3$ by means of double and triangle join operations.

$(n_e, n_e^+)=(2,2)$: Let $E_1, E_2$ be end vertices of $\Phi$ and $\Xi$ a tee-family
composed of tees of the form $\lfloor E_i, B_i|C\rceil$, \,i=1, 2, with $B_i$ running through 
bottom vertices of $P_{E_i}$ and $C$ trough center vertices of $\Phi$. Obviously, $Xi$
belongs to the union of casings of $P_{E_i}$'s. Then
$\Psi_{2,0}\cup\Xi\cup\Phi_{mp}\subset\Phi$ is a cluster and $\Phi$ is obtained from it
by means of double, p-twain, bridge and triangle join operations. Hence in the 
considered case $n_{r}, n_{tr}, \mathbf{t}, \mathbf{B}$ are arbitrary, while $\mathbf{p}$
is subjected by the condition $p_1p_2>0$.

$(n_e, n_e^+)=(2,1)$: Let $E_1, E_2$ be end vertices of $\Phi$. We may assume that 
$p_1>0, \,p_2=0$. Obviously, tees $\lfloor E_1, B|E_2\rceil$ with $B$ running bottom 
vertices of $P_{E_1}$ do not belong to $\Phi$ and,
therefore, are blocked. As earlier, the blocking base structures in the considered context 
are among tees belonging to treys and bridges of $\Phi$. A simple check shows that the 
blocking tees may come either from a rooted at $E_2$ trey, or from two bridges connecting
$E_1$ and $E_2$. This implies that, in the considered case $\Phi$ is a cluster iff either
or both of the inequalities $t_2>0$ and $b_{12}\geq 2$ holds.

$(n_e, n_e^+)=(2,0)$: In this case $\Phi_h\cup\Phi_{rt}\cup\Phi_{mp}$ is, as it is easily verified,
a cluster with the same set of vertices as $\Phi$. So, $\Phi$ is obtained from it by means of
p-twain and bridge join operations. This shows that $[\mathbf{t}, \mathbf{B}]$ is arbitrary in
the considered case.

$(n_e, n_e^+)=(1,1)$: In this case $\Phi_1^{im}$ contains only one dee, say, 
$\theta=\lfloor E|B\rceil$ and, obviously, $\mathbf{B}=0$. This dee belongs to a poker 
$\Pi=\{\theta,  \lfloor E, C|B\rceil\}\subset \Phi$. In the considered context $C$ may be 
either a bottom vertex of $P_E$, or a side vertex of a trey in $\Phi$.  As previously, this
proves that in the considered case $\Phi$ is a cluster iff one of the inequalities 
$p=p_1\geq 2$ or $t=t_1>0$ holds.

$(n_e, n_e^+)=(1,0)$: The family $\Phi\setminus(\Phi_h\cup\Phi_{rt})$ consists of a
number of multi-treys, which are rooted at the unique end vertex, and their casings. 
This number is positive. Indeed, otherwise the graph $\Upsilon_{\Phi}$ would non be 
connected. Since $\Phi$ is a cluster, this end vertex is stable. This is so iff 
$t=t_1\geq2$. Moreover, this condition guarantees that $\Phi$ is a cluster.\newline
 
Ultimately, we get the following list of clusters of type II where
$\mathbf{D}=\|d_{ij}\|, \,\mathbf{B}=\|b_{ij}\|$. In this and other lists that follow the
alternatives which are ``embraced" by square bracket do not exclude one another.

$$\mathbf{\mbox{\bf{\underline{Type II}\,:}}\quad n_t>0, \,n_d=0, \;n_{tr}+n_r>0}.$$
 \begin{equation}\label{list-type II}
\begin{array}{l}
\mathrm{II_{+}}: \,n_e>3.\\
\mathrm{II_{32}}: \,n_e=3, \; d_{12}d_{13}>0 . \\
\mathrm{II_{31}}: \,n_e=3,  \;d_{12}>0 \Rightarrow
\left \{
\begin{array}{l}
 p_3>0.\\
p_3=0, \,t_3>0.\\
p_3=t_3=0, \,\mathbf{B}\neq 0.
\end{array}
\right. \\
\mathrm{II_{30}}: \,n_e=3, \,\mathbf{D}=0 \Rightarrow
\left \{
\begin{array}{l}
p_1p_2p_3>0.\\
p_1p_2>0, \,p_3=0 \Rightarrow 
\left [
\begin{array}{l}
t_3>0.\\b_{13}+b_{23}>0.
\end{array}
\right. \\
p_1>0, \,p_2=p_3=0 \Rightarrow 
\left [
\begin{array}{l} 
t_2t_3>0. \\t_2b_{13}>0. \\b_{12}b_{13}>0. \\b_{23}>0.
\end{array}
\right. \\
\end{array} 
\right.\\
\mathrm{II_{300}}: \,n_e=3, \,\mathbf{D}=0, \,\mathbf{p}=0 \Rightarrow
\left \{
\begin{array}{l}
t_1t_2t_3>0.\\
t_1t_2>0, \,t_3=0, \,b_{13}+b_{23}>0.\\
t_1>0, \,t_2=t_3=0 \Rightarrow 
\left [
\begin{array}{l}
b_{23}>0.\\b_{12}b_{13}>0.
\end{array} 
\right.\\
\end{array} 
\right.
\end{array}
\end{equation}
\begin{equation}\label{list-type II2}
\begin{array}{l}
\mathrm{II_2}: \,n_e=2, \;\mathbf{D}\neq 0. \\
\mathrm{II_{22}}: \,n_e=2, \,p_1p_2>0, \,\mathbf{D}=0.\\
\mathrm{II_{21}}: \,n_e=2, \,p_1>0, \,p_2=0,  \,\mathbf{D}=0 \Rightarrow
\left [
\begin{array}{l}
t_2>0.\\ b_{12}\geq 2.
\end{array} 
\right.\\
\mathrm{II_{20}}: \,n_e=2, \,p_1=p_2=0, \,\mathbf{D}=0.\\
\mathrm{II_{11}}: \,n_e=1 \Rightarrow
\left \{
\begin{array}{l}
p_1t_1>0.\\
p_1\geq 2, \,t_1=0.
\end{array} 
\right.\\
\mathrm{II_{10}}: \,n_e=1, \,p_1=0, \,t_1\geq 2.
\end{array}
\end{equation}
\newline

 {\bf Type III.} First, note that joining to a cluster of type III some rotators and triangles
 we get a cluster of type II. So, clusters of type III are among families of the form
 $\Phi_{cstr}=\Phi\setminus(\Phi_h\cup\Phi_{rt})$ with $\Phi$ being a cluster of type II. 
 So, we shall get a description of clusters of type III just by running through lists 
 (\ref{list-type II}) and  (\ref{list-type II2}) of clusters of type II and singling out those of
 them for which families $\Phi_{cstr}$ are clusters. It should be stressed that the end 
 and center vertices of $\Phi$ are those of $\Phi_{cstr}$.
 
 It is easy to see that  a compatible family $\Phi_{cstr}$
 is a cluster iff its center vertices are stable. This is so, if $n_e\geq 3$.
 Indeed, in this case the multiped $\Phi_{mp}\subset\Phi_{cstr}$ has stable center 
 vertices, and these are center vertices of $\Phi_{cstr}$. Hence relations of list
 (\ref{list-type II}) describe clusters  also in the case when
 $\Phi_h\cup\Phi_{rt}=\emptyset$ also.
 
 Let $n_e=2$ and $E_1, E_2$ be end vertices of $\Phi$.  In this case center vertices 
 of $\Phi_{cstr}$ are stable, if $\Phi$ (equivalently, $\Phi_{cstr}$) contains a trey or a 
 bridge or two nonempty pyramids $P_{E_1}$ and $P_{E_2}$. Indeed, 
 $\Phi_{pm}$ together with the casing of 
 any of these families contains a tripod or a cross and, at the same time, is contained 
 in $\Phi_{cstr}$. In other words, $\Phi_{cstr}$ may not be a cluster only if relations
 \begin{equation}\label{list-3}
 p_1p_2=0, \quad\mathbf{t}=0, \quad\mathbf{B}=0.
 \end{equation}
 holds. So, it remains to describe clusters whose cards satisfy relations (\ref{list-type II2})
 and (\ref{list-3}). We shall examine cases $\mathrm{II_2},\dots,\mathrm{II_{20}}$  (see
 \ref{list-type II2}) one after another. 

 $\mathbf{II_2.}$ Assume that $P_{E_1}\neq\emptyset, \,P_{E_2}=\emptyset$,
 i.e., that $p_1>0, \,p_2=0$, and let $C$ be a center vertex of $\Phi$. Then 
 $\theta=\lfloor E_1, E_2|C\rceil$ is the only tee in $\Phi_{cstr}$ with the center
 at $C$. Since $\mathbf{t}=0, \,\mathbf{B}=0$, this shows that all tees from
 $\Phi_{cstr}$ are rooted at $E_1$. Let $B$ be a bottom vertex of $P_{E_1}$ 
 or $C_{E_1}$ if $p_1=1$. Then $\rho=\lfloor E_1, C|B\rceil\notin\Phi_{cstr}$. 
 But being rooted at $E_1 \,\rho$ is compatible  with $\Phi_{cstr}$. 
 Hence in the considered case $\Phi_{cstr}$ is not a cluster, if one of 
 pyramids $P_{E_i}$ is nonempty.
 
 In the remaining subcase $\mathbf{p}=0$ assume that $C_1, C_2$ are center
 vertices of $\Phi_{cstr}$. Then $\theta=\lfloor E_1, C_1|C_2\rceil\notin\Phi$. But
by the same reasons as above, $\theta$ is compatible with $\Phi_{cstr}$. So, 
$\Phi_{cstr}$ is not a cluster, if $n_t\geq 2$. On the contrary, this is so, if $n_t=1$. 
The corresponding cluster is an s-raft composed of one d-bridge and having one 
center vertex. Thus under conditions  $\mathbf{II_2}$ and (\ref{list-3}) $\Phi_{cstr}$ 
is a cluster iff $\mathbf{p}=0, \,n_t=1$.
   
$\mathbf{II_{22}} - \mathbf{II_{21}.}$ The corresponding relations in (\ref{list-type II2})
are, obviously, inconsistent with (\ref{list-3}). So, in the considered case $\Phi_{cstr}$
is a cluster.
 
 $\mathbf{II_{20.}}$ It follows from (\ref{list-3}) that in this case $\Phi_{cstr}$ consists
 of only one tee and hence is not a cluster.\newline

 Finally, consider the situation when $n_e=1$. 
 In the case $\mathbf{II_{11}}$ center vertices of $\Phi_{cstr}$ are stable only if 
 $p_1t_1>0$. Indeed, the dee $\lfloor E|C\rceil$ where $E=E_1$ is the end vertex 
 of $\Phi$ and $C$ is one of its center vertices is compatible with $\Phi_{cstr}$ but
 does not belong to it. The condition $t=t_1\geq 2$ guarantees stability of 
 $\Phi_{cstr}$ in the case $\mathbf{II_{10}}$. So, if $n_e=1$, then $\Phi_{cstr}$ is a
 cluster either or both $p_1t_1>0$ and $t_1\geq 2$.
 
 The results concerning type III are synthesized in the following list:\newline
 $$\mathbf{\mbox{\bf{\underline{Type III}\,:}}\quad n_t>0, \,n_d=n_{tr}=n_r=0}.$$
 \begin{equation}\label{list-type III}
\begin{array}{l}
\mathrm{III_3}: \mbox{the same relations as in (\ref{list-type II})}.\\ 
\mathrm{III_2}: \,n_e=2, \;p_1p_2>0.\\
\mathrm{III_{2d}}: \,n_e=2, \;n_t=1,  \;\mathbf{p}=\mathbf{t}=0,  \;\mathbf{B}=0, 
\;\mathbf{D}\neq 0.\\
\mathrm{III_{21}}: \,n_e=2, \,p_1>0, \,p_2=0,  \,\mathbf{D}=0 \Rightarrow
\left [
\begin{array}{l}
t_2>0.\\ b_{12}\geq 2.
\end{array} 
\right.\\
\mathrm{III_{11}}: \,n_e=1, \,p_1t_1>0.\\
\mathrm{III_{10}}: \,n_e=1, \,p_1=0, \,t_1\geq 2.
\end{array}
\end{equation}
\newline

 {\bf Type IV.} Clusters of this type are made of framed pyramids (including singles), 
 bridges, d-bridges and the corresponding connectives. These connectives are of 
 the form $\lfloor E,A|B\rceil$ where $E$ is an end vertex of  $\Phi, \,B$ is a 
 bottom vertex of $P_{E}$ and $A$ is either an end vertex  of  $\Phi$ or the center 
 of a bridge contained in  $\Phi$. They will be called \emph{e-connectives} and 
 \emph{b-connectives}, respectively. Denote by $\Psi_{ \Phi}$ the family composed 
 of all framings of $P_{E}$'s and all e-connectives. Also, recall that $n_e^+$ stands 
 for the number of 
 end vertices $E$ of $\Phi$ such that $p_E\neq 0$ and $\mathbf{p}^+$ for the 
 vector obtained from $\mathbf{p}$ by canceling its zero components. End vertices 
 $E$ of $\Phi$ for which $p_E=0$ will be reffered as \emph{free}. We shall 
 classify clusters of type IV according to the value of  $n_e^+$.  
 
 $n_e^+\geq 3$: Since $n_e^+\geq 3$, it follows from 
 lemma \,\ref{psi-clusters},\,(1), that $\Psi_{ \Phi}$ is equivalent to 
 $\bar{\Psi}_{k,l}(\mathbf{p}^+)$ with $k=n_e^+$ and $l=n_e-n_e^+$ 
 (see subsection\,\ref{t-generated}) and hence is a cluster. Since  
 $\Psi_{ \Phi}$ and of $\Phi$ have common end vertices, $\Phi$ is obtained
 from $\Psi_{ \Phi}$ by means of bridge and d-bridge join operations. Hence 
 $[\mathbf{B}, \mathbf{D}]$ may be arbitrary in this case.
  
 $n_e^+=2$: Let $E_1, E_2$ be non-free end vertices of $\Phi$, i.e., $p_1p_2>0$ 
 and $p_i=0$ if $i>2$. A cluster of the considered type is completely characterized 
 by the following two tautological properties:
 \begin{enumerate}
 \item all its free end vertices are stable;
 \item dees $\lfloor E_1 | B_2\rceil$ and $\lfloor E_2 | B_1\rceil$ with $B_i$ being a bottom
 vertex of $P_{E_i}, \,i=1,2,$ are blocked.
 \end{enumerate}
 Examine them separately.
 
(1) First, note that free end vertices of $\Phi$ are not stable as end vertices of 
 $\Psi_{ \Phi}$ (see lemma\,\ref{psi-clusters}).  Since such a free  end vertex is not
 a vertex of a b-connective, it is stable only if it is an end of a d-bridge belonging to $\Phi$.
 In other words,  in the considered context free end vertices of a cluster $\Phi$ are among 
 end vertices of the raft $\mathcal{R}_{\Phi}$ formed by all d-bridges contained in $\Phi$.
 This may by expressed algebraically in terms of the following inequalities:
\begin{equation}\label{br-IV}
\alpha_i(\mathbf{p}, \mathbf{D})\df p_i+\sum_{j=1}^{n_e}d_{ij}>0, \quad i=1,\dots,n_e.
\end{equation}

(2) The dees in question are blocked if any bottom vertex $B_i$ of $P_{E_i}$ is the center
of at least two pokers, which belong to $\Phi$ rooted at $E_i, \,i=1,2$. This manifestly 
takes place if one of the following conditions is fulfilled: (a) $P_{E_i}$ is not a single 
$\Leftrightarrow p_i>1$ for $i=1,2$; (b) there is at least one bridge in $\Phi$
$\Leftrightarrow\mathbf{B}\neq 0$; (c) $\Phi$ possesses at least one free end vertex
$\Leftrightarrow n_e>2$. On the other hand, one easily sees that $\Phi$ is not a cluster,
if none of this conditions is satisfied. 

Thus in the considered case an abstract card is the card of a cluster if inequalities
(\ref{br-IV}) and one of conditions (a) - (c) hold.

$n_e^+=1$: First, note that $\mathbf{B}=0$ in this case. Assume that $p_1>0$ and 
$p_i=0$ if $i>1$. The same arguments as previously show that free end vertices
of  $\Phi$ are end vertices of the raft $\mathcal{R}_{\Phi}$, i.e., inequalities (\ref{br-IV})
hold in the considered situation too. Moreover, it is easily verified that this is an unique
restriction on $\mathrm{Card}(\Phi)$, if $\Phi$ has at least two free end vertices, i.e., if
$n_e\geq 3$. If $n_e=2$, then, obviously, $\Phi$ is a cluster iff 
$d_{12}>0\Leftrightarrow \mathbf{D}\neq 0$ and $p_1\geq 2$. Finally, a framed 
pyramid is the only cluster if $n_e=1$, i.e., if $p_1\geq 2$.

 $n_e^+=0$:  In this case $\Phi=\mathcal{R}_{\Phi}$ is a raft cluster. Its card is 
 characterized by two obvious requirements: $n_e\geq 2$ and the matrix 
 $\mathbf{D}$ is ``connected".  The last means that the associated with $\mathbf{D}$ 
 graph $\Upsilon_{\mathbf{D}}$ whose $i$-th and $j$-the vertices are connected by 
 $d_{ij}$ edges, $1\leq i,j,\leq n_e,$ is \emph{connected}.  Indeed, in the considered
context connectedness of $\Upsilon_{\mathbf{D}}$ is equivalent to that of 
$\Upsilon_{\Phi}$. 
 
 $$\mathbf{\mbox{\bf{\underline{Type IV}\,:}}\quad n_t=n_d=n_{tr}=n_r=0}.$$
 \begin{equation}\label{list-type IV}
\begin{array}{l}
\mathrm{IV_+}: p_1p_2p_3>0.\\ 
\mathrm{IV_2}: p_1p_2>0, \,p_i=\alpha_i(\mathbf{p}, \mathbf{D})=0, 
\mathrm{if} \,i>2 \Rightarrow
\left [
\begin{array}{l}
p_1>1, \,p_2>1.\\
\mathbf{B}\neq 0.\\
n_e>2.
\end{array} 
\right.\\
\mathrm{IV_1}: \,p_1>0,  \,p_i=0, \mathrm{if} \,i>1 \Rightarrow
\left \{
\begin{array}{l}
n_e\geq 3.\\
n_e=2, \,p_1\geq 2, \,\mathbf{D}\neq 0.\\
n_e=1, \,p_1\geq 2.
\end{array} 
\right.\\
\mathrm{IV_1}: \,\mathbf{p}=0, \,\mathbf{D}  \;\mbox{is connected}.
\end{array}
\end{equation}

Thus we have proven
\begin{thm}\label{cluster classification}
Equivalence classes of clusters are in one-to-one correspondence with equivalence 
classes of cards listed in
 $\mathrm{(\ref{list-type I}),
(\ref{list-type II}), (\ref{list-type II2}), (\ref{list-type III}), (\ref{list-type IV})}$.
\end{thm}

\subsection{Low dimensional clusters}
The use of join operations simplifies much classification of clusters. This, however,
leads to a certain loss of control of dimensions of the so-constructed clusters. In 
particular, the only way to describe $n$-dimensional clusters for a given $n$ is 
to extract a list of them from lists 
$\mathrm{(\ref{list-type I}), (\ref{list-type II}),  (\ref{list-type II2}),  (\ref{list-type III})}$
and $\mathrm{(\ref{list-type IV})}$. This, however, would be  hardly instructive. Moreover, 
the expected result would be too cumbersome to report it here. Below we shall list 
clusters of dimensions $\leq 5$ to illustrate the situation.  To this end we have to
introduce three special clusters before.
\begin{itemize}
\item A \emph{single-center-single cluster}  has one center and two end vertices 
and is composed of two singles and the corresponding connectives and casings.\\
\item \emph{single-twain cluster} has $2$ end vertices in which are rooted a framed
twain and a single together with necessary external connectives.\\
\item  \emph{bridge-d-bridge cluster} consists of one bridge and one d-bridge with
common end vertices, which are suspended from one t-center vertex.
\end{itemize}
A \emph{d-multiplex} is a multiplex to which is added the dee $\lfloor O|C\rceil$ where
$O$ (resp., $C$ is the origin (resp., center) of the multiplex. 

Below we use the notation $\Psi=\Psi'\Join J$, which tells that the family $\Psi$ 
is obtained from the family $\Psi'$ by means of the join operation $J$. 
\newline

\noindent{\bf n=2:} double. \\
{\bf n=3:} triangle, (1,1)-rotator, framed twain, d-bridge.\\
{\bf n=4:} (3,1)-spider, (1,1)-hedgehog, (1,2)-rotator, (2,2)-raft, 
framed 3-pyramid. \\
{\bf n=5:}  (4,1)- and (3,2)-spiders, (3,1)-spider$\Join$d-bridge, 
tripod\,$\Join$(3,1)-spider, (3,1)-spider$\Join$single, (1,2)-hedgehog,  
(1,3)- and (2,1)-rotators,   (4,1)-multiped, 2-bridge cluster, 
single-center-single cluster, bridge-d-bridge cluster,
suspended (2,2)-raft,  trey$\Join$single, 
d-2-plex$\Join$bridge, (2,3)-raft, single-twain cluster, 
(d-bridge$\Join$single)$\Join$single, d-bridge$\Join$(framed twain), 
(3,2)-raft, framed 4-pyramid.\\
\newline

One can see from these lists that their irregularity is due to a small number of end
and center vertices. In that case their stability is not guaranteed on a common basis.
By this reason, it is natural to call a cluster $\Phi$  \emph{stable}, if $n_t\geq 1$ and 
$n_\geq 4$. Indeed, in this case $\Phi_{mp}$ is a cluster with stable end and center 
vertices, so that  the numbers of structural groups (triangles, doubles, etc) composing 
$\Phi$ are no longer constrained by some conditions. \newline

\subsection{On the structure of coaxial Lie algebras}

Basic structure elements of a coaxial algebra $\gG$ can be directly 
read from $\Phi_{\gG}$. The key is the following simple lemma, an analogue of 
lemma\,\ref{coax-sub-id}.
\begin{lem}\label{coax-s-i}
Let  $\gG$ be a Lie algebra associated with a compatible family $\Phi$,
$S_i\subset S(\Phi)$, \,i=1,2, and $V_i$ the subset of $|\gG|$
spanned by $S_i$.  Then 
\begin{enumerate} 
\item the subspace $[V_1,V_2]=\mathrm{Span}\{[v_1,v_2]\,|\,v_i\in V_i, \,i=1,2\}$ of 
$|\gG|$ is spanned by center vertices of all tees $\lfloor e_i,e_j|e_k\rceil\in\Phi$  and 
end vertices of all dees $\lfloor e_p|e_q\rceil\in\Phi$ such that 
$e_i, e_j$ and $e_p, e_q$ belong to different subsets $S_i$, respectively.
\item the subspace of $|\gG|$ spanned by a subset $S\subset S(\Phi)$
is a subalgebra of $\gG$, if center vertices of tees $\theta\in\Phi$ with ends in $S$ 
also belong to $S$.
\item the subspace of $|\gG|$ spanned by a subset $S\subset S(\Phi)$ is an ideal 
of $\gG$, if (1) the center vertex of any tee in $\Phi$  with one end in $S$, and (2) 
the end vertex of any dee in $\Phi$ with the origin in $S$ also belong to $S$.
\item $|[\gG,\gG]|$ belongs to the subspace spanned by center and mixing
vertices of $\Phi_0$ and end vertices of dees  $\varrho\in\Phi_1$.
\item The modular vector of $\gG$ is a linear combination of end vertices of dees 
belonging to $\Phi$.
\end{enumerate}                    
\end{lem}
\begin{proof}
The first assertion is obvious, while 2) - 4) are its immediate consequences. The last 
assertion follows from proposition\,\ref{PFA} and the fact that the modular vector of a 
$\between$-lieon which is proportional to $\lfloor e_p|e_q\rceil$ is proportional to 
$e_q$.
\end{proof}
Now, let $\gG$ be a coaxial Lie algebra such that $\Phi=\Phi_{\gG}$ is a cluster. 
Lemma\,\ref{coax-s-i} allows us to single out  a some basic ideals in $\gG$. For
a subset $S\subset S(\Phi)t$ we shall 
use the notation $S\Rightarrow \gG_{(an \;index)}\subset\gG$ to express
the fact that the subspace $\mathrm{Span}(S)\subset |\gG|$ supports the ideal
$\gG_{(an \;index)}$ of $\gG$.
\begin{equation}\label{cluster id}
\begin{array}{cc}
\mathrm{Span}\{\mbox{t-center vertices}\}\Rightarrow \gG_{c}, & 
 \mathrm{Span}\{\mbox{d-center vertices}\}\Rightarrow \gG_{cd}, \\
S(\Phi_h)\Rightarrow \gG_h,  \quad S(\Phi_{rt}) \Rightarrow \gG_{rt}, &
 S(\Phi\setminus \Phi_h) \Rightarrow \gG_{rad},  \\
S(\Phi\setminus (\Phi_h\cup\Phi_{rt})) 
 \Rightarrow \gG_{sr}, & S(\Phi\setminus \Phi_{rt} )\Rightarrow \gG_{0} .
\end{array}
\end{equation}
The ideal $\gG_c$ belongs to the center of $\gG$. Even more, it coincides with it if the 
linear combination of base lieons of $\Phi$ that defines $\gG$ is generic.  In particular,
$\gG_c$ commute with all ideals (\ref{cluster id}). Nontrivial commutation relations
involving these ideals are
\begin{equation}\label{com-rel-id}
\begin{array}{ccc}
[\gG_{cd}, \gG_{rad}]\subset\gG_{cd},  & [\gG_{cd}, \gG_{sr}]\subset\gG_{cd}, &
[\gG_{cd}, \gG_{0}]\subset\gG_{cd}, \\ 
{[ \gG_{rad}, \gG_{sr}}]\subset \gG_{sr}, &
[\gG_{rad}, \gG_{0}]\subset\gG_{sr},  & [\gG_{sr}, \gG_{0}]\subset\gG_{sr}. 
\end{array}
\end{equation}

Obviously, the quotient algebra $\gG_{smpl}=\gG_h/\gG_c=\gG/\gG_{rad}$ is associated 
with the family $\Phi_0^{tr}$ of al triangles contained in $\Phi$. Hence it is isomorphic to 
the  direct sum of 3-dimensional simple Lie algebras associated with these triangles. In its 
turn this shows that the ideal $\gG_h$ is an abelian extension of $\gG_{smpl}$, i.e., that
$\gG_h=\gG_{smpl}\oplus_{\rho}|\gG_c|$ for a suitable representation $\rho$. 

On the other hand, $\gG_{rad}$ is the radical of $\gG$. To prove this it is sufficient to 
show that $\gG_{rad}$ is solvable, since $\gG_{smpl}=\gG/\gG_{rad}$is semisimple. 
Obviously,  solvability of  $\gG_{rad}$ is equivalent to solvability of 
$\gG_{rad}^{\prime}=\gG_{rad}/\gG_c$. To this end note that the algebra 
$\gG_{rad}^{\prime}$ is associated with the family composed of  $\Phi_1$ and the 
family $\Phi_0^{mix}$ of all tees with mixing (in $\Phi$) center vertices, which are 
rooted at end vertices of $\Phi$. It is easy to see that  $\gG_{rad}^{\prime}$ splits into 
a direct sum of an algebra associated with $\Phi_1^{db}$ and an algebra $\gG''_{rad}$ 
associated with $(\Phi_1\setminus\Phi_1^{db})\cup\Phi_0^{mix}$. The first of them is 
isomorphic to the direct sum of $n_{rt}$ copies of $\between$ and an abelian algebra 
and, therefore, is solvable. On the other hand, $[\gG''_{rad}, \gG''_{rad}]\subset\gG'''_{rad}$ 
where $\gG'''_{rad}\subset\gG''_{rad}$ is an abelian subalgebra generated by mixing 
vertices of $\Phi_0$, center vertices of d-bridges and d-center vertices of $\Phi$. This 
shows that the derived series of $\gG_{rad}^{\prime}$ is of length 2. Consequently,  the 
derived series of $\gG_{rad}$ is of length 3.  Finally, we see that the image of a natural 
imbedding $\gG_{smpl}\rightarrow\gG_{smpl}\oplus_{\rho}|\gG_c|=\gG_h\subset\gG$ is 
the Levi subalgebra of $\gG$.

The algebra $\gG_{\between}=\gG_{rt}/\gG_c$ is associated with $\Phi_1^{db}$
and $\gG_{rt}$ is isomorphic to $\gG_{\between}\oplus_{\varrho}|\gG_c|$ for a suitable
representation $\varrho$. There is a certain similitude between ideals $\gG_h$ and 
$\gG_{rt}$. Namely, $\gG_h$ is an abelian extension of a direct sum of simple 
3-dimensional Lie algebras. These are associated with triangles, which are simplest 
tee-clusters. Similarly, the ideal $\gG_{rt} $ is an abelian extension of a direct sum of 
nonabelian 2-dimensional Lie algebras, which are associated with doubles, i.e., with 
simplest dee-clusters. Since the semisimple part of $\gG$ intrinsically defined, the number
$n_{tr}$ does not depend on the way $\gG$ is assembled from basic lieons. Similarly,
it can be shown that the number $n_{rt}$ is intrinsically defined for ``generic" in a sense
coaxial algebras. This simple example illustrates the fact that $\Phi_{\gG}$ reveal some 
finer details of the structure of $\gG$, which go unnoticed in the standard approach.

\subsection{Infinite-dimensional ``disassemblable" Lie algebras}

There are natural infinite-dimensional analogues of coaxial Lie algebras. Namely,
let $\{e_i\}_{1\leq i<\infty}$ be a base of a numerable vector space $V$ over a field 
$\gk$. Recall that elements of $V$ are,  by definition, linear combinations 
$\sum_{i<\infty}\lambda_ie_i, \,\lambda_i\in\gk,$ with finitely many nonzero 
coefficients. Then a formal combination 
\begin{equation}\label{infinite coaxial}
\sum_{1\leq i,j,k<\infty}\alpha_{ijk}\lfloor e_i, e_j|e_k\rceil +
\sum_{1\leq p,q<\infty}\beta_{pq}\lfloor e_p|e_q\rceil, \;\alpha_{ijk}, \,\beta_{pq}\in\gk,
\end{equation}
defines a Lie algebra structure in $V$, if the figuring in (\ref{infinite coaxial}) base lieons
with nonzero coefficients are mutually compatible. We shall call \emph{countable}
this kind of coaxial algebras.

The pro-finite version of this construction is as follows. Let $W$ be a pro-finite vector space 
and $\{e_i\}_{1\leq i<\infty}$ a (pro-finite)
base of it. Elements of $W$ are formal linear combinations 
$\sum_{i<\infty}\lambda_ie_i, \,\lambda_i\in\gk$. Assume that formal combination 
(\ref{infinite coaxial}) is \emph{locally finite}. This means that for a given $k$ (resp., $q$)
only a finite number of coefficients $\alpha_{ijk}$ (resp., $\beta_{pq}$) is different from 
zero. Then a locally finite combination correctly defines a Lie algebra structure in $W$,
if the occurring base structures with nonzero coefficients are mutually compatible.
\begin{ex}
Let $\{e_{ij}\}_{-\infty< i,j<\infty}$ be a base in a pro-finite vector space $W$. The tee-family 
$\Theta_{ij}, \,i,j\in \Z,$ composed of four tee-structures 
$\lfloor e_{i\pm 1,j}, e_{i,j\pm 1}|e_{ij}\rceil$, is compatible. It is easy to see that $\Theta_{ij}$ 
and $\Theta_{kl}$ are compatible iff $i+j\equiv k+l  \;\mathrm{mod}\,2$. So, each of families
$$
\Theta_{even}=\bigcup_{i+j \,\mathrm{is \;even}}\Theta_{ij}, \qquad 
\Theta_{odd}=\bigcup_{i+j \,\mathrm{is \;odd}}\Theta_{ij}
$$
is compatible but they are not compatible one another. Similarly, the dee-family 
$\mathrm{D}_{ij}$ composed of four dee-structures $\lfloor e_{ij}|e_{i\pm 1,j}\rceil$ and 
$\lfloor e_{ij}|e_{i,j\pm 1}\rceil$ is compatible as well as families 
$$
\mathrm{D}_{even}=\bigcup_{i+j \,\mathrm{is \;even}}\mathrm{D}_{ij}, \qquad 
\mathrm{D}_{odd}=\bigcup_{i+j \,\mathrm{is \;odd}}\mathrm{D}_{ij},
$$
which are not compatible one another. Also, a family $\Theta$ and  a family $\mathrm{D}$
are compatible iff they are of different ``parities". In particular, any formal linear combination
of base structures belonging to $\Theta_{even}\cup\mathrm{D}_{odd}$ defines a Lie algebra
structure in $W$. 

An interesting property of families $\Theta_{even}$ and $\Theta_{odd}$ is that they are
\emph{absolutely incompatible}. This means that any Lie algebras associated
with $\Theta_{even}$ is incompatible with any algebra associated with 
$\Theta_{odd}$. 
\end{ex}
\noindent The above example is naturally related with a 2-dimensional lattice and is easily 
generalized to lattices of higher dimensions.

By combining compatible countable/pro-finite coaxial Lie algebras we can construct infinite-
dimensional  ``disassemblable" Lie algebras of second level, etc. Infinite  analogues of
classical Lie algebras are examples of this kind algebras. For instance, the countable special
orthogonal algebra $\gs\go_{\infty}$ is given by the formal combination
$$
\gs\go_{\infty}(\lambda_1, \lambda_2,\dots)=\sum_{k=1}^{\infty}\lambda_k\mathbf{P}_k \quad\mathrm{with}\quad
\mathbf{P}_k=\sum_{1\leq i,j<\infty}\lfloor e_{ik}, e_{kl}|e_{ij}\rceil, \quad\lambda_k\in\gk,
$$
assuming that $e_{ij}=-e_{ji}$. Countable coaxial algebras corresponding to $\mathbf{P}_k$'s
are compatible each other  and hence disassemble algebra $\gs\go_{\infty}(\lambda_1, \lambda_2,\dots)$.

This way we get a new class of infinite-dimensional Lie algebras, which merits to be studied 
in its own right (see \cite{K} in this connection). 

\section{Some problems and perspectives}\label{problems}
In connection with topics discussed in this paper a series of natural questions arise 
togetherwith various perspective applications. Below we shall mention some of them, 
which, at present, look most interesting. 

{\bf Complete disassembling for arbitrary ground fields}. It is rather plausible that the 
complete disassembling theorem takes place for arbitrary ground fields of characteristic 
zero(see a more detailed discussion at the end of section\,\ref{dis-probl}). But what
about nonzero characteristic?

{\bf Algebraic variety of Lie algebra structures.} The algebraic variety $\mathcal{L}ie(V)$ 
of all Lie algebra structures is an intersection 
 of quadrics in $\mathcal{A}(V)=\mathrm{Hom}_{\gk}(V\otimes V,V)$. The 
subspace of $\mathcal{A}(V)$ spanned by a family of mutually compatible structures 
belongs to $\mathcal{L}ie(V)$. In this sense $\mathcal{L}ie(V)$ is ``woven" of such 
subspaces. This suggests to use this ``web structures" in describing $\mathcal{L}ie(V)$.
An instructive example of this kind is given in \cite{Mor}. Also, a more deep understanding
of the structure of $\mathcal{L}ie(V)$ for algebraically closed ground field $\gk$ could
shed some light on the general disassembling problem.

{\bf Deformations.} On the basis of a disassembling of a Lie algebra $\gG$ one can 
construct some deformations of it by substituting  $\lambda_v\gG_{v}, \,\lambda_n\in\gk,$ 
for $\gG_{v}$ in the corresponding a-scheme (see subsection\,\ref{statemant}). Here 
factors $\lambda_n$ are constrained by some relations, which are absent in the case
of first level algebras. The conjecture that all essential deformations of a given Lie 
algebra are of this kind is, at present, rather plausible.

{\bf Length of disassemblings.} The procedure we have used in the proof of the complete 
disassembling theorem allows to estimate the minimal number of necessary 
for this steps as $n+\mathrm{const}$. On the other hand, classical Lie algebras can be completely 
disassembled in $\leq 4$ steps independently of their dimensions (see section\,\ref{classical}). 
So, a natural question is: whether there is an universal constant $N$ such that any Lie algebra 
can be completely disassembled in no more than $N$ steps. In the case of positive answer 
the problem of describing the variety $\mathcal{L}ie(V)$ passes to a more constructive basis.

{\bf Lie algebras of second level.} It seems to be still possible to  explicitly describe second 
level Lie algebras, i.e., Lie algebras, which can be completely disassembled in two 
steps, as it has been done for coaxial algebras in section\,\ref{coaxial}.  
This could shed, among other things, a new light  on problems mentioned before. 

{\bf Invariants of Lie algebras.} A natural question is : how to construct invariants of a Lie 
algebra if a disassembling of it is known. This question has an evident homological flavour 
and, probably, leads to \emph{constructive} theory of Lie algebras, i.e., a theory, which from
the very beginning considers Lie algebras as compound structures.

{\bf Cohomological aspects.} Poisson (resp., Lie algebra) structures, which are compatible 
with a given one, are closed 2-forms in the associated Lichnerowicz-Poisson 
(resp.,  Chevalley-Eilenberg) complex. This fact was not explicitly exploited in this paper. 
Nevertheless, a more deep
understanding of topics we have discussed here is related with this cohomological aspect.
In this connection we mention that one of methods to completely disassemble a classical 
Lie algebra $\gG$ is to represent the corresponding Poisson bivector $P_{\gG}$ in the 
form $P_{\gG}=\sum_i\ls X_i, P_{\gG}\rs$ for suitable vector fields $X_i$ on $|\gG|$. Here
the terms $\ls X_i, P_{\gG}\rs$ are exact 2-cohains in the associated with $P_{\gG}$ 
Lichnerowicz-Poisson complex.

It is naturally to think that the cohomology a Lie algebra assembled from some other ones 
is, in a sense, ``assembled" from their cohomology. An exact 
formalization of this idea requires some special techniques of homological algebra and
will be discussed in our subsequent paper. 

{\bf Generalizations.} Assemblage techniques and some results of this paper can be directly 
extended in many directions. First of all, there should be mentioned graded and multiple 
Lie algebras (see \cite{Fil, HW, MVV, VV1, VV2}). An interesting point here is that natural 
compatibility of hereditary structures associated with an n-ary Lie algebra links the 
disassembling problems for Lie algebra of different multiplicities together. More generally,
compatibility problems for any kind of Poisson structures, say, algebroids and their n-ary
analogues (see \cite{Mac}), are tightly intertwined.

{\bf Physical applications.} The concept of two compatible Poisson structures was
introduced by F.Magri at 1977 (see \cite{Mag}) and since that was studied and widely 
exploited by numerous authors in the context integrable systems. We do not discuss 
here these well-known aspects. The fundamental question arising in connection with  
\emph{compound nature} of Lie algebras is :
\begin{quote}
{\it Let $S$ be a physical system and $\gG$ a Lie algebra of its infinitesimal symmetries. 
What one can say about intrinsic structure of $S$,  if a disassembling $\gG$  is known?}
\end{quote}
This question is, of course, too general to allow an universal answer. It should be duly 
specified  each time according to the nature of the system in question. 

Interpretation of compounds of  the symmetry algebra as \emph{symmetry waves} is 
suggested itself. Since these waves are nonlinear, it leads to the conclusion that 
characterizing them quantities need not be additive. Probably, this kind
of considerations could be useful in the theory of quarks.

 \smallskip
{\bf Acknowledgements.} The author is very grateful to M.M. and D.A. Vinogradov for
their help in preparing figures for this paper.

\end{document}